\documentclass{amsart}
\usepackage{geometry}
\usepackage{a4wide}
\usepackage{graphicx}
\usepackage{float}
\usepackage{caption}
\usepackage{subcaption}
\graphicspath{ {images/} }
\usepackage{amsmath,mathrsfs,amssymb,amsthm,amscd}
\usepackage{amsfonts}
\usepackage{mathtools}
\usepackage[]{fontenc}
\usepackage[all]{xy}
\usepackage{hyperref}
\usepackage{enumerate}
\usepackage{tikz}
\usepackage{tikz-cd}
\usepackage{tikz-3dplot}
\usepackage{mathtools}
\usepackage{afterpage}
\usepackage{hyperref}
\usepackage{tensor}
\usepackage{xr}
\usepackage{enumitem}
\usepackage{xcolor}
\usepackage{orcidlink}
\usepackage{mathrsfs}
\usepackage{bbm}
\usepackage{orcidlink}
\usetikzlibrary{patterns}
\usetikzlibrary{3d}
\usetikzlibrary{arrows,automata,backgrounds}
\tikzstyle pt=[fill,black]
\tikzstyle ptbl=[fill,black]

\setlist[enumerate,1]{label=(\roman*)}
\setlist[enumerate,2]{label=(\alph*)}

\theoremstyle{plain}
\newtheorem{thm-int}{Theorem}
\newtheorem{thm}{Theorem}[subsection]
\newtheorem{prop}[thm]{Proposition}
\newtheorem{cor}[thm]{Corollary}
\newtheorem{lem}[thm]{Lemma}

\theoremstyle{definition}
\newtheorem{dfn}[thm]{Definition}
\theoremstyle{remark}
\newtheorem{ex}[thm]{Example}
\newtheorem{rem}[thm]{Remark}
\newtheorem{cons}[thm]{Construction}

\begin{document}
	
	\title{Heights on toric varieties for singular metrics: Local theory}
	\author[G. Peralta]{Gari Y. Peralta Alvarez \orcidlink{0000-0002-1362-8126}}
	\thanks{The author acknowledges support from the Deutsche Forschungsgemeinschaft (DFG, German Research Foundation) under Germany’s Excellence Strategy – The Berlin Mathematics Research Center MATH+ (EXC-2046/1, project ID: 390685689). The author was also supported by the collaborative research center SFB 1085 Higher Invariants - Interactions between Arithmetic Geometry and Global Analysis, funded by the DFG}
	\email{gari.peralta@mathematik.uni-regensburg.de}
	\date{January 20, 2026}
	\begin{abstract}
		We show that the (toric) local height of a toric variety with respect to a semipositive torus-invariant singular metric is given by the integral of a concave function over a compact convex set. This generalizes a result of Burgos, Philippon, and Sombra for the case of continuous metrics and answers a question raised by Burgos, Kramer, and Kühn in 2016.
	\end{abstract}
	
	\maketitle
	\setcounter{tocdepth}{1}
	\tableofcontents
	
	\section{Introduction}
	The height of a projective variety $X$ over $\mathbb{Q}$ is an invariant which measures its arithmetic complexity, in an analogous way as the (intersection-theoretic) degree measures its algebraic complexity. This notion generalizes the height of points considered by Northcott, Weil, and others in Diophantine geometry and is a central tool in Arakelov theory. Heights are useful for establishing finiteness statements in arithmetic geometry; the most famous result is Faltings's Theorem \cite{Fal83}, previously known as Mordell's conjecture. It states that a curve $C$ over $\mathbb{Q}$ of genus at least two has a finite number of rational points.
	
	A seminal idea of Arakelov \cite{Ar} is that, by considering an integral model $\mathcal{X}$ of $X$ and adjoining the associated complex-analytic variety $X(\mathbb{C})$, the resulting object behaves similarly to a proper flat variety over a projective curve. Then, for the case of $X$ being a curve, he showed that the height of point in $X$ admits an intersection-theoretic interpretation on the ``arithmetic surface'' $\mathcal{X} \sqcup X(\mathbb{C})$. Gillet and Soulé later generalized this to arbitrary dimension $d = \textup{dim}(X)$ \cite{GS}, \cite{BGS}; we briefly recall this definition. First, they introduce the arithmetic Chow ring $\widehat{\textup{CH}}^{\bullet} (\mathcal{X})_{\mathbb{Q}}$, which comes equipped with a trace map  $\int \colon \widehat{\textup{CH}}^{d+1} (\mathcal{X})_{\mathbb{Q}} \rightarrow \mathbb{R}$. Then, for each hermitian line bundle $\overline{\mathcal{L}} = (\mathcal{L}, \| \cdot \|)$, where $\mathcal{L}$ is a line bundle on $\mathcal{X}$ and $\| \cdot \|$ is a smooth hermitian metric on the line bundle $L_{\mathbb{C}} = \mathcal{L} \otimes \mathbb{C}$ over $X(\mathbb{C})$, they define its first arithmetic Chern class $\widehat{c}_{1}(\overline{\mathcal{L}}) \in \widehat{\textup{CH}}^{1} (\mathcal{X})_{\mathbb{Q}}$. Finally, the height of $\mathcal{X}$ with respect to $\overline{\mathcal{L}}$ is defined as the real number
	\begin{displaymath}
		\textup{h}(\mathcal{X}; \overline{\mathcal{L}}) = \int  \widehat{c}_{1}(\overline{\mathcal{L}})^{d+1}.
	\end{displaymath}
	
	An alternative approach, considered by Zhang \cite{Z95} and Gubler \cite{Gub98}, is to define heights \textit{adelically}, in the sense that a global arithmetic object is completely described by assembling its local pieces. For instance, a hermitian line bundle $\overline{\mathcal{L}}$ on $\mathcal{X}$ induces a metrized line bundle $\overline{L}$ on $X$. That is, a pair consisting of a line bundle $L$ on the generic fibre $X$ of $\mathcal{X}$ and a family of continuous $v$-adic metrics $\lbrace \| \cdot \|_v \rbrace_{v \in \mathfrak{M}(\mathbb{Q})}$ on the line bundles $L_v = L \otimes \mathbb{Q}_v$ over $X_v = X \times_{\mathbb{Q}} \textup{Spec}(\mathbb{Q}_v)$, where $v$ runs over all places $\mathfrak{M}(\mathbb{Q})$ of $\mathbb{Q}$. Then, one defines the local height $\textup{h}(X_v; \overline{L}_v)$ (\ref{local-h-def}), and the global height of Gillet and Soulé is recovered by taking the sum over all places. Concretely, we have the equation
	\begin{displaymath}
		\textup{h}(X; \overline{L}) = \sum_{v \in \mathfrak{M}(\mathbb{Q})} \textup{h}(X_v; \overline{L}_v)  = 	\textup{h}(\mathcal{X}; \overline{\mathcal{L}}).
	\end{displaymath}
	The heights obtained by both approaches generalize to arbitrary cycles on varieties over a global field.
	
	While theoretically satisfactory, the machinery described above has two shortcomings: the technicalities involved in these definitions complicate the computation of examples. Additionally, these definitions do not apply to singular metrics, whose importance is discussed in~Subsection~\ref{int-sm}. The first shortcoming was addressed by Burgos, Philippon, and Sombra, whose work lays out a comprehensive convex-analytic description of Arakelov theory for proper toric varieties~\cite{BPS}, which we briefly recall in~Subsection~\ref{int-tv}. The second shortcoming was studied by Burgos and Kramer in \cite{BK24}, where they define heights for a large class of singular metrics. This article is the first of two parts, where we extend the results of \cite{BPS} to compute heights of toric varieties for the class of singular metrics considered in \cite{BK24}. To do this, we developed a formalism inspired by the theory of adelic line bundles of Yuan and Zhang \cite{Y-Z}. In this article, we present the theory over a local field (\ref{notation}), in both the archimedean and non-archimedean cases. The global theory over a number field will be treated in the second part of the series.
	
	\subsection{The local Arakelov theory of proper toric varieties}\label{int-tv}
	Let $K$ be a local field and $d$ be a positive integer. We denote by $N$ the group of cocharacters of the $d$-dimensional split torus $\mathbb{G}_{m}^{d}$ over $K$. Then, for a complete fan  $\Sigma$ in $N_{\mathbb{R}}$, there are functorial assignments
	\begin{center}
		\begin{tikzcd}
			\Sigma \arrow[r, mapsto] & X_{\Sigma} \arrow[r, mapsto, "{ \textup{an} }"] &  X_{\Sigma}^{\textup{an}}  \arrow[r, mapsto, "{ \textup{Trop}}"] & N_{\Sigma},
		\end{tikzcd}
	\end{center}
	where $X_{\Sigma}/K$ is the toric variety associated to $\Sigma$, $X_{\Sigma}^{\textup{an}}$ is its Berkovich analytification, and $N_{\Sigma}$ is the tropical toric variety associated to $\Sigma$ (see~\ref{tropical-toric-polytope}). The guiding principle is that torus invariant objects on $X_{\Sigma}$ or $X_{\Sigma}^{\textup{an}}$ are in correspondence with an object on $N_{\Sigma}$. A \textit{toric arithmetic divisor} $\overline{D}$ on $X_{\Sigma}$ is a pair $(D,g_{D})$ consisting of a Cartier divisor $D$ on $X_{\Sigma}$ and a Green's function on $X_{\Sigma}^{\textup{an}}$ of continuous type, both being torus invariant. Any such function $g_{D}$ induces a continuous metric on the line bundle $\mathcal{O}_{X_{\Sigma}}(D)$. These divisors form a group, denoted $\overline{\textup{Div}}_{\mathbb{T}}(X_{\Sigma})$. Then, Burgos, Philippon, and Sombra establish a correspondence
	\begin{center}
		\begin{tikzcd}
			\lbrace  D \textup{ nef, } g_{D} \textup{ of continuous and psh} \rbrace \arrow[r , leftrightarrow] & \lbrace  \Delta \textup{ convex polytope, } \vartheta \colon \Delta \rightarrow \mathbb{R} \textup{ continuous concave} \rbrace
		\end{tikzcd}
	\end{center}
	This is obtained by assigning to a nef toric divisor $D$ its corresponding convex polytope $\Delta_D$, and to a Green's function $g_D$ of continuous and psh type the Legendre-Fenchel transform $\vartheta_{\overline{D}}$ of its tropicalization. Moreover, they give the following integral formula for the local toric height of $X_{\Sigma}$ with respect to $\overline{D}$
	\begin{displaymath}
		\textup{h}^{\textup{tor}}(X_{\Sigma}; \overline{D}) = (d+1)!  \int_{\Delta_{D}} \vartheta_{\overline{D}}.
	\end{displaymath}
	The local toric height is defined as a difference of local heights (\ref{local-toric-h-proj-def}). However, in the global setting, one recovers the usual global height by summing the local toric heights over all places. In this paper, we generalize these results to the following kind of singular metrics.
	
	\subsection{Singular metrics}\label{int-sm} For each natural number $g \geq 1$, denote by $\mathcal{A}_{g}$ the moduli stack of principally polarized abelian schemes of dimension $g$ over $\mathbb{Z}$. Then, let $\overline{\omega} = (\omega, \| \cdot \|_{\textup{inv}})$ be the hermitian vector bundle on $\mathcal{A}_{g}$ consisting of the Hodge bundle $\omega$ with its natural invariant metric $ \| \cdot \|_{\textup{inv}}$. The metric $ \| \cdot \|_{\textup{inv}}$ is logarithmically singular as it approaches the boundary of $\mathcal{A}_{g}$. In particular, the arithmetic intersection theory of Gillet and Soul\'{e} is no longer valid in this situation. 
	
	Building on the results of Kühn~\cite{K01} for the case of $g=1$ (also independently obtained by Bost in~\cite{B98}), Burgos, Kramer and Kühn generalized the theory of Gillet and Soul\'{e}, taking into account hermitian vector bundles with logarithmically singular metrics (\cite{BKK5},~\cite{BKK7}). In particular, they were able to define the height $\textup{h}(\mathcal{A}_{g}; \, \overline{\omega})$. Their constructions rely on the fact that logarithmic singularities of metrics are mild enough so that Chern-Weil theory remains valid. That is, the geometric degree $\textup{deg}(\omega_{\mathbb{C}})$ of the line bundle $\omega_{\mathbb{C}}$ is given by the formula
	\begin{displaymath}
		\textup{deg}(\omega_{\mathbb{C}})= d! \int_{\mathcal{A}_{g}(\mathbb{C})} c_1(\omega_{\mathbb{C}},  \| \cdot \|_{\textup{inv}})^{d},
	\end{displaymath}
	where $d = g(g+1)/2$ is the (complex) dimension of $\mathcal{A}_{g}(\mathbb{C})$, and $c_1(\omega_{\mathbb{C}},  \| \cdot \|_{\textup{inv}})$ is the first Chern class of $(\omega_{\mathbb{C}},  \| \cdot \|_{\textup{inv}})$. Note that the volume of $\mathcal{A}_{g}(\mathbb{C})$ appearing as a factor on the right-hand side was made explicit by Siegel in \cite{Sie36}. The height $\textup{h}(\mathcal{A}_{g}; \, \overline{\omega})$ has been computed in \cite{B98} and \cite{K01} for $g=1$, in \cite{JvP22} for $g=2$, and in \cite{BBK07} for some Hilbert modular surfaces.
	
	For each $g\geq 1$, there is a universal abelian scheme $\pi \colon \mathcal{B}_g \rightarrow \mathcal{A}_g$. A natural question is whether Burgos-Kramer-Kühn's extension can be used to compute the height of $\mathcal{B}_g$ with respect to the hermitian line bundle $\overline{\mathcal{J}} \coloneqq \pi^{\ast} \overline{\omega} \otimes \overline{\mathcal{L}}$, where $ \overline{\mathcal{L}} = (\mathcal{L}, \| \cdot \|_{\textup{inv}})$ is the hermitian line bundle corresponding to twice the principal polarization on $\mathcal{B}_g$ together with its natural invariant metric. In their subsequent article~\cite{BKK16}, the authors investigate the case of $g=1$. They show that the natural invariant metric on the line bundle $L_{\mathbb{C}} = \mathcal{L} \otimes \mathbb{C}$ acquires singularities more severe than logarithmic. In particular, Chern-Weil theory no longer holds. Nevertheless, they compute a degree by replacing divisors with the so-called \textit{b-divisors}. This degree is obtained as a limit of geometric intersection numbers, computed over an infinite chain of successive blow-ups of $\mathcal{B}_g (\mathbb{C})$ along the boundary where the metric becomes singular.
	
	Finally, we mention the connection with the toric setting. At the end of the article~\cite{BKK16}, the authors give the following example of a toric metric with the same kind of singularities as the invariant metric on $L$. Let $\mathbb{P}^{2}_{\mathbb{C}}$ be the projective plane over $\mathbb{C}$, with homogeneous coordinates $(x_0; x_1; x_2)$. Then, we consider the singular metric $\| \cdot \|_{\textup{sing}}$ on the line bundle $\mathcal{O}_{\mathbb{P}^{2}_{\mathbb{C}}}(1)$ induced by the Green's function
	\begin{displaymath}
		- \log \| x_0 \|_{\textup{sing}} \coloneqq \begin{cases} \frac{\log |x_1/x_0|\log |x_2/x_0|}{\log |x_1/x_0| + \log |x_2/x_0|} & |x_0| \geq \max \lbrace |x_1|,|x_2| \rbrace \\
			\max \lbrace \log |x_1/x_0|, \log |x_2/x_0| \rbrace & \textup{otherwise.} \end{cases}
	\end{displaymath}
	This metric is singular, $\mathbb{S}$\textit{-invariant} and \textit{plurisubharmonic}. Moreover, they conjecture that this metric corresponds to a nef toric b-divisor, and in turn, it should correspond to the compact convex set
	\begin{displaymath}
		\Delta_{\textup{sing}} \coloneqq \lbrace (x,y) \in M_{\mathbb{R}} \, \vert \, x,y \geq 0, \, x+y \leq 1, \, \sqrt{x} + \sqrt{y} \geq 1 \rbrace.
	\end{displaymath}
	This toric conjecture was shown to be true by Botero in~\cite{Bot19}, where she developed a theory of toric b-divisors and defined the corresponding geometric intersection numbers. This marks the starting point of this article, which is motivated by two main questions:
	\begin{enumerate}
		\item Is there a correspondence between the set of $\mathbb{S}$-invariant Green's functions of continuous and psh type for the nef toric b-divisor $D$ and the set of concave functions $\vartheta$ on the compact convex set $\Delta_D$?
		\item Does the integral formula for the local toric height hold in this situation?
	\end{enumerate}
	In this article, we answer both questions affirmatively.
	
	\subsection{Main results}
	While the theory of b-divisors is nicely suited for geometric problems, such as the computation of intersection numbers previously mentioned, Yuan and Zhang's theory of \textit{adelic divisors} on quasi-projective varieties~\cite{Y-Z} is better suited for arithmetic problems. We propose toric analogs of their constructions and develop the corresponding convex-analytic descriptions. Subsequently, we use these descriptions to extend the toric height of Burgos, Philippon, and Sombra to account for the singular metrics considered by Burgos, Kramer, and Kühn~(\cite{BKK16}).
	
	Let $X_{\Sigma}/K$ be a projective toric variety and $\overline{L}=(L,\| \cdot \|)$ be a toric metrized line bundle on $X_{\Sigma}$, whose metric $\| \cdot \|$ is singular. A crucial observation is that the singularities of a toric metric are contained in a torus-invariant Zariski closed proper subset. Then, the restriction of the metric $\| \cdot \|$ to the split torus $U = \mathbb{G}_{m}^{d} / K$ is a continuous. We emulate Yuan and Zhang's idea, that is, to consider the class of toric metrics on $U^{\textup{an}}$ that can be approximated by restrictions of continuous metrics on the compactifications $X_{\Sigma^{\prime}}^{\textup{an}}$ of  $U^{\textup{an}}$. Thus, we define the group of \textit{toric compactified arithmetic divisors} of $U$ over $K$ as the completion $\overline{\textup{Div}}_{\mathbb{T}}(U/K)$ with respect to the boundary topology of the direct limit
	\begin{displaymath}
		\varinjlim_{\Sigma^{\prime}} \overline{\textup{Div}}_{\mathbb{T}}(X_{\Sigma^{\prime}}) \otimes {\mathbb{Q}},
	\end{displaymath}
	where $\Sigma^{\prime}$ runs over all projective fans in $N_{\mathbb{R}}$, ordered by refinement (see~\ref{cptf-dfn-local} for details). This is an example of an \textit{abstract divisoral space}, a notion recently introduced by Cai and Gubler in~\cite{CGub24} which formalizes the compactification procedure first described~in~\cite{Y-Z}.  An element $\overline{D}$ of this group is given by a Cauchy sequence $\lbrace (D_{n}, g_{D_n}) \rbrace_{n \in \mathbb{N}}$ of toric arithmetic divisors $\overline{D}_{n}$ on $X_{\Sigma_n}$. If the toric compactified arithmetic divisor $\overline{D}$ is nef, we show that the convex polytopes $\Delta_{D_n}$ converge in the Hausdorff distance to a compact convex set $\Delta_{D}$, and the functions $ \vartheta_{\overline{D}_{n}}$ converge to a closed concave function  $\vartheta_{\overline{D}}: \Delta_{D} \rightarrow \mathbb{R}_{- \infty }$ (which is finite on the relative interior of $\Delta_D$). The convergence of polytopes recovers Botero's result on toric b-divisors~\cite{Bot19}. Then, we have the following correspondence.
	\begin{thm-int}\label{1}
		The assignment $\overline{D} \mapsto  \vartheta_{\overline{D}}$ induces a bijective correspondence
		\begin{center}
			\begin{tikzcd}
				\lbrace \overline{D} \textup{ nef divisor on } U/K \rbrace \arrow[r , leftrightarrow] & \lbrace \vartheta: \Delta \rightarrow \mathbb{R}_{ - \infty } \textup{ closed concave on } \Delta \textup{ compact convex} \rbrace.
			\end{tikzcd}
		\end{center}
		The singularities of the associated metric are determined by the growth of $\vartheta_{\overline{D}}$ along the boundary of the compact convex set $\Delta_{D}$ associated to the toric compactified divisor $D$.
	\end{thm-int}
	This result is proven in~\ref{roof-functions-UK}. Note that it already answers the first question posed in the motivation for this article (in the local case). It is worth mentioning that in \cite{Son24}, Song has independently arrived at a related result using substantially different methods. Concretely, in~Theorem~4.5~of~\cite{Son24}, he characterizes a certain class of toric compactified divisors which are given by the pullback under the tropicalization map of a conical and continuous function on $N_{\mathbb{R}}$. In Theorem~\ref{arith-toric-div-UK}, we estabish an isomorphism 
	\begin{displaymath}
		\overline{\textup{Div}}_{\mathbb{T}}(U/K) \cong \mathcal{C}(N_{\mathbb{R}}) \oplus \mathcal{SL}(N_{\mathbb{R}}),
	\end{displaymath}
	where $\mathcal{C}(N_{\mathbb{R}})$ is the space of conical continuous functions in $N_{\mathbb{R}}$, and $\mathcal{SL}(N_{\mathbb{R}})$ is the space of sublinear functions on $N_{\mathbb{R}}$. Therefore, the toric divisors considered by Song form a subgroup of the group of toric compactified arithmetic divisors considered in this article.
	
	Our main result generalizes the local toric height of Burgos, Philippon, and Sombra to our formalism of toric compactified arithmetic divisors. In Definition~\ref{local-toric-h-dfn}, we introduced the \textit{local toric height} of $U$ with respect to $\overline{D}$. This definition is justified by a crucial example, discussed in~\ref{CEH-1}~and~\ref{motivation-local-toric-h}. Then, in Theorems~\ref{local-toric-h-UK-1}~and~\ref{local-toric-h-UK-2}, we use arguments from convex analysis to prove the following formula, which answers the second question. 
	\begin{thm-int}\label{2}
		The local toric height of $U/K$ with respect to $\overline{D} = \lbrace \overline{D}_{n} \rbrace_{n \in \mathbb{N}}$ satisfies
		\begin{align*}
			\textup{h}^{\textup{tor}}(U; \overline{D}) = (d+1)!  \int_{\Delta_{D}} \vartheta_{\overline{D}} = (d+1)! \,  \lim_{n \in \mathbb{N}} \int_{\Delta_{D_n}} \vartheta_{\overline{D}_{n}}.
		\end{align*}
		The height is either finite or equal to $- \infty$; it is finite if and only if $\vartheta_{\overline{D}} \in L^1 (\Delta_{D})$.
	\end{thm-int}
	As mentioned in Remark~\ref{mixed-energy-2}, our generalized local toric height is a particular instance of the mixed relative energy of Burgos and Kramer, introduced~in~\cite{BK24}. Nevertheless, the integral formula above makes the local toric height suitable for concrete computations, such as the ones in Examples~\ref{sum-not-finite}~and~\ref{local-height-singular}.
	
	\subsection{Acknowledgments}
	We sincerely thank José I. Burgos Gil and Jürg Kramer for introducing the author to this subject and for the numerous discussions regarding the results presented here. We also thank Marco Flores, Walter Gubler, and Klaus Künnemann for their helpful observations. 
	
	\subsection{Notation and conventions}\label{notation}
	We introduce the notations and conventions that will appear throughout this manuscript. The natural numbers $\mathbb{N}$ include $0$. For each $a \in \mathbb{R}$ denote:
	\begin{itemize}
		\item The set $\mathbb{R}_{>a} \coloneqq \lbrace x \in \mathbb{R} \, \vert \, x > a \rbrace.$ When $a=0$, these are the positive reals.
		\item The set $\mathbb{R}_{\geq a} \coloneqq \lbrace x \in \mathbb{R} \, \vert \, x \geq a \rbrace.$ When $a=0$, these are the non-negative reals.
		\item The set $\mathbb{R}_{- \infty} \coloneqq \lbrace - \infty \rbrace \cup \mathbb{R}$ with the conventions $-\infty < x$ and $-\infty + x = -\infty$ for all $x \in \mathbb{R}$. We equip it with the topology that makes the map $\log : \mathbb{R}_{\geq 0} \rightarrow \mathbb{R}_{- \infty}$ a homeomorphism. We define $\mathbb{R}_{\infty} \coloneqq \mathbb{R} \cup \lbrace \infty \rbrace$ and $\mathbb{R}_{\pm \infty} \coloneqq \lbrace \pm \infty \rbrace \cup \mathbb{R} $ analogously.
	\end{itemize}
	A ring $R$ is always assumed to be commutative, with multiplicative unit $1$. Denote by $R^{\times}$ the group of invertible elements. If $A$ is a finitely generated Abelian group, we write $A_{R} \coloneqq A \otimes_{\mathbb{Z}} R$, the extension of scalars to $R$.
	
	By a \textit{lattice} $N$, we mean a free Abelian group of finite rank $\textup{rk}(N)=d$. Denote by $ M \coloneqq N^{\vee} = \textup{Hom}(N,\mathbb{Z})$ its dual lattice. The natural pairing $\langle \cdot , \cdot \rangle \colon M \times N \rightarrow \mathbb{Z}$ given by evaluation
	\begin{displaymath}
		\langle m , n \rangle \coloneqq m(n).
	\end{displaymath}
	To ease notation, we write $\langle \cdot, \cdot \rangle$ for the induced pairing $\langle \cdot, \cdot \rangle_{R}$ obtained by extending scalars to $R$. When $R = \mathbb{R}$, we equip $N_{\mathbb{R}}$ with the Euclidean topology induced by a norm $\| \cdot \|$. Let $x \in N_{\mathbb{R}}$ and $r \geq 0$, we denote by $\textup{B}(x,r)$ and $\overline{\textup{B}}(x,r)$ the open and closed balls of radius $r$ centred at $x$ respectively. Given a closed subset $C \subset N_{\mathbb{R}}$, the distance from $x$ to $C$ is denoted by $\textup{dist}(x, C) \in \mathbb{R}_{\geq 0}$. We abuse notation slightly by using the same symbols for the corresponding objects on $M_{\mathbb{R}}$, which we equip with the dual norm. Finally, let $\textup{Vol}_M$ be the Haar measure on $M_{\mathbb{R}}$, normalized so that $M$ has covolume $1$.
	
	If $X$ is a topological space, we denote by $C^0 (X)$ the space of continuous real-valued functions and by $C^{0}_{\textup{bd}} (X) \subset C^0 (X)$ the subspace of bounded functions. The space $C^{0}_{\textup{bd}} (X)$ is a Banach space with the uniform norm $\| \cdot \|_{\infty}$. If $X$ is compact, then $C^{0}_{\textup{bd}} (X) = C^0 (X)$. If $X$ admits a differentiable structure, denote by $C^k (X)$ the space of $k$-continuously differentiable real-valued functions on $X$, with $k=\infty$ for smooth functions.
	
	By a local field $(K,| \cdot |)$, we mean either $(\mathbb{R}, | \cdot |_{\infty})$, $(\mathbb{C}, |\cdot |_{\infty})$ with the archimedean absolute value $|\cdot |_{\infty}$, or a field $K$ complete with respect to a discrete absolute value. In the non-archimedean case, we denote by $v \coloneqq - \log | \cdot |$ the associated discrete valuation, $\mathcal{O}_K$ its valuation ring, $\mathfrak{m}$ the unique prime ideal of $\mathcal{O}_K$. We let $k \coloneqq \mathcal{O}_K / \mathfrak{m}$ be its residue field.
	
	All of our schemes are assumed to be Noetherian. We denote by $\mathcal{O}_{\mathcal{X}}$ and $K(\mathcal{X})$ the structure sheaf and function field of the (integral) scheme $\mathcal{X}$, respectively.  By a $d$-dimensional \textit{variety} $\mathcal{X}$ over a base scheme $\mathcal{S}$, denoted $\mathcal{X}/\mathcal{S}$, we mean an integral, separated, normal scheme $\mathcal{X} \rightarrow \mathcal{S}$, which is flat and of finite type over $\mathcal{S}$ and has relative dimension $d = \textup{dim}_{\mathcal{S}} (\mathcal{X})$. We say that the variety $\mathcal{X}/\mathcal{S}$ is projective (resp. quasi-projective, smooth) if the structural morphism $\mathcal{X} \rightarrow \mathcal{S}$ is projective (resp. quasi-projective, smooth). A morphism of varieties over $\mathcal{S}$ (or $\mathcal{S}$-morphism) is a morphism $\mathcal{X}_1 \rightarrow \mathcal{X}_2$ compatible with the respective structural morphisms.  We reserve plain Roman letters for schemes over a field, e.g., $X \rightarrow \textup{Spec}(K)$.
	
	Let $\mathcal{X}/\mathcal{S}$ be a variety and $\textup{Div}(\mathcal{X})$ be the group of Cartier divisors on $\mathcal{X}$. We have a morphism $\textup{div}\colon K(\mathcal{X})^{\times} \rightarrow \textup{Div}(\mathcal{X})$ given by the assignment $f \mapsto \textup{div}(f)$. The image of the morphism $\textup{div}$ is the subgroup $\textup{Pr}(\mathcal{X})$ of \textit{principal} Cartier divisors on $\mathcal{X}$. The \textit{Cartier class group} of $\mathcal{X}$ is the quotient $\textup{Cl}(\mathcal{X})\coloneqq \textup{Div}(\mathcal{X}) / \textup{Pr}(\mathcal{X})$. Denote by $\textup{Pic}(\mathcal{X})$ the Picard group of $\mathcal{X}$, i.e., the group of isomorphism classes of line bundles over $\mathcal{X}$. Let $\mathcal{L}$ be a line bundle over $\mathcal{X}$ and $s \neq 0$ be a rational section of $\mathcal{L}$. The variety $\mathcal{X}$ is integral, therefore the assignment $s \mapsto \textup{div}(s)$ induces an isomorphism $\textup{Pic}(\mathcal{X}) \cong \textup{Cl} (\mathcal{X})$. Since $\mathbb{Q}$ is flat over $\mathbb{Z}$, this isomorphism is preserved after scalar extension. Given $\mathcal{D} \in \textup{Div}(\mathcal{X})$, we denote by $\mathcal{O}_{\mathcal{X}}(\mathcal{D})$ the associated line bundle.
	
	\section{Review of Yuan-Zhang's theory of compactified divisors}
	In this section, we summarize some concepts and results on the Arakelov theory of quasi-projective varieties, introduced by Yuan and Zhang in \cite{Y-Z}. Our approach differs slightly; we avoid the use of $(\mathbb{Q},\mathbb{Z})$-divisors and base-valued schemes. Since we only consider the geometric and local cases, we opt for the term \textit{``compactified divisor''} instead of \textit{``adelic divisor''}. In the literature, the latter is only used in the global arithmetic setting, for example, in~\cite{BPS}~and~\cite{BMPS}.
	
	\subsection{The geometric case} Throughout this subsection, we let the base scheme $\mathcal{S}$ be the spectrum of a field  $K$ or a Dedekind domain $\mathcal{O}_K$, and fix a quasi-projective variety $\mathcal{U} / \mathcal{S}$. When we want to specify $\mathcal{S}$, we often abuse notation by writing the ring instead of its spectrum. This should not cause any confusion.
	\begin{dfn}
		By a \textit{projective model of }$\mathcal{U}$\textit{ over }$\mathcal{S}$, we mean a projective variety $\mathcal{X} / \mathcal{S}$ together with an open embedding $\pi \colon \mathcal{U} \rightarrow \mathcal{X}$ over $\mathcal{S}$. A \textit{morphism of projective models of} $\mathcal{U}$ \textit{over} $\mathcal{S}$ is a proper $\mathcal{S}$-morphism $f \colon \mathcal{X}_1 \rightarrow \mathcal{X}_2$ between projective models of $\mathcal{U}$ satisfying $f \circ \pi_1 = \pi_2$. In this case, we say that $\mathcal{X}_1$ \textit{dominates} $\mathcal{X}_2$. We denote by $\textup{PM}(\mathcal{U}/\mathcal{S})$ the category of projective models of $\mathcal{U}$ over $\mathcal{S}$.
	\end{dfn}
	\begin{rem}
		If $\mathcal{S}$ is the spectrum of a Dedekind domain $\mathcal{O}_K$ and $U/K$ is a variety, an \textit{arithmetic model of} $U$ \textit{over} $\mathcal{O}_K$ is a projective variety $\mathcal{X}/\mathcal{O}_K$ together with an open embedding of $U$ into the generic fibre $X$ of $\mathcal{X}$. Assume further that $\mathcal{O}_K$ is a local ring (and thus a DVR). In this setting, the notions of arithmetic and projective model of $U$ over $\mathcal{O}_K$ coincide. 
	\end{rem}
	\begin{rem}\label{cofiltered}
		The category $\textup{PM}(\mathcal{U}/\mathcal{S})$ is cofiltered. Indeed, given projective models $\pi_i  \colon \mathcal{U} \rightarrow \mathcal{X}_i$, $i~\in~\lbrace1,2\rbrace$, denote by $\Delta_{\mathcal{U}}$ the diagonal morphism. The normalization $\mathcal{X}$ of the scheme-theoretic image of $\mathcal{U}$ under the map $(\pi_1 \times \pi_2) \circ \Delta_{\mathcal{U}}$ induces a model dominating both of them. For details, see~p.61~in~Section~2.7~of~\cite{Y-Z}.
	\end{rem}
	Given a morphism of varieties $f \colon \mathcal{X} \rightarrow \mathcal{Y}$ that is either flat or dominant. there is an induced pullback morphism $f^{\ast} \colon \textup{Div}(\mathcal{Y}) \rightarrow \textup{Div}(\mathcal{X})$ (See \href{https://stacks.math.columbia.edu/tag/01WQ}{Section~01WQ}~of~\cite{Stacks}). In the case of a projective model $\pi \colon \mathcal{U} \rightarrow \mathcal{X}$, the pullback morphism is the restriction map to $\mathcal{U}$. We often denote the restriction $\pi^{\ast} \mathcal{D}$ of a divisor $\mathcal{D}$ as $\mathcal{D}|_{\mathcal{U}}$. This map gives a notion of models for divisors.
	\begin{dfn}
		Let $\mathcal{X}$ be a projective model of $\mathcal{U}$ over $\mathcal{S}$, and consider a divisor $\mathcal{D} \in \textup{Div}(\mathcal{U})_{\mathbb{Q}}$. A \textit{model of} $\mathcal{D}$ \textit{on} $\mathcal{X}$ is a $\mathbb{Q}$-divisor $\mathcal{D}^{\prime}$ on $\mathcal{X}$ such that $\mathcal{D}^{\prime}|_{\mathcal{U}} = \mathcal{D}$. The \textit{group of model divisors of }$\mathcal{U}$ \textit{over }$\mathcal{S}$ is defined as the direct limit
		\begin{displaymath}
			\textup{Div}(\mathcal{U}/\mathcal{S})_{\textup{mod}}\coloneqq \varinjlim_{\mathcal{X} \in \textup{PM}(\mathcal{U}/\mathcal{S})} \textup{Div}(\mathcal{X})_{\mathbb{Q}}.
		\end{displaymath}
	\end{dfn}
	In other words, a \textit{model divisor} $\mathcal{D} \in \textup{Div}(\mathcal{U}/\mathcal{S})_{\textup{mod}}$ is an equivalence class of models of a given $\mathbb{Q}$-divisor $\mathcal{E}$ on $\mathcal{U}$. Here,we declare the models $\mathcal{D}_1$ and $\mathcal{D}_2$ of $\mathcal{E}$ equivalent, denoted by $\mathcal{D}_1 \sim \mathcal{D}_2$, if their pullbacks to a common projective model $\mathcal{X}$ coincide. We abuse notation by denoting a model divisor $\mathcal{D}$ and a representative of its associated equivalence class by the same symbol. This should not cause any confusion. Moreover, if $\mathcal{D}$ has a representative on a model $\mathcal{X}$, we often say that $\mathcal{D}$ \textit{is defined on }$\mathcal{X}$.
	
	The notion of effectivity extends to the level of model divisors. Concretely, a model divisor $\mathcal{D}$ is \textit{effective} if it is represented by an effective Cartier $\mathbb{Q}$-divisor. By Lemmas~2.3.5~and~2.3.5~in~\cite{Y-Z}, this is well-defined. As usual, effectivity induces a partial order: Given model divisors $\mathcal{D}_1$ and $\mathcal{D}_2$, we say that $\mathcal{D}_1 \geq \mathcal{D}_2$ if the model divisor $\mathcal{D}_1 - \mathcal{D}_2$ is effective. This partial order can be used to define a topology on the group $\textup{Div}(\mathcal{U}/\mathcal{S})_{\textup{mod}}$.
	\begin{dfn}
		By a \textit{boundary divisor} of $\mathcal{U}/\mathcal{S}$, we mean a pair $(\mathcal{X}, \mathcal{B})$ consisting of a projective model $\pi \colon \mathcal{U} \rightarrow \mathcal{X}$ and an effective divisor $\mathcal{B} \in \textup{Div}(\mathcal{X})$ satisfying the support condition $|\mathcal{B}| = \mathcal{X} \setminus \mathcal{U}$. The \textit{boundary norm} $\| \cdot \|_{\mathcal{B}}$ on $\textup{Div} (\mathcal{U}/\mathcal{S})_{\textup{mod}}$ is given by
		\begin{displaymath}
			\| \mathcal{D} \|_{\mathcal{B}} \coloneqq \inf \lbrace \varepsilon \in \mathbb{Q}_{> 0} \, \vert \, - \varepsilon \cdot \mathcal{B} \leq \mathcal{D} \leq \varepsilon \cdot \mathcal{B} \rbrace,
		\end{displaymath}
		with the convention that $\inf \emptyset = \infty$. The \textit{boundary topology} on $\textup{Div} (\mathcal{U}/\mathcal{S})_{\textup{mod}}$ is the topology induced by $\| \cdot \|_{\mathcal{B}}$.
	\end{dfn}
	\begin{rem}\label{extended-norm}
		The boundary norm $\| \cdot \|_{\mathcal{B}}$ is not a norm in the usual sense; it attains the value~$\infty$. Indeed, for a model divisor $\mathcal{D}$ we have $\| \mathcal{D} \|_{\mathcal{B}} = \infty$ if and only if $\mathcal{D}|_{\mathcal{U}} \neq 0$. Nevertheless, the boundary norm is a so-called \textit{extended norm}, satisfying the following properties:
		\begin{enumerate}
			\item $\| \mathcal{D} \|_{\mathcal{B}} =0$ if and only if $\mathcal{D}=0$.
			\item $\| a \cdot \mathcal{D} \|_{\mathcal{B}}  = |a| \cdot\| \mathcal{D} \|_{\mathcal{B}} $ for all $a \in \mathbb{Q}$.
			\item The triangle inequality holds: $ \| \mathcal{D}_1 + \mathcal{D}_2 \|_{\mathcal{B}}  \leq \| \mathcal{D}_1 \|_{\mathcal{B}}  + \| \mathcal{D}_2 \|_{\mathcal{B}} $.
		\end{enumerate}
		Moreover, if $(\mathcal{X}, \mathcal{B})$ and $(\mathcal{X}^{\prime}, \mathcal{B}^{\prime})$ are boundary divisors, the respective boundary norms are equivalent in the usual sense: There is a real number $r>1$ such that $r^{-1} \| \cdot \|_{\mathcal{B}^{\prime}} \leq \| \cdot \|_{\mathcal{B}} \leq r \| \cdot \|_{\mathcal{B}^{\prime}}$. Therefore, the boundary topology is independent of the choice of boundary divisor. For details, see Lemma~2.4.1~of~\cite{Y-Z}.
	\end{rem}
	\begin{dfn}
		The \textit{group of compactified divisors of }$\mathcal{U}$\textit{ over }$\mathcal{S}$ is the completion $\textup{Div} (\mathcal{U}/\mathcal{S})$ of the group $\textup{Div} (\mathcal{U}/\mathcal{S})_{\textup{mod}}$ with respect to the boundary topology. A \textit{compactified divisor of }$\mathcal{U}$ \textit{over} $\mathcal{S}$ is an element of $\textup{Div} (\mathcal{U}/\mathcal{S})$.
	\end{dfn}
	Concretely, a compactified divisor $\mathcal{D} \in \textup{Div} (\mathcal{U}/\mathcal{S})$ naturally identifies with an equivalence class of Cauchy sequences in $\textup{Div}(\mathcal{U}/\mathcal{S})_{\textup{mod}}$, where two Cauchy sequences $\lbrace \mathcal{D}_n \rbrace_{n \in \mathbb{N}}$ and $\lbrace \mathcal{E}_n \rbrace_{n \in \mathbb{N}}$ are equivalent if their difference converges to $0$. We abuse notation by writing $\mathcal{D} = \lbrace \mathcal{D}_n \rbrace_{n \in \mathbb{N}}$.
	\begin{cons}\label{decreasing}
		For many of our arguments, it will be convenient to represent a compactified divisor $\mathcal{D}$ by a decreasing Cauchy sequence $\lbrace \mathcal{D}_n \rbrace_{n \in \mathbb{N}}$. This is always possible: Start with any Cauchy sequence $\lbrace \mathcal{E}_k \rbrace_k$ representing $\mathcal{D}$. For each positive integer $n$, there is a $k_n$ such that for all $k , m \geq k_n$ we have $-1/n \cdot \mathcal{B} \leq \mathcal{E}_k - \mathcal{E}_m \leq 1/n \cdot \mathcal{B}$. Then, it suffices to consider the sequence $\lbrace \mathcal{D}_n \rbrace_{n \in \mathbb{N}}$ given by $\mathcal{D}_n \coloneqq \mathcal{E}_{k_n} + 1/n \cdot \mathcal{B}$ for all $n$ positive and $\mathcal{D}_0 = \mathcal{D}_1$.
	\end{cons}
	The notion of nefness can be extended to compactified divisors. Recall that a Cartier $\mathbb{Q}$-divisor $\mathcal{D}$ on a projective variety $\mathcal{X} / \mathcal{S}$ is \textit{(relatively) nef} if for every vertical curve $\mathcal{C}$ in $\mathcal{X}$, the intersection product $\mathcal{D} \cdot \mathcal{C}$ is non-negative. Note that for $\mathcal{S}=\textup{Spec}(K)$, this is the usual notion of nefness. Then, a model divisor $\mathcal{D}$ is \textit{(relatively) nef} if it has a nef representative. Since the morphisms of projective models are proper and birational, the projection formula shows that this is well-defined. We denote by $\textup{Div}^{\textup{nef}}(\mathcal{U}/\mathcal{S})_{\textup{mod}}$ the cone consisting of all nef model divisors of $\mathcal{U}/\mathcal{S}$. We then present the following definition.
	\begin{dfn}\label{rel-nef-dfn}
		The  \textit{cone of nef compactified divisors of }$\mathcal{U}$\textit{ over} $\mathcal{S}$ is the closure $\textup{Div}^{\textup{nef}}(\mathcal{U}/\mathcal{S})$ of the cone $\textup{Div}^{\textup{nef}}(\mathcal{U}/\mathcal{S})_{\textup{mod}}$ with respect to the boundary topology. The \textit{space of integrable compactified divisors} is defined as the difference of cones
		\begin{displaymath}
			\textup{Div}^{\textup{int}} (\mathcal{U}/\mathcal{S}) \coloneqq \textup{Div}^{\textup{nef}} (\mathcal{U}/\mathcal{S}) -\textup{Div}^{\textup{nef}}(\mathcal{U}/\mathcal{S}).
		\end{displaymath}
	\end{dfn}
	Let $K$ be a field and $U / K$ be a quasi-projective variety of dimension $d$. On each projective model $X$ of $U$ over $K$, there is a pairing given by intersection numbers of divisors. Then, the projection formula yields a well-defined intersection pairing $\textup{Div}(U/K)_{\textup{mod}}^{d} \rightarrow \mathbb{Q}$, where the intersection number $D_1 \cdot \ldots \cdot D_d$ is computed over a projective model on which all of the $D_i$ are defined. We state Proposition~4.1.1~of~\cite{Y-Z}, which extends continuously this pairing to the space of integrable compactified divisors.
	\begin{thm}\label{int-prod-K}
		Let $K$ be a field and $U / K$ be a quasi-projective variety of dimension $d$. The intersection pairing on model divisors extends continuously to a pairing $\textup{Div}^{\textup{nef}}(U/K)^{d} \rightarrow \mathbb{R}$ on the nef cone. It is given by the assignment
		\begin{displaymath}
			D_1 \cdot \ldots \cdot D_d \coloneqq \lim_{n \rightarrow \infty}  D_{1,n} \cdot \ldots \cdot D_{d,n}.
		\end{displaymath}
		The number $D_1 \cdot \ldots \cdot D_d$ is independent of the choice of nef sequences $D_i = \lbrace D_{i,n} \rbrace_{n \in \mathbb{N}}$. This extends by linearity to a symmetric multilinear pairing $\textup{Div}^{\textup{int}}(U/K)^{d} \rightarrow \mathbb{R}$.
	\end{thm}
	Note that all of these constructions can be translated to the language of line bundles via the correspondence induced by $\mathcal{D} \mapsto \mathcal{O}_{\mathcal{X}}(\mathcal{D})$. For details, see~Section~2.5~of~\cite{Y-Z}.
	
	\subsection{The local arithmetic case}\label{local-theory} Throughout this subsection, $K = (K, | \cdot |)$ is a local field, and the varieties $X/K$ and $U/K$ are projective and quasi-projective respectively.  The analytic structure of $K$ allows us to enhance the constructions in the previous section with analytic objects: Green's functions for divisors and metrics on line bundles. 
	
	A variety $Y/K$ has an associated \textit{Berkovich analytification} $Y^{\textup{an}}$. We sketch its construction and recall some of its properties. For an open affine subset $V = \textup{Spec}(A)$ of $Y$, the Berkovich analytification of $V$ over $K$ is the set $V^{\textup{an}}$ of multiplicative seminorms $y = | \cdot |_y$ on $A$ whose restriction to $K$ is bounded by $| \cdot |$. It is endowed with the coarsest topology such that for all $f \in A$, the function $|f ( \cdot )| \colon  V^{\textup{an}} \rightarrow \mathbb{R}_{\geq 0}$ given by $|f(y)| \coloneqq | f |_{y}$ is continuous, and the assignment $y \mapsto \textup{ker}(|\cdot |_y)$ induces a continuous map $V^{\textup{an}} \rightarrow V$. For each $y \in V^{\textup{an}}$, the seminorm $| \cdot |_y$ induces an absolute value on the field of fractions of the integral domain $A / \textup{ker}( | \cdot |_y )$. We denote by $H_y$ its completion, and write $f(y)$ for the image of $f \in A$ in $H_y$. On $V^{\textup{an}}$ there is a sheaf of \textit{analytic functions} $\mathcal{O}_{V^{\textup{an}}}$, where a function
	\begin{displaymath}
		g \colon W \longrightarrow \bigsqcup_{y \in W} H_y
	\end{displaymath}
	on an open subset $W$ of $V^{\textup{an}}$ is analytic if for each $y \in W$ we have $g(y) \in H_y$ and there exists an open neighbourhood $W_y$ of $y$ such that, for all $z \in W_y$ and $\varepsilon >0$, there are $a,b \in A$ with $b(z)\neq 0$ satisfying $\left\vert g(z) - a(z)/b(z) \right\vert < \varepsilon$. The locally ringed space $(Y^{\textup{an}},\mathcal{O}_{Y^{\textup{an}}})$ is obtained by gluing an open affine cover of $Y$. Then, we have the following properties:
	\begin{enumerate}
		\item The space $Y^{\textup{an}}$ is locally compact Hausdorff. The variety $Y$ is proper if and only if the space $Y^{\textup{an}}$ is compact.
		\item Given a morphism $\pi :Y \rightarrow Z$ over $K$, there is an induced map $\pi^{\textup{an}} \colon Y^{\textup{an}} \rightarrow Z^{\textup{an}}$. Thus, we have an analytification functor from the category of varieties over $K$ to the category of Berkovich $K$-analytic spaces. The morphism $\pi$ is flat (resp. unramified, smooth, \'{e}tale, separated, injective, surjective, open immersion, isomorphism) if and only if $\pi^{\textup{an}}$ has the same property. For any point $z \in Z^{\textup{an}}$, the fibre $Y^{\textup{an}}_{z} = (\pi^{\textup{an}})^{-1}(z)$ is canonically isomorphic to the analytification of the fibre $Y \times_{K} \textup{Spec}(H_z) \rightarrow \textup{Spec}(H_z)$.
		\item Let $K$ be non-archimedean, $Y / K$ be proper, $\pi \colon Y \rightarrow \mathcal{Y}$ be a model of $Y$ over $\mathcal{O}_{K}$ with special fibre $\mathcal{Y} \times_{\mathcal{O}_K} k$. Then, there is a map $\textup{red} \colon Y^{\textup{an}} \rightarrow \mathcal{Y} \times_{\mathcal{O}_K} k$ known as the \textit{reduction} which is surjective and anticontinuous. It is defined via the valuative criterion of properness, extending the morphism $\textup{Spec}(H_y) \rightarrow Y$ to a morphism $\textup{Spec}(\mathcal{O}_{H_y}) \rightarrow \mathcal{Y}$ and mapping $y$ to the image of the special fiber of the latter.
		\item  Assume that $K$ is Archimedean. Denote by $Y(\mathbb{C})$ the complex-analytic variety associated with $Y$. If $K = \mathbb{C}$, then $Y^{\textup{an}}$ is homeomorphic to $Y(\mathbb{C})$. If $K = \mathbb{R}$, then $Y^{\textup{an}}$ is homeomorphic to the quotient of $Y(\mathbb{C})$ by the action of complex conjugation $\zeta$. To have a uniform treatment of the archimedean case, if $K=\mathbb{R}$, we regard $Y^{\textup{an}}$ as the pair $(Y(\mathbb{C}), \zeta)$ and require every object on $Y^{\textup{an}}$ to be invariant under the action of $\zeta$.
	\end{enumerate}
	For details on Berkovich spaces and complex-analytic varieties, we refer~to~\cite{Ber}~and~\cite{Dem12} respectively. Now, we recall the definition of a Green's function for a divisor.
	\begin{dfn}\label{arith-div-dfn}
		Let $D \in \textup{Div}(U)_{\mathbb{Q}}$ and $P$ be a property of functions (e.g., continuous, smooth, locally bounded, etc.). A function $g \colon U^{\textup{an}} \rightarrow \mathbb{R}_{\pm \infty}$ is a \textit{Green's function for} $D$ \textit{of} $P$ \textit{type} if for every Zariski open $V \subset U$ and every local equation $f$ for $D$ on $V$, the function $g + \log |f|^2$ satisfies the property $P$ on $V^{\textup{an}}$. The pair $\overline{D}\coloneqq (D,g_D)$ is an \textit{arithmetic divisor of }$P$ \textit{type on} $U$. We denote by $\overline{\textup{Div}}(U)_{\mathbb{Q}}$ the group of arithmetic $\mathbb{Q}$-divisors of continuous type on $U$.
	\end{dfn}
	There is a notion of linear equivalence for arithmetic divisors. Let $f \in K(U)^{\times}$ be a rational function. Trivially, the function $g_{f} \coloneqq - \log |f|^2$ on $U^{\textup{an}}$ is a Green's function for $\textup{div}(f)$ of continuous type. This yields a morphism $\overline{\textup{div}} \colon K(U)^{\times} \rightarrow \overline{\textup{Div}}(U)$ given by $\overline{\textup{div}}(f)\coloneqq (\textup{div}(f), g_f )$. The image of $\overline{\textup{div}}$ is the group $\overline{\textup{Pr}}(U)$ of \textit{principal arithmetic divisors on} $U$, and the \textit{arithmetic class group} $\overline{\textup{Cl}}(U)$ is its cokernel.
	
	If $K$ is archimedean, the role of a ``well-behaved function'' is played by smooth functions. On the other hand, when $K$ is non-archimedean, this role is played by \textit{algebraic} functions. We sketch their construction: Consider a divisor $D \in \textup{Div}(X)_{\mathbb{Q}}$ on a projective variety $X/K$. Let $\pi \colon X \rightarrow \mathcal{X}$ be a projective model over $\mathcal{O}_K$ and $\mathcal{D}$ be a model of $D$ on $\mathcal{X}$. We choose any integer $m>0$ such that $m \cdot \mathcal{D}$ is an integral Cartier divisor. For $x \in X^{\textup{an}}$, choose an open affine neighbourhood $\mathcal{V} \subset \mathcal{X}$ of $\textup{red}(x)$. In particular, we get that $x \in V^{\textup{an}}$, where $V = \mathcal{V} \cap X$. Then, define $g_{\mathcal{D},D}(x) \coloneqq - \log | f (x)|^{2/m}$. Since any two local equations cutting out $m \cdot \mathcal{D}$ differ by a unit, this value does not depend on the choice of $f$. Therefore, we obtain a well-defined Green's function $g_{\mathcal{D},D}$ for $D$ of continuous type. A Green's function $g$ for $D$ is said to be of \textit{algebraic type} if it is of the form $g = g_{ \mathcal{D},D}$. If additionally $D=0$, then $g$ is said to be an \textit{algebraic} function on $X^{\textup{an}}$. In the literature, these are often referred to as \textit{model} functions. For instance, in~Definition~2.1~of~\cite{BFJ16}. Moreover, $\mathcal{D}$ is effective if and only if the Green's function $g_{\mathcal{D}, D}$ is non-negative (Lemma~3.3.3~of~\cite{Y-Z}). 
	
	The assignment $\mathcal{D} \mapsto (D, g_{\mathcal{D},D})$ induces a group morphism $\textup{an} \colon \textup{Div}(\mathcal{X})_{\mathbb{Q}} \rightarrow \overline{\textup{Div}}(X)_{\mathbb{Q}}$, known as the \textit{analytification map} on divisors. The analytification map is functorial in the following sense. Let $f \colon \mathcal{X} \rightarrow \mathcal{Y}$ be a morphism of projective models of $X$ over $\mathcal{O}_K$. By Proposition~1.3.6~of~\cite{BPS}, the pullback morphism $f^{\ast} \colon \textup{Div}(\mathcal{Y})_{\mathbb{Q}} \rightarrow \textup{Div}(\mathcal{X})_{\mathbb{Q}}$ commutes with the analytification map. That is, $\textup{an} \circ f^{\ast} = \textup{an}$. Taking direct limits, there is a (unique) induced analytification map on the group of model divisors $\textup{an} \colon \textup{Div}(X/\mathcal{O}_K)_{\textup{mod}} \rightarrow \overline{\textup{Div}}(X)_{\mathbb{Q}}$. The image of this map is a group consisting of all arithmetic divisors of algebraic type on $X$. This motivates the following definition.
	\begin{dfn}\label{model-type}
		Let $X/K$ be a projective variety and $D \in \textup{Div}(X)_{\mathbb{Q}}$. We say that a Green's function $g$ for $D$ is of \textit{model type} if $g$ is of smooth or algebraic type, depending on $K$. If additionally $D=0$, we say that $g$ is a model function on $X^{\textup{an}}$. The group of arithmetic divisors of model type on $X$ is denoted by $\overline{\textup{Div}}(X)_{\textup{mod}}$. The space of model functions on $X^{\textup{an}}$ is denoted by $C^{0}(X^{\textup{an}})_{\textup{mod}}$.
	\end{dfn}
	The following proposition shows that the arithmetic divisors of continuous type are uniform limits of arithmetic divisors of model type.
	\begin{prop}\label{model-dense}
		Let $(D,g)$ be an arithmetic divisor of continuous type on $X$. Then, there exists a sequence $\lbrace g_n \rbrace_{n \in \mathbb{N}}$ of Green's functions of model type for $D$ converging uniformly to $g$.
	\end{prop}
	\begin{proof}
		Let $g^{\prime}$ be any reference Green's function for $D$ of model type, and denote $f = g- g'$. Then, it suffices to find a sequence $\lbrace f_n \rbrace_{n \in \mathbb{N}}$ of model functions converging uniformly to $f$.  If $K$ is archimedean, this is the well-known fact that the space $C^{\infty}(X^{\textup{an}})$ of smooth functions is dense in the space $C^{0}(X^{\textup{an}})$ of continuous functions with respect to the uniform norm. If $K$ is non-archimedean, it is also well-known that the space of model functions is dense in the space of continuous functions. For instance, see Theorem~7.12~of~\cite{Gub98} or Corollary~2.3~of~\cite{BFJ16}.
	\end{proof}
	Now, we extend these concepts to the quasi-projective case. Let $K$ be non-archimedean and $\pi \colon U \rightarrow \mathcal{X}$ be a projective model over $\mathcal{O}_K$. The image of $\pi$ is contained in the generic fibre $X$ of $\mathcal{X}$, which is projective over $K$. The open embedding $\pi \colon U \rightarrow X$ induces a dense embedding $\pi^{\textup{an}} \colon U^{\textup{an}} \rightarrow X^{\textup{an}}$. Then, there is an induced pullback map $\pi^{\ast} \colon \overline{\textup{Div}}(X)_{\mathbb{Q}} \rightarrow \overline{\textup{Div}}(U)_{\mathbb{Q}}$ given by restriction. Note that every projective model of $X$ over $\mathcal{O}_K$ is also a model of $U$. Therefore, the universal property of direct limits gives a commutative diagram
	\begin{center}
		\begin{tikzcd}
			\textup{Div}(X/\mathcal{O}_K)_{\textup{mod}} \arrow[d] \arrow[r, "{\textup{an}}"] & \overline{\textup{Div}}(X)_{\textup{mod}} \arrow[r] \arrow[d] &\overline{\textup{Div}}(X)_{\mathbb{Q}} \arrow[d] \\
			\textup{Div}(U/\mathcal{O}_K)_{\textup{mod}} \arrow[r, "{\textup{an}}"] & \displaystyle \varinjlim_{Y \in \textup{PM}(U/K)} \overline{\textup{Div}}(Y)_{\textup{mod}} \arrow[r] & \overline{\textup{Div}}(U)_{\mathbb{Q}}.
		\end{tikzcd}
	\end{center}
	In the archimedean setting, the group $\textup{Div}(U/\mathcal{O}_K)_{\textup{mod}}$ is not defined. However, we still have the commutative diagram below. These diagrams motivate the next definition.
	\begin{center}
		\begin{tikzcd}
			\overline{\textup{Div}}(X)_{\textup{mod}} \arrow[r] \arrow[d] &\overline{\textup{Div}}(X)_{\mathbb{Q}} \arrow[d] \\
			\displaystyle \varinjlim_{Y \in \textup{PM}(U/K)} \overline{\textup{Div}}(Y)_{\textup{mod}} \arrow[r] & \overline{\textup{Div}}(U)_{\mathbb{Q}}
		\end{tikzcd}
	\end{center}
	\begin{dfn}\label{arith-model-div}
		The \textit{group of model arithmetic divisors of }$U$ \textit{over} $K$ is the direct limit
		\begin{displaymath}
			\overline{\textup{Div}}(U/K)_{\textup{mod}} \coloneqq \varinjlim_{X \in \textup{PM}(U/K)} \overline{\textup{Div}}(X)_{\textup{mod}}.
		\end{displaymath}
	\end{dfn}
	Recall that an arithmetic divisor $(D,g)$ is \textit{effective} (resp. \textit{strictly effective}) if $D$ is effective and the function $g$ is non-negative (resp. positive). As in the geometric case, effectivity induces a partial order on $\overline{\textup{Div}}(U/K)_{\textup{mod}}$, which we use to replicate the definitions of boundary divisors (which are strictly effective) and boundary topology in this setting. We then present the following definition.
	\begin{dfn}\label{cptf-dfn}
		The \textit{group of compactified arithmetic divisors of} $U$ \textit{over }$K$ is the completion $\overline{\textup{Div}}(U/K)$ of $\overline{\textup{Div}}(U/K)_{\textup{mod}}$ with respect to the boundary topology.
	\end{dfn}
	Shrinking the quasi-projective variety is functorial. This is Corollary~3.4.2 of~\cite{Y-Z}.
	\begin{lem}
		Let $U \rightarrow V$ be an open immersion of quasi-projective varieties over $K$. Then, the restriction map induces a continuous injective map $\overline{\textup{Div}}(V/K) \rightarrow \overline{\textup{Div}}(U/K)$.
	\end{lem}
	\begin{rem}\label{cptf-1}
		As in the geometric case, compactified arithmetic divisors are represented by Cauchy sequences of model divisors. By abuse of notation, we will write $\overline{D} = \lbrace (D_n ,g_n) \rbrace_{n \in \mathbb{N}}$. By definition of the boundary norm, the sequence $\lbrace g_n \rbrace_{n \in \mathbb{N}}$ converges uniformly on every compact subset of $U^{\textup{an}}$. Thus, its pointwise limit $g \colon U^{\textup{an}} \rightarrow \mathbb{R}_{\pm \infty}$ is a Green's function for $D|_U$ of continuous type. Moreover, one can show that the compactified arithmetic divisor $\overline{D}$ naturally identifies with the arithmetic divisor $(D|_U, g) \in \overline{\textup{Div}}(U)_{\mathbb{Q}}$:
		\begin{enumerate}
			\item For a projective variety $X/K$, the category $\textup{PM}(X/K)$ consist of a single object; the identity morphism. It follows that a boundary divisor is of the form $\overline{B}=(0,g_{B})$, where $g_{B}$ is smooth and positive on $X^{\textup{an}}$. By compactness, there are $c_1, c_2 \in \mathbb{R}_{>0}$ satisfying $c_1 > g_{B} > c_2$. Therefore, a sequence $\lbrace (D_n,g_n) \rbrace_{n \in \mathbb{N}}$ converges in the boundary topology if and only if the sequence $\lbrace D_n \rbrace_{n \in \mathbb{N}}$ is eventually constant, and the sequence $\lbrace g_n \rbrace_{n \in \mathbb{N}}$ converges uniformly on $X^{\textup{an}}$. By Proposition~\ref{model-dense}, we get that $ \overline{\textup{Div}}(X/K) =  \overline{\textup{Div}}(X)_{\mathbb{Q}}$.
			\item As in the projective case, $\overline{D} \in \overline{\textup{Div}}(U)_{\mathbb{Q}}$ is of \textit{model type} if it belongs to the image of $\overline{\textup{Div}}(U/K)_{\textup{mod}}$ under the restriction map. We denote by $\overline{\textup{Div}}(U)_{\textup{mod}}$ and $C^{0}(U^{\textup{an}})_{\textup{mod}}$ the groups of arithmetic divisors and continuous functions of model type on $U$. A boundary divisor $\overline{B}=(B,g_B)$ of $U/K$ restricts to $(0, g_{B}|_{U^{\textup{an}}})$ and induces a boundary topology on $\overline{\textup{Div}}(U)_{\mathbb{Q}}$. By Lemma~3.6.2~of~\cite{Y-Z}, the space $\overline{\textup{Div}}(U)_{\mathbb{Q}}$ is complete with respect to this boundary topology. Then, Proposition~3.6.1~of~\cite{Y-Z} gives an isomorphism of topological groups $\overline{\textup{Div}}(U/K) \rightarrow \overline{\textup{Div}}(U)_{\textup{cptf}}$, where $\overline{\textup{Div}}(U)_{\textup{cptf}}$ is the closure of $\overline{\textup{Div}}(U)_{\textup{mod}}$ in $\overline{\textup{Div}}(U)_{\mathbb{Q}}$. Similarly, we define $C^{0}(U^{\textup{an}})_{\textup{cptf}}$ as the closure of $C^{0}(U^{\textup{an}})_{\textup{mod}}$ in $C^{0}(U^{\textup{an}})$.
			\item If $K$ is non-archimedean, there is an analytification map $\textup{an} \colon \textup{Div}(U/\mathcal{O}_K) \rightarrow \overline{\textup{Div}}(U)_{\mathbb{Q}}$ induced by the universal property of direct limits and extension by continuity. Then, Proposition~3.6.1~of~\cite{Y-Z} gives a sequence
			\begin{center}
				\begin{tikzcd}
					\textup{Div}(U/\mathcal{O}_K) \arrow[r, "{\textup{an}}"] & \overline{\textup{Div}}(U/K) \arrow[r] & \overline{\textup{Div}}(U)_{\textup{cptf}}.
				\end{tikzcd}
			\end{center}
			where each map is an isomorphism.
		\end{enumerate}
		By a similar argument, $\overline{\textup{Div}}(U/K)$ also coincides with the completion in the boundary topology of the direct limit of arithmetic divisors of continuous type on projective models of $U/K$. In the following, we identify the groups $\overline{\textup{Div}}(U/K)$ and $\overline{\textup{Div}}(U)_{\textup{cptf}}$, omitting the latter notation.
	\end{rem}
	The assignment $(D,g) \mapsto D$ induces a continuous group morphism $\textup{for} \colon \overline{\textup{Div}}(U/K) \rightarrow \textup{Div}(U/K)$, known as the \textit{forgetful map}. Theorem~3.6.4~of~\cite{Y-Z} below describes the kernel of this map in terms of the asymptotic growth of the associated Green's functions as they approach the boundary.
	\begin{thm}\label{exact-seq-local}
		The forgetful map induces a short exact sequence of topological groups
		\begin{center}
			\begin{tikzcd}
				0 \arrow[r] & C^{0}(U^{\textup{an}})_{\textup{cptf}} \arrow[r] & \overline{\textup{Div}}(U/K) \arrow[r, "{\textup{for}}"] & \textup{Div}(U/K) \arrow[r] & 0.
			\end{tikzcd}
		\end{center}
		Let $(X, (B,g_B))$ be any boundary divisor of $U/K$. Then, the kernel $C^{0}(U^{\textup{an}})_{\textup{cptf}}$ satisfies
		\begin{displaymath}
			C^{0}(U^{\textup{an}})_{\textup{cptf}} = \lbrace g_{B} \cdot h|_{U^{\textup{an}}} \, \vert \, h \in C^{0}(X^{\textup{an}}) \textup{ and } h(X^{\textup{an}} \setminus U^{\textup{an}}) = 0 \rbrace.
		\end{displaymath}
	\end{thm}
	In other words, an element $\overline{D} \in  \overline{\textup{Div}}(U/K)$ is completely determined by a pair $\overline{D}=(D,g)$, where $D \in \textup{Div}(U/K)$ and $g$ is a Green's function for $D|_U$ of continuous type on $U^{\textup{an}}$. Intuitively, the above result says that a function $f$ belongs to the kernel if and only if the function $f/g_B$ tends to $0$ as it ``approaches'' the boundary. It is immediate that the inclusions $\overline{\textup{Div}}(U/K) \subset \overline{\textup{Div}}(U)_{\mathbb{Q}}$ and $C^{0}(U^{\textup{an}})_{\textup{cptf}} \subset C^{0}(U^{\textup{an}})$ are strict.
	\begin{rem}[Singular metrics]\label{singularities-metrics}
		Let $L$ be a line bundle on $U/K$ and $L^{\textup{an}}$ be its analytification. A \textit{singular metric} $\| \cdot \|$ on $L^{\textup{an}}$ as an assignment which, to each pair $(V,s_V)$ consisting of an open set $V$ and a non-vanishing section $s_V \in \Gamma(L,V)$, associates a function $- \log \| s_V \|^2 \colon V^{\textup{an}} \rightarrow \mathbb{R}_{\pm \infty }$ satisfying the gluing condition
		\begin{displaymath}
			- \log \| s_W \|^2 = \log |s_V /s_W|^2 - \log \| s_W \|^2 \quad \textup{on } V^{\textup{an}} \cap W^{\textup{an}}.
		\end{displaymath}
		The pair $\overline{L}=(L, \| \cdot \|)$ is called a \textit{singular metrized line bundle}. Given $P$ a property of functions (e.g., continuous, smooth, locally bounded, etc.), a singular metric is $P$ if each function $- \log \| s \|^2$ satisfies $P$. If $\| \cdot \|$ is continuous and finite, we remove the prefix ``singular''. Observe that the assignment 	$(\overline{L}, s) \mapsto \overline{\textup{div}}(s) \coloneq (\textup{div}(s), - \log \| s  \|^2 )$ induces an isomorphism $\overline{\textup{Pic}}(U) \cong \overline{\textup{Cl}}(U)$, where $\overline{\textup{Pic}}(U)$ is the group of isomorphism classes of continuous metrized line bundles on $U$. Then, the theory of compactified arithmetic divisors allows us to study the following typical situation:  Let $\overline{L}=(L,\| \cdot \|)$ be a singular metrized line bundle on a projective variety $X/K$. Suppose that the singular locus of $\| \cdot \|$ is contained in a proper closed subset $Z$ of $X$, and let $s$ be a rational section such that $Z \subset |\textup{div}(s)|$. Then, the restriction of $\overline{L}$ to the quasi-projective variety $U= X \setminus|\textup{div}(s)|$ is a continuous metrized line bundle, and the arithmetic divisor $\overline{\textup{div}}(s)|_U$ is of continuous type. Theorem \ref{exact-seq-local} gives necessary conditions for $\overline{\textup{div}}(s)|_U$ to be a compactified arithmetic divisor. Examples of this situation can be found~in~~\cite{K01},~\cite{BKK7}~and~\cite{BKK16}. In summary, we can move back and forth between the divisorial and line bundle perspectives. In this article, we opt for a divisorial approach. For the corresponding constructions on line bundles, see~Sections~2.5~and~3.4~of~\cite{Y-Z}.
	\end{rem}
	\subsection{Nefness, measures and local heights.} This subsection is devoted to local heights, which play the role of intersection numbers for the local case. As before, $U$ and $X$ are quasi-projective and projective varieties over $K$, respectively. If $K$ is archimedean, we further assume that these varieties are smooth. This is a mild assumption for intersection-theoretic purposes: We can always shrink $U$ so that it is smooth, and then pass from a possibly singular projective model $X'$ to a smooth projective model $X$ by taking successive blowups with center contained in the boundary $X \setminus U$. By the projection formula, intersection numbers are not affected by these operations.
	
	First, we recall the notion of nefness for arithmetic divisors. First, we assume $K=\mathbb{C}$, so the analytification of $X/\mathbb{C}$ is the projective complex manifold $X(\mathbb{C})$. Let $V \subset X(\mathbb{C})$ be an open set and $\phi \in L^{1}_{\textup{loc}}(V)$ locally integrable. Any such function induces a $(0,0)$-current on $V$. To avoid redundancy, one identifies locally integrable functions as in the $L^{1}_{\textup{loc}}$-topology; two functions are identified if they coincide almost everywhere. Regarding $\phi$ as a $(0,0)$-current on $V$, we may form the closed $(1,1)$-current $\textup{dd}^{\textup{c}} \phi$ on $V$, where $\textup{d}^{\textup{c}} \coloneqq \frac{i}{4 \pi} (\overline{\partial} - \partial )$. By~Theorem~5.8~in~Ch.~I~of~\cite{Dem12}, the function $\phi$ is \textit{plurisubharmonic} (\textit{psh} for short) if and only if the current $\textup{dd}^{\textup{c}} \phi$ is positive. The analogous definition for Green's functions is stated below.
	\begin{dfn}\label{psh-type}
		Let $\overline{D}=(D,g)$ be an arithmetic divisor on $X/\mathbb{C}$ of locally integrable type. Define the $(1,1)$-current $\omega_{D}(g) \coloneqq \textup{dd}^{\textup{c}} g+ \delta_D$, where $\delta_D$ is the current of integration along $D$. We say that the Green's function $g$ is of \textit{plurisubharmonic} type (psh type for short) if the $(1,1)$-current $\omega_{D}(g)$ is positive. If additionally $D$ is nef and $g$ is of continuous type, we say that $\overline{D}$ is \textit{nef}.
	\end{dfn}
	Now, we let $K$ be non-archimedean. By Remark~\ref{cptf-1}, there is an isomorphism of topological groups $\textup{an} \colon \textup{Div}(X/\mathcal{O}_K) \rightarrow \overline{\textup{Div}}(X)_{\mathbb{Q}}$. Then, an arithmetic divisor $\overline{D}=(D,g)$ on $X$ is \textit{nef} if it belongs to the image of the nef cone $ \textup{Div}^{\textup{nef}}(X/\mathcal{O}_K)$, introduced in Definition~\ref{rel-nef-dfn}. This means that the divisor $D$ is nef and the Green's function $g$ is the uniform limit of a sequence $\lbrace g_{n} \rbrace_{n \in \mathbb{N}}$ of Green's functions of model type for $D$, where each $g_n$ is induced by the analytification of a nef model $\mathcal{D}_n$ of $D$ over $\mathcal{O}_K$. Then, we extend these definitions to the quasi-projective case.
	\begin{dfn}\label{nef-local}
		The \textit{cone of semipositive compactified divisors of }$U$ \textit{over }$K$ is the closure $\overline{\textup{Div}}^{\textup{nef}}(U/K)$ of the cone $\overline{\textup{Div}}^{\textup{nef}} (U/K)_{\textup{mod}}$ of nef model arithmetic divisors with respect to the boundary topology. The \textit{space of integrable compactified arithmetic divisors} is defined as the difference of cones
		\begin{displaymath}
			\overline{\textup{Div}}^{\textup{int}} (U/K) \coloneqq \overline{\textup{Div}}^{\textup{nef}} (U/K) - \overline{\textup{Div}}^{\textup{nef}} (U/K).
		\end{displaymath}
	\end{dfn}
	\begin{rem}
		A posteriori, we recover the same notion of nefness as in the projective case. If $K$ is non-archimedean, by Remark~\ref{cptf-1}, we may identify the cones $\textup{Div}^{\textup{nef}} (U/\mathcal{O}_K)$ and $\overline{\textup{Div}}^{\textup{nef}} (U/K)$. On the other hand, if $K$ is archimedean, we may always shrink $U$ to find a nef boundary divisor $\overline{B}$ of $U/K$. By the argument in Remark~\ref{decreasing}, for each $\overline{D} \in \overline{\textup{Div}}^{\textup{nef}} (U/K)$ there exist a decreasing sequence $\lbrace \overline{D}_n \rbrace_{n \in \mathbb{N}}$ of nef model arithmetic divisors representing $\overline{D}=(D,g_D)$. It follows that the compactified divisor $D$ is nef and the current $\omega_{D}(g) =\lim_{n\in\mathbb{N}} \omega_{D_n}(g_n)$ is positive, as it is the limit of a decreasing sequence of positive currents. For details, see Section~3.6.6~of~\cite{Y-Z}.
	\end{rem}
	Let $X / K$ be a smooth projective variety of dimension $d$. Given a closed subvariety $Y$ of $X$ of dimension $e$, and model divisors $\overline{D}_1, \ldots ,\overline{D}_{e}$ on $X$, there is a signed measure
	\begin{displaymath}
		\langle \omega_{D_1}(g_1) \wedge \ldots \wedge \omega_{D_e}(g_e) \rangle \wedge \delta_Y
	\end{displaymath}
	on $X^{\textup{an}}$. It is given by the product of smooth (resp. algebraic) currents. Moreover, if the arithmetic divisors are nef, then this signed measure is a measure. For details, we refer~to~\cite{Dem12}~and~\cite{CD12} for the archimedean and non-archimedean cases, respectively. By Proposition~1.4.5~of~\cite{BPS}, this construction is extended by weak convergence to nef divisors $\overline{D}_i \in \textup{Div}^{\textup{nef}}(X)_{\mathbb{Q}}$. If $Y=X$ (and therefore $e=d$), the measure $\langle \omega_{D_1}(g_1) \wedge \ldots \wedge \omega_{D_d}(g_d) \rangle \wedge \delta_X$ is known as the \textit{Monge-Amp\`{e}re measure} in the archimedean case, and the \textit{Chambert-Loir measure} in the non-archimedean case. Now, consider arithmetic divisors $\overline{D}_0, \ldots ,\overline{D}_{e}$ of continuous type on $X$ such that the geometric divisor $D_0$ intersects $Y$ properly. By Theorem~1.4.10~from~\cite{BPS}, the Green's function $g_0$ is integrable with respect to the measure $\langle \omega_{D_1}(g_1) \wedge \ldots \wedge \omega_{D_e}(g_e) \rangle \wedge \delta_Y$. Then, we recall the definition and properties of local heights.
	\begin{dfn}\label{local-h-def}
		Let $X / K$ be a smooth projective variety of dimension $d$. Let $Y $ be a subvariety of dimension $e$ and $\overline{D}_1, \ldots, \overline{D}_{e} \in \overline{\textup{Div}}(X)_{\mathbb{Q}}$ be nef arithmetic divisors intersecting $Y$ properly. The \textit{local height of }$Y$\textit{ with respect to} $\overline{D}_0, \ldots, \overline{D}_{e}$ is inductively defined as the number
		\begin{displaymath}
			\textup{h}(Y; \overline{D}_0, \ldots ,\overline{D}_e) \coloneqq \textup{h}(Y \cdot D_0; \overline{D}_1, \ldots, \overline{D}_e) + \int_{X^{\textup{an}}} g_0 \langle \omega_{D_1}(g_1) \wedge \ldots \wedge \omega_{D_e}(g_e) \rangle \wedge \delta_Y.
		\end{displaymath}
		If $Y =X$, define the local arithmetic intersection numbers $\overline{D}_0\cdot \ldots \cdot \overline{D}_{d}$ as $\textup{h}(X; \overline{D}_0, \ldots, \overline{D}_d)$. These definitions extend by linearity to an arbitrary cycle $Y$ on $X$.
	\end{dfn}
	\begin{thm}\label{local-h-thm}
		Let $X/K$ be a projective variety, $Y$ be $e$-dimensional cycle on $X$, and $\overline{D}_0, \ldots , \overline{D}_e$ be nef arithmetic divisors whose divisorial parts $D_0, \ldots , D_e$ meet $Y$ properly. Then, the following statements hold:
		\begin{enumerate}
			\item The local height $\textup{h}(Y; \overline{D}_0, \ldots ,\overline{D}_e)$ is symmetric and multilinear whenever it is defined.
			\item Let $Y^{\prime}$ be an $e$-dimensional cycle on the projective variety $X^{\prime}$ over $K$, and let $\phi  \colon X' \rightarrow X$ be a proper morphism over $K$. Then, we have a projection formula
			\begin{displaymath}
				\textup{h}(Y^{\prime}; \phi^{\ast} \overline{D}_0, \ldots, \phi^{\ast} \overline{D}_e) = \textup{h}(\phi_{\ast} Y^{\prime}; \overline{D}_0, \ldots, \overline{D}_e).
			\end{displaymath}
			\item Let $f_0 \in K(X)^{\times}$ be a rational function such that the divisors $D_0 + \textup{div}(f_0),  \ldots, D_e$ meet $Y$ properly. Consider the zero-cycle $Y \cdot D_1 \cdot \ldots \cdot D_e = \sum_{i} m_i \cdot p_i$ on $X$, where $m_i \in \mathbb{Q}$ and $p_i \in X$, and the principal arithmetic divisor $\overline{\textup{div}}(f_0) = (\textup{div}(f_0), - \log |f_0|^2 )$. Then, we can compute the following difference
			\begin{displaymath}
				\textup{h}(Y ; \overline{D}_0, \ldots , \overline{D}_e) - \textup{h}(Y ; \overline{D}_{0} + \overline{\textup{div}}(f_0), \overline{D}_1, \ldots , \overline{D}_e) = \log \prod_{i} |f_0(p_i)^{m_i}|.
			\end{displaymath}
			\item Let $\overline{D}_{0}^{\prime}$ be an arithmetic divisor on $X$ sharing the same divisorial part as $\overline{D}_{0}$. Then,
			\begin{displaymath}
				\overline{D}_0\cdot \ldots \cdot \overline{D}_{d} - \overline{D}_{0}^{\prime}\cdot \ldots \cdot \overline{D}_{d}  = \int_{X^{\textup{an}}} (g_0 - g_{0}^{\prime}) \langle \omega_{D_1}(g_1) \wedge \ldots \wedge \omega_{D_d}(g_d) \rangle \wedge \delta_X.
			\end{displaymath}
		\end{enumerate}
	\end{thm}
	\begin{proof}
		See Theorem~1.4.17~of~\cite{BPS}.
	\end{proof}
	By part \textit{(iii)} of the above theorem, the local height is not invariant under linear equivalence. Nevertheless, one can show that certain differences of local heights satisfy this property.
	\begin{lem}\label{difference-h}
		Let $X/K$ be a projective variety, $Y$ be $e$-dimensional cycle on $X$, and $\overline{D}_0, \ldots , \overline{D}_e$ be nef arithmetic divisors whose divisorial parts $D_0, \ldots , D_e$ meet $Y$ properly. For each $0 \leq i \leq e$, let $\overline{D}_{i}^{\prime}$ be a nef arithmetic divisor on $X/K$ whose underlying geometric divisor is equal to $D_i$, and let $f_i  \in K(X)^{\times}$ be a rational function such that the divisors $D_0 + \textup{div}(f_0),\ldots , D_e + \textup{div}(f_e)$ meet $Y$ properly. Then, have an identity
		\begin{displaymath}
			\textup{h}(Y ; \overline{D}_0, \ldots , \overline{D}_e) - \textup{h}(Y ; \overline{D}_{0}^{\prime}, \ldots , \overline{D}_{e}^{\prime}) = \textup{h}(Y ; \overline{E}_0, \ldots , \overline{E}_e) - \textup{h}(Y ; \overline{E}_{0}^{\prime}, \ldots , \overline{E}_{e}^{\prime}),
		\end{displaymath}
		where $\overline{E}_i \coloneqq \overline{D}_i + \overline{\textup{div}}(f_i)$ and $\overline{E}_{i}^{\prime} \coloneqq \overline{D}_{i}^{\prime} + \overline{\textup{div}}(f_i)$.
	\end{lem}
	\begin{proof}
		Applying Theorem~\ref{local-h-thm}.\textit{(iii)} twice, we obtain
		\begin{align*}
			\delta_0 \coloneqq \log  \prod_{i} |f_0 (p_i)^{m_i} |  &= \textup{h}(Y ; \overline{D}_0, \ldots , \overline{D}_e) - \textup{h}(Y ; \overline{E}_0,\overline{D}_1, \ldots , \overline{D}_e) \\
			&=  \textup{h}(Y ; \overline{D}_{0}^{\prime}, \ldots , \overline{D}_{e}^{\prime}) - \textup{h}(Y ; \overline{E}_{0}^{\prime}, \overline{D}_{1}^{\prime}, \ldots , \overline{D}_{e}^{\prime}).
		\end{align*}
		Note that $\delta_0$ depends only on the rational function $f_0$ and the divisors $D_1, \ldots, D_e$. Proceeding similarly, for each $1 \leq i \leq e$, we obtain the identity
		\begin{align*}
			\delta_i  &\coloneqq \textup{h}(Y ; \overline{E}_{0}, \ldots , \overline{E}_{i-1}, \overline{D}_{i}, \ldots , \overline{D}_{e}) - \textup{h}(Y ; \overline{E}_{0}, \ldots , \overline{E}_{i-1},  \overline{E}_{i}, \overline{D}_{i+1}, \ldots , \overline{D}_{e})\\
			& = \textup{h}(Y ; \overline{E}_{0}^{\prime} , \ldots , \overline{E}_{i-1}^{\prime}, \overline{D}_{i}^{\prime}, \ldots , \overline{D}_{e}^{\prime}) - \textup{h}(Y ; \overline{E}_{0}^{\prime}, \ldots, \overline{E}_{i-1}^{\prime}, \overline{E}_{i}^{\prime}, \overline{D}_{i+1}^{\prime}, \ldots , \overline{D}_{e}^{\prime}),
		\end{align*}
		where the number $\delta_i$ depends only on the rational function $f_i$ and the underlying geometric divisors $E_0, \ldots, E_{i-1}$ and $D_{i+1}, \ldots, D_{e}$. Summing over $i$, the right-hand side is a telescopic sum. After simplification, it yields
		\begin{displaymath}
			\sum_{i=0}^{e} \delta_i = \textup{h}(Y ; \overline{D}_0, \ldots , \overline{D}_e) - \textup{h}(Y ; \overline{E}_0, \ldots , \overline{E}_e) =  \textup{h}(Y ; \overline{D}_{0}^{\prime}, \ldots , \overline{D}_{e}^{\prime}) -  \textup{h}(Y ; \overline{E}_{0}^{\prime}, \ldots , \overline{E}_{e}^{\prime}).
		\end{displaymath}
		Rearranging the above equation, we obtain the desired identity.
	\end{proof}
	As in the geometric case, the local arithmetic intersection numbers induce a symmetric, multilinear map on the group of model divisors of a quasi-projective variety $U$ over $K$. 
	\begin{cor}
		Let $U/K$ be a quasi-projective variety. Then, the local arithmetic intersection numbers extend to a symmetric, multilinear map on $\overline{\textup{Div}}^{\textup{nef}} (U/K)_{\textup{mod}}$ and  $\overline{\textup{Div}}^{\textup{int}} (U/K)_{\textup{mod}}$, whenever the underlying geometric divisors intersect properly.
	\end{cor}
	\begin{proof}
		This is implied by parts \textit{(i)} and \textit{(ii)} of Theorem~\ref{local-h-thm}.
	\end{proof}
	\begin{rem}\label{mixed-energy-1}
		In the article~\cite{BK24}, Burgos and Kramer extended Yuan and Zhang's arithmetic intersection pairing using the notion of \textit{mixed relative energy}. This construction is based on the \textit{relative energy}~of~\cite{Darvas}. The key observation is that over the archimedean places, the mixed relative energy coincides with a certain difference of local arithmetic intersection numbers, similar to those considered in Lemma~\ref{difference-h}. We relate this notion to the local toric height in Remark~\ref{mixed-energy-2}.
	\end{rem}
	
	\section{A toric geometric analog of Yuan-Zhang's theory}\label{3}
	In this section, we study the group of toric compactified divisors of a quasi-projective toric variety. In some important cases, for instance, the split torus of dimension $d$ over a field, this group and its nef cone can be computed explicitly. We emphasize that some of these results can be traced back to the work of Botero on toric b-divisors~\cite{Bot19}. However, we also discuss the case of a split torus over a discrete valuation ring. 
	
	\subsection{Toric compactified divisors} We fix notation that will be repeatedly used throughout the section. For an affine scheme $\mathcal{S}=\textup{Spec}(R)$, the \textit{multiplicative group} $\mathbb{G}_{m}/\mathcal{S}$ is the affine group scheme $\textup{Spec}(R[T, T^{-1}])$. In the following, we let $\mathcal{S}$ be either the spectrum of a field $K$ or a Dedekind domain $\mathcal{O}_K$ with field of fractions $K$. For a positive integer $d$, we let $\mathcal{U}\coloneqq \mathbb{G}^{d}_{m}/\mathcal{S}$ be the \textit{split} $d$-\textit{dimensional torus} over $\mathcal{S}$. When the base scheme is a Dedekind domain, we denote by $U \coloneqq \mathbb{G}_{m}^{d}/K$ the corresponding split tori. In this case, $U$ naturally identifies with the generic fibre of $\mathcal{U}$. Then, we recall the following definition.
	\begin{dfn}\label{toric-scheme}
		A $d$-dimensional \textit{toric scheme over} $\mathcal{S}$ is a triple $(\mathcal{X},\pi,\mu)$ consisting of a variety $\mathcal{X}/\mathcal{S}$ of relative dimension $d$, an open embedding $\pi \colon U \rightarrow X$ into the generic fibre $X$ of $\mathcal{X}$, and an $\mathcal{S}$-group scheme action $\mu \colon \mathcal{U} \times_{\mathcal{S}} \mathcal{X} \rightarrow \mathcal{X}$ extending the action of $U$ on itself by translations.
	\end{dfn}
	We often obviate the embedding and action to lighten up notation and say that $\mathcal{X}/\mathcal{S}$ is a toric scheme. When the base scheme is a field, $\mathcal{X}$ coincides with its generic fibre and we recover the usual definition of a toric variety (for instance, Definition~3.1.1~of~\cite{CLS}). We have the corresponding notions for divisors and morphisms.
	\begin{dfn}
		Let $\mathcal{X}/\mathcal{S}$ be a toric scheme. A divisor $\mathcal{D} \in \textup{Div}( \mathcal{X})_{\mathbb{Q}}$ is \textit{toric} if it satisfies $\mu^{\ast} \mathcal{D} = p_{2}^{\ast} \mathcal{D}$, where $\mu$ is the action of $\mathcal{U}$ and $p_{2} \colon \mathcal{U} \times_{\mathcal{S}}  \mathcal{X} \rightarrow \mathcal{X}$ is the projection to the second factor. Denote by $ \textup{Div}_{\mathbb{T}}( \mathcal{X})_{\mathbb{Q}}$ the group of toric Cartier $\mathbb{Q}$-divisors on $\mathcal{X}$.
	\end{dfn}
	\begin{dfn}
		For each $i \in \lbrace 1,2 \rbrace$, let $(\mathcal{X}_i, \pi_i, \mu_i)$ be a toric scheme over $\mathcal{S}$ and $\rho \colon \mathcal{U}_{1} \rightarrow \mathcal{U}_{2}$ be an $\mathcal{S}$-homomorphism between their tori. An $\mathcal{S}$-morphism $f \colon \mathcal{X}_1 \rightarrow \mathcal{X}_2$ is $\rho$\textit{-equivariant} if the following diagram
		\begin{center}
			\begin{tikzcd}
				\mathcal{U}_{1} \times_{\mathcal{S}} \mathcal{X}_1 \arrow[d, "{\rho \times f}"] \arrow[r, "{\mu_1}"] & \mathcal{X}_1  \arrow[d, "{f}"] \\
				\mathcal{U}_{2} \times_{\mathcal{S}} \mathcal{X}_2  \arrow[r, "{\mu_2}"] & \mathcal{X}_2 
			\end{tikzcd}
		\end{center}
		commutes. A $\rho$-equivariant morphism $f$ is $\rho$-\textit{toric} if it restricts to a morphism of tori $U_1 \rightarrow U_2$. If the homomorphism $\rho$ is clear, we say that $f \colon \mathcal{X}_1 \rightarrow \mathcal{X}_2$ is equivariant (resp. toric).
	\end{dfn}
	Throughout this subsection, we fix a quasi-projective toric scheme $\mathcal{X}_0 /\mathcal{S}$. The notion of a projective model readily generalizes to the toric setting.
	\begin{dfn}
		A projective model $\pi \colon \mathcal{X}_0 \rightarrow \mathcal{X}$ over $\mathcal{S}$ is \textit{toric} if both $\mathcal{X}/\mathcal{S}$ and $\pi$ are toric. A \textit{morphism of toric projective models of} $\mathcal{X}_0$ \textit{over} $\mathcal{S}$ is a morphism of projective models which is toric. We denote by $\textup{PM}_{\mathbb{T}}(\mathcal{X}_0 /\mathcal{S})$ the category of toric projective models of $\mathcal{X}_{0}$ over $\mathcal{S}$.
	\end{dfn}
	\begin{rem}
		The category $\textup{PM}_{\mathbb{T}}(\mathcal{X}_0 /\mathcal{S})$ is defined as a subcategory of $\textup{PM}(\mathcal{X}_0 /\mathcal{S})$, and these are not equivalent. To carry over the notion of compactified divisors to the toric setting, we need to show that $\textup{PM}_{\mathbb{T}}(\mathcal{X}_0 /\mathcal{S})$ is cofiltered. Later in this section, we will give a proof of this using a combinatorial description of the category $\textup{PM}_{\mathbb{T}}(\mathcal{X}_0 /\mathcal{S})$ for the case when $\mathcal{S}$ is the spectrum of a field or a discrete valuation ring (See~Propositions~\ref{tqpmod}~and~\ref{PM-over-DVR}).
	\end{rem}
	Given a toric projective model $\pi \colon \mathcal{X}_{0} \rightarrow \mathcal{X}$ over $\mathcal{S}$ and a toric divisor $\mathcal{D} \in \textup{Div}_{\mathbb{T}}( \mathcal{X}_{0})_{\mathbb{Q}}$, a \textit{toric model} of $\mathcal{D}$ on $\mathcal{X}$ is a toric divisor $\mathcal{D}^{\prime}$ on $\mathcal{X}$ such that $\mathcal{D}^{\prime}|_{\mathcal{X}_{0}} = \mathcal{D}$. Then, the \textit{group of toric model divisors of} $\mathcal{X}_0$ \textit{over} $\mathcal{S}$ is defined as the direct limit
	\begin{displaymath}
		\textup{Div}_{\mathbb{T}}(\mathcal{X}_0/\mathcal{S})_{\textup{mod}} \coloneqq \varinjlim_{\mathcal{X} \in \textup{PM}_{\mathbb{T}}(\mathcal{X}_0/\mathcal{S})} \textup{Div}_{\mathbb{T}}(\mathcal{X})_{\mathbb{Q}}.
	\end{displaymath}
	We denote by $\textup{Div}_{\mathbb{T}}^{\textup{nef}}(\mathcal{X}_0/\mathcal{S})_{\textup{mod}}$ the $\mathbb{Q}$-cone of nef toric model divisors. By a \textit{toric boundary divisor of} $\mathcal{U}/\mathcal{S}$ we mean a boundary divisor $(\mathcal{X},\mathcal{B})$ such that both the projective model $\pi \colon \mathcal{X}_0 \rightarrow \mathcal{X}$ and the divisor $\mathcal{B}$ are toric. This induces a boundary norm on $	\textup{Div}_\mathbb{T}(\mathcal{X}_0/\mathcal{S})_{\textup{mod}}$, giving rise to a boundary topology. As before, this topology is independent of the choice of boundary divisor. We are now ready to introduce the analogous compactified objects.
	\begin{dfn}
		The \textit{group of toric compactified divisors of} $\mathcal{X}_0/\mathcal{S}$ is the completion $\textup{Div}_{\mathbb{T}}(\mathcal{X}_0/\mathcal{S})$ of $\textup{Div}_{\mathbb{T}}(\mathcal{X}_0/\mathcal{S})_{\textup{mod}}$ with respect to the boundary topology. The cone $\textup{Div}_{\mathbb{T}}^{\textup{nef}}(\mathcal{X}_0/\mathcal{S})$ of nef toric divisors is the closure of $\textup{Div}_{\mathbb{T}}^{\textup{nef}}(\mathcal{X}_0/\mathcal{S})_{\textup{mod}}$, and the space of integrable toric divisors is the difference
		\begin{displaymath}
			\textup{Div}_{\mathbb{T}}^{\textup{nef}}(\mathcal{X}_0/\mathcal{S}) \coloneq \textup{Div}_{\mathbb{T}}^{\textup{nef}}(\mathcal{X}_0/\mathcal{S}) - \textup{Div}_{\mathbb{T}}^{\textup{nef}}(\mathcal{X}_0/\mathcal{S}).
		\end{displaymath}
	\end{dfn}
	Shrinking the quasi-projective toric scheme is functorial in the following sense.
	\begin{lem}\label{shrinking}
		Let $f \colon \mathcal{X}_{0} \rightarrow \mathcal{X}_{0}^{\prime}$ be a toric open immersion. Then, the following diagram in the category of topological abelian groups commutes
		\begin{center}
			\begin{tikzcd}
				\textup{Div}_{\mathbb{T}}(\mathcal{X}_{0}^{\prime}/\mathcal{S}) \arrow[r] \arrow[d] &  \textup{Div}_{\mathbb{T}}(\mathcal{X}_{0}/\mathcal{S}) \arrow[d] \\
				\textup{Div}(\mathcal{X}_{0}^{\prime}/\mathcal{S}) \arrow[r] & \textup{Div}(\mathcal{X}_{0}/\mathcal{S}).
			\end{tikzcd}
		\end{center}
	\end{lem}
	\begin{proof}
		Since $\textup{PM}_{\mathbb{T}}(\mathcal{X}_0 /\mathcal{S})$ is defined as a subcategory of $\textup{PM}(\mathcal{X}_0 /\mathcal{S})$ and a toric boundary divisor of $\mathcal{X}_0/\mathcal{S}$ is a boundary divisor in the Yuan-Zhang sense, we obtain a canonical continuous group morphism $\textup{Div}_{\mathbb{T}}(\mathcal{X}_0/\mathcal{S}) \rightarrow  \textup{Div}(\mathcal{X}_0/\mathcal{S})$. The second vertical map is obtained in the same way. The map $f \colon \mathcal{X}_{0} \rightarrow \mathcal{X}_{0}^{\prime}$ is a toric open immersion, and so, every (toric) projective model of $\mathcal{X}_{0}^{\prime}$ over $\mathcal{S}$ is trivially a (toric) projective model of $\mathcal{X}_{0}$ over $\mathcal{S}$. Then, we have horizontal maps at the level of model divisors. Consider toric boundary divisors $\mathcal{B}$ and $\mathcal{B}^{\prime}$ for $\mathcal{X}_{0}$ and $\mathcal{X}_{0}^{\prime}$ respectively, which we can assume are defined on the same toric projective model. By Remark~\ref{extended-norm}, the condition $\mathcal{X}_{0} \subset \mathcal{X}_{0}^{\prime}$ implies $ \| \mathcal{B}^{\prime} \|_{\mathcal{B}} < b < \infty $ for some positive rational number $b$. It follows from the definition of the boundary norm that $\| \mathcal{D} \|_{\mathcal{B}^{\prime}} < b \cdot \| \mathcal{D} \|_{\mathcal{B}}$ for each $\mathcal{D} \in \textup{Div}_{\mathbb{T}}(\mathcal{X}_{0}^{\prime}/\mathcal{S})$, showing continuity. Therefore, the horizontal maps exist and are continuous. Commutativity of the diagram is obvious.
	\end{proof}
	\begin{rem}
		The vertical maps in the previous lemma are actually embeddings of topological groups. It only remains to show injectivity. Later, we will establish this fact in the case where $\mathcal{S}$ is a field or a DVR. To do this, we will use the convex analytic description of $\textup{Div}_{\mathbb{T}}(\mathcal{X}_0/\mathcal{S})$. Similar assertions for $\textup{Div}_{\mathbb{T}}^{\textup{nef}}(\mathcal{X}_0/\mathcal{S})$ and $\textup{Div}_{\mathbb{T}}^{\textup{int}}(\mathcal{X}_0/\mathcal{S})$ also hold. Using these techniques, one can also show that the horizontal maps are injective (See Lemmas~\ref{toric-top-emb},~\ref{toric-top-emb-DVR}). 
	\end{rem}
	
	\subsection{The case over a field}\label{3-1} The so-called ``\textit{toric dictionary}'' describes the geometry of toric varieties over a field $K$ in terms of convex geometry. We begin with a summary of these descriptions, then extend them to the context of compactified divisors. Denote by $N \coloneqq \textup{Hom}_{K}(\mathbb{G}_{m},U)$ the group of \textit{cocharacters} of $U/K$. Its dual $M \coloneqq \textup{Hom}_K ( U , \mathbb{G}_{m})$ is called the group of \textit{characters} of $U/K$. Both $M$ and $N$ are lattices of rank~$d$. For each character $m \in M$, we denote by $\chi^{m}$ its corresponding element in the group algebra $K \left[ M \right]$.
	
	\begin{cons}[The toric variety of a fan]
		For each fan $\Sigma$ in $N_{\mathbb{R}}$, there is an associated toric variety $X_{\Sigma}/K$. We sketch its construction below.
		\begin{enumerate}
			\item For each $\sigma \in \Sigma$, denote by $M_{\sigma} \coloneqq \sigma^{\vee} \cap M $ the semigroup of lattice points in the dual cone $\sigma^{\vee} \coloneqq \lbrace x \in M_{\mathbb{R}} \, \vert \, \langle x, u \rangle \geq 0 \, \textup{for all } u \in \sigma\rbrace$. Then, $M_{\sigma}$ generates the commutative $K$-algebra $K\left[ M_{\sigma} \right]$. Define the affine variety $X_{\sigma} \coloneqq \textup{Spec}( K \left[ M_{\sigma}\right] )$. This is an affine toric variety over $K$. The collection $\lbrace X_{\sigma} \rbrace_{\sigma \in \Sigma}$ of affine toric varieties constitutes the building blocks of $X_{\Sigma}$.
			\item Let $\tau$ be a face of $\sigma$. By definition, there exists $m \in M_{\sigma}$ such that $\tau = \sigma \cap H_m$, where $H_m$ is the hyperplane determined by the equation $\langle m, \cdot \rangle = 0$. The variety $X_{\tau}$ corresponds to the principal open set $(X_{\sigma})_m$, given by the localization of $K \left[ M_{\sigma} \right]$ at $\chi^{m}$. This induces an open embedding $X_{\tau} \rightarrow X_{\sigma}$, identifying $X_{\tau}$ with an open subset of $X_{\sigma}$. Every pair of cones $ \sigma_1, \sigma_2 \in \Sigma$ intersect at a common face $ \tau$, hence the identifications $X_{\tau} \rightarrow X_{\sigma_i}$, $i  \in \lbrace 1,2 \rbrace$, allows us to glue the family $\lbrace X_{\sigma} \rbrace_{\sigma \in \Sigma}$.
			\item The affine toric variety $X_{\lbrace 0 \rbrace}$ is isomorphic to  the torus $U$. For each cone $\sigma$, the assignment $\chi^m \mapsto \chi^m \otimes \chi^m$ determines a ring morphism $\mu^{\sharp}_{\sigma} \colon K \left[  M_{\sigma} \right] \rightarrow K \left[  M_{\sigma} \right] \otimes_K K \left[ M \right]$. This induces an action $\mu_{\sigma} \colon U \times_K X_{\sigma} \rightarrow X_{\sigma}$, and the family $\lbrace \mu_{\sigma} \rbrace_{\sigma \in \Sigma}$ glues into an action $\mu$ of $U$ on $X_{\Sigma}$, giving it the structure of toric variety.
		\end{enumerate}
		Many properties of the variety $X_{\Sigma}$ correspond with properties of the fan $\Sigma$ and vice versa. For example, the toric variety $X_{\Sigma}$ is proper (resp. smooth) if and only if the fan $\Sigma$ is complete (resp. smooth). For details, see sections \S~3.1--3.3~of~\cite{CLS}.
	\end{cons} 
	A toric morphism  $f \colon X_{\Sigma_1} \rightarrow X_{\Sigma_2}$  induces a $\mathbb{Z}$-linear map $\phi \colon  N_1 \rightarrow N_2$, where $\phi (u)$ is the cocharacter $f \circ u \colon \mathbb{G}_m \rightarrow U_2$. The induced $\mathbb{R}$-linear map $\phi \colon (N_1)_{\mathbb{R}} \rightarrow (N_2)_{\mathbb{R}}$ is compatible with $\Sigma_1$ and $\Sigma_2$ (see~Definition~\ref{compatible}). The following result shows that every toric morphism arises in this way.
	\begin{prop}\label{toricmorphism}
		There is a bijection between toric morphisms $f \colon X_{\Sigma_1} \rightarrow X_{\Sigma_2}$ and  $\mathbb{Z}$-linear maps $\phi \colon N_1 \rightarrow N_2$ compatible with the fans $\Sigma_1$ and $\Sigma_2$. The morphism is proper if and only if $|\Sigma_1| = \phi^{-1}(|\Sigma_2|)$. In particular, proper birational toric morphisms are equivalent to refinements (up to isomorphism of the torus).
	\end{prop} 
	\begin{proof}
		The first two assertions are Theorems~3.3.4~and~3.4.11~of~\cite{CLS}, respectively. The last one is Exercise~3.4.14, which we prove for completeness. Clearly, a refinement determines a proper birational morphism. Conversely, let $f$ be such a map. Then, the induced map $\phi \colon N_1  \rightarrow N_2$ is an isomorphism, with inverse $\psi$. Since $f$ is proper, the fan $\Sigma \coloneqq \lbrace \psi(\sigma) \, \vert \, \sigma \in \Sigma_2 \rbrace$ has support $|\Sigma_1|$. The compatibility condition on $\phi$ implies that $\Sigma_1$ is a refinement of $\Sigma$. Therefore, the map $g \colon X_{\Sigma_2} \rightarrow X_{\Sigma}$ induced by $\psi$ is an isomorphism, and the composition $g \circ f \colon X_{\Sigma_1} \rightarrow X_{\Sigma}$ corresponds to a refinement.
	\end{proof}
	Theorem~11.1.9~of~\cite{CLS} gives a toric resolution of singularities, which does not depend on the characteristic of the field $K$. We will use it repeatedly throughout this article. For the reader's convenience, it is included below.
	\begin{thm}\label{toricres}
		Every fan $\Sigma$ has a refinement $\Sigma'$ with the following properties:
		\begin{enumerate}[label=(\roman*)]
			\item The fan $\Sigma'$ is smooth and contains every smooth cone of $\Sigma$.
			\item $\Sigma'$ is obtained from $\Sigma$ by a sequence of star subdivisions.
			\item The induced toric morphism $X_{\Sigma'} \rightarrow X_{\Sigma}$ is a projective resolution of singularities.
		\end{enumerate}
	\end{thm}
	\begin{cons}[The cone-orbit correspondence]\label{cone-orbit}
		Let $\Sigma$ be a fan in $N_{\mathbb{R}}$. For each cone $\sigma$, we have lattices $N(\sigma) \coloneqq N/(N \cap \textup{Span}(\sigma))$ and $M(\sigma) \coloneqq N(\sigma)^{\vee} = M \cap \sigma^{\perp}$, where $\sigma^{\perp}$ is the linear space orthogonal to $\sigma$. Then, $O(\sigma)=\textup{Spec}(K \left[ M(\sigma) \right] )$ is a torus of dimension $d-\textup{dim}(\sigma)$. The surjective ring morphism $K \left[ M_{\sigma} \right]  \rightarrow K \left[ M( \sigma ) \right] $ determined by the assignment
		\begin{displaymath}
			\chi^m \mapsto \begin{cases} \chi^m, & m\in \sigma^{\perp} \\ 0, & \textup{otherwise} \end{cases}
		\end{displaymath}
		induces an equivariant locally closed immersion $O(\sigma ) \rightarrow X_{\sigma} \rightarrow X_{\Sigma}$. Under this immersion, $O(\sigma )$ identifies with the torus orbit of the image of the identity element in $O(\sigma )$. This gives a correspondence between the $n$-dimensional cones of $\Sigma$ and the $(d-n)$-dimensional orbits of~$X_{\Sigma}$. The Zariski closure $V(\sigma )$ of $O(\sigma )$ in $X_{\Sigma}$ admits the following description: For each cone $\tau$ containing $\sigma$, write $\overline{\tau}$ for its image in the quotient $N(\sigma)$. The collection $\Sigma(\sigma) \coloneqq \lbrace \overline{ \tau } \, | \, \tau \in \Sigma, \,  \sigma \leq \tau \rbrace$ is a fan in $N(\sigma)$, and the family of immersions $\lbrace O(\tau) \rightarrow X_{\Sigma} \, \vert \, \sigma \leq \tau \in \Sigma \rbrace$ glue into an equivariant (but not toric) closed immersion $\iota_{\sigma}\colon X_{\Sigma(\sigma)}\rightarrow X_{\Sigma}$. It identifies $X_{\Sigma(\sigma)}$ with $V(\sigma)$. For details, see Section~\S~3.3~of~\cite{CLS}.
	\end{cons}
	The cone-orbit correspondence implies that a ray $\rho$ of $\Sigma$ corresponds with the closure $V(\rho)$ of an orbit of codimension $1$. Then, a Weil divisor on $X_{\Sigma}$ is \textit{toric} if it belongs to the free abelian group generated by the set $\lbrace V(\rho) \, | \, \rho \in \Sigma(1) \rbrace$. On the other hand, a toric integral Cartier divisor on $X_{\Sigma}/K$ (as opposed to a $\mathbb{Q}$-Cartier divisor) is of the form $D = \lbrace (X_{\sigma}, \chi^{-m_{\sigma}})\rbrace_{\sigma \in \Sigma}$, where $\chi^{-m_{\sigma}}$ is the rational function on $X_{\sigma}$ coming from a character $m_{\sigma}$. The glueing condition means that the functions $\chi^{m_{\sigma'}-m_{\sigma}}$ are units on the overlaps $X_{\sigma \cap \sigma'}$. Using this representation, we can attach to $D$ a virtual support function $\Psi_{D} \in \mathcal{SF}(\Sigma, \mathbb{Q})$, which on each cone $\sigma$ is defined as $\Psi_{D}(u)\coloneqq \langle m_{\sigma} , u \rangle$. The support function $\Psi_D$ is integral, that is, $\Psi_{D} (N \cap |\Sigma|) \subset \mathbb{Z}$. This construction is extended to toric $\mathbb{Q}$-divisors by linearity. The following result states that every toric divisor arises in this way.
	\begin{prop}\label{toric-divisors}
		The map $\mathcal{SF} \colon \textup{Div}_{\mathbb{T}}(X_{\Sigma})_{\mathbb{Q}} \rightarrow \mathcal{SF}(\Sigma, \mathbb{Q})$ given by the assignment $D \mapsto \Psi_D$ is an isomorphism. The toric divisors $D$ and $D'$ are linearly equivalent if and only if $\Psi_{D} - \Psi_{D'}$ is linear. If $\Sigma$ is complete, there is an exact sequence
		\begin{center}
			\begin{tikzcd}
				0 \arrow[r] & M \arrow[r] & \textup{Div}_{\mathbb{T}}(X_{\Sigma}) \arrow[r] &  \textup{Cl}(X_{\Sigma}) \arrow[r] & 0
			\end{tikzcd}
		\end{center}
		and the Cartier class group $\textup{Cl}(X_{\Sigma})$ is a free $\mathbb{Z}$-module. Moreover, the groups of toric Cartier divisors and toric Weil divisors coincide if and only if $\Sigma$ is smooth.
	\end{prop}
	\begin{proof}
		See Propositions~4.2.6,~4.2.12 and Theorems~4.1.3,~4.2.1~of~\cite{CLS}.
	\end{proof}
	Proposition~3.3.17~of~\cite{BPS} below shows that support functions of divisors behave well under pullback morphisms. This will be a key ingredient in our descriptions.
	\begin{lem}\label{pbsup}
		Let $f \colon X_{\Sigma_1} \rightarrow X_{\Sigma_2}$ be a dominant toric morphism inducing the homomorphism $\phi \colon N_1 \rightarrow N_2$. Let $D$ be a toric divisor on $X_{\Sigma_2}$ with support function $\Psi_{D}$. Then, the pullback $f^{\ast} D$ is a toric divisor on $X_{\Sigma_1}$, and its support function is $\Psi_{D} \circ \phi$. In particular, if $f$ restricts to the identity on tori, then $\Psi_{D} = \Psi_{f^{\ast} D}$.
	\end{lem}
	Many properties of a toric divisor $D$ can be read from its support function $\Psi_D$. A guiding principle is that under the toric dictionary, positivity conditions translate into convexity properties. For instance, this is true for the conditions of being nef or ample. Before stating the relevant result, recall that a toric divisor $D$ has an associated polyhedron $\Delta_D \in N_{\mathbb{R}}$, given by
	\begin{displaymath}
		\Delta_{D}\coloneqq \lbrace m \in M_{\mathbb{R}} \, \vert \, \langle m , v_{\tau} \rangle \geq \Psi (v_{\tau}), \textup{ for all } \tau \in \Sigma(1)\rbrace,
	\end{displaymath}
	where $v_{\tau}$ is the ray generator of $\tau$ (i.e. $\mathbb{N} \cdot v_{\tau} = \tau \cap N$). We have the following characterization.
	\begin{thm}\label{nefdiv}
		Let $\Sigma$ be a complete fan in $N_{\mathbb{R}}$ and $X_{\Sigma}/K$ its toric variety. The assignments $D \mapsto \Psi_D$ and $D \mapsto \Delta_D$ induce bijections between the sets:
		\begin{enumerate}[label=(\roman*)]
			\item Nef toric divisors $D \in \textup{Div}_{\mathbb{T}}^{\textup{nef}}(X_{\Sigma})_{\mathbb{Q}}$.
			\item Concave rational support functions $\Psi \in  \mathcal{SF}^{+}(\Sigma, \mathbb{Q})$.
			\item Rational polytopes $\Delta \subset M_{\mathbb{R}}$ with support function $\Psi_{\Delta} \in \mathcal{SF}(\Sigma, \mathbb{Q})$.
		\end{enumerate}
		Moreover, the toric divisor $D$ is ample if and only if the support function $\Psi_D$ is strictly concave.
	\end{thm}
	\begin{proof}
		See Theorem~6.1.7~and~6.1.14~of~\cite{CLS}.
	\end{proof}
	\begin{rem}[Quasi-projective toric varieties]\label{t-q-p}
		By a \textit{projective fan} $\Sigma$ in $N_{\mathbb{R}}$, we mean a complete fan that admits a strictly convex rational support function $\Psi \in \mathcal{SF}(\Sigma, \mathbb{Q})$. Then, the toric variety $X_{\Sigma}/K$ is projective if and only if the fan $\Sigma$ is projective. Indeed, if $X_{\Sigma}/K$ is projective, it is proper and so $\Sigma$ is complete. By Proposition~\ref{toric-divisors}, an ample divisor $E$ on $X_{\Sigma}$ is linearly equivalent to a toric divisor $D$. It follows from Theorem~\ref{nefdiv} that $\Psi_D$ is strictly convex, and therefore $\Sigma$ is projective. The converse is immediate from the same theorem. Now, consider a toric variety $X_{\Sigma_{0}}/K$ which is quasi-projective, hence an open subvariety of a projective variety. The existence of a toric projective model of $X_{\Sigma_0}/K$ means that we can identify $X_{\Sigma_0}/K$ as an open subvariety of a toric projective variety, and this identification is compatible with the toric structure. Then, a natural question is: When does $X_{\Sigma_{0}}/K$ admit a toric projective model? By Proposition~\ref{toricmorphism}, $\pi \colon X_{\Sigma_{0}} \rightarrow X_{\Sigma}$ is a toric projective model if and only if the fan $\Sigma_{0}$ is a subfan of the projective fan $\Sigma$. Here, a fan $\Sigma_0$ is a \textit{subfan} of $\Sigma$ if every cone of $\Sigma_0$ also belongs to $\Sigma$. Finally, Theorem~7.2.4~of~\cite{CLS} gives a partial answer: If the support $|\Sigma_0|$ is convex, the toric variety $X_{\Sigma_{0}}$ admits a projective model.
	\end{rem}
	Now we describe the category of toric projective models $\textup{PM}_{\mathbb{T}}( X_{\Sigma_0}/K)$.  Denote by $\textup{PF}(N_{\mathbb{R}}, \Sigma_0)$ the poset of projective fans $\Sigma \in N_{\mathbb{R}}$ having $\Sigma_0$ as a subfan, ordered by refinements. If the fan $\Sigma_{0}$ is smooth, we write  $\textup{SPM}_{\mathbb{T}}( X_{\Sigma_0}/K)$ and $\textup{SPF}(N_{\mathbb{R}}, \Sigma_0)$ for the subcategories of smooth toric projective models and smooth projective fans, respectively. If $\Sigma_0$ is the fan $\lbrace \lbrace 0 \rbrace \rbrace$ corresponding to the torus $U/K$, we omit it from the notation. We then present the following results.
	\begin{prop}\label{tqpmod}
		Let $X_{\Sigma_0}/K$ be quasi-projective. Then, $\textup{PM}_{\mathbb{T}}( X_{\Sigma_0}/K)$ is equivalent to the cofiltered category $\textup{PF}(N_{\mathbb{R}}, \Sigma_0)$. Assume further that $X_{\Sigma_0}$ is smooth, then $\textup{SPM}_{\mathbb{T}}( X_{\Sigma_0}/K)$ is cofinal in $\textup{PM}_{\mathbb{T}}( X_{\Sigma_0}/K)$.
	\end{prop}
	\begin{proof}
		Both equivalences of categories follow from Proposition~\ref{toricmorphism} and Remark~\ref{t-q-p}. If $X_{\Sigma_0}$ is smooth, Theorem~\ref{toricres} shows the cofinality of $\textup{SPM}_{\mathbb{T}}( X_{\Sigma_0}/K)$ in $\textup{PM}_{\mathbb{T}}( X_{\Sigma_0}/K)$. Indeed, the resolution of singularities fixes the fan $\Sigma_0$. We only need to show that $\textup{PF}(N_{\mathbb{R}}, \Sigma_0)$ is cofiltered. This follows from Remark~\ref{refin}: Given two fans $\Sigma_1$ and $\Sigma_2$ in $\textup{PF}(N_{\mathbb{R}}, \Sigma_0)$, they admit a common refinement $\Sigma_{3}$ fixing $\Sigma_0$. By definition, there are a strictly concave support functions $\Psi_i \in \mathcal{SF}(\Sigma_i, \mathbb{Q})$, $i\in\lbrace1,2\rbrace$. Then, the function $\Psi_1 + \Psi_2 \in \mathcal{SF}(\Sigma_3, \mathbb{Q})$ is strictly concave.
	\end{proof}
	\begin{lem}\label{toricmodeldivisors}
		Let $X_{\Sigma_0}/K$ be quasi-projective and $D$ a model divisor. Then, the assignment $D \mapsto \Psi_D$ is well-defined and induces an isomorphism
		\begin{displaymath}
			\mathcal{SF} \colon \textup{Div}_{\mathbb{T}} (X_{\Sigma_0}/K)_{\textup{mod}} \longrightarrow  \bigcup_{\Sigma \in \textup{PF}(N_{\mathbb{R}}, \Sigma_0 )} \mathcal{SF}(\Sigma, \mathbb{Q}).
		\end{displaymath}
	\end{lem}
	\begin{proof}
		Two toric divisors induce the same model divisor if they coincide after pullback to a common model. Then, by Proposition~\ref{pbsup}, the assignment $D \mapsto \Psi_D$ is well-defined. Proposition~\ref{tqpmod} gives that $\textup{PM}_{\mathbb{T}}(X_{\Sigma_0}/K)$ is equivalent to $\textup{PF}(N_{\mathbb{R}}, \Sigma_0)$. Using the isomorphism in Proposition~\ref{toric-divisors} and taking direct limits, we obtain an isomorphism
		\begin{displaymath}
			\mathcal{SF}\colon \textup{Div}_{\mathbb{T}} (X_{\Sigma_0}/K)_{\textup{mod}} \longrightarrow \varinjlim_{\Sigma \in \textup{PF}(N_{\mathbb{R}}, \Sigma_0 )} \mathcal{SF}(\Sigma, \mathbb{Q}).
		\end{displaymath}
		Note that $\Sigma' \geq \Sigma$ implies that $\mathcal{SF}(\Sigma, \mathbb{Q})  \subset \mathcal{SF}(\Sigma', \mathbb{Q})$. It follows that
		\begin{displaymath}
			\varinjlim_{\Sigma \in \textup{PF}(N_{\mathbb{R}}, \Sigma_0 )} \mathcal{SF}(\Sigma, \mathbb{Q}) \cong \bigcup_{\Sigma \in \textup{PF}(N_{\mathbb{R}}, \Sigma_0 )} \mathcal{SF}(\Sigma, \mathbb{Q}).
		\end{displaymath}
	\end{proof}
	The following lemma characterizes a toric boundary divisor $B$ and the induced boundary norm $\| \cdot \|_B$ in terms of its support function $\Psi_B$.
	\begin{lem}\label{toric-bdry-div}
		Let $X_{\Sigma_0}/K$ be quasi-projective, $\pi \colon X_{\Sigma_0} \rightarrow X_{\Sigma}$ be a toric projective model, and $B$ be a toric divisor on $X_{\Sigma}$ with support function $\Psi_B$. Then, $B$ is a boundary divisor of $X_{\Sigma_0}/K$ if and only if  $\Psi_{B}(x) = 0$ for all $x \in |\Sigma_0|$ and $\Psi_{B} (x)<0$ for all $x \in N_{\mathbb{R}} \setminus |\Sigma_{0}|$. Moreover, if $D$ is a toric model divisor of $X_{\Sigma_0}/K$ with support function $\Psi_D$, we have the identity
		\begin{displaymath}
			\| D \|_{B} = \inf \lbrace \varepsilon \in \mathbb{Q}_{> 0} \, \vert \, |\Psi_{D} (x) | \leq \varepsilon | \Psi_{B} (x) | \textup{ for all } x \in N_{\mathbb{R}} \rbrace.
		\end{displaymath}
	\end{lem}
	\begin{proof}
		We use the cone-orbit correspondence~\ref{cone-orbit} to write $X_{\Sigma_{0}}$ as a union of torus orbits
		\begin{displaymath}
			X_{\Sigma_0} = \bigcup_{\sigma \in \Sigma_0} X_{\sigma} = \bigcup_{\sigma \in \Sigma_0} \bigcup_{\tau \leq \sigma } O(\tau) = \bigcup_{\tau \in \Sigma_0 } O(\tau ).
		\end{displaymath}
		Let $\rho$ be a ray in $\Sigma$. The closure of torus orbit $V(\rho)$ corresponding to $\rho$ can be written as
		\begin{displaymath}
			V(\rho) = \bigcup_{ \rho \leq \tau } O(\tau ).
		\end{displaymath}
		The orbits $O(\tau)$ are pairwise disjoint, therefore $V(\rho) \cap X_{\Sigma_0} = \emptyset$ if and only if $\rho \notin \Sigma_0$. The Weil divisor associated to $B$ is of the form
		\begin{displaymath}
			\sum_{\rho \in  \Sigma(1)} -\Psi_{B}(v_{\rho}) \cdot V(\rho)
		\end{displaymath}
		where $v_{\rho}$ is the ray generator of $\rho$. Then, its support $|B|$ is given by the union of the $V(\rho)$ such that $\Psi_{B}(v_{\rho}) \neq 0$. The divisor $B$ is effective if and only if $\Psi_{B}(v_{\rho}) \leq 0$ for every ray $\rho$. Since $\Psi_{B}$ is linear on each cone of $\Sigma$, we conclude that $B$ is a boundary divisor of $X_{\Sigma_0}/K$ if and only if $\Psi_{B}(x) = 0$ for all $x \in |\Sigma_0|$ and $\Psi_{B} (x)<0$ for all $x \in N_{\mathbb{R}} \setminus |\Sigma_{0}|$. 
		
		Similarly, a toric model divisor $D$ is effective if and only if $\Psi_{D} (x) \leq 0$ for all $x \in N_{\mathbb{R}}$. Therefore, the inequality $- \varepsilon B \leq D \leq \varepsilon B$ holds if and only if $|\Psi_{D} (x)| \leq  \varepsilon | \Psi_{B} (x) |$ for all $x \in N_{\mathbb{R}}$. The desired identity is clear from the definition of boundary norm.
	\end{proof}
	The next step is to extend the notion of support function to toric compactified divisors.
	\begin{lem}\label{toric-comp-SF}
		Let $X_{\Sigma_0}/K$ be a quasi-projective and $D = \lbrace D_n \rbrace_{n \in \mathbb{N}}$ be a toric compactified divisor. Then, the function $\Psi_{D}$ given by the pointwise limit $\lim_{n \in \mathbb{N}} \Psi_{D_n}$ is conical, continuous and satisfies $\Psi_{D}|_{|\Sigma_0|} \in \mathcal{SF}(\Sigma_0, \mathbb{Q})$. Moreover, the assignment $D \mapsto \Psi_D$ induces an injective continuous group morphism $\mathcal{SF} \colon \textup{Div}_{\mathbb{T}} (X_{\Sigma_0} /K) \rightarrow \mathcal{C}(N_{\mathbb{R}})$, where $\mathcal{C}(N_{\mathbb{R}})$ is the space of conical continuous functions on $N_{\mathbb{R}}$.
	\end{lem}
	\begin{proof}
		Let $B$ be a boundary divisor of $X_{\Sigma_{0}}/K$. By the identity in Lemma~\ref{toric-bdry-div}, the Cauchy condition of $ \lbrace D_n \rbrace_{n \in \mathbb{N}}$ means: For every $\varepsilon >0$ there is $n_{\varepsilon}$ such that for all $n,m > n_{\varepsilon}$ and all $x \in N_{\mathbb{R}}$, we have $|\Psi_{D_n} (x) - \Psi_{D_m} (x) | < \varepsilon \cdot | \Psi_{B}(x)|$. Then, the sequence $\lbrace \Psi_{D_n} \rbrace_{n \in \mathbb{N}}$ converges uniformly on each compact subset of $N_{\mathbb{R}}$, and the function $\Psi_{D}$ given by the limit is continuous. The conical property is preserved by limits, therefore $\Psi_D$ is conical. Moreover, the sequence $\lbrace \Psi_{D_n} (x) \rbrace_{n \geq n_{\varepsilon}}$ is constant for all $x \in |\Sigma_0|$. Therefore, the restriction of $\Psi_{D}$ to $|\Sigma_0|$ is a rational support function on the fan $\Sigma_0$. It is clear that the map $\mathcal{SF}$ is a group morphism; now we show that it is continuous. It is enough to show that its restriction to $\textup{Div}_{\mathbb{T}} (X_{\Sigma_0}/K)_{\textup{mod}}$ is continuous. Let $E$ be a toric model divisor satisfying $\| E \|_{B} < \varepsilon $. By Lemma~\ref{toric-bdry-div}, for all $x \in N_{\mathbb{R}}$ we have
		\begin{displaymath}
			| \Psi_E (x) | < \varepsilon  | \Psi_{B}(x) | \leq \varepsilon \cdot \| \Psi_{B} \|_{\mathcal{C}} \cdot \| x \|, \quad  \| \Psi_{B} \|_{\mathcal{C}} \coloneqq \sup \lbrace |\Psi_B(x) | \, | \, \| x \| =1 \rbrace. 
		\end{displaymath}
		Then, $\| \Psi_E \|_{\mathcal{C}} <   \varepsilon \cdot \| \Psi_{B} \|_{\mathcal{C}}$. The map $\mathcal{SF}$ is bounded, and therefore continuous. It remains only to show that $\mathcal{SF}$ is injective. Assume that $D = \lbrace D_n \rbrace_{n \in \mathbb{N}} \neq 0$, we want to show that $\Psi_{D} \neq 0$. The norm of $\| D \|_{B}$ is given by $\lim_n \| D_n \|_B$, and so $\|D\|_B > 0$. If $\|D \|_B = \infty$, we get that $\| D_n \|_B = \infty$ for all but a finite number of $n$'s. On the other hand, we have shown that the sequence of functions $\lbrace \Psi_{D_n}|_{|\Sigma_{0}|} \rbrace_{n \in \mathbb{N}}$ is eventually constant. Combining these two facts, there exists $x \in |\Sigma_0|$ such that $\Psi_{D_n}(x) = \Psi_D (x) \neq 0$ for all but a finite number of $n$'s. If $\|D \|_B = \delta < \infty$, possibly after passing to a subsequence, we may assume that $\|D_n\| < \infty$ for all $n \in \mathbb{N}$. By definition of a limit, we may find a positive integer $n_0$ such that for all $n> n_0$, we have $\| D_n \|_{B} > \delta / 2$. By the identity in Lemma~\ref{toric-bdry-div}, for each $n> n_0$ there exist a point $x_n \in N_{\mathbb{R}}$ such that 
		\begin{displaymath}
			|\Psi_{D_n} (x_n)| > \delta / 2 \cdot | \Psi_{B} (x_n) |.
		\end{displaymath}
		Note that $\Psi_{D_n} (x) = \Psi_{B}(x) =0$ for all $x \in |\Sigma_0|$. Then, we get that $x_n \in N_{\mathbb{R}} \setminus | \Sigma_0 |$. By Lemma~\ref{toric-bdry-div}, $|\Psi_{B} (x_n)| > 0$. Since the sequence $\lbrace D_n \rbrace_{n \in \mathbb{N}}$ is Cauchy, we can find $n_1 > n_0$ such that for all $n, \, m > n_1$  and all $x \in N_{\mathbb{R}}$ we have
		\begin{displaymath}
			\delta / 4 \cdot | \Psi_{B} (x) | \geq |\Psi_{D_m} (x) - \Psi_{D_n} (x)|. 
		\end{displaymath}
		By the reverse triangle inequality and the two previous estimates, for all $m,n > n_1$ we have
		\begin{align*}
			| \Psi_{D_m} (x_n) | & \geq  | |\Psi_{D_n} (x_n)| -  |\Psi_{D_m} (x_n) - \Psi_{D_n}(x_n) | | \\ 
			& = |\Psi_{D_n} (x_n)| -  |\Psi_{D_m} (x_n) - \Psi_{D_n} (x_n) | \\
			& \geq  (\delta / 2 - \delta/4) | \Psi_{B} (x_n) | = \delta /4 \cdot | \Psi_{B} (x_n) | > 0.
		\end{align*}
		We take the limit over $m$ to get $	| \Psi_{D} (x_n) | \geq  \delta / 4 \cdot | \Psi_{B} (x_n) | > 0$ and conclude that $\Psi_{D} \neq 0$.
	\end{proof}
	\begin{dfn}
		Let $D \in \textup{Div}_{\mathbb{T}} (X_{\Sigma_{0}}/K)$ and let $\Psi_D \in \mathcal{C}(N_{\mathbb{R}})$ be the function defined in~\ref{toric-comp-SF}. The function $\Psi_{D}$ is called the \textit{(virtual) support function} of $D$.
	\end{dfn}
	We have the following functoriality property, which shows that the top horizontal map in Lemma~\ref{shrinking} is injective.
	\begin{lem}\label{toric-top-emb}
		Let $\iota \colon X_{\Sigma_0} \rightarrow X_{\Sigma_1}$ be a toric birational morphism between quasi-projective toric varieties over $K$. Then, there is an induced pullback map $\iota^{\ast} \colon \textup{Div}_{\mathbb{T}} (X_{\Sigma_1}/K) \rightarrow \textup{Div}_{\mathbb{T}} (X_{\Sigma_0}/K)$ which is continuous, injective, and satisfies  $\mathcal{SF} \circ \iota^{\ast} = \mathcal{SF}$, where $\mathcal{SF}$ is the map defined in Lemma~\ref{toric-comp-SF}. Moreover, $\iota^{\ast}$ is a topological group embedding if and only if $\iota$ is proper.
	\end{lem}
	\begin{proof}
		Following the proof of Proposition~\ref{toricmorphism}, we may assume that $\iota$ is the identity on the torus $U$. Then, the induced linear map $\phi(\iota)$ is the identity on $N_{\mathbb{R}}$. It follows that every cone $\sigma_0 \in \Sigma_0$ is a subset of a cone $\sigma_1 \in \Sigma_1$. In particular, $|\Sigma_0| \subset |\Sigma_1|$. For each $i =0,1$, let $\Sigma_{i}^{\prime}$ be a projective fan having $\Sigma_i$ as a subfan. Consider the fan $\Sigma = \Sigma_{0}^{\prime} \cdot \Sigma_{1}^{\prime}$ consisting of the intersections $\sigma_0 \cap \sigma_1$, where $\sigma_i \in \Sigma_{i}^{\prime}$. Arguing as in Proposition~\ref{tqpmod}, the fan $\Sigma$ is projective and refines each $\Sigma_{i}^{\prime}$. Since each cone $\sigma_0 \in \Sigma_0$ is a subset of some cone $\sigma_1 \in \Sigma_1$, the cone $\sigma_0 = \sigma_0 \cap \sigma_1$ belongs to $\Sigma$. Therefore, $\Sigma_0$ is a subfan of $\Sigma$. Taking pullbacks with respect to the toric morphisms induced by these refinements, we obtain morphisms $\textup{Div}_{\mathbb{T}} (X_{\Sigma_{i}^{\prime}})_{\mathbb{Q}} \rightarrow \textup{Div}_{\mathbb{T}} (X_{\Sigma})_{\mathbb{Q}}$. Note that the set $F$ consisting of all such fans $\Sigma$ is a cofinal subposet of $\textup{PF}(N_{\mathbb{R}}, \Sigma_0)$. Taking direct limits, we obtain a commutative diagram
		\begin{center}
			\begin{tikzcd}[row sep = tiny]
				\textup{Div}_{\mathbb{T}} (X_{\Sigma_1}/K)_{\textup{mod}} \arrow[r] &\displaystyle \varinjlim_{\Sigma \in F} \textup{Div}_{\mathbb{T}} (X_{\Sigma})_{\mathbb{Q}} & \textup{Div}_{\mathbb{T}} (X_{\Sigma_{0}}/K)_{\textup{mod}} \arrow[l, "{\cong}"'].
			\end{tikzcd}
		\end{center}
		The arrow on the left is injective, and the arrow on the right is an isomorphism. Composing with the inverse of the latter, we obtain a morphism $\iota^{\ast} \colon \textup{Div}_{\mathbb{T}} (X_{\Sigma_1}/K)_{\textup{mod}} \longrightarrow \textup{Div}_{\mathbb{T}} (X_{\Sigma_0}/K)_{\textup{mod}}$ which is injective. The construction of $\iota^{\ast}$ together with Proposition~\ref{pbsup} imply that $\mathcal{SF} \circ \iota^{\ast} = \mathcal{SF}$. We only need to show that it is continuous with respect to the boundary topology. For each $i = 0,1$, let $B_i$ be a boundary divisor of $X_{\Sigma_i}/K$. We can assume that both boundary divisors are defined on the same model of $X_{\Sigma_0}/K$. The condition $|\Sigma_0| \subset |\Sigma_1|$ and Lemma~\ref{toric-bdry-div} imply that $|B_1| \subset |B_0|$. Then, continuity is proven in the same way as in Lemma~\ref{shrinking}. Therefore, we obtain the continuous group morphism $\iota^{\ast} \colon \textup{Div}_{\mathbb{T}} (X_{\Sigma_1}/K) \longrightarrow \textup{Div}_{\mathbb{T}} (X_{\Sigma_0}/K)$, which satisfies $\mathcal{SF} \circ \iota^{\ast} = \mathcal{SF}$. By Lemma~\ref{toric-comp-SF}, the map $\mathcal{SF}$ is injective. Thus, the map $ \iota^{\ast}$ is injective. Finally, the morphism $\iota$ is proper if and only if $|\Sigma_0| = |\Sigma_1|$. In this case, Lemma~\ref{toric-bdry-div} implies that $|B_1| = |B_0|$. In particular, the boundary norms induced by these divisors are equivalent.
	\end{proof}
	Let $X_{\Sigma_{0}}/K$ be quasi-projective. By the previous lemma, the toric open immersion $U \rightarrow X_{\Sigma_0}$ given by the toric structure induces a canonical identification of the group $ \textup{Div}_{\mathbb{T}} (X_{\Sigma_0}/K)$ as a subgroup of the toric compactified divisors of the torus $U/K$. Similar statements hold for the nef cones and the spaces of integrable divisors. Therefore, for intersection-theoretic purposes, we will restrict our discussion to the case of $X_{\Sigma_0} = U$. This has an additional advantage: The boundary topology of $\textup{Div}_{\mathbb{T}} (U/K)$ is easier to work with. In particular, we can compute the image of the map $\mathcal{SF}$, giving a complete description of the group of toric compactified divisors.
	\begin{thm}\label{torgeomdiv}
		The map $\mathcal{SF} \colon  \textup{Div}_{\mathbb{T}} (U/K) \rightarrow \mathcal{C}(N_{\mathbb{R}})$, which to a toric compactified divisor $D$ assigns its support function $\Psi_D$, is an isomorphism of Banach spaces. Here, $\mathcal{C}(N_{\mathbb{R}})$ denotes the space of real-valued, conical and continuous functions on $N_{\mathbb{R}}$.
	\end{thm}
	\begin{proof}
		By Lemma~\ref{toric-comp-SF}, the map $\mathcal{SF}$ is continuous and injective. By Lemma~\ref{toric-bdry-div}, a boundary divisor $B$ of $U/K$ satisfies $\Psi_B (x) < 0$ for all $x \neq 0$. The function $\Psi_B$ is continuous, and therefore bounded on the unit sphere. Define the constants
		\begin{displaymath}
			c \coloneqq \inf \lbrace |\Psi_B (x)| \, \vert \, \| x \| =1 \rbrace, \quad C \coloneqq \sup \lbrace |\Psi_B (x)| \, \vert \, \| x \| =1 \rbrace.
		\end{displaymath}
		By definition, we have an inequality $0< c \leq \| \Psi_B \|_{\mathcal{C}} \leq C $. This, together with the identity
		\begin{displaymath}
			\| D \|_{B} = \inf \lbrace \varepsilon \in \mathbb{Q}_{> 0} \, \vert \, |\Psi_{D} (x) | \leq \varepsilon \cdot | \Psi_{B} (x) | \textup{ for all } x \in N_{\mathbb{R}} \rbrace
		\end{displaymath}
		induce an equivalence of norms. Therefore, $\textup{Div}_{\mathbb{T}} (U/K)$ is isomorphic to its image under $\mathcal{SF}$. Using Lemma~\ref{toricmodeldivisors}, the group of toric compactified divisors is isomorphic to the closure in $\mathcal{C}(N_{\mathbb{R}})$ of the $\mathbb{Q}$-vector space of rational support functions
		\begin{displaymath}
			\mathcal{SF}(N_{\mathbb{R}}, \mathbb{Q}) = \bigcup_{\Sigma \in \textup{PF}(N_{\mathbb{R}})}  \mathcal{SF}(\Sigma,\mathbb{Q}).
		\end{displaymath}
		Here, we used Theorem~\ref{toricres} to show that the set on the right-hand side contains the one on the left: Every complete fan $\Sigma^{\prime}$ admits a smooth projective refinement $\Sigma \in  \textup{PF}(N_{\mathbb{R}})$. By the Density Theorem~\ref{density-1}, the result follows.
	\end{proof}
	As in the projective case (Theorem\ref{nefdiv}), the nefness of toric compactified divisors corresponds with the concavity of support functions.
	\begin{prop}\label{m-d-approx-1}
		Let $D \in \textup{Div}_{\mathbb{T}} (U/K)$. The toric compactified divisor $D$ is nef if and only if its support function $\Psi_{D}$ is concave. Moreover, we can choose the sequence $D = \lbrace D_n \rbrace_{n \in \mathbb{N}}$ of nef toric model divisors to be decreasing, i.e., $D_n \geq D_{n+1}$ for all $n \in \mathbb{N}$.
	\end{prop}
	\begin{proof}
		Suppose that $D$ is nef and let $\lbrace D_n \rbrace_{n \in \mathbb{N}}$ be a Cauchy sequence of nef toric model divisors representing $D$. By Theorem~\ref{nefdiv}, the support function $\Psi_n$ of $D_n$ is concave. Continuity of $\mathcal{SF}$ implies that the sequence of concave support functions $\lbrace \Psi_{D_n} \rbrace_{n \in \mathbb{N}}$ converges pointwise to the conical continuous function $\Psi_{D}$. Therefore, $\Psi_D$ is concave. Conversely, assume that $\Psi_{D}$ is concave. By Theorem~\ref{mono-approx-2}, there exist an increasing sequence $\lbrace \Psi_n \rbrace_{n \in \mathbb{N}}$ in $\mathcal{SF}^{+}(N_{\mathbb{R}}, \mathbb{Q})$ which converges in the $\mathcal{C}$-norm to $\Psi_{D}$. Define $D_n = \mathcal{SF}^{-1}(\Psi_n)$. The inverse $\mathcal{SF}^{-1}$ is continuous, and so the sequence $D= \lbrace D_n \rbrace_{n \in \mathbb{N}}$ of nef toric model divisors converges in the boundary topology to $D$. Note that the map $\mathcal{SF}$ sends decreasing sequences into increasing sequences. Therefore, the constructed sequence $\lbrace D_n \rbrace_{n \in \mathbb{N}}$ is decreasing. 
	\end{proof}
	A conical concave function $\Psi \colon N_{\mathbb{R}} \rightarrow \mathbb{R}$ is entirely determined by its stability set $\textup{stab}(\Psi)$, which is compact and convex. If $\Psi = \Psi_D$ is the support function of a nef compactified divisor $D$, we say that $\Delta_D \coloneqq \textup{stab}(\Psi_D)$ is the compact convex set associated to $D$. Then, the following result is a compactified analog of Theorem~\ref{nefdiv}.
	\begin{thm}\label{nefcorrespondence}
		The assignments $D \mapsto \Psi_D$ and $D \mapsto \Delta_D$ induce isomorphisms of cones which are also homeomorphisms between:
		\begin{enumerate}
			\item The cone $\textup{Div}_{\mathbb{T}}^{\textup{nef}} (U/K)$ of nef toric compactified divisors of $U/K$.
			\item The cone $\mathcal{C}^{+}(N_{\mathbb{R}})$ of concave conical functions on $N_{\mathbb{R}}$.
			\item The cone $\mathcal{K}^{+}(M_{\mathbb{R}})$ of compact convex subsets of $M_{\mathbb{R}}$.
		\end{enumerate}
	\end{thm}
	\begin{proof}
		This follows from Theorem~\ref{torgeomdiv}, Proposition~\ref{m-d-approx-1}, and Corollary~\ref{Kompact}.
	\end{proof}
	Let $A_1, ... , A_r \subset M_{\mathbb{R}}$ be bounded convex sets and $I \subset \lbrace 1,2, \dots, r \rbrace$. We write $A_I$ for the Minkowski sum $\sum_{i\in I} A_i$. The \textit{mixed volume of } $A_1, \dots , A_r$ is given by the formula
	\begin{displaymath}
		\textup{MV}_M ( A_1, \dots , A_r) \coloneqq \sum_{I \subset \lbrace 1,2, \dots, r \rbrace } (-1)^{r -|I|} \textup{Vol}_M (A_I).
	\end{displaymath}
	Then, Theorem~13.4.1~of~\cite{CLS} states that the intersection numbers of nef toric divisors can be computed in terms of the volumes of their associated lattice polytopes. 
	\begin{thm}\label{toricdegree}
		Let $\Sigma$ be a complete fan in $N_{\mathbb{R}}$ and $X_{\Sigma}/K$ its toric variety of dimension $d$. Let $D, \, D_1, \ldots , \, D_d \in \textup{Div}_{\mathbb{T}}^{\textup{nef}}(X_{\Sigma})_{\mathbb{Q}}$ with associated polytopes $\Delta, \Delta_1, \ldots ,\Delta_d$ respectively. Then, the intersection number $D_1 \cdot \ldots \cdot D_d$ is equal to $\textup{MV}_M (\Delta_1, \ldots ,\Delta_d)$. In particular, $D^d = d! \, \textup{Vol}_{M}(\Delta)$.
	\end{thm}
	Using Proposition~\ref{m-d-approx-1} and Beer's~Theorem~\cite{Beer}~below, we can extend Theorem~\ref{toricdegree} to the compactified setting.
	\begin{thm}\label{cont-vol}
		The volume function $ \textup{Vol}_M \colon \mathcal{K}^{+}(M_{\mathbb{R}}) \rightarrow \mathbb{R}_{\geq 0}$ is continuous with respect to the Hausdorff distance $\textup{d}_{\mathcal{H}}$.
	\end{thm}
	\begin{prop}\label{comp-int-num}
		Let $D_1, \ldots , D_d \in \textup{Div}_{\mathbb{T}}^{\textup{nef}}(U/K)$. Write $ D_i = \lbrace D_{i,n} \rbrace_{n \in \mathbb{N}}$ as a Cauchy sequence of nef toric model divisors and let $\Delta_{i}$ and $\Delta_{i,n}$ be the corresponding compact convex sets. Then, the sequence of intersection numbers $\lbrace D_{1,n} \cdot \ldots \cdot D_{d,n} \rbrace_{n \in \mathbb{N}}$ converges and it is given by the formula
		\begin{displaymath}
			D_1 \cdot \ldots \cdot D_d \coloneqq \lim_{n \in \mathbb{N}} D_{1,n} \cdot \ldots \cdot D_{d,n} = \textup{MV}_M (\Delta_{1}, \ldots ,\Delta_{d}).
		\end{displaymath}
		In particular, the degree of a nef toric compactified divisor is given by $D^{d} = d! \, \textup{Vol}_{M} (\Delta_{D})$.
	\end{prop}
	\begin{proof}
		The proof is straightforward. For each $I \subset \lbrace 1,2, \ldots ,d \rbrace$ and $n \in \mathbb{N}$, we consider the divisors $D_I \coloneqq \sum_{i \in I} D_i$ and $D_{I,n} \coloneqq \sum_{i \in I} D_{i,n}$. Since $\textup{Div}_{\mathbb{T}}^{\textup{nef}} (U/K)$ is a closed cone, we know that $D_I$ and $D_{I,n}$ are nef and satisfy
		\begin{displaymath}
			D_I = \lim_{n\in\mathbb{N}} D_{I,n}.
		\end{displaymath}
		Their corresponding compact convex sets are given by $\Delta_I = \sum_{i \in I} \Delta_i$ and $\Delta_{I,n} = \sum_{i \in I} \Delta_{i,n}$, respectively. By Theorem~\ref{toricdegree}, for each $n \in \mathbb{N}$ we have
		\begin{displaymath}
			D_{1,n} \cdot \ldots \cdot D_{d,n}  = \textup{MV}_M ( \Delta_{1,n}, \ldots , \Delta_{d,n}) = \sum_{I \subset \lbrace 1,2, \ldots,d \rbrace } (-1)^{d -|I|} \textup{Vol}_M (\Delta_{I,n}).
		\end{displaymath}
		Applying Theorem~\ref{cont-vol} to each compact convex set $\Delta_I = \lim_n \Delta_{I,n}$, we conclude that
		\begin{displaymath}
			D_1 \cdot \ldots \cdot D_d = \lim_{n \in \mathbb{N}} D_{1,n} \cdot \ldots \cdot D_{d,n} =\lim_{n \in \mathbb{N}} \textup{MV}_M ( \Delta_{1,n}, \ldots , \Delta_{d,n}) = \textup{MV}_M (\Delta_{1}, \ldots ,\Delta_{d}).
		\end{displaymath}
	\end{proof}
	Then, we state the following toric analog of Theorem~\ref{int-prod-K}. 
	\begin{thm}\label{comp-int-pair}
		The intersection pairing $\textup{Div}_{\mathbb{T}}^{\textup{nef}} (U/K)^d \rightarrow\mathbb{R}$ given by the formula
		\begin{displaymath}
			D_1 \cdot ... \cdot D_d =  \textup{MV}_M (\Delta_{D_1}, \ldots ,\Delta_{D_d}), \quad  \Delta_i = \textup{stab}(\Psi_{D_i}),
		\end{displaymath}
		extends continuously the intersection pairing on nef toric model divisors. By linearity, it is extended further to a multilinear, symmetric pairing $\textup{Div}_{\mathbb{T}}^{\textup{int}} (U/K)^d \longrightarrow \mathbb{R}$ on the space of integrable divisors.
	\end{thm}
	\begin{proof}
		Let $D \in \textup{Div}_{\mathbb{T}}^{\textup{int}}(U/K)$. Write $D = D_1 - D_2$, where each $D_i=\lbrace D_{i,n} \rbrace_{n \in \mathbb{N}}$ is a Cauchy sequence of nef toric model divisors. For each $n \in \mathbb{N}$, we have
		\begin{displaymath}
			(D_{1,n} - D_{2,n})^d = \sum_{k =0}^{d} (-1)^{d-k} \binom{d}{k} D_{1,n}^{k} \cdot D_{2,n}^{d-k}.
		\end{displaymath}
		Observe that the sum in the right is finite, and the number of terms does not depend on $n$. By Proposition~\ref{comp-int-num}, it follows that
		\begin{align*}
			(D_{1} - D_{2})^d & = \lim_{n \in \mathbb{N}}  \sum_{k =0}^{n} (-1)^{n-k} \binom{n}{k} D_{1,n}^{k} \cdot D_{2,n}^{n-k} = \sum_{k =0}^{d} (-1)^{d-k} \binom{d}{k} D_{1}^{k} \cdot D_{2}^{d-k}.
		\end{align*}
	\end{proof}
	We finish this section with a comparison between our toric compactified divisors and the toric b-divisors~of~\cite{Bot19}.
	\begin{rem}[Toric compactified divisors and toric b-divisors]
		Let $\Sigma$ be a complete and smooth fan in $N_{\mathbb{R}}$ and denote by $R(\Sigma)$ the directed poset of smooth and complete fans $\Sigma' \geq \Sigma$, ordered by refinement. The groups of toric Cartier and Weil b-divisors on $X_{\Sigma}$ are defined as
		\begin{displaymath}
			\textup{Ca}(X_{\Sigma})_{\mathbb{Q}} \coloneqq \varinjlim_{\Sigma' \in R(\Sigma)} \textup{Div}_{\mathbb{T}}(X_{\Sigma'})_{\mathbb{Q}}, \quad \textup{We}(X_{\Sigma})_{\mathbb{Q}}  \coloneqq \varprojlim_{\Sigma' \in R(\Sigma)} \textup{Div}_{\mathbb{T}}(X_{\Sigma'})_{\mathbb{Q}},
		\end{displaymath}
		Where maps in the limits are pullbacks and pushforwards induced by toric birational morphisms, respectively. By Remark~2.5~of~\cite{Bot19}, there is an injective morphism $\textup{Ca}(X_{\Sigma_0})_{\mathbb{Q}} \rightarrow \textup{We}(X_{\Sigma_0})_{\mathbb{Q}}$. Theorem~\ref{toricres} implies that the directed set $\textup{SPF}(N_{\mathbb{R}})$ is cofinal in $R(\Sigma)$. Therefore, we have an isomorphism 
		\begin{displaymath}
			\textup{Ca}(X_{\Sigma})_{\mathbb{Q}} \cong  \textup{Div}_{\mathbb{T}} (U/K)_{\textup{mod}}.
		\end{displaymath}
		Now, we compare toric Weil b-divisors and toric compactified divisors. From Lemma~2.9~of~\cite{Bot19}, to each toric Weil b-divisor one can associate a $\mathbb{Q}$-conical function $\Phi_D \colon N_{\mathbb{Q}} \rightarrow \mathbb{Q}$. The assignment $D \mapsto \Phi_D$ induces an isomorphism between $\textup{We}(X_{\Sigma_0})_{\mathbb{Q}}$ and the $\mathbb{Q}$-vector space of $\mathbb{Q}$-conical functions $\Phi\colon N_{\mathbb{Q}} \rightarrow \mathbb{Q}$. After extending scalars to $\mathbb{R}$, elements $D\in \textup{We}(X_{\Sigma_0})_{\mathbb{R}}$ correspond to $\mathbb{Q}$-conical functions $\Phi\colon N_{\mathbb{Q}} \rightarrow \mathbb{R}$. On the other hand, we have an isomorphism $\mathcal{SF} \colon \textup{Div}_{\mathbb{T}} (U/K) \longrightarrow \mathcal{C}(N_{\mathbb{R}})$. Since every support function $\Psi_D$ is continuous, it is determined uniquely by its values on the dense subset $N_{\mathbb{Q}}$. Therefore, we get an injective map $\textup{Div}_{\mathbb{T}} (U/K) \rightarrow \textup{We}(X_{\Sigma_0})_{\mathbb{R}}$ inducing the identification
		\begin{displaymath}
			\textup{Div}_{\mathbb{T}} (U/K) \cong \lbrace D \in \textup{We}(X_{\Sigma_0})_{\mathbb{R}} \, \vert \, \Phi_{D} \colon N_{\mathbb{Q}} \rightarrow \mathbb{R}  \textup{ extends continuously to } N_{\mathbb{R}} \rbrace.
		\end{displaymath}
		Finally, Remark~3.7~of~\cite{Bot19} shows that nefness implies concavity. In contrast, Theorem~\ref{m-d-approx-1} states that for a toric compactified divisor, nefness is equivalent to concavity. Moreover, every conical concave function is the support function of a toric compactified divisor. Finally, the formula in Theorem~3.10~of~\cite{Bot19} for intersection numbers of nef toric b-divisors agrees with our formula in Theorem~\ref{comp-int-pair}.
	\end{rem}
	\subsection{The case over a DVR}\label{DVR-1}
	This subsection describes similar constructions as the ones appearing in Subsection~\ref{3-1}, where we consider a discrete valuation ring as our base scheme (DVR for short). First, we give a summary of the theory of toric schemes over a DVR. These results first appeared in Chapter~4~of~\cite{KKMS}, while Chapter~3~of~\cite{BPS} serves as a modern reference. Then, we extend these results to the compactified setting.
	
	Throughout this subsection, $\mathcal{O}_K$ is a discrete valuation ring with field of fractions $K$. Let $\mathfrak{m}$ be its unique maximal ideal, $\varpi$ be a generator of $\mathfrak{m}$, and $k = \mathcal{O}_K / \mathfrak{m}$ be its residue field. We remind the reader of the notation introduced at the start of Sections~3.1~and~3.2: The split $d$-dimensional torus $\mathcal{U}/\mathcal{O}_K$, its generic fibre $U/K$, and the lattice $M$ (resp. $N$) of (co)characters of $U$. 
	\begin{cons}[The toric scheme of a fan]
		For each fan $\widetilde{\Sigma}$ in $N_{\mathbb{R}} \times \mathbb{R}_{\geq 0}$ there is an associated toric scheme $\mathcal{X}_{\widetilde{\Sigma}}/\mathcal{O}_K$. We sketch its construction below:
		\begin{enumerate}
			\item For each cone $\sigma  \in \widetilde{\Sigma}$, we let $\sigma^{\vee} \subset M_{\mathbb{R}} \times \mathbb{R}$ be its dual cone, $\widetilde{M}_{\sigma} \coloneqq (M \times \mathbb{Z}) \cap \sigma^{\vee}$ be the associated semigroup, and $\mathcal{O}_K [ \widetilde{M}_{\sigma} ]$ be the induced semigroup algebra. Since $\sigma$ is contained in the upper half-space $N_{\mathbb{R}} \times \mathbb{R}_{\geq 0}$, we know that $(0,1) \in \sigma^{\vee}$. Then, the affine toric scheme $\mathcal{X}_{\sigma}$ associated to $\sigma$ is defined as
			\begin{displaymath}
				\mathcal{X}_{\sigma} \coloneqq \textup{Spec}( \mathcal{O}_{K} [ \widetilde{M}_{\sigma} ] / ( \chi^{(0,1)} - \varpi) ).
			\end{displaymath}
			Note that the toric scheme associated to the vertical ray $\tau_{\textup{vert}} \coloneqq \mathbb{R}_{\geq 0} \cdot (0,1)$ is canonically isomorphic to the torus $\mathcal{U}/\mathcal{O}_K$.
			\item For each face $\tau$ of $\sigma$, there is an induced open immersion $\mathcal{X}_{\tau} \rightarrow \mathcal{X}_{\sigma}$ which identifies $\mathcal{X}_{\tau}$ as an open subscheme of $\mathcal{X}_{\sigma}$. We use these identifications to glue the family of affine toric schemes $\lbrace \mathcal{X}_{\sigma} \rbrace_{\sigma \in  \widetilde{\Sigma} } $ and obtain a scheme $\mathcal{X}_{\widetilde{\Sigma}} / \mathcal{O}_K$. The action of $\mathcal{U}$ on $\mathcal{X}_{\widetilde{\Sigma}}$ is constructed as in the case of toric varieties.
			\item We regard $\Sigma \coloneqq \lbrace \sigma = \widetilde{\sigma} \cap N_{\mathbb{R}} \times \lbrace 0 \rbrace \, \vert \, \widetilde{\sigma} \in \widetilde{\Sigma} \rbrace$ as a fan in $N_{\mathbb{R}}$. Its associated toric variety $X_{\Sigma}/K$ is isomorphic to the generic fibre of $\mathcal{X}_{\widetilde{\Sigma}} / \mathcal{O}_K$.
		\end{enumerate}
		Properties of the toric scheme $\mathcal{X}_{\widetilde{\Sigma}} / \mathcal{O}_K$ correspond to properties of the fan $\widetilde{\Sigma}$. For instance, the toric scheme $\mathcal{X}_{\widetilde{\Sigma}} / \mathcal{O}_K$ is proper if and only if the fan $\widetilde{\Sigma}$ is complete in $N_{\mathbb{R}} \times \mathbb{R}_{\geq 0}$. By this we mean $|\widetilde{\Sigma}| = N_{\mathbb{R}}\times \mathbb{R}_{\geq 0}$. This is not standard notation, but it should not cause any confusion.
	\end{cons}
	Now, we discuss toric morphisms. Let $\phi \colon  (N_1)_{\mathbb{R}} \times \mathbb{R}_{\geq 0} \rightarrow (N_2)_{\mathbb{R}} \times \mathbb{R}_{\geq 0}$ be an $\mathbb{R}$-linear map which is compatible with the lattice structures and the fans $\widetilde{\Sigma}_1$ and $\widetilde{\Sigma}_2$ (see Definition~\ref{compatible}). Then, the map $\phi$ induces a toric morphism $f \colon \mathcal{X}_{\widetilde{\Sigma}_1} \rightarrow \mathcal{X}_{\widetilde{\Sigma}_2}$. The map $f$ is proper if and only if $\phi(|\widetilde{\Sigma}_1|) =  |\widetilde{\Sigma}_2|$. The next proposition summarizes the results appearing in~\S~IV.3.(e)-(g)~of~\cite{KKMS}, describing the category of proper toric schemes over $\mathcal{O}_K$.
	\begin{prop}\label{toric-over-S}
		The assignments $\widetilde{\Sigma} \mapsto \mathcal{X}_{\widetilde{\Sigma}}$ and $\phi \mapsto f$ induce an equivalence between the category of complete fans in $N_{\mathbb{R}}\times \mathbb{R}_{\geq 0}$ with compatible linear maps, and the category of proper toric schemes over $\mathcal{O}_K$ with toric morphisms. In particular, proper birational morphisms of proper toric schemes over $\mathcal{O}_K$ correspond (up to isomorphism) with refinements of fans.
	\end{prop}
	\begin{rem}
		We have a cone-orbit correspondence for the toric scheme $\mathcal{X}_{\widetilde{\Sigma}}/\mathcal{O}_K$. As in the case over a field (Construction~\ref{cone-orbit}), cones of dimension $n$ correspond to torus orbits of dimension $d-n$. The orbits $\textbf{O}(\sigma)$ are divided into two kinds, depending on the cone $\sigma$ being contained in $N_{\mathbb{R}} \times \lbrace 0 \rbrace$ or not. In the first case, $\textbf{O}(\sigma)$ identifies with a torus orbit in the toric variety $X_{\Sigma}/K$ given by the generic fibre, and its Zariski closure $\textbf{V}(\sigma)$ in $\mathcal{X}_{\widetilde{\Sigma}}$ is a toric scheme over $\mathcal{O}_K$ of relative dimension $d-n$. This type of orbit is called \textit{horizontal}. In the second case, the torus orbit $\textbf{O}(\sigma)$ is contained in the special fibre of $\mathcal{X}_{\widetilde{\Sigma}}/\mathcal{O}_K$, and its Zariski closure $\textbf{V}(\sigma)$ has the structure of a toric variety of dimension $n$ over the residue field $k$. This type of orbit is called \textit{vertical}. In particular, the irreducible components of the special fibre of $\mathcal{X}_{\widetilde{\Sigma}}/\mathcal{O}_K$ are toric varieties over the residue field $k$, and correspond with the rays of $\widetilde{\Sigma}$ which are not contained in $N_{\mathbb{R}} \times \lbrace 0 \rbrace$. For details, see~Section~3.5,~pp.~97--99~of~\cite{BPS}.
	\end{rem}
	Let $\widetilde{\Sigma}$ be a complete fan in $N_{\mathbb{R}} \times \mathbb{R}_{\geq 0}$. Given an integral toric Cartier divisor $\mathcal{D}$ on $\mathcal{X}_{\widetilde{\Sigma}}$, we can write $\mathcal{D} = \lbrace (\mathcal{X}_{\sigma}, \chi^{(-m_{\sigma},-\ell_{\sigma})}) \rbrace_{\sigma \in \widetilde{\Sigma}}$, where $(m_{\sigma}, \ell_{\sigma}) \in M \times \mathbb{Z}$. Then, the divisor $\mathcal{D}$ has an associated rational (virtual) support function $\Phi_{\mathcal{D}} \in \mathcal{SF}(\widetilde{\Sigma}, \mathbb{Q})$ which, resticted to a cone $\sigma \in \widetilde{\Sigma}$, is given by the equation $\Phi_{\mathcal{D}}(x)\coloneqq \langle (m_{\sigma}, \ell_{\sigma}), x \rangle$. This is extended by linearity to all toric Cartier $\mathbb{Q}$-divisors. Many properties of the divisor $\mathcal{D}$ can be read from its support function $\Phi_{\mathcal{D}}$. For instance, the restriction $D$ of $\mathcal{D}$ to the generic fibre $X_{\Sigma}$ of $\mathcal{X}_{\widetilde{\Sigma}}/\mathcal{O}_K$ has support function $\Psi_D = \Phi_{\mathcal{D}}|_{N_{\mathbb{R}}\times \lbrace 0 \rbrace}$. Now, consider the polyhedral complex $\Pi \coloneqq \lbrace \sigma \cap N_{\mathbb{R}} \times \lbrace 1 \rbrace \, \vert \, \sigma \in \widetilde{\Sigma} \rbrace$. Its support $|\Pi|$ is equal to $N_{\mathbb{R}} \times \lbrace 1 \rbrace$, so the restriction of $\Phi_{\mathcal{D}}$ to $N_{\mathbb{R}}\times \lbrace 1 \rbrace$ identifies with a function $\gamma_{\mathcal{D}} \colon N_{\mathbb{R}} \rightarrow \mathbb{R}$, which is rational piecewise affine on  $\Pi$. The function $\Psi_D$ can also be recovered from $\gamma_{\mathcal{D}}$, it is equal the recession function $\textup{rec}(\gamma_{\mathcal{D}})$ of $\gamma_{\mathcal{D}}$ (see~Definition~\ref{G-def}). The following results summarize the properties of these constructions.
	\begin{prop}\label{toric-divisors-DVR-1}
		Let $\mathcal{X}_{\widetilde{\Sigma}} / \mathcal{O}_K$ be a proper toric scheme, and $\mathcal{D}$ be a toric divisor on $\mathcal{X}_{\widetilde{\Sigma}}$. Then:
		\begin{enumerate}
			\item The assignment $\mathcal{D} \mapsto \Phi_{\mathcal{D}}$  induces an isomorphism $\mathcal{SF} \colon \textup{Div}_{\mathbb{T}}(\mathcal{X}_{\widetilde{\Sigma}})_{\mathbb{Q}} \rightarrow \mathcal{SF}( \widetilde{\Sigma}, \mathbb{Q})$, where $ \mathcal{SF}( \widetilde{\Sigma}, \mathbb{Q})$ is the space of rational support functions on the fan $\widetilde{\Sigma}$.
			\item The assignment $\mathcal{D} \mapsto \gamma_{\mathcal{D}}$ induces an isomorphism $\mathcal{PA} \colon \textup{Div}_{\mathbb{T}}(\mathcal{X}_{\widetilde{\Sigma}})_{\mathbb{Q}} \rightarrow \mathcal{PA}( \Pi, \mathbb{Q})$, where $ \mathcal{PA}( \Pi, \mathbb{Q})$ is the space of rational piecewise affine functions on the polyhedral complex $\Pi$.
		\end{enumerate}
		Moreover, if $D$ is the restriction of $\mathcal{D}$ to the generic fibre $X_{\Sigma}$ of $\mathcal{X}_{\widetilde{\Sigma}}$, its support function satisfies $\Psi_D = \Phi_{\mathcal{D}}|_{N_{\mathbb{R}}\times \lbrace 0 \rbrace} = \textup{rec}(\gamma_{\mathcal{D}})$, where $\textup{rec}(\gamma_{\mathcal{D}})$ is the recession function of $\gamma_{\mathcal{D}}$ (Definition~\ref{G-def}).
	\end{prop}
	\begin{proof}
		The first isomorphism is~\S~IV.3.(h)~of~\cite{KKMS}. The second isomorphism follows the first, together with Corollary~2.1.13~of~\cite{BPS}. The last assertion is immediate from~Proposition~2.6.9~of~\cite{BPS}. See also Theorems~3.6.7~and~3.6.8~of~\cite{BPS}.
	\end{proof}
	The isomorphism $\mathcal{SF}$ satisfies the following functorial property.
	\begin{prop}\label{toric-divisors-DVR-2}
		Let $f \colon \mathcal{X}_{\widetilde{\Sigma}_1} \rightarrow \mathcal{X}_{\widetilde{\Sigma}_2}$ be a dominant toric morphism of proper toric schemes over $\mathcal{O}_K$ induced by the linear map $\phi$, and $\mathcal{D}$ be a toric divisor on $\mathcal{X}_{\widetilde{\Sigma}_2}$ with support function $\Phi_{\mathcal{D}}$. Then, the pullback $f^{\ast} \mathcal{D}$ is a toric divisor on $\mathcal{X}_{\widetilde{\Sigma}_1}$ and its support function is $\Phi_{\mathcal{D}} \circ \phi$. In particular, if $\phi$ is the identity, then $\Phi_{\mathcal{D}} = \Phi_{f^{\ast} \mathcal{D}}$ and $\gamma_{\mathcal{D}} = \gamma_{f^{\ast} \mathcal{D}}$.
	\end{prop}
	\begin{proof}
		This is Proposition~3.6.15~of~\cite{BPS}.
	\end{proof}
	The following result is an analog of Proposition~\ref{nefdiv} for the DVR case. It relates the positivity properties of a toric divisor to the concavity properties of its support function.
	\begin{prop}\label{nef-dvr}
		Let $\mathcal{X}_{\widetilde{\Sigma}} / \mathcal{O}_K$ be proper. Let $\mathcal{D} \in \textup{Div}_{\mathbb{T}}(\mathcal{X}_{\widetilde{\Sigma}})_{\mathbb{Q}}$ with corresponding functions $\Phi_{\mathcal{D}}$ and $\gamma_{\mathcal{D}}$. Then, the following are equivalent:
		\begin{enumerate}
			\item The divisor $\mathcal{D}$ is relatively nef, i.e, $\mathcal{D} \cdot C \geq 0$ for every vertical curve $C$. 
			\item The function $\Phi_{\mathcal{D}}$ is concave.
			\item The function $\gamma_{\mathcal{D}}$ is concave.
		\end{enumerate}
		Moreover, the toric divisor $\mathcal{D}$ is ample if and only if its support function $\Phi_{\mathcal{D}}$ is strictly concave.
	\end{prop}
	\begin{proof}
		This is Theorem~3.7.1~of~\cite{BPS}. See also~\S~IV.3.(k)~of~\cite{KKMS}.
	\end{proof}
	As in the previous subsection, the category of toric projective models of a quasi-projective toric scheme $\mathcal{X}_{\widetilde{\Sigma}_0}/\mathcal{O}_K$ admits a combinatorial description. First, we introduce notation and then state the description below. We say that a fan $\widetilde{\Sigma}$ in $N_{\mathbb{R}} \times \mathbb{R}_{\geq 0}$ is \textit{projective} if it is complete and it admits a strictly concave rational support function. The above result implies that a toric scheme $\mathcal{X}_{\widetilde{\Sigma}}/\mathcal{O}_K$ is projective if and only if the fan $\widetilde{\Sigma}$ is projective. Then, we let $\textup{PF}(N_{\mathbb{R}}\times \mathbb{R}_{\geq 0}, \widetilde{\Sigma}_0)$ be the poset of projective fans in $N_{\mathbb{R}} \times \mathbb{R}_{\geq 0}$ having $\widetilde{\Sigma}_0$ as a subfan, ordered by refinement.
	\begin{prop}\label{PM-over-DVR}
		Let $\mathcal{X}_{\widetilde{\Sigma}_0}/\mathcal{O}_K$ be quasi-projective. Then, the category $\textup{PM}_{\mathbb{T}}(\mathcal{X}_{\widetilde{\Sigma}_0}/\mathcal{O}_K)$ of toric projective models is equivalent to the cofiltered category $\textup{PF}(N_{\mathbb{R}}\times \mathbb{R}_{\geq 0}, \widetilde{\Sigma}_0)$.
	\end{prop}
	\begin{proof}
		Using Propositions~\ref{toric-over-S}, this is proven in the same way as Proposition~\ref{tqpmod}.
	\end{proof}
	Now, we use these results to compute the group of toric model divisors and characterize the toric boundary divisors of a given quasi-projective toric scheme.
	\begin{lem}\label{toric-model-div-S}
		Let $\mathcal{X}_{\widetilde{\Sigma}_0}/\mathcal{O}_K$ be quasi-projective and $\mathcal{D}$ be a model divisor. Then, the assignments $\mathcal{D} \mapsto \Phi_{\mathcal{D}}$ and $\mathcal{D}\mapsto \gamma_{\mathcal{D}}$ are well-defined and induce isomorphisms
		\begin{align*}
			\mathcal{SF} \colon \textup{Div}_{\mathbb{T}}(\mathcal{X}_{\widetilde{\Sigma}_0}/\mathcal{O}_K)_{\textup{mod}} & \longrightarrow  \bigcup_{\widetilde{\Sigma} \in \textup{PF}(N_{\mathbb{R}} \times \mathbb{R}_{\geq 0}, \widetilde{\Sigma}_0)} \mathcal{SF}( \widetilde{\Sigma}, \mathbb{Q}) \\
			\mathcal{PA} \colon \textup{Div}_{\mathbb{T}}(\mathcal{X}_{\widetilde{\Sigma}_0}/\mathcal{O}_K)_{\textup{mod}} & \longrightarrow  \bigcup_{\widetilde{\Sigma} \in \textup{PF}(N_{\mathbb{R}} \times \mathbb{R}_{\geq 0}, \widetilde{\Sigma}_0)} \mathcal{PA}( \Pi, \mathbb{Q})
		\end{align*}
		where the polyhedral complex $\Pi$ is obtained from intersecting $\widetilde{\Sigma}$ with the hyperplane $N_{\mathbb{R}}\times \lbrace 0 \rbrace$.
	\end{lem}
	\begin{proof}
		By Proposition~\ref{toric-divisors-DVR-2}, the assignments $\mathcal{D} \mapsto \Phi_{\mathcal{D}}$ and $\mathcal{D}\mapsto \gamma_{\mathcal{D}}$ are well-defined. The isomorphisms follow from Proposition~\ref{toric-divisors-DVR-1} and taking direct limits. See the proof of Lemma~\ref{toricmodeldivisors}.
	\end{proof}
	\begin{lem}\label{toric-bdy-top-DVR}
		Let $\mathcal{X}_{\widetilde{\Sigma}_0}/\mathcal{O}_K$ be quasi-projective, $\pi \colon \mathcal{X}_{\widetilde{\Sigma}_0} \rightarrow \mathcal{X}_{\widetilde{\Sigma}}$ be a toric projective model, and $\mathcal{B}$ be a toric divisor on $\mathcal{X}_{\widetilde{\Sigma}}$ with associated functions $\Phi_{\mathcal{B}}$ and $\gamma_{\mathcal{B}}$. Then, $\mathcal{B}$ is a boundary divisor of $\mathcal{X}_{\widetilde{\Sigma}_0}/\mathcal{O}_K$ if and only if $\Phi_{\mathcal{B}}(x) = 0$ for all $x \in |\widetilde{\Sigma}_0|$ and  $\Phi_{\mathcal{B}}(x) < 0$ otherwise. Moreover, for each toric model divisor $\mathcal{D}$ with associated functions $\Phi_{\mathcal{D}}$ and $\gamma_{\mathcal{D}}$, we have the following identities
		\begin{align*}
			\| \mathcal{D} \|_{\mathcal{B}} & = \inf \lbrace \varepsilon \in \mathbb{Q}_{>0} \, \vert \, | \Phi_{\mathcal{D}}(w) | \leq \varepsilon \cdot  | \Phi_{\mathcal{B}}(w) | \textup{ for all } w \in N_{\mathbb{R}} \times \mathbb{R}_{\geq 0} \rbrace \\
			& = \inf \lbrace \varepsilon \in \mathbb{Q}_{>0} \, \vert \, | \gamma_{\mathcal{D}}(x) | \leq \varepsilon \cdot  | \gamma_{\mathcal{B}}(x) | \textup{ for all } x \in N_{\mathbb{R}} \rbrace.
		\end{align*}
	\end{lem}
	\begin{proof}
		The characterization of toric boundary divisors is proven in the same way as Lemma~\ref{toric-bdry-div}. It remains to show the equality of infima. Let $\varepsilon >0$ and suppose we have a functional inequality $|\Phi_{\mathcal{D}}| \leq  \varepsilon \cdot |\Phi_{\mathcal{B}}|$. By definition, we have $\gamma_{\mathcal{D}} = \Phi_{\mathcal{D}} |_{N_{\mathbb{R}}\times \lbrace 1 \rbrace}$ (and the same statement for $\mathcal{B}$). Therefore, we get a functional inequality $|\gamma_{\mathcal{D }}| \leq  \varepsilon \cdot |\gamma_{\mathcal{B}}|$. Conversely, by the conical property of the support functions, the functional inequality $|\gamma_{\mathcal{D }}| \leq  \varepsilon \cdot |\gamma_{\mathcal{B}}|$ implies that $|\Phi_{\mathcal{D}}(x)| \leq  \varepsilon \cdot |\Phi_{\mathcal{B}}(x)|$ for all $x \in N_{\mathbb{R}} \times \mathbb{R}_{>0}$. This inequality is extended to $N_{\mathbb{R}} \times \mathbb{R}_{\geq0}$ by continuity of the support functions. Taking the infimum over $\varepsilon$, we get the desired identity.
	\end{proof}
	For each toric compactified divisor $\mathcal{D}$, we will associate a support function $\Phi_{\mathcal{D}}$. As in the case over a field (Lemma~\ref{toric-comp-SF}), this assignment involves the space $\mathcal{C}(N_{\mathbb{R}} \times \mathbb{R}_{\geq 0})$ of continuous conical functions on $N_{\mathbb{R}} \times \mathbb{R}_{\geq 0}$, which becomes a Banach space when endowed with the $\mathcal{C}$-norm (see Definition~\ref{C-norm}). We will also attach a function $\gamma_{\mathcal{D}}$ on $N_{\mathbb{R}}$, which will no longer be piecewise affine. Instead, it will be given by the sum of a conical function and a sublinear function. This will involve the Banach space $\mathcal{G}(N_{\mathbb{R}})$, introduced in Definition~\ref{G-def}. The following results are analogs of Lemma~\ref{toric-comp-SF}~and~\ref{toric-top-emb} in the DVR setting.
	\begin{lem}\label{sup-fun-dvr-1}
		Let $\mathcal{X}_{\widetilde{\Sigma}_0}/\mathcal{O}_K$ be quasi-projective, $X_{\Sigma_{0}}/K$ be its generic fibre, and $\mathcal{D}= \lbrace \mathcal{D}_n \rbrace_{n \in \mathbb{N}}$ be a toric compactified divisor of $\mathcal{X}_{\widetilde{\Sigma}_0}/\mathcal{O}_K$. Then:
		\begin{enumerate}
			\item The function $\Phi_{\mathcal{D}}$ given by the pointwise limit $\lim_{n \in \mathbb{N}} \Phi_{\mathcal{D}_n}$ is conical, continuous, and its restriction $\Phi_{\mathcal{D}}|_{|\widetilde{\Sigma}_{0}|}$ is a rational support function on the fan $\widetilde{\Sigma}_0$.
			\item The sequence $\lbrace D_n \rbrace_{n\in \mathbb{N}}$ given by the generic fibres of the divisors $\mathcal{D}_n$ determines a toric compactified divisor $D$ of $X_{\Sigma_{0}}/K$. Its support function satisfies $\Psi_D = \Phi_{\mathcal{D}}|_{N_{\mathbb{R}}\times \lbrace 0 \rbrace}$.
			\item The function $\gamma_{\mathcal{D}} \coloneqq  \Phi_{\mathcal{D}}|_{N_{\mathbb{R}}\times \lbrace 1 \rbrace}$ is continuous and the difference  $\gamma_{\mathcal{D}} - \Psi_D$ is sublinear.
		\end{enumerate}
		Moreover, the assignments $\mathcal{D} \mapsto \Phi_{\mathcal{D}}$ and $\mathcal{D} \mapsto \gamma_{\mathcal{D}}$ induce continuous injective group morphisms $\mathcal{SF} \colon \textup{Div}_{\mathbb{T}}(\mathcal{X}_{\widetilde{\Sigma}_0}/\mathcal{O}_K)  \rightarrow \mathcal{C}(N_{\mathbb{R}} \times \mathbb{R}_{\geq 0})$ and $\mathcal{PA} \colon \textup{Div}_{\mathbb{T}}(\mathcal{X}_{\widetilde{\Sigma}_0}/\mathcal{O}_K) \rightarrow \mathcal{G}(N_{\mathbb{R}})$, respectively.
	\end{lem}
	\begin{proof}
		To show \textit{(i)}, we copy the proof of Lemma~\ref{toric-comp-SF}. Additionally, we get that the map $\mathcal{SF}$ is continuous and injective. For \textit{(ii)}, we use Proposition~\ref{toric-divisors-DVR-1} to see that $\lbrace D_n \rbrace_{n\in \mathbb{N}}$ is a sequence of toric model divisors of $X_{\Sigma_0}/K$. By Lemmas~\ref{toric-bdry-div}~and~\ref{toric-bdy-top-DVR}, the restriction to the generic fibre of a toric boundary divisor $\mathcal{B}$ of $\mathcal{X}_{\widetilde{\Sigma}_0}/\mathcal{O}_K$ gives a toric boundary divisor $B$ of $X_{\Sigma_0}/K$. In particular, the induced sequence $\lbrace D_n \rbrace_{n\in \mathbb{N}}$ is Cauchy in the boundary norm induced by $B$. Hence, it determines a toric compactified divisor $D$ of $X_{\Sigma_0}/K$. Additionally, the support function of $D_n$ is given by $\Psi_{D_n} =  \Phi_{\mathcal{D}_n}|_{N_{\mathbb{R}}\times \lbrace 0 \rbrace}$. Then, Lemma~\ref{toric-comp-SF} shows that $\Psi_D = \Phi_{\mathcal{D}}|_{N_{\mathbb{R}}\times \lbrace 0 \rbrace}$. Now, we prove \textit{(iii)}. For each $n \in \mathbb{N}$, let $\gamma_{\mathcal{D}_n}$ be the piecewise affine function associated to $\mathcal{D}_n$. By definition, $\gamma_{\mathcal{D}_n} = \Phi_{\mathcal{D}_n}|_{N_{\mathbb{R}}\times \lbrace 1 \rbrace} $. It follows that the function
		\begin{displaymath}
			\gamma_{\mathcal{D}_n} - \Psi_{D_n} = \Phi_{\mathcal{D}_n}|_{N_{\mathbb{R}}\times \lbrace 1 \rbrace} - \Phi_{\mathcal{D}_n}|_{N_{\mathbb{R}}\times \lbrace 0 \rbrace}
		\end{displaymath}
		is bounded on $N_{\mathbb{R}}$. It follows that $\gamma_{\mathcal{D}_n} \in \mathcal{G}(N_{\mathbb{R}})$ for every $n \in \mathbb{N}$. On the other hand, the sequence $\lbrace \mathcal{D}_n \rbrace$ is Cauchy in the boundary norm. Then, the identity in Lemma~\ref{toric-bdy-top-DVR} shows that the sequence $\lbrace \gamma_{\mathcal{D}_n} \rbrace_{n\in\mathbb{N}}$ is Cauchy in the $\mathcal{G}$-norm (See Definition~\ref{G-def}) and its limit is $\gamma_{\mathcal{D}}$. Since $\mathcal{G}(N_{\mathbb{R}})$ is a Banach space, the function $\gamma_{\mathcal{D}}$ belongs to it. Therefore, the difference $\gamma_{\mathcal{D}} - \Psi_D$ is sublinear. Continuity of the group morphism $\mathcal{PA}$ follows from the identity in Lemma~\ref{toric-bdy-top-DVR}. It remains to show the injectivity of $\mathcal{PA}$. Suppose that $\gamma_{\mathcal{D}} = \Phi_{\mathcal{D}}|_{N_{\mathbb{R}}\times \lbrace 1 \rbrace} = 0$ for some toric compactified divisor $\mathcal{D}$. Then, the conical property of $\Phi_{\mathcal{D}}$ and continuity shows that $\Phi_{\mathcal{D}} = 0$. The group morphism $\mathcal{SF}$ is injective, therefore $\mathcal{D} = 0$. 
	\end{proof}
	\begin{lem}\label{toric-top-emb-DVR}
		Let $\iota \colon \mathcal{X}_{\widetilde{\Sigma}_0} \rightarrow \mathcal{X}_{\widetilde{\Sigma}_1}$ be a birational toric morphism between quasi-projective toric schemes over $\mathcal{O}_K$. Then, the induced pullback $\iota^{\ast} \colon  \textup{Div}_{\mathbb{T}} (\mathcal{X}_{\widetilde{\Sigma}_1}/\mathcal{O}_K) \rightarrow  \textup{Div}_{\mathbb{T}} (\mathcal{X}_{\widetilde{\Sigma}_0}/\mathcal{O}_K)$ is continuous, injective, and satisfies $\mathcal{SF} \circ \iota^{\ast} = \mathcal{SF}$ (resp. $\mathcal{PA} \circ \iota^{\ast} = \mathcal{PA}$), where $\mathcal{SF}$ (resp $\mathcal{PA}$) is the map defined in Lemma~\ref{sup-fun-dvr-1}. Moreover, the morphism $\iota^{\ast}$ is a topological group embedding if and only if the map $\iota$ is proper.
	\end{lem}
	\begin{proof}
		Copy the proof of Lemma~\ref{toric-top-emb}.
	\end{proof}
	Let $\mathcal{X}_{\widetilde{\Sigma}_0}/\mathcal{O}_K$ be a quasi-projective toric scheme. By definition, there is a toric open immersion $\iota \colon U \rightarrow \mathcal{X}_{\widetilde{\Sigma}_0}$. Regarding the torus $U/K$ as a quasi-projective toric scheme over $\mathcal{O}_K$, the above result induces a canonical identification of the group of toric compactified divisors of $\mathcal{X}_{\widetilde{\Sigma}_0}/\mathcal{O}_K$ as a subgroup of  $\textup{Div}_{\mathbb{T}} (U/\mathcal{O}_K)$, the toric compactified divisors of $U/\mathcal{O}_K$. We will study this group in detail in the next section using the analytification map on divisors. Now, we focus our attention on the canonical model of $U$ over $\mathcal{O}_K$: the split torus $\mathcal{U}/\mathcal{O}_K$ of relative dimension $d$.
	\begin{ex}[Canonical models]\label{canonical-model}
		Let $\Sigma$ be a fan in $N_{\mathbb{R}}$ and $X_{\Sigma}/K$ the corresponding toric variety. There is a canonical way to construct a toric scheme $\mathcal{X}_{\Sigma_{\textup{can}}}$ over $\mathcal{O}_K$ whose generic fibre is $X_{\Sigma}$. This toric scheme is known as the \textit{canonical model} of the toric variety $X_{\Sigma}/K$. First, we regard $\Sigma$ as a fan in $N_{\mathbb{R}} \times \mathbb{R}_{\geq 0}$ by identifying the latter with $N_{\mathbb{R}}$. Then, we define the fan $\Sigma_{\textup{can}}\coloneqq \Sigma \cup \lbrace c( \sigma \cup \lbrace (0,1) \rbrace ) \, \vert \, \sigma \in \Sigma \rbrace$, where $c( A )$ is the convex cone spanned by the set $A$. Clearly, the intersection of $\Sigma_{\textup{can}}$ with $N_{\mathbb{R}} \times \mathbb{R}_{\geq 0}$ is $\Sigma$. Additionally, the fan $\Sigma_{\textup{can}}$ is complete in $N_{\mathbb{R}} \times \mathbb{R}_{\geq 0}$ if and only if $\Sigma$ is complete in $N_{\mathbb{R}}$. We abuse notation by writing $\mathcal{X}_{\Sigma} = \mathcal{X}_{\Sigma_{\textup{can}}}$ whenever there is no risk of confusion.
		
		Similarly, every toric divisor $D \in  \textup{Div}_{\mathbb{T}}( X_{\Sigma})_{\mathbb{Q}}$ admits a \textit{canonical model} $\mathcal{D}_{\textup{can}} \in \textup{Div}_{\mathbb{T}}( \mathcal{X}_{\Sigma})_{\mathbb{Q}}$. Equivalently, the function $\Psi_D \in \mathcal{SF}(\Sigma, \mathbb{Q})$ extends to a support function $\Phi_{\mathcal{D}_{\textup{can}}} \in  \mathcal{SF}(\Sigma_{\textup{can}}, \mathbb{Q})$. Indeed, we define $\Phi_{\mathcal{D}_{\textup{can}}}(x,t) \coloneqq \Psi(x)$ for all $t \in \mathbb{R}_{\geq 0}$. Trivially, the function $\Phi_{\mathcal{D}_{\textup{can}}}$ is strictly concave if and only if $\Psi_D$ is strictly concave. Therefore, the fan $\Sigma_{\textup{can}}$ is projective if and only if the fan $\Sigma$ is projective. By Proposition~\ref{PM-over-DVR}, every projective variety $X_{\Sigma}/K$ admits a toric arithmetic model $\mathcal{X}_{\Sigma}/\mathcal{O}_K$. It follows from Lemmas~\ref{toric-bdry-div}~and~\ref{toric-bdy-top-DVR} that a pair $(X_{\Sigma},B)$ is a toric boundary divisor of $X_{\Sigma_{0}}/K$ if and only if the pair $(\mathcal{X}_{\Sigma}, \mathcal{B}_{\textup{can}})$ is a boundary divisor of $\mathcal{X}_{\Sigma_0}/ \mathcal{O}_K$, and they satisfy $\gamma_{\mathcal{B},\textup{can}} = \Psi_{B}$ (See~Lemma~\ref{sup-fun-dvr-1}).
		
		We apply these constructions to the case of the torus $U/K$, whose fan consists of the cone $\lbrace 0 \rbrace$. The canonical fan extending it is the set $\lbrace \lbrace 0 \rbrace, \tau_{\textup{vert}} \rbrace$, where $\tau_{\textup{vert}}$ is they ray generated by the vector $(0,1)$.  Identify $N = \mathbb{Z}^{d}$ and let $e_1,...,e_d$ be the standard basis. Write $e_0 = - (e_1 + ... + e_d)$. Let $\Sigma$ be the fan whose cones are spanned by all proper subsets of $\lbrace e_0, ..., e_d \rbrace$. The toric variety $X_{\Sigma}/K$ is the projective space $\mathbb{P}^{d}/K$ with the usual toric structure. Let $B$ be the toric divisor given by the sum of the hyperplanes $H_i$ determined by the equations $\chi^{e_i} = 0$. Its support function is determined by $\Psi_B (e_i) = -1$. By Lemma~\ref{toric-bdry-div}, $B$ is a boundary divisor of $U/K$. Then, the canonical model $\mathcal{X}_{\Sigma}/\mathcal{O}_K$ is the projective space $\mathbb{P}^{d}/\mathcal{O}_K$ with the usual toric structure, and $\mathcal{B}_{\textup{can}}$ is the toric boundary divisor of $\mathcal{U}/\mathcal{O}_K$ given by the sum of the respective hyperplanes.
	\end{ex}
	Now, we describe the group of toric compactified divisors of the torus $\mathcal{U}/\mathcal{O}_K$.
	\begin{thm}\label{CD-torus-DVR}
		Let $\mathcal{U}$ be the split $d$-dimensional torus over the discrete valuation ring $\mathcal{O}_K$. The continuous injective group morphism $\mathcal{PA} \colon \textup{Div}_{\mathbb{T}} (\mathcal{U} /\mathcal{O}_K) \rightarrow \mathcal{G}(N_{\mathbb{R}})$ given by $\mathcal{D} \mapsto \gamma_{\mathcal{D}}$ identifies the group of toric compactified divisors of $\mathcal{U}/\mathcal{O}_K$ with the closure in the $\mathcal{C}$-norm of the space $\mathcal{PA}(N_{\mathbb{R}}, \mathbb{Q})$ of rational piecewise affine functions on $N_{\mathbb{R}}$ (See Definitions~\ref{C-norm}~and~\ref{G-def}).
	\end{thm}
	\begin{proof}
		First, we show that the target of the isomorphism $\mathcal{PA}$ in Lemma~\ref{toric-model-div-S} is the $\mathbb{Q}$-vector space $\mathcal{PA}(N_{\mathbb{R}}, \mathbb{Q})$. By Example~\ref{canonical-model}, the torus $\mathcal{U}/\mathcal{O}_K$ is the toric scheme given by the fan $\lbrace \lbrace 0 \rbrace, \tau_{\textup{vert}} \rbrace$, that is, the faces of the ray $\tau_{\textup{vert}}$ generated by the vector $(0,1)$. By abuse of notation, we write $\tau_{\textup{vert}}$ for this fan. We claim that the poset $\textup{PF}(N_{\mathbb{R}} \times \mathbb{R}_{\geq 0}, \tau_{\textup{vert}} )$, consisting of projective fans in $N_{\mathbb{R}}\times \mathbb{R}_{\geq 0}$ having $\tau_{\textup{vert}}$ as a subfan, is cofinal in the poset $\textup{PF}(N_{\mathbb{R}} \times \mathbb{R}_{\geq 0})$ of all projective fans in $N_{\mathbb{R}}\times \mathbb{R}_{\geq 0}$. Let $\widetilde{\Sigma}$ be a projective fan in $N_{\mathbb{R}} \times \mathbb{R}_{\geq 0}$ and let $\Sigma$ be any projective fan in $N_{\mathbb{R}}$. Following Example~\ref{canonical-model}, the fan $\Sigma_{\textup{can}}$ in $N_{\mathbb{R}}\times \mathbb{R}_{\geq 0}$ is projective and $\tau_{\textup{vert}} \in \Sigma_{\textup{can}}$. Arguing as in Proposition~\ref{tqpmod}, the fan $\widetilde{\Sigma} \cdot \Sigma_{\textup{can}}$ given by intersecting their cones refines the fan $\widetilde{\Sigma}$, is projective, and contains the ray $\tau_{\textup{vert}}$, which proves the claim. This implies the equality of $\mathbb{Q}$-vector spaces
		\begin{displaymath}
			\bigcup_{\widetilde{\Sigma} \in \textup{PF}(N_{\mathbb{R}} \times \mathbb{R}_{\geq 0}, \tau_{\textup{vert}}  )} \mathcal{SF}(\widetilde{\Sigma}, \mathbb{Q}) = \bigcup_{\widetilde{\Sigma} \in \textup{PF}(N_{\mathbb{R}} \times \mathbb{R}_{\geq 0})} \mathcal{SF}(\widetilde{\Sigma}, \mathbb{Q}) = \mathcal{SF}(N_{\mathbb{R}} \times \mathbb{R}_{\geq 0}, \mathbb{Q}).
		\end{displaymath}
		By Proposition~\ref{toric-divisors-DVR-1}, restricting each support function to the hyperplane $N_{\mathbb{R}} \times \lbrace 1 \rbrace$, we get
		\begin{displaymath}
			\bigcup_{\widetilde{\Sigma} \in \textup{PF}(N_{\mathbb{R}} \times \mathbb{R}_{\geq 0}, \tau_{\textup{vert}}  )} \mathcal{PA}( \Pi, \mathbb{Q}) = \mathcal{PA}(N_{\mathbb{R}}, \mathbb{Q}).
		\end{displaymath}
		Summarizing, we have shown that $\mathcal{PA} \colon \textup{Div}_{\mathbb{T}} (\mathcal{U} /\mathcal{O}_K)_{\textup{mod}} \rightarrow \mathcal{PA}(N_{\mathbb{R}}, \mathbb{Q})$ is an isomorphism (\ref{toric-model-div-S}).
		
		Following Example~\ref{canonical-model}, $\mathcal{B}_{\textup{can}}$ is a toric boundary divisor of $\mathcal{U}/\mathcal{O}_K$, whose piecewise affine function $\gamma_{\mathcal{B},\textup{can}}$ coincides with the support function $\Psi_{B}$. By Proposition~\ref{toric-bdy-top-DVR}, for each model divisor $\mathcal{D}$ we have the identity
		\begin{displaymath}
			\| \mathcal{D} \|_{\mathcal{B}_{\textup{can}}} = \inf \lbrace \varepsilon \in \mathbb{Q}_{>0} \, \vert \, | \gamma_{\mathcal{D}}(x) | \leq \varepsilon \cdot  | \Psi_B (x) | \textup{ for all } x \in N_{\mathbb{R}} \rbrace.
		\end{displaymath}
		Since $B$ is a boundary divisor of $U/K$, by the proof of Theorem~\ref{torgeomdiv}, the extended norm induced by the right-hand side is equivalent to the $\mathcal{C}$-norm. Therefore, the image of the group morphism $\mathcal{PA}$ identifies with the closure in the $\mathcal{C}$-norm of  $\mathcal{PA}(N_{\mathbb{R}}, \mathbb{Q})$.
	\end{proof}
	We have the final theorem of this section: A characterization of the cone of nef toric compactified divisors of $\mathcal{U}/\mathcal{O}_K$. This is an analog of Theorem~\ref{nefcorrespondence} in the DVR setting.
	\begin{thm}\label{nef-cone-US}
		Let $\mathcal{U}/\mathcal{O}_K$ be the split torus of dimension relative $d$, and $\mathcal{D}$ be a toric compactified divisor. The assignment $\mathcal{D} \mapsto \gamma_{\mathcal{D}}$ induces a continuous bijection between:
		\begin{enumerate}
			\item The cone $\textup{Div}_{\mathbb{T}}^{\textup{nef}} (\mathcal{U} /\mathcal{O}_K)$ of nef toric compactified divisors of $\mathcal{U}/\mathcal{O}_K$.
			\item The set of globally Lipschitz concave functions $\gamma \colon N_{\mathbb{R}} \rightarrow \mathbb{R}$ satisfying $\gamma(0) \in \mathbb{Q}$.
		\end{enumerate}
	\end{thm}
	\begin{proof}
		Suppose that $\mathcal{D} \in \textup{Div}_{\mathbb{T}} (\mathcal{U} /\mathcal{O}_K)$ is nef. By definition, there is a Cauchy sequence $ \lbrace \mathcal{D}_n \rbrace_{n \in \mathbb{N}}$ of nef toric model divisors representing $\mathcal{D}$. By the previous theorem, the induced sequence $\lbrace \gamma_{\mathcal{D}_n} \rbrace_{n \in \mathbb{N}}$ of rational piecewise affine concave functions converges in the $\mathcal{C}$-norm to the function $\gamma_{\mathcal{D}}$. Therefore, $\gamma_{\mathcal{D}} \in \mathcal{G}^{+}(N_{\mathbb{R}})$ (See~\ref{G-def}). It is immediate from the definition of the $\mathcal{C}$-norm that the sequence of rational numbers $ \lbrace \gamma_{\mathcal{D}_n} (0) \rbrace_{n \in \mathbb{N}}$ is eventually constant.
		
		Conversely, let $\gamma$ be a globally Lipschitz concave function on $N_{\mathbb{R}}$ such that $\gamma (0) \in \mathbb{Q}$. Apply the approximation Theorem~\ref{closure-P-C} to obtain an increasing sequence of concave rational piecewise affine functions $\lbrace \gamma_n \rbrace_{n \in \mathbb{N}}$ converging in the $\mathcal{C}$-norm to $\gamma$. By Proposition~\ref{nef-dvr}, the sequence $\lbrace \gamma_n \rbrace_{n \in \mathbb{N}}$ corresponds to a sequence of nef toric model divisors $\lbrace \mathcal{D}_n \rbrace_{n \in \mathbb{N}}$ which, after Lemma~\ref{toric-bdy-top-DVR}, is Cauchy in the boundary norm. Theorem~\ref{CD-torus-DVR} gives that the limit $\mathcal{D}$ of the Cauchy sequence $\lbrace \mathcal{D}_n \rbrace_{n \in \mathbb{N}}$ is a nef toric compactified divisor satisfying $\gamma_{\mathcal{D}} = \gamma$. The result follows.
	\end{proof}
	
	\section{A toric local arithmetic analog of Yuan-Zhang's theory}
	This section focuses on the local arithmetic theory of toric quasi-projective varieties, where we prove the main theorems announced in the introduction. These results generalize the work of Burgos, Philippon, and Sombra on toric arithmetic divisors of continuous type on projective toric varieties (see~\cite{BPS}) to admit Green's functions with singularities which are worse than logarithmic (see Remark~\ref{singularities-metrics}). This section is divided into three parts. First, we give a summary of the previously mentioned results of \cite{BPS}. In the second subsection, we introduce the group of toric compactified divisors and describe its properties in convex-analytic terms. The last subsection is devoted to the study of the nef cone and the integral formula for the local toric height. Throughout this section, we fix a local field $K$, and use freely the notation introduced in Subsection~\ref{notation} and the previous Section~\ref{3}. 
	
	\subsection{The arithmetic of projective toric varieties over a local field}
	Let $X_{\Sigma}/K$ be a toric variety with torus action $\mu \colon U \times X_{\Sigma} \rightarrow X_{\Sigma}$. The analytification functor, introduced in Section~\ref{local-theory}, associates to $X_{\Sigma}$ a Berkovich $K$-analytic variety $X_{\Sigma}^{\textup{an}}$ and $K$-analytic group action $\mu^{\textup{an}} \colon U^{\textup{an}} \times X_{\Sigma}^{\textup{an}} \rightarrow X_{\Sigma}^{\textup{an}}$, which restricts to an action $\mu^{\textup{an}} \colon \mathbb{S} \times X_{\Sigma}^{\textup{an}} \rightarrow X_{\Sigma}^{\textup{an}}$ of the \textit{compact torus} $\mathbb{S}$; the $K$-analytic subgroup of $U^{\textup{an}}$ given by $\mathbb{S}:= \lbrace p \in U^{\textup{an}} \, | \, |\chi^m (p)| = 1 \textup{ for all } m \in M \rbrace$. For details on $K$-analytic groups and their actions, we refer to Chapter~5~of~\cite{Ber} or Section~4.2~of~\cite{BPS} for a summary. The $\mathbb{S}$-orbits of this action are described nicely by the \textit{tropicalization functor} $\textup{Trop}$. We sketch its construction below and refer to Sections~4.1~and~4.2~of~\cite{BPS} for details.
	\begin{cons}[Tropicalization functor]\label{tropical-toric-polytope}
		A point $x \in N_{\mathbb{R}}$ is a linear functional on $M_{\mathbb{R}}$; thus, it is determined by its values at the characters $m \in M$. Then, the \textit{tropicalization map} is the continuous map $\textup{Trop} \colon U^{\textup{an}} \rightarrow N_{\mathbb{R}}$ induced by the assignment
		\begin{displaymath}
			p \longmapsto (m \mapsto - \log |\chi^m (p)| ).
		\end{displaymath}
		Observe that $\textup{Trop}(t) = 0$ if and only if $t \in \mathbb{S}$. By Proposition~4.2.6~of~\cite{BPS}, for all $t, p  \in U^{\textup{an}}$ we have $\textup{Trop}(t \cdot p) = \textup{Trop}(t) + \textup{Trop}(p)$. Therefore,  $\textup{Trop}^{-1}(\textup{Trop}(p)) = \mathbb{S} \cdot p$, and so the tropicalization map encodes the $\mathbb{S}$-orbits. The cone-orbit correspondence (\ref{cone-orbit}) gives a stratification of $X_{\Sigma}$ by tori $O(\sigma)$. Then, the family of tropicalization maps $\lbrace \textup{Trop}: O(\sigma )^{\textup{an}} \rightarrow N(\sigma)_{\mathbb{R}} \rbrace_{\sigma \in \Sigma}$ induces a surjective map $\textup{Trop}: X_{\Sigma}^{\textup{an}} \rightarrow N_{\Sigma}$, where $N_{\Sigma}$ is called the \textit{tropical toric variety associated to }$\Sigma$. As a set, it is given by the disjoint union of the $N(\sigma)_{\mathbb{R}}$. For the definition of its topology, we refer to Section~4.1.~p.118~of~\cite{BPS}. We list some of its properties below:
		\begin{enumerate}
			\item The map $\textup{Trop}: X_{\Sigma}^{\textup{an}} \rightarrow N_{\Sigma}$ is continuous, (topologically) proper, and it identifies $N_{\Sigma}$ with the topological quotient of $X_{\Sigma}^{\textup{an}}$ by $\mathbb{S}$.
			\item  Let $f \colon X_{\Sigma_1} \rightarrow X_{\Sigma_2}$ be a toric morphism and $\phi \colon (N_1)_{\mathbb{R}} \rightarrow (N_2)_{\mathbb{R}}$ be the linear map given by Proposition~\ref{toricmorphism}. Then, there is an induced continuous map $\phi \colon N_{1,\Sigma_1} \rightarrow N_{2,\Sigma_2}$ extending $\phi$, compatible with the respective stratifications, and satisfying $\textup{Trop} \circ f^{\textup{an}} = \phi \circ \textup{Trop}$.
			\item If $X_{\Sigma}/K$ is projective and $D$ is an ample toric divisor with corresponding convex polytope $\Delta_D$, then  $N_{\Sigma}$ and $\Delta_D$ are homeomorphic. In particular, $N_{\Sigma}$ is compact.
			\item If $K$ is non-archimedean, there exists a continuous proper section $\theta_{\Sigma} \colon N_{\Sigma} \rightarrow X_{\Sigma}^{\textup{an}}$ of $\textup{Trop}$. For all $p \in X_{\Sigma}^{\textup{an}}$, it satisfies $\theta_{\Sigma} (\textup{Trop} (p)) = \xi \cdot p$, where $\xi$ is the \textit{Gauss point} in $\mathbb{S}$. That is, the point corresponding to the multiplicative valuation on $K[M]$ given by $\sum_{m \in M} \alpha_m \, \chi^m \mapsto  \max_{m \in M} | \alpha_m |$.
		\end{enumerate}
		Therefore, we get a \textit{tropicalization functor} $\textup{Trop}$ from the category of toric varieties over $K$ to the category of topological spaces.
	\end{cons}
	By the above discussion, a function $f : X_{\Sigma}^{\textup{an}} \rightarrow \mathbb{R}_{\pm \infty}$ is $\mathbb{S}$\textit{-invariant} if there exists a function $\widetilde{f} \colon N_{\Sigma} \rightarrow \mathbb{R}_{\pm \infty}$ such that $f = \widetilde{f} \circ \textup{Trop}$.	We denote by $C^{0}_{\mathbb{S}}(X_{\Sigma}^{\textup{an}})$ the space of $\mathbb{S}$-invariant continuous functions on $X_{\Sigma}^{\textup{an}}$. Then, we have the following definition.
	\begin{dfn}
		Let $X_{\Sigma}/K$ be a quasi-projective toric variety. An arithmetic divisor $\overline{D}= (D,g)$ on $X_{\Sigma}$ is said to be \textit{toric} if its divisorial part $D$ is toric and its Green's function $g$ is $\mathbb{S}$-invariant. We denote by $\overline{\textup{Div}}_{\mathbb{T}}(X_{\Sigma})_{\mathbb{Q}}$ the group of toric arithmetic divisors of continuous type on $X_{\Sigma}$.
	\end{dfn}
	Using the tropicalization map, we can construct an analytic analog of a support function. Let $\overline{D}=(D,g)$ be a toric divisor on a quasi-projective variety $X_{\Sigma}/K$. The support $|D|$ of the toric divisor $D$ is contained in the complement of the torus $X_{\Sigma} \setminus U$. Then, the restriction of $g$ to the analytic torus $U^{\textup{an}}$ is a continuous function. By definition of $\mathbb{S}$-invariance, there exists a continuous function $\gamma_{\overline{D}} \colon N_{\mathbb{R}} \rightarrow \mathbb{R}$ satisfying $g|_{U^{\textup{an}}} = - \gamma_{\overline{D}} \circ \textup{Trop}$. We call $\gamma_{\overline{D}}$ the \textit{tropical Green's function associated to }$\overline{D}$. This nomenclature is not standard, as there is a minus sign in the definition of $\gamma_{\overline{D}}$. The justification of this sign comes from the fact that nef toric arithmetic divisors $\overline{D}$ will correspond to concave functions $\gamma_{\overline{D}}$. First, we show that this assignment is injective.
	\begin{prop}\label{G}
		Let $X_{\Sigma}/K$ be quasi-projective and $\overline{D} = (D,g)$ be a toric arithmetic divisor on $X_{\Sigma}$. The assignment $\overline{D} \mapsto \gamma_{\overline{D}}$ induces an injective group morphism $\mathcal{G} \colon \overline{\textup{Div}}_{\mathbb{T}}(X_{\Sigma})_{\mathbb{Q}} \longrightarrow C^{0}(N_{\mathbb{R}})$.
	\end{prop}
	\begin{proof}
		It is clear that $\mathcal{G}$ is a group morphism. Now, we show that it is injective. Since $U^{\textup{an}}$ is dense in $X_{\Sigma}^{\textup{an}}$, a Green's function of continuous type is determined by its values on $U^{\textup{an}}$. Then, for toric divisors $\overline{D}_i =(D_i ,g_i )$ on $X_{\Sigma}$, the equality $\gamma_{\overline{D}_1} = \gamma_{\overline{D}_2}$ holds if and only if $g_1 = g_2$ holds. Assuming that $g_1 = g_2$, the pair $(D_1 - D_2, 0)$ is an arithmetic divisor of continuous type on $X_{\Sigma}$. Therefore, for every local equation $(V,h)$ of $D$, the function $\log |h|$ is continuous and finite on $V^{\textup{an}}$. Then, $h$ does not have poles or zeros on $V$, which implies $D_1 - D_2 = 0$.
	\end{proof}
	The morphism $\mathcal{G}$ satisfies the following functorial property.
	\begin{prop}\label{gamma-birational}
		Let $f\colon X_{\Sigma_0} \rightarrow X_{\Sigma_1}$ be a dominant toric morphism between quasi-projective toric varieties over $K$, inducing the linear map $\phi \colon (N_1)_{\mathbb{R}} \rightarrow (N_2)_{\mathbb{R}}$. Let $\overline{D}$ be a toric arithmetic divisor on $X_{\Sigma_1}$ with tropical Green's function $\gamma_{\overline{D}}$. Then, the pullback $f^{\ast} \overline{D}$ is a toric divisor on $X_{\Sigma_0}$, and its tropical Green's function satisfies  $\gamma_{f^{\ast} \overline{D}} = \gamma_{\overline{D}} \circ \phi$. In particular, if $f$ restricts to the identity on tori, we get $\gamma_{f^{\ast} \overline{D}} = \gamma_{\overline{D}}$.
	\end{prop}
	\begin{proof}
		By hypothesis, the toric morphism $f$ is dominant. Then, there is an induced pullback map $f^{\ast} \colon \overline{\textup{Div}}(X_{\Sigma_2})_{\mathbb{Q}} \rightarrow \overline{\textup{Div}}(X_{\Sigma_1})_{\mathbb{Q}}$. Since $f$ is toric, the pullback map $f^{\ast}$ restricts to the subgroup of toric elements. The identity  $\gamma_{f^{\ast} \overline{D}} = \gamma_{\overline{D}} \circ \phi$ follows from the functoriality of $\textup{Trop}$.
	\end{proof}
	From now on, and until the end of the subsection, the toric variety $X_{\Sigma}/K$ is projective. We state below Proposition~4.3.10~of~\cite{BPS}, which characterizes the $\mathbb{S}$-invariant Green's functions of continuous type of a given toric divisor on $X_{\Sigma}$ in terms of the induced tropical Green's functions.
	\begin{prop}\label{gamma}
		Let $X_{\Sigma}/K$ be a projective toric variety, and $D$ be a toric divisor on $X_{\Sigma}$ with support function $\Psi_D$. Then, an $\mathbb{S}$-invariant continuous function $g \colon U^{\textup{an}} \rightarrow \mathbb{R}$ extends to a Green's function of continuous type for $D$ if and only if $\Psi_D - \gamma$ extends to a continuous function on $N_{\Sigma}$, where $g = - \gamma \circ \textup{Trop}$. If this is the case, then the function $|\Psi_D - \gamma|$ is bounded.
	\end{prop}
	Recall the notion of an arithmetic divisor of model type from Definition~\ref{model-type}. Then, we denote by $\overline{\textup{Div}}_{\mathbb{T}}(X_{\Sigma})_{\textup{mod}}$ the group of toric arithmetic divisors of model type on $X_{\Sigma}/K$. If $K$ is non-archimedean, we have the following characterization of the group of toric arithmetic divisors of model type.
	\begin{prop}\label{toric-local-analytification}
		Let $K$ be non-archimedean, $X_{\Sigma}/K$ be projective, and $\pi \colon X_{\Sigma} \rightarrow \mathcal{X}_{\widetilde{\Sigma}}$ be a toric projective model of $X_{\Sigma}$ over $\mathcal{O}_K$. Then, for each toric divisor $\mathcal{D}$ on $\mathcal{X}_{\widetilde{\Sigma}}$, its analytification $\overline{D}$ is toric and the tropical Green's function $ \gamma_{\overline{D}}$ coincides with the piecewise affine function $\gamma_{\mathcal{D}}$ associated to $\mathcal{D}$ (\ref{toric-divisors-DVR-1}). Moreover, the assignment $\overline{D} \mapsto \gamma_{\overline{D}}$ induces an isomorphism
		\begin{displaymath}
			\mathcal{G} \colon \overline{\textup{Div}}_{\mathbb{T}}(X_{\Sigma})_{\textup{mod}} \longrightarrow   \bigcup_{\widetilde{\Sigma} \in \textup{PF}(N_{\mathbb{R}}\times \mathbb{R}_{\geq 0}, \Sigma)}  \mathcal{PA}( \Pi, \mathbb{Q}),
		\end{displaymath}
		where $\textup{PF}(N_{\mathbb{R}}\times \mathbb{R}_{\geq 0}, \Sigma)$ is the category of projective fans $\widetilde{\Sigma}$ in $N_{\mathbb{R}}\times \mathbb{R}_{\geq 0}$ having $\Sigma$ as a subfan, and $\mathcal{PA}(\Pi, \mathbb{Q})$ is the space of rational piecewise affine functions on the polyhedral complex $\Pi$ obtained by intersecting the fan $\widetilde{\Sigma}$ with the hyperplane $N_{\mathbb{R}} \times \lbrace 1 \rbrace$.
	\end{prop}
	\begin{proof}
		The first part follows from Proposition~4.5.3 and Corollary~4.5.5~of~\cite{BPS}. Then, the isomorphism is immediate from Lemma~\ref{toric-model-div-S}.
	\end{proof}
	\begin{ex}[Canonical Green's functions]\label{canonical-GF}
		Let $D$ be a toric divisor on $X_{\Sigma}/K$ with support function $\Psi_D$. There is a \textit{canonical Green's function} $g_{D,\textup{can}}$ for $D$ of continuous type, and we write $\overline{D}_{\textup{can}} = (D,g_{D,\textup{can}})$.  Indeed, consider $\gamma_{D,\textup{can}}:= \Psi_D$. Trivially, the function $\Psi_D - \gamma_{D,\textup{can}}$ extends continuously to $N_{\Sigma}$. By Proposition~\ref{gamma}, the $\mathbb{S}$-invariant function $g_{D,\textup{can}}$ given by $\gamma_{D,\textup{can}} \circ \textup{Trop}$ extends to a Green's function for $D$. Additionally, if $K$ is non-archimedean, the arithmetic divisor $\overline{D}_{\textup{can}}$ is given by analytification of the canonical model $\mathcal{D}_{\textup{can}}$ of $D$ (See Example~\ref{canonical-model}). Finally, the assignment $D \mapsto \overline{D}_{\textup{can}}$ induces a group morphism  $\textup{can} \colon \textup{Div}_{\mathbb{T}}(X_{\Sigma})_{\mathbb{Q}} \rightarrow \overline{\textup{Div}}_{\mathbb{T}}(X_{\Sigma})_{\mathbb{Q}}$, which we call the \textit{canonical map}. It is a section of the \textit{forgetful map} $\textup{for} \colon \overline{\textup{Div}}_{\mathbb{T}}(X_{\Sigma})_{\mathbb{Q}} \rightarrow \textup{Div}_{\mathbb{T}}(X_{\Sigma})_{\mathbb{Q}}$, induced by the assignment $\overline{D} \mapsto D$.
	\end{ex}
	\begin{cons}[$\mathbb{S}$-symmetrization]\label{S-inv-proj}
		Let $X_{\Sigma}/K$ be a projective toric variety and $f$ be a continuous function on $X_{\Sigma}^{\textup{an}}$. There is an averaging process that takes the function $f$ and returns an $\mathbb{S}$-invariant function $f^{\mathbb{S}}$. The assignment $f \mapsto f^{\mathbb{S}}$ induces an $\mathbb{R}$-linear map $( \cdot )^{\mathbb{S}} \colon C^{0}(X_{\Sigma}^{\textup{an}}) \rightarrow C^{0}_{\mathbb{S}}(X_{\Sigma}^{\textup{an}})$ which we call the $\mathbb{S}$\textit{-symmetrization}. It is defined as follows: If $K$ is archimedean, for each $p \in X_{\Sigma}^{\textup{an}}$ the function $f^{\mathbb{S}}$ is given by the integral
		\begin{displaymath}
			f^{\mathbb{S}}(p):= \int_{\mathbb{S}} f(t \cdot p) \, \textup{d}\mu_{\mathbb{S}}(t),
		\end{displaymath}
		where $\mu_{\mathbb{S}}$ is the Haar probability measure on $\mathbb{S}$. If $K$ is non-archimedean, the function $f^{\mathbb{S}}$ is defined as the compostion $f \circ \theta_{\Sigma} \circ \textup {Trop}$. It is clear that $f^{\mathbb{S}}$ is continuous and $\mathbb{S}$-invariant. We have a similar statement for arithmetic divisors. Let $\overline{D}=(D,g)$ be an arithmetic divisor of continuous type on $X_{\Sigma}$ such that its divisorial part $D$ is toric. The analytification of the support $|D|^{\textup{an}}$ is an $\mathbb{S}$-invariant set. Then, define the function $g^{\mathbb{S}} \colon X_{\Sigma}^{\textup{an}} \setminus |D|^{\textup{an}} \rightarrow \mathbb{R}$ in the same way as above. Proposition~4.3.4~of~\cite{BPS} gives the following properties:
		\begin{enumerate}
			\item The pair $\overline{D}^{\mathbb{S}} \coloneqq (D, g^{\mathbb{S}})$ is a toric arithmetic divisor of continuous type on $X_{\Sigma}$.
			\item Let $\lbrace g_n \rbrace_{n \in \mathbb{N}}$ be a sequence of Green's functions of continuous type for $D$ converging uniformly to $g$. Then, the sequence $\lbrace g_{n}^{\mathbb{S}} \rbrace_{n \in \mathbb{N}}$ converges uniformly to $g^{\mathbb{S}}$.
			\item If $K$ is archimedean and $g$ is of smooth type (resp. psh type), then $g^{\mathbb{S}}$ is of smooth type (resp. psh type).
		\end{enumerate}
		By Proposition~\ref{model-dense}, we conclude that every $\mathbb{S}$-invariant Green's function of continuous type for $D$ is the uniform limit of $\mathbb{S}$-invariant Green's functions of smooth type for $D$.
	\end{cons}
	Recall the notion of a nef arithmetic divisor, introduced in Definition~\ref{nef-local}. Then, we denote by $\textup{Div}^{\textup{nef}}_{\mathbb{T}}(X_{\Sigma})_{\mathbb{Q}}$ the cone of nef toric arithmetic divisors on $X_{\Sigma}$. The space of integrable toric arithmetic divisors is the difference
	\begin{displaymath}
		\textup{Div}^{\textup{int}}_{\mathbb{T}}(X_{\Sigma})_{\mathbb{Q}} \coloneqq \textup{Div}^{\textup{nef}}_{\mathbb{T}}(X_{\Sigma})_{\mathbb{Q}} - \textup{Div}^{\textup{nef}}_{\mathbb{T}}(X_{\Sigma})_{\mathbb{Q}}.
	\end{displaymath}
	As in the geometric case, the nefness of a toric arithmetic divisor $\overline{D}$ corresponds to the concavity of the tropical Green's function $\gamma_{\overline{D}}$. Indeed, we have the following theorem.Theorem~4.8.1~of~\cite{BPS}.
	\begin{thm}\label{nef-tor-arith-proj}
		Let $X_{\Sigma}/K$ be a projective toric variety and $D$ be a nef toric divisor with support function $\Psi_D$ and convex polytope $\Delta_D$. Then, the assignments $\overline{D} \mapsto \gamma_{\overline{D}}$ and $\overline{D} \mapsto \gamma_{\overline{D}}^{\vee}$, where $\gamma_{\overline{D}}^{\vee}$ is the Legendre-Fenchel transform of $\gamma_{\overline{D}}$ (\ref{LF-dfn}) induce bijections between:
		\begin{enumerate}
			\item The set of Green's functions $g$ of continuous type for $D$ such that $\overline{D}=(D,g)$ is nef.
			\item The set of concave functions $\gamma \colon N_{\mathbb{R}} \rightarrow \mathbb{R}$ such that $|\Psi_D - \gamma|$ is bounded.
			\item The set of concave continuous functions $\vartheta \colon \Delta_D \rightarrow \mathbb{R}$.
		\end{enumerate}
		Assume further that $K$ is non-archimedean. Then, a nef divisor $\overline{D}$ is of model type if and only if its tropical Green's function $\gamma_{\overline{D}}$ is rational piecewise affine. 
	\end{thm}
	\begin{proof}
		This is an amalgamation of Theorems~4.8.1~and~4.5.10~of~\cite{BPS}.
	\end{proof}
	Let $\overline{D}$ be a nef toric arithmetic divisor on the projective variety $X_{\Sigma}/K$. The above theorem introduces the concave function $\vartheta_{\overline{D}} \colon \Delta_{D} \rightarrow \mathbb{R}$, which is given by the Legendre-Fenchel transform of the tropical Green's function $\gamma_{\overline{D}}$. It is known as the \textit{roof function} of $\overline{D}$. It will play a key role when computing toric heights, which we now recall. Chapter~5~of~\cite{BPS} introduces toric heights using the language of toric metrized line bundles. The following remark shows that this definition can be stated in divisorial terms.
	\begin{rem}\label{remark-toric-h-def}
		Let $X_{\Sigma}/K$ be a projective toric variety of dimension $d$. Suppose we are given an $e$-dimensional irreducible subvariety $Y$ of $X_{\Sigma}$ and nef toric arithmetic divisors $\overline{D}_0, \ldots, \overline{D}_e$ on $X_{\Sigma}$. By successive applications of Chow's moving lemma, we may find rational functions $f_0, \ldots, f_e \in K(X_{\Sigma})^{\times}$ such that the divisors $D_i + \textup{div}(f_i)$, $i \in \lbrace 0, \ldots, e\rbrace$ meet $Y$ properly. Moreover, for each $i$, the toric arithmetic divisor $\overline{D}_{i, \textup{can}}$ is nef and has the same divisorial part as $\overline{D}_i$. Then, by Definition~\ref{local-h-def}, the difference of heights
		\begin{displaymath}
			\textup{h}(Y ; \overline{D}_0 + \overline{\textup{div}}(f_0), \ldots , \overline{D}_e + \overline{\textup{div}}(f_e))  - \textup{h}(Y ; \overline{D}_{0,\textup{can}} + \overline{\textup{div}}(f_0), \ldots , \overline{D}_{e,\textup{can}} + \overline{\textup{div}}(f_e))
		\end{displaymath}
		exist. By Lemma~\ref{difference-h}, the quantity above does not depend on the choice of the rational functions $f_0, \ldots, f_e \in K(X_{\Sigma})^{\times}$. This motivates the following definition.
	\end{rem}
	\begin{dfn}\label{local-toric-h-proj-def}
		Let $X_{\Sigma}/K$ be a projective toric variety of dimension $d$. Given an $e$-dimensional irreducible subvariety $Y$ of $X_{\Sigma}$ and nef toric arithmetic divisors $\overline{D}_0, \ldots, \overline{D}_e$ on $X_{\Sigma}$, the \textit{local toric height} of $Y$ with respect to $\overline{D}_0, \ldots, \overline{D}_e$, denoted by $\textup{h}^{\textup{tor}}(Y; \overline{D}_0 , \ldots , \overline{D}_e ) $, is defined as the difference
		\begin{displaymath}
			\textup{h}(Y ; \overline{D}_0 + \overline{\textup{div}}(f_0), \ldots , \overline{D}_e + \overline{\textup{div}}(f_e))  - \textup{h}(Y ; \overline{D}_{0,\textup{can}} + \overline{\textup{div}}(f_0), \ldots , \overline{D}_{e,\textup{can}}+\overline{\textup{div}}(f_e)),
		\end{displaymath}
		where $f_0, \ldots, f_e \in K(X_{\Sigma})^{\times}$ are as in Remark~\ref{remark-toric-h-def}. When $\overline{D}_{0} = \ldots = \overline{D}_{e}= \overline{D}$, we will denote $\textup{h}^{\textup{tor}}(Y; \overline{D}) \coloneqq \textup{h}^{\textup{tor}}(Y; \overline{D}_0 , \ldots , \overline{D}_e )$. This definition is extended to $e$-dimensional cycles and integrable toric arithmetic divisors by linearity.
	\end{dfn}
	\begin{rem}
		The local toric height is defined as a difference of local heights. Thus, it is symmetric and multilinear. It does not agree with the local height introduced in Definition~\ref{local-h-def}. However, in the global situation, the sum over all places of the local heights of a toric divisor equipped with the canonical Green's function at all places will be zero. Therefore, it can be used to compute global heights.
	\end{rem}
	The following lemma shows that the local toric arithmetic self-intersection numbers of a nef toric arithmetic divisor can be computed directly instead of using the inductive process described in Definition~\ref{local-h-def}.
	\begin{lem}\label{toric-mixed-energy}
		Let $X_{\Sigma}/K$ be a projective toric variety of dimension $d$ and $\overline{D}$ be a nef arithmetic toric divisor on $X_{\Sigma}$. Then,
		\begin{displaymath}
			\textup{h}^{\textup{tor}}(X_{\Sigma}; \overline{D}) = \sum_{j=0}^{d} \int_{X_{\Sigma}^{\textup{an}}} (g_D - g_{D,\textup{can}}) \langle \omega_{D}(g_D)^{\wedge j} \wedge \omega_{D}(g_{D,\textup{can}})^{\wedge (d-j)} \rangle \wedge \delta_{X_{\Sigma}}.
		\end{displaymath}
	\end{lem}
	\begin{proof}
		For each $j=0,\ldots,d$, define the number
		\begin{displaymath}
			h(j) \coloneqq \textup{h}(X_{\Sigma}; \underbrace{\overline{D}, \ldots, \overline{D}}_{j \textup{ times}},\underbrace{\overline{D}_{\textup{can}}, \ldots, \overline{D}_{\textup{can}}}_{d+1-j \textup{ times}}).
		\end{displaymath}
		The definition of toric height and an elementary algebraic manipulation show
		\begin{displaymath}
			\textup{h}^{\textup{tor}}(X_{\Sigma}; \overline{D}) = h(d+1) - h(0) = \sum_{j=0}^{d} h(j+1)-h(j).
		\end{displaymath}
		Apply part \textit{(iii)} of Theorem~\ref{local-h-thm} to obtain the identity
		\begin{displaymath}
			h(j+1)-h(j) = \int_{X_{\Sigma}^{\textup{an}}} (g_D - g_{D,\textup{can}}) \langle \omega_{D}(g_D)^{\wedge j} \wedge \omega_{D}(g_{D,\textup{can}})^{\wedge (d-j)} \rangle \wedge \delta_{X_{\Sigma}}.
		\end{displaymath}
		The result follows from substituting the second equation into the first.
	\end{proof}
	Now, we state Theorem~5.1.6~of~\cite{BPS}. This formula computes the height of a toric variety with respect to a nef toric arithmetic divisor in terms of its associated convex analytic objects. Our main result generalizes this formula to the compactified setting.
	\begin{thm}\label{local-toric-height-proj}
		Let $X_{\Sigma}/K$ be a projective toric variety of dimension $d$ and $\overline{D}$ a nef arithmetic toric divisor on $X_{\Sigma}$. Then,
		\begin{displaymath}
			\textup{h}^{\textup{tor}}(X_{\Sigma}; \overline{D}) = (d+1)! \int_{\Delta_D} \vartheta_{\overline{D}} \, \textup{dVol}_M, 
		\end{displaymath}
		where $\Delta_D$ and $\vartheta_{\overline{D}}$ are the corresponding convex polytope and roof function respectively.
	\end{thm}
	Finally, we state a mixed version of the above theorem, which includes the integrable case. It combines Corollary~5.1.9 and Remark~5.1.10~of~\cite{BPS}.
	\begin{dfn}\label{mixed-integral}
		For each $i = 0, \ldots , d$, let $\Delta_i$ be a compact convex subset of $M_{\mathbb{R}}$ and $g_i \colon \Delta_i \rightarrow \mathbb{R}$ be a concave function. The \textit{mixed integral} of $g_0, \ldots , g_d$ is defined as
		\begin{displaymath}
			\textup{MI}_M (g_0, \ldots , g_d ) \coloneqq  \sum_{j=0}^{d} (-1)^{d-j} \sum_{0 \leq i_0 < \ldots < i_j \leq d } \int_{\Delta_{i_0} + \ldots + \Delta_{i_j} } g_{i_0} \boxplus \ldots \boxplus g_{i_j} \textup{ dVol}_M . 
		\end{displaymath}
	\end{dfn}
	\begin{cor}\label{local-toric-height-proj-integrable}
		Let $\overline{D}_0, \ldots, \overline{D}_d \in  \overline{\textup{Div}}^{\textup{int}}_{\mathbb{T}}(X_{\Sigma})_{\mathbb{Q}}$. For each $i = 0, \ldots , d$, write $\overline{D}_{i}$ as a difference $\overline{D}_{i,+} - \overline{D}_{i,-}$ of nef arithmetic toric divisors. Denote by $\vartheta_{i,+}$ and $\vartheta_{i,-}$ the corresponding roof function of $\overline{D}_{i,+}$ and $\overline{D}_{i,-}$, respectively. Then,
		\begin{displaymath}
			\textup{h}^{\textup{tor}}(X_{\Sigma}; \overline{D}_0 , \ldots , \overline{D}_d ) = \sum_{\epsilon_0 , \ldots , \epsilon_d \in \lbrace \pm 1 \rbrace} \epsilon_0 \ldots \epsilon_d \, \textup{MI}_M ( \vartheta_{0,\epsilon_0}, \ldots , \vartheta_{d,\epsilon_d }).
		\end{displaymath}
		In particular, if  $\overline{D}_0, \ldots, \overline{D}_d \in  \overline{\textup{Div}}^{\textup{nef}}_{\mathbb{T}}(X_{\Sigma})_{\mathbb{Q}}$ $\vartheta_i$ is the roof function of $\overline{D}_i$, the above formula yields
		\begin{displaymath}
			\textup{h}^{\textup{tor}}(X_{\Sigma}; \overline{D}_0 , \ldots , \overline{D}_d ) = \textup{MI}_M ( \vartheta_{0}, \ldots , \vartheta_{d}).
		\end{displaymath}
	\end{cor}
	
	\subsection{Toric compactified arithmetic divisors}
	Throughout this subsection, $X_{\Sigma_0}/K$ is a quasi-projective toric variety. If $K$ is archimedean, we also require, for simplicity, that $X_{\Sigma_0}$ be smooth. This is a minor restriction; $X_{\Sigma_0}$ is normal. Hence, its singular locus is contained in a set of codimension at least 2. Moreover, the torus $U/K$ is an open smooth subvariety of $X_{\Sigma_0}$. Hence, we can always shrink $X_{\Sigma_0}$ to obtain a smooth toric variety over $K$. Then, we define the group of \textit{toric model arithmetic divisors} of $X_{\Sigma_{0}}$ over $K$ as the direct limit
	\begin{displaymath}
		\overline{\textup{Div}}_{\mathbb{T}}(X_{\Sigma_0}/K)_{\textup{mod}} \coloneqq \varinjlim_{X_{\Sigma} \in \textup{PM}_{\mathbb{T}}(X_{\Sigma_0}/K)} \overline{\textup{Div}}_{\mathbb{T}}(X_{\Sigma})_{\mathbb{Q}}.
	\end{displaymath}
	By Proposition~\ref{tqpmod} and the smoothness assumption on the archimedean case, the category $\textup{SPM}_{\mathbb{T}}(X_{\Sigma_0}/K)$ of smooth toric projective models is cofinal in the category $\textup{PM}_{\mathbb{T}}(X_{\Sigma_0}/K)$ of toric projective models of $X_{\Sigma_0}/K$. Then, the direct limit above can also be taken over the category of smooth projective toric models, yielding the same group. By a \textit{toric boundary (arithmetic) divisor of} $X_{\Sigma_0}/K$, we mean a boundary divisor $(X_{\Sigma},\overline{B})$ such that both the projective model $X_{\Sigma}$ and the arithmetic divisor $\overline{B}$ are toric. A toric boundary divisor $\overline{B}$ induces a boundary norm $\| \cdot \|_{\overline{B}}$ and a boundary topology on $\overline{\textup{Div}}_{\mathbb{T}} (X_{\Sigma_0}/K)_{\textup{mod}}$. Then, we have the following definition.
	\begin{dfn}\label{cptf-dfn-local}
		The group of \textit{toric compactified arithmetic divisors of }$X_{\Sigma_0}$ \textit{over} $K$ is the completion $\overline{\textup{Div}}_{\mathbb{T}} (X_{\Sigma_0}/K)$ of the group $\overline{\textup{Div}}_{\mathbb{T}} (X_{\Sigma_0}/K)_{\textup{mod}}$ with respect to the boundary topology. The cone $\overline{\textup{Div}}^{\textup{nef}}_{\mathbb{T}} (X_{\Sigma_0}/K)$ of nef toric compactified arithmetic divisors is the closure in the boundary topology of the cone of nef toric model arithmetic divisors. The space of integrable toric arithmetic divisors is the difference
		\begin{displaymath}
			\overline{\textup{Div}}^{\textup{int}}_{\mathbb{T}} (X_{\Sigma_0}/K) \coloneqq \overline{\textup{Div}}^{\textup{nef}}_{\mathbb{T}} (X_{\Sigma_0}/K)- \overline{\textup{Div}}^{\textup{nef}}_{\mathbb{T}} (X_{\Sigma_0}/K).
		\end{displaymath}
	\end{dfn}
	The following lemma describes the boundary norm in terms of tropical Green's functions.
	\begin{lem}\label{local-toric-bdry-norm}
		Let $X_{\Sigma_0}/K$ be a quasi-projective toric variety and $(X_{\Sigma}, \overline{B})$ be a toric boundary divisor of $X_{\Sigma_0}/K$. For each toric model arithmetic divisor $\overline{D}$, we have an identity
		\begin{displaymath}
			\| \overline{D} \|_{\overline{B}} = \inf \lbrace \varepsilon \in \mathbb{Q}_{>0} \, \vert \, | \gamma_{\overline{D}}(x) | \leq \varepsilon \cdot | \gamma_{\overline{B}}(x) | \textup{ for all } x \in N_{\mathbb{R}} \rbrace. 
		\end{displaymath}
		Moreover, if $B$ is a toric boundary divisor of $X_{\Sigma_0}/K$, then $\overline{B} = (B, g_{B,\textup{can}} + 1)$ is an arithmetic boundary divisor satisfying $\gamma_{\overline{B}} = \Psi_B -1$.
	\end{lem}
	\begin{proof}
		Write $\overline{D} = (D,g_D)$ and $\overline{B}= (B,g_B)$. By density of $U^{\textup{an}}$ and continuity of the Green's functions, if the condition $|g_D | \leq \varepsilon \cdot g_B$ holds on $U^{\textup{an}}$ then it holds globally. By definition of the tropical Green's function, for all $p \in U^{\textup{an}}$ we have $\gamma_{\overline{D}} (\textup{Trop}(p)) = - g_{D} (p)$. Therefore, the condition $|g_D | \leq \varepsilon \cdot g_B$ on $U^{\textup{an}}$ is equivalent to $|\gamma_{\overline{D}}| \leq \varepsilon \cdot | \gamma_{\overline{B}}|$ on $N_{\mathbb{R}}$. Since $\Psi_D = \textup{rec}(\gamma_{\overline{D}})$ and the recession map is monotonous, the condition $|\gamma_{\overline{D}}| \leq \varepsilon \cdot | \gamma_{\overline{B}}|$ implies the corresponding inequality on divisors. The identity follows. The second statement is immediate from Proposition~\ref{gamma}.
	\end{proof}
	We can use the previous lemma to extend the notion of tropical Green's function to the level of toric compactified arithmetic divisors. This relates the group of toric compactified arithmetic divisors with the Banach space $\mathcal{G}(N_{\mathbb{R}})$ introduced in Definition~\ref{G-def}.
	\begin{prop}\label{G/K}
		Let $X_{\Sigma_0}/K$ be a quasi-projective toric variety and $\overline{D} = \lbrace \overline{D}_m \rbrace_{n \in \mathbb{N}}$. The assignment $\overline{D} \mapsto \gamma_{\overline{D}} \coloneqq \lim_{n \in \mathbb{N}} \gamma_{\overline{D}_n}$ induces a group morphism $\mathcal{G} \colon \overline{\textup{Div}}_{\mathbb{T}}(X_{\Sigma_0}/K) \rightarrow \mathcal{G}(N_{\mathbb{R}})$, which is continuous and injective.
	\end{prop}
	\begin{proof}
		By Proposition~\ref{gamma-birational} and taking direct limits, the assignment $\overline{D} \mapsto \gamma_{\overline{D}}$ induces a group morphism on model divisors $\mathcal{G} \colon \overline{\textup{Div}}_{\mathbb{T}}(X_{\Sigma_0}/K)_{\textup{mod}} \rightarrow \mathcal{G}(N_{\mathbb{R}})$. Proposition~\ref{G} shows that this morphism is injective. To see that it is continuous, define $C \coloneqq \| \gamma_{\overline{B}} \|_{\mathcal{G}}$. Then, for all $x \in N_{\mathbb{R}}$, the inequality $ | \gamma_{\overline{D}}(x) | \leq \varepsilon \cdot | \gamma_{\overline{B}}(x) |$ implies $| \gamma_{\overline{D}}(x) | \leq C \, \varepsilon \cdot (1 + \| x \|)$. By Lemma~\ref{local-toric-bdry-norm}, we have the estimate $\| \gamma_{\overline{D}} \|_{\mathcal{G}} \leq C \cdot \| \overline{D} \|_{\mathcal{B}}$. Therefore, $\mathcal{G}$ is continuous and extends to the completion. The injectivity of the map $\mathcal{G}$ on the completion is proven in the same way as Lemma~\ref{toric-comp-SF}, replacing $D$ and $\Psi_D$ with their arithmetic counterparts $\overline{D}$ and $\gamma_{\overline{D}}$.
	\end{proof}
	\begin{lem}\label{G-functorial}
		Let $\iota \colon X_{\Sigma_0} \rightarrow X_{\Sigma_1}$ be a toric birational morphism. Then, there is a continuous injective pullback map $\iota^{\ast} \colon \overline{\textup{Div}}_{\mathbb{T}}(X_{\Sigma_1}/K) \longrightarrow \overline{\textup{Div}}_{\mathbb{T}}(X_{\Sigma_0}/K)$ satisfying $\mathcal{G} = \mathcal{G} \circ \iota^{\ast}$. Moreover, the pullback  $\iota^{\ast}$ is a topological embedding if and only if the map $\iota$ is proper.
	\end{lem}
	\begin{proof}
		Copy the proof of Lemma~\ref{toric-top-emb}. 
	\end{proof}
	\begin{ex}[Canonical Green's functions]\label{can-cptf}
		We can extend the construction in Example~\ref{canonical-GF} to the compactified setting. Let $D = \lbrace D_n \rbrace_{n \in \mathbb{N}}$ be a toric compactified divisor on $X_{\Sigma_0}$. Denote by $\Psi_D$ and $\Psi_{D_n}$ the support functions of $D$ and $D_n$, respectively. Then, consider the sequence $\lbrace \overline{D}_{n,\textup{can}} \rbrace_{n \in \mathbb{N}}$, where $\overline{D}_{n,\textup{can}}$ is the model divisor $D_n$ with its canonical Green's function $g_{D_n,\textup{can}}$. By definition, the associated tropical Green's functions satisfy $\gamma_{D_n,\textup{can}} = \Psi_{D_n}$ and the sequence$\lbrace \gamma_{D_{n,\textup{can}}} \rbrace_{n \in \mathbb{N}}$ converges to the support function $\Psi_D$ of $D$. Lemma~\ref{local-toric-bdry-norm} implies that the sequence $\lbrace \overline{D}_{n,\textup{can}} \rbrace_{n \in \mathbb{N}}$ is Cauchy in the boundary topology, and we write $\overline{D}_{\textup{can}} = (D, g_{D,\textup{can}})$ for its limit. We call $g_{D,\textup{can}}$ the \textit{canonical Green's function} of the toric compactified divisor $D$. By Proposition~\ref{G/K}, its associated tropical Green's function is $\gamma_{D,\textup{can}} = \Psi_D$. The assignment $D \mapsto \overline{D}_{\textup{can}}$ induces a continuous injective morphism
		\begin{displaymath}
			\textup{can} \colon \textup{Div}_{\mathbb{T}}(X_{\Sigma_0}/K) \longrightarrow \overline{\textup{Div}}_{\mathbb{T}}(X_{\Sigma_0}/K),
		\end{displaymath}
		which we also call the \textit{canonical map}. The assignment $(D,g) \mapsto D$ induces a \textit{forgetful map}
		\begin{displaymath}
			\textup{for} \colon \overline{\textup{Div}}_{\mathbb{T}}(X_{\Sigma_0}/K) \longrightarrow \textup{Div}_{\mathbb{T}}(X_{\Sigma_0}/K).
		\end{displaymath}
		One trivially verifies that the canonical map is a continuous section of the forgetful map.
	\end{ex}
	Our next task is to compare the theory of toric compactified arithmetic divisors with the compactified arithmetic divisors of Yuan and Zhang (see Definition~\ref{cptf-dfn}).
	\begin{rem}[Compatibility with Yuan-Zhang's theory]\label{toric-cptf-1}
		Let $\overline{D} = \lbrace \overline{D}_n \rbrace_{n \in \mathbb{N}}$ be a toric compactified arithmetic divisor on $X_{\Sigma_0}/K$. Write $\overline{D}_n = (D_n, g_n)$ and observe that:
		\begin{enumerate}
			\item The sequence $\lbrace D_n \rbrace_{n \in \mathbb{N}}$ converges to a toric compactified divisor $D \in \textup{Div}_{\mathbb{T}}(X_{\Sigma}/K)$. In particular, $D|_{X_{\Sigma_0}} \in \textup{Div}_{\mathbb{T}}(X_{\Sigma_0})_{\mathbb{Q}}$.
			\item The sequence of functions $\lbrace g_n \rbrace_{n \in \mathbb{N}}$ converges uniformly on compact subsets of $X_{\Sigma_0}^{\textup{an}}$ to an $\mathbb{S}$-invariant Green's function $g$ of continuous type for the divisor $D|_{X_{\Sigma_0}}$. 
		\end{enumerate}
		Following Remark~\ref{cptf-1}, the assignment $\overline{D} \mapsto (D|_{X_{\Sigma_0}}, g)$ induces an isometric group morphism $\overline{\textup{Div}}_{\mathbb{T}}(X_{\Sigma_0}/K) \rightarrow \overline{\textup{Div}}(X_{\Sigma_0})_{\mathbb{Q}}$, where the target is equipped with the boundary topology induced by the image of a toric boundary divisor $\overline{B}$ of $X_{\Sigma_0}/K$. Observe that the subspace $\overline{\textup{Div}}_{\mathbb{T}}(X_{\Sigma_0})_{\mathbb{Q}}$ of toric divisors of $\overline{\textup{Div}}(X_{\Sigma_0})_{\mathbb{Q}}$ is a closed subgroup. Indeed, for every Cauchy sequence $\lbrace (E_n,g_n) \rbrace_{n \in \mathbb{N}}$, the induced sequence $\lbrace E_n \rbrace_{n \in \mathbb{N}}$ is eventually constant, and the limit of the sequence $\lbrace g_n \rbrace_{n \in \mathbb{N}}$ of $\mathbb{S}$-invariant functions is $\mathbb{S}$-invariant. Then, the morphism
		\begin{displaymath}
			\overline{\textup{Div}}_{\mathbb{T}}(X_{\Sigma_0}/K) \longrightarrow \overline{\textup{Div}}_{\mathbb{T}}(X_{\Sigma_0})_{\mathbb{Q}} \subset \overline{\textup{Div}}(X_{\Sigma_0})_{\mathbb{Q}}
		\end{displaymath}
		identifies the group of toric compactified arithmetic divisors as a closed topological subgroup of the arithmetic divisors of continuous type $\overline{\textup{Div}}_{\mathbb{T}}(X_{\Sigma_0})_{\mathbb{Q}}$ on the quasi-projective variety $X_{\Sigma_0}/K$. The injectivity follows from the commutativity of the diagram
		\begin{center}
			\begin{tikzcd}
				\overline{\textup{Div}}_{\mathbb{T}}(X_{\Sigma_0}/K) \arrow[r] \arrow[d, "{\mathcal{G}}"'] & \overline{\textup{Div}}_{\mathbb{T}}(X_{\Sigma_0})_{\mathbb{Q}} \arrow[d, "{\mathcal{G}}"] \\
				\mathcal{G}(N_{\mathbb{R}}) \arrow[r, hook] & C^{0}(N_{\mathbb{R}})
			\end{tikzcd}
		\end{center}
		and the fact that the maps $\mathcal{G}$ are injective (Propositions~\ref{G}~and~\ref{G/K}). It is clear from the definition of the group $\overline{\textup{Div}}_{\mathbb{T}}(X_{\Sigma_0}/K)$ that, under the above identification, it becomes a subgroup of the compactified arithmetic divisors $\overline{\textup{Div}}(X_{\Sigma_0}/K)$. Now, by example \ref{can-cptf}, the composition of $\textup{for} \circ \textup{can}$ is the identity on the group of toric compactified divisors $\textup{Div}_{\mathbb{T}}(X_{\Sigma_0}/K)$. Then, by the exact sequence in Theorem~\ref{exact-seq-local}, the group of toric compactified divisors identifies as a closed topological subgroup of the compactified divisors  $\textup{Div}(X_{\Sigma_0}/K)$. This shows the injectivity of the vertical maps in Lemma~\ref{shrinking}.
	\end{rem}
	The $\mathbb{S}$-symmetrization morphism introduced in Construction~\ref{S-inv-proj} can be defined in the same way for a toric quasi-projective $X_{\Sigma_0}/K$. Then, we have the following proposition.
	\begin{prop}\label{S-inv-cont}
		The $\mathbb{S}$-symmetrization morphism is continuous in the boundary topology. In particular, it induces a continuous morphism $( \cdot )^{\mathbb{S}} \colon C^{0}(X_{\Sigma_0}^{\textup{an}})_{\textup{cptf}} \rightarrow C^{0}_{\mathbb{S}}(X_{\Sigma_0}^{\textup{an}})_{\textup{cptf}}$, where $C^{0}_{\mathbb{S}}(X_{\Sigma_0}^{\textup{an}})_{\textup{cptf}}$ is the subspace of $\mathbb{S}$-invariant functions of the space of $C^{0}(X_{\Sigma_0}^{\textup{an}})_{\textup{cptf}}$ introduced in Remark~\ref{cptf-1}.
	\end{prop}
	\begin{proof}
		Let us prove the continuity. Let $f \in C^{0}(X_{\Sigma_0}^{\textup{an}})$. Without loss of generality, we may choose a toric boundary divisor as in Remark~\ref{toric-cptf-1}, and denote by $(0,g_B)$ the Green's function inducing the boundary topology. If $K$ is archimedean, for all $p \in X_{\Sigma_0}^{\textup{an}}$ we have
		\begin{displaymath}
			f^{\mathbb{S}}(p) = \int_{\mathbb{S}} f(t \cdot p) \, \textup{d}\mu_{\mathbb{S}}(t).
		\end{displaymath}
		If $K$ is non-archimedean, for all $p \in X_{\Sigma_0}^{\textup{an}}$ we have $	f^{\mathbb{S}}(p) = f (\theta_{\Sigma}( \textup {Trop}(p)))$. In both cases, we obtain an estimate
		\begin{displaymath}
			|f^{\mathbb{S}}(p)| \leq \max_{q \in \mathbb{S} \cdot p} |f(q)|.
		\end{displaymath}
		Since $g_{B}$ is $\mathbb{S}$-invariant, for each $\varepsilon>0$, the condition $|f(p)| \leq \varepsilon \cdot g_B (p)$ for all $p \in  X_{\Sigma_0}^{\textup{an}}$ implies that 
		\begin{displaymath}
			|f^{\mathbb{S}}(p)| \leq \varepsilon \cdot g_B (p), \quad \textup{for all } p \in X_{\Sigma}^{\textup{an}}.
		\end{displaymath}
		The continuity follows. By definition, the space $C^{0}(X_{\Sigma_0}^{\textup{an}})_{\textup{cptf}}$ is the closure in the boundary topology of $C^{0}(X_{\Sigma_0}^{\textup{an}})_{\textup{mod}}$. Then, the second statement follows.
	\end{proof}
	We are now ready to state and prove the main results of this subsection. This is a toric version of Yuan-Zhang's Theorem~\ref{exact-seq-local}.
	\begin{thm}\label{exact-seq-local-toric}
		Let $X_{\Sigma_0}/K$ be a toric quasi-projective variety. Then:
		\begin{enumerate}[label=(\roman*)]
			\item The forgetful map induces a split exact sequence
			\begin{center}
				\begin{tikzcd}
					0 \arrow[r] & C^{0}_{\mathbb{S}}(X_{\Sigma_0}^{\textup{an}})_{\textup{cptf}} \arrow[r] &  \overline{\textup{Div}}_{\mathbb{T}}(X_{\Sigma_0}/K) \arrow[r, "{ \textup{for}}"] & \textup{Div}_{\mathbb{T}}(X_{\Sigma_0}/K) \arrow[l, bend right = 30, "{\textup{can}}"'] \arrow[r] &0.
				\end{tikzcd}
			\end{center}
			\item There is an induced isomorphism
			\begin{displaymath}
				\overline{\textup{Div}}_{\mathbb{T}}(X_{\Sigma_0}/K) \cong  C^{0}_{\mathbb{S}}(X_{\Sigma_0}^{\textup{an}})_{\textup{cptf}} \oplus \textup{Div}_{\mathbb{T}}(X_{\Sigma_0}/K).
			\end{displaymath}
			\item Let $(X_{\Sigma}, (B,g_B))$ be a toric boundary divisor of $X_{\Sigma_0}/K$. Then, 
			\begin{displaymath}
				C^{0}_{\mathbb{S}}(X_{\Sigma_0}^{\textup{an}})_{\textup{cptf}} = g_{B} \cdot  \lbrace h \in C^{0}_{\mathbb{S}}(X_{\Sigma}^{\textup{an}}) \, \vert \, h( X_{\Sigma}^{\textup{an}} \setminus X_{\Sigma_0}^{\textup{an}}) = 0 \rbrace.
			\end{displaymath}
		\end{enumerate}
	\end{thm}
	\begin{proof}
		Part \textit{(ii)} follows immediately from \textit{(i)}, which we now prove. By Example~\ref{can-cptf}, the forgetful map induces a split exact sequence
		\begin{center}
			\begin{tikzcd}
				0 \arrow[r] & \textup{Ker}(\textup{for}) \arrow[r] &  \overline{\textup{Div}}_{\mathbb{T}}(X_{\Sigma_0}/K) \arrow[r, "{ \textup{for}}"] & \textup{Div}_{\mathbb{T}}(X_{\Sigma_0}/K) \arrow[l, bend right = 30, "{\textup{can}}"'] \arrow[r] &0.
			\end{tikzcd}
		\end{center}
		We just need to show that the kernel of the forgetful map $ \textup{Ker}(\textup{for})$ is exactly $C^{0}_{\mathbb{S}}(X_{\Sigma_0}^{\textup{an}})_{\textup{cptf}}$. By Remark~\ref{toric-cptf-1} and the exact sequence in Theorem~\ref{exact-seq-local}, we have inclusions 
		\begin{displaymath}
			C^{0}_{\mathbb{S}}(X_{\Sigma_0}^{\textup{an}})_{\textup{mod}} \subset \textup{Ker}(\textup{for}) \subset C^{0}_{\mathbb{S}}(X_{\Sigma_0}^{\textup{an}})_{\textup{cptf}}.
		\end{displaymath}
		Observe that $\textup{Ker}(\textup{for})$ is a closed subgroup in the complete group $ \overline{\textup{Div}}_{\mathbb{T}}(X_{\Sigma_0}/K) $, therefore, it is complete. We will use this completeness to show the inclusion
		\begin{displaymath}
			C^{0}_{\mathbb{S}}(X_{\Sigma_0}^{\textup{an}})_{\textup{cptf}} \subset  \textup{Ker}(\textup{for}).
		\end{displaymath}
		Let $f \in C^{0}_{\mathbb{S}}(X_{\Sigma_0}^{\textup{an}})_{\textup{cptf}}$, by definition there is a sequence $\lbrace f_n \rbrace_{n \in \mathbb{N}}$ of functions in $C^{0}(X_{\Sigma_0}^{\textup{an}})_{\textup{mod}}$ converging in the boundary topology to $f$. By Proposition~\ref{S-inv-cont}, the sequence $\lbrace f_{n}^{\mathbb{S}} \rbrace_{n \in \mathbb{N}}$ converges in the boundary topology to $f$. Note that the functions $f_{n}^{\mathbb{S}}$ are not necessarily of model type. However, Construction~\ref{S-inv-proj} implies that each $f_n \in  \textup{Ker}(\textup{for})$. It follows that $f \in \textup{Ker}(\textup{for})$. It remains to show part \textit{(iii)}. By Theorem~\ref{exact-seq-local}, we know that
		\begin{displaymath}
			C^{0}(X_{\Sigma_0}^{\textup{an}})_{\textup{cptf}} = g_{B} \cdot \lbrace h \in C^{0}(X_{\Sigma}^{\textup{an}}) \, \vert \, h( X_{\Sigma}^{\textup{an}} \setminus X_{\Sigma_0}^{\textup{an}}) = 0 \rbrace.
		\end{displaymath}
		Since $g_B$ is $\mathbb{S}$-invariant, the result follows by imposing $\mathbb{S}$-invariance on both sides of the equation.
	\end{proof}
	\begin{cor}\label{ker-for-toric}
		With the notation of Theorem~\ref{exact-seq-local-toric}, the morphism $\mathcal{G}$ induces a bijection
		\begin{displaymath}
			C^{0}_{\mathbb{S}}(X_{\Sigma_0}^{\textup{an}})_{\textup{cptf}} \cong \gamma_{\overline{B}} \cdot \lbrace \eta \in  C^{0}(N_{\Sigma}) \, \vert \, \eta( N_{\Sigma} \setminus N_{\Sigma_0}) = 0 \rbrace.
		\end{displaymath}
		In other words, the set on the right-hand side is the subspace of $C^{0}(N_{\Sigma})$ consisting of functions vanishing at the boundary $N_{\Sigma} \setminus N_{\Sigma_0}$.
	\end{cor}
	\begin{proof}
		By Theorem~\ref{exact-seq-local-toric}, we have 
		\begin{displaymath}
			C^{0}_{\mathbb{S}}(X_{\Sigma_0}^{\textup{an}})_{\textup{cptf}} = g_{B} \cdot \lbrace h \in C_{\mathbb{S}}^{0}(X_{\Sigma}^{\textup{an}}) \, \vert \, h( X_{\Sigma}^{\textup{an}} \setminus X_{\Sigma_0}^{\textup{an}}) = 0 \rbrace.
		\end{displaymath}
		By the functorial property of the tropicalization map (Construction~\ref{tropical-toric-polytope}), we have a commutative diagram
		\begin{center}
			\begin{tikzcd}
				X_{\Sigma_0}^{\textup{an}} \arrow[r, "{\textup{Trop}}"] \arrow[d] & N_{\Sigma_0}  \arrow[d]  \\
				X_{\Sigma}^{\textup{an}} \arrow[r, "{\textup{Trop}}"] & N_{\Sigma}
			\end{tikzcd}
		\end{center}
		Moreover, the tropicalization map identifies the toric variety $N_{\Sigma}$ (resp. $N_{\Sigma_0}$) with the topological quotient of $X_{\Sigma}^{\textup{an}}$ (resp. $X_{\Sigma_0}^{\textup{an}}$) by compact torus $\mathbb{S}$. Combining these facts, the identity $h = \eta \circ \textup{Trop}$ given by $\mathbb{S}$-invariance induces a bijection
		\begin{displaymath}
			\lbrace h \in C_{\mathbb{S}}^{0}(X_{\Sigma}^{\textup{an}}) \, \vert \, h( X_{\Sigma}^{\textup{an}} \setminus X_{\Sigma_0}^{\textup{an}}) = 0 \rbrace \cong \lbrace \eta \in  C^{0}(N_{\Sigma}) \, \vert \, \eta( N_{\Sigma} \setminus N_{\Sigma_0}) = 0 \rbrace.
		\end{displaymath}
		The result follows from the equality $\mathcal{G}(\overline{B}) = \gamma_{\overline{B}}$.
	\end{proof}
	Finally, we compute the group of toric compactified arithmetic divisors on the torus.
	\begin{thm}\label{arith-toric-div-UK}
		Let $U/K$ be the split $d$-dimensional torus. The morphism 
		\begin{displaymath}
			\mathcal{G} \colon  \overline{\textup{Div}}_{\mathbb{T}}(U/K) \longrightarrow  \mathcal{G}(N_{\mathbb{R}})
		\end{displaymath}
		is an isomorphism of Banach spaces, where $\mathcal{C}(N_{\mathbb{R}})$ is the space of continuous conical functions, $\mathcal{SL}(N_{\mathbb{R}})$ is the space of continuous sublinear functions and $\mathcal{G}(N_{\mathbb{R}})$ is their direct sum, all equipped with the $\mathcal{G}$-norm (see Definition~\ref{G-def}). In particular,
		\begin{displaymath}
			\overline{\textup{Div}}_{\mathbb{T}}(U/K) \cong \mathcal{SL}(N_{\mathbb{R}}) \oplus  \mathcal{C}(N_{\mathbb{R}}).
		\end{displaymath}
	\end{thm}
	\begin{proof}
		By Theorem~\ref{exact-seq-local-toric}, we know that the map $\mathcal{G}$ induces an isomorphism
		\begin{displaymath}
			\overline{\textup{Div}}_{\mathbb{T}}(U/K) \cong  C^{0}_{\mathbb{S}}(U^{\textup{an}})_{\textup{cptf}} \oplus \textup{Div}_{\mathbb{T}}(U/K).
		\end{displaymath}
		By Theorem~\ref{torgeomdiv}, we have an isomorphism of Banach spaces $\textup{Div}_{\mathbb{T}}(U/K) \cong  \mathcal{C}(N_{\mathbb{R}})$. We only need to show that $\mathcal{G}$ restricts to an isomorphism of Banach spaces $C^{0}_{\mathbb{S}}(U^{\textup{an}})_{\textup{cptf}} =  \mathcal{SL}(N_{\mathbb{R}})$. By Lemma~\ref{local-toric-bdry-norm}, if $B$ is a toric boundary divisor of $U/K$, then, the function $\Psi_B - 1$ corresponds to the tropical Green's function $\gamma_{\overline{B}}$ of the toric boundary divisor $(B, g_{B,\textup{can}} -1)$. Moreover, we have the identity
		\begin{displaymath}
			\| \overline{D} \|_{\overline{B}} = \inf \lbrace \varepsilon \in \mathbb{Q}_{>0} \, \vert \, | \gamma_{\overline{D}}(x) | \leq \varepsilon \cdot | \gamma_{\overline{B}}(x) | \textup{ for all } x \in N_{\mathbb{R}} \rbrace. 
		\end{displaymath}
		Observe that the norm induced by the right-hand side is equivalent to the $\mathcal{G}$-norm. Then, by Corollary~\ref{ker-for-toric}, we only need to verify that
		\begin{displaymath}
			\mathcal{SL}(N_{\mathbb{R}}) = \gamma_{\overline{B}} \ \lbrace \eta \in  C^{0}(N_{\Sigma}) \, \vert \, \eta( N_{\Sigma} \setminus N_{\Sigma_0}) = 0 \rbrace.
		\end{displaymath}
		Let $f$ be a sublinear function. By Definition~\ref{G-def} and the equivalence of norms, for each $\varepsilon > 0$, there exists a radius $R_{\varepsilon} > 0$ such that for all $\| x \| > R_{\varepsilon}$ we have $|f(x)| \leq \varepsilon \cdot |\gamma_{\overline{B}}(x)|$. Then, the function $f / \gamma_{\overline{B}}$ extends continuously to a function vanishing at the boundary $N_{\Sigma} \setminus N_{\Sigma_0}$. Therefore,
		\begin{displaymath}
			\mathcal{SL}(N_{\mathbb{R}}) \subset \gamma_{\overline{B}} \cdot  \lbrace \eta \in  C^{0}(N_{\Sigma}) \, \vert \, \eta( N_{\Sigma} \setminus N_{\Sigma_0}) = 0 \rbrace.
		\end{displaymath}
		To show the reverse inclusion, note that the toric variety $X_{\Sigma}$ is projective; therefore, the tropical toric variety $N_{\Sigma}$ is homeomorphic to a convex polytope. In particular, it is compact. Consider a continuous function $h$ vanishing on the set $N_{\Sigma} \setminus N_{\Sigma_0}$. Then, for each $\varepsilon >0$, there is a radius $R_{\varepsilon} > 0$ such that for all $\| x \| > R_{\varepsilon}$ we have $	|h(x)| \leq \varepsilon$. It follows that the function $h \cdot \gamma_{\overline{B}}$ is sublinear.
	\end{proof}
	
	\subsection{The nef cone and local toric heights}
	In this section, we extend the definition of the local toric height of Burgos, Philippon, and Sombra~\cite{BPS} to the quasi-projective case of Yuan and Zhang~\cite{Y-Z}. The main result is that the formula in Theorem~\ref{local-toric-height-proj} holds in this setting. A crucial difference with the projective case is that the local toric height of a toric quasi-projective variety with respect to an arbitrary nef toric compactified arithmetic divisor is not necessarily finite. Thus, we characterize the finiteness of the local toric height in terms of the associated tropical Green's function of the nef toric divisor. To do this, we first describe the cone of nef toric compactified arithmetic divisors.
	\begin{lem}\label{nef-tor-model-arith-div}
		Let $K$ be non-archimedean and $X_{\Sigma_0}/K$ be a quasi-projective toric variety. The map $\mathcal{G}$ induces an identification between the cone $\overline{\textup{Div}}_{\mathbb{T}}^{\textup{nef}}(X_{\Sigma_0}/K)_{\textup{mod}}$ of nef model arithmetic divisors and the set of all concave rational piecewise affine functions $f$ on $N_{\mathbb{R}}$ whose recession function $\textup{rec}(f)$ is a support function on some fan $\Sigma \in \textup{PF}(N_{\mathbb{R}}, \Sigma_0)$.
	\end{lem}
	\begin{proof}
		The result follows from Theorem~\ref{nef-tor-arith-proj} and taking direct limits as in Remark~\ref{direct-lim}.
	\end{proof}
	\begin{rem}\label{nef-cont-type}
		We have a similar description for the nef toric model arithmetic divisors of continuous type. Indeed, Proposition~\ref{nef-tor-arith-proj} implies
		\begin{displaymath}
			\varinjlim_{X_{\Sigma} \in \textup{PM}_{\mathbb{T}}(X_{\Sigma_0})} \overline{\textup{Div}}_{\mathbb{T}}^{\textup{nef}}(X_{\Sigma_0}/K)_{\mathbb{Q}} \cong \bigcup_{\Sigma \in \textup{PF}(N_{\mathbb{R}}, \Sigma_0)} \mathcal{G}^{+}_{\textup{rbd}}(\Sigma, \mathbb{Q}).
		\end{displaymath}
		where $\mathcal{G}^{+}_{\textup{rbd}}(\Sigma, \mathbb{Q})$ is the set of concave functions $f$ on $N_{\mathbb{R}}$ whose recession function $\textup{rec}(f)$ is a rational support function on $\Sigma$ and $|f- \textup{rec}(f)|$ is bounded. Moreover, parts \textit{(ii)} and \textit{(iii)} of Proposition~\ref{S-inv-cont} show that every $\mathbb{S}$-invariant Green's function of continuous and psh type is the limit in the uniform norm of a sequence of $\mathbb{S}$-invariant Green's functions of smooth and psh type. Then, the closures in the boundary topology of
		\begin{displaymath}
			\overline{\textup{Div}}_{\mathbb{T}}^{\textup{nef}}(X_{\Sigma_0}/K)_{\textup{mod}} \quad \textup{and} \varinjlim_{X_{\Sigma} \in \textup{PM}_{\mathbb{T}}(X_{\Sigma_0})} \overline{\textup{Div}}_{\mathbb{T}}^{\textup{nef}}(X_{\Sigma}/K)_{\mathbb{Q}}
		\end{displaymath}
		are the same. Therefore, we can use the nef toric model arithmetic divisors of continuous type to compute the nef cone $\overline{\textup{Div}}_{\mathbb{T}}^{\textup{nef}}(U/K)$.
	\end{rem}
	\begin{thm}\label{nef-cone-UK}
		Let $U/K$ be the $d$-dimensional split torus over $K$. The morphism $\mathcal{G}$, which to a toric compactified arithmetic divisor $\overline{D}$ assigns its tropical Green's function $\gamma_{\overline{D}}$, restricts to a continuous isomorphism of cones $\mathcal{G} \colon \overline{\textup{Div}}_{\mathbb{T}}^{\textup{nef}}(U/K) \rightarrow \mathcal{G}^{+}(N_{\mathbb{R}})$, where $\mathcal{G}^{+}(N_{\mathbb{R}})$ is the cone of globally Lipschitz concave functions on $N_{\mathbb{R}}$.
	\end{thm}
	\begin{proof}
		By Theorem~\ref{arith-toric-div-UK}, we have an isomorphism of Banach spaces $	\mathcal{G} \colon \overline{\textup{Div}}_{\mathbb{T}}(U/K) \rightarrow \mathcal{G}(N_{\mathbb{R}})$. Since the tropical Green's function $\gamma_{\overline{D}}$ of a nef toric model arithmetic divisor $\overline{D}$ is concave, we know that the image of $\overline{\textup{Div}}_{\mathbb{T}}^{\textup{nef}}(U/K)$ under $\mathcal{G}$ is contained in the cone $\mathcal{G}^{+}(N_{\mathbb{R}})$.
		
		For the reverse inclusion, it is enough to show that the image of $\overline{\textup{Div}}_{\mathbb{T}}^{\textup{nef}}(U/K)_{\textup{mod}}$ is dense in $\mathcal{G}^{+}(N_{\mathbb{R}})$. If $K$ is non-archimedean, Lemma~\ref{nef-tor-model-arith-div} identifies this image with the set $\mathcal{PA}^{+}(N_{\mathbb{R}}, \mathbb{Q})$ of concave piecewise affine rational functions on $N_{\mathbb{R}}$. Then, apply Theorem~\ref{closure-P}, which states that $\mathcal{PA}^{+}(N_{\mathbb{R}}, \mathbb{Q})$ is dense in $ \mathcal{G}^{+}(N_{\mathbb{R}})$. If $K$ is archimedean, by Remark~\ref{nef-cont-type} we know that the image under $\mathcal{G}$ of the nef cone $\overline{\textup{Div}}_{\mathbb{T}}^{\textup{nef}}(X_{\Sigma_0}/K) $ is the closure in the $\mathcal{G}$-norm of the set
		\begin{displaymath}
			\bigcup_{\Sigma \in \textup{PF}(N_{\mathbb{R}})} \mathcal{G}^{+}_{\textup{rbd}}(\Sigma, \mathbb{Q}),
		\end{displaymath} 
		which contains the cone $\mathcal{PA}^{+}(N_{\mathbb{R}}, \mathbb{Q})$. Applying Theorem~\ref{closure-P}, we obtain the result.
	\end{proof}
	By Legendre-Fenchel duality, we obtain an alternative description of the nef cone $\overline{\textup{Div}}_{\mathbb{T}}^{\textup{nef}}(U/K)$ in terms of the corresponding roof functions.
	\begin{cor}\label{roof-functions-UK}
		There is a bijection between the nef cone $\overline{\textup{Div}}_{\mathbb{T}}^{\textup{nef}}(U/K)$ and the set of closed concave functions $\vartheta \colon M_{\mathbb{R}} \rightarrow \mathbb{R}_{-\infty}$ with bounded effective domain $\textup{dom}(\vartheta)$. It is given by the assignment $\overline{D} \mapsto \vartheta_{\overline{D}} = \gamma_{\overline{D}}^{\vee}$, where $\gamma_{\overline{D}}^{\vee}$ is the Legendre-Fenchel dual of $\gamma_{\overline{D}}$. Moreover, the set $\Delta \subset M_{\mathbb{R}}$ given by the closure of $\textup{dom}(\vartheta)$ is the compact convex set corresponding to the nef toric compactified arithmetic divisor $\overline{D} = \mathcal{G}^{-1}(\vartheta^{\vee})$.
	\end{cor}
	\begin{proof}
		It is immediate from Theorem~\ref{nef-cone-UK} and Proposition~\ref{conical-compact}
	\end{proof}
	
	The next task is to continuously extend the local toric height of \cite{BPS} to the compactified arithmetic setting. First, have the following observation. Let $\iota \colon X_{\Sigma_0} \rightarrow X_{\Sigma_1}$ be a toric open embedding of quasi-projective toric varieties over $K$. Then:
	\begin{enumerate}
		\item The pullback morphism $\iota^{\ast} \colon \overline{\textup{Div}}_{\mathbb{T}}(X_{\Sigma_1}/K) \rightarrow  \overline{\textup{Div}}_{\mathbb{T}}(X_{\Sigma_0}/K)$ from Corollary~\ref{G-functorial} maps the cone $\overline{\textup{Div}}_{\mathbb{T}}^{\textup{nef}}(X_{\Sigma_1}/K)_{\textup{mod}}$ into the cone $\overline{\textup{Div}}_{\mathbb{T}}^{\textup{nef}}(X_{\Sigma_0}/K)_{\textup{mod}}$. Therefore, the map $\iota^{\ast}$ restricts to a continuous injective morphism of cones
		\begin{displaymath}
			\iota^{\ast} \colon \overline{\textup{Div}}_{\mathbb{T}}^{\textup{nef}}(X_{\Sigma_1}/K) \longrightarrow  \overline{\textup{Div}}_{\mathbb{T}}^{\textup{nef}}(X_{\Sigma_0}/K).
		\end{displaymath}
		\item The projection formula in Theorem~\ref{local-h-thm} applied to the local toric height induces a local toric height pairing on the cone of nef toric model divisors on $X_{\Sigma_0}$. It is given by
		\begin{displaymath}
			\textup{h}^{\textup{tor}}(X_{\Sigma_0} ; \overline{D}_0 , \ldots , \overline{D}_d) \coloneqq \textup{h}^{\textup{tor}}(X_{\Sigma} ; \overline{D}_0 , \ldots , \overline{D}_d),
		\end{displaymath}
		where $X_{\Sigma}/K$ is a projective toric model of $X_{\Sigma_0}$ where the nef toric arithmetic divisors $\overline{D}_0, \ldots , \overline{D}_d$ are all defined. Since every projective toric model of $X_{\Sigma_1}$ is a projective toric model of $X_{\Sigma_0}$, the local toric height on model divisors satisfies
		\begin{displaymath}
			\textup{h}^{\textup{tor}}(X_{\Sigma_0} ; \iota^{\ast} \overline{D}_0 , \ldots , \iota^{\ast} \overline{D}_d) = \textup{h}^{\textup{tor}}(X_{\Sigma_1} ; \overline{D}_0 , \ldots , \overline{D}_d).
		\end{displaymath}
	\end{enumerate}
	We conclude that a continuous extension of the local toric height pairing to the nef cone of toric compactified arithmetic divisors satisfies the functorial property
	\begin{displaymath}
		\textup{h}^{\textup{tor}}(X_{\Sigma_0} ; \iota^{\ast} \overline{D}_0 , \ldots , \iota^{\ast} \overline{D}_d) = \textup{h}^{\textup{tor}}(X_{\Sigma_1} ; \overline{D}_0 , \ldots , \overline{D}_d).
	\end{displaymath}
	Since our primary goal is to study heights, we will restrict ourselves to the initial object in the category of $d$-dimensional quasi-projective toric varieties over $K$, namely, the $d$-dimensional split torus $U/K$.
	
	In Subsection~\ref{3-1}, we extended the intersection pairing on model toric divisors to the nef cone in two steps:
	\begin{enumerate}
		\item First, we showed that for each nef toric compactified divisor $D$, represented by a Cauchy sequence $\lbrace D_n \rbrace_{n \in \mathbb{N}}$ consisting of nef toric model divisors, the limit $\lim D_{n}^{d}$ exists and it is equal to $d! \, \textup{Vol}_M (\Delta_{D})$. In particular, it is independent of the choice of nef sequence $\lbrace D_n \rbrace_{n \in \mathbb{N}}$.
		\item Then, we used the multilinearity of intersection products and limits to give a formula for the intersection number $D_1 \cdot \ldots \cdot D_d$ of nef toric compactified divisors.
	\end{enumerate}
	This naive approach no longer works. The following example shows a new phenomenon appearing in the arithmetic situation. Concretely, the limit of local toric heights of an arbitrary Cauchy sequence of nef toric model arithmetic divisors might not converge.
	\begin{ex}\label{CEH-1}
		Let $K$ be non-archimedean and $U = \mathbb{G}_m/K$. For each $n \geq 1$, consider the nef toric model arithmetic divisor $\overline{D}_n$ corresponding to the concave rational piecewise affine function $\gamma_n \colon N_{\mathbb{R}} \rightarrow \mathbb{R}$, given by
		\begin{displaymath}
			\gamma_n (x) \coloneqq \begin{cases} x/n + n^2, & x \leq -n^3 \\ 0, & -n^3 \leq x \leq 0 \\ -x, & 0 \leq x. \end{cases}
		\end{displaymath}
		Observe that the sequence $\lbrace \gamma_n \rbrace_{n \in \mathbb{N}}$ is increasing, and it is Cauchy in the $\mathcal{G}$-norm. Indeed, if $n \geq m$ we have
		\begin{displaymath}
			(\gamma_n - \gamma_m ) (x) = \begin{cases} x/n - x/m + n^2 - m^2, & x \leq -n^3 \\ -x/m - m^{2}, & -n^3 \leq x \leq -m^3 \\ 0, & -m^3 \leq x. \end{cases}
		\end{displaymath}
		A quick calculation shows that $| \gamma_n (x) - \gamma_m (x)| \leq 4/m \cdot (1 + |x|)$ for all  $x\in N_{\mathbb{R}}$. Then, the sequence $\lbrace \gamma_n \rbrace_{n \in \mathbb{N}}$ converges in the $\mathcal{G}$-norm to the concave rational piecewise affine function 
		\begin{displaymath}
			\gamma (x) \coloneqq \begin{cases} 0, & x \leq 0 \\ -x,  & 0\leq x. \end{cases}
		\end{displaymath}
		The sequence $\lbrace \overline{D}_n \rbrace_{n \in \mathbb{N}}$ represents the nef toric model arithmetic divisor $\overline{D}$ corresponding to $\gamma$. Computing the Legendre-Fenchel duals $\vartheta_n \coloneqq \gamma_{n}^{\vee} $ and $\vartheta \coloneqq \gamma^{\vee}$, we obtain
		\begin{displaymath}
			\vartheta_n (y) = \begin{cases} 0, & -1 \leq y \leq 0 \\ -n^3 y, & 0 \leq y \leq 1/n \\ - \infty, & \textup{otherwise.} \end{cases} \quad \textup{and} \quad \vartheta(y) = \begin{cases} 0, & -1 \leq y \leq 0 \\ - \infty, & \textup{otherwise.} \end{cases}
		\end{displaymath}
		Now, we use Theorem~\ref{local-toric-height-proj} to compute the heights. We get
		\begin{displaymath}
			\textup{h}^{\textup{tor}}(U; \overline{D}_n ) = 2 \int_{-1}^{1/n} \vartheta_n (y) \textup{d}y = -n \quad \textup{ and }  \quad \textup{h}^{\textup{tor}}(U; \overline{D}) = 2 \int_{-1}^{0} \vartheta(y) \textup{d}y = 0.
		\end{displaymath}
		Summarizing, the sequence of nef toric model arithmetic divisors $\lbrace \overline{D}_n \rbrace_{n \in \mathbb{N}}$ is decreasing and converges in the boundary topology to the nef toric model arithmetic divisor $\overline{D}$. Nevertheless, we have
		\begin{displaymath}
			\lim_{n \rightarrow \infty} \textup{h}^{\textup{tor}}(U; \overline{D}_n ) = - \infty \quad \textup{while}  \quad \textup{h}^{\textup{tor}}(U; \overline{D}) =0.
		\end{displaymath}
	\end{ex}
	The following remark dissects the divergence phenomenon in the previous example and motivates our approach for the rest of this section.
	\begin{rem}\label{motivation-local-toric-h}
		Let $\lbrace \overline{D}_n \rbrace_{n \in \mathbb{N}}$ be a decreasing sequence of nef toric compactified arithmetic divisors converging to the toric compactified arithmetic divisor $\overline{D}$. Then, the sequence of roof functions $\lbrace \vartheta_n \rbrace_{n \in \mathbb{N}}$ is decreasing, bounded above by a constant $C$, and converges pointwise to $\vartheta$, the roof function of $\overline{D}$. In particular, the sequence of effective domains $\lbrace \textup{dom}(\vartheta_n) \rbrace_{n \in \mathbb{N}}$ is decreasing and their intersection is $\textup{dom}(\vartheta)$. Then, we have the following observations:
		\begin{enumerate}[label=(\arabic*)]
			\item For all $y \in \textup{dom}(\vartheta)$, we have an inequality $C \geq \vartheta_n (y) \geq \vartheta (y) > - \infty$. Therefore, the integral of $\vartheta_n$ on $\textup{dom}(\vartheta)$ is bounded below by the integral of $\vartheta$.
			\item Outside of the set $\textup{dom}(\vartheta)$, we have no control over the behaviour of the functions $\vartheta_n$ as $n$ grows. In the previous example, as $n$ tends to $\infty$, the function $\vartheta_n$ grows faster than the rate at which the sets $\textup{dom}(\vartheta_n)$ shrink. Hence, the sequence of integrals diverges.
		\end{enumerate}
		These observations suggest that the divergence phenomenon of the previous example can be avoided by imposing additional conditions on the functions $\vartheta_n$ or the sets $\textup{dom}(\vartheta_n)$. The following two options naturally come to mind:
		\begin{enumerate}[label=(\arabic*')]
			\item The first option is to control the growth of the functions $\vartheta_n$. For example, we may require the sequence of functions $\lbrace \vartheta_n  \rbrace_{n \in \mathbb{N}}$ to be bounded uniformly in $n$. That is,
			\begin{displaymath}
				\sup \lbrace |\vartheta_n (y) | \, \vert \, y \in \textup{dom}(\vartheta_n) \textup{ and } n \in \mathbb{N} \rbrace < \infty.
			\end{displaymath}
			By elementary properties of the Legendre-Fenchel transform and the definition of the map $\mathcal{G}$, this is equivalent to the sequence of functions $\lbrace g_{D_n} - g_{D_n, \textup{can}} \rbrace_{n \in \mathbb{N}}$ being bounded uniformly in $n$. By the isomorphism
			\begin{displaymath}
				\overline{\textup{Div}}_{\mathbb{T}}(U/K) \cong  C^{0}_{\mathbb{S}}(U^{\textup{an}})_{\textup{cptf}} \oplus \textup{Div}_{\mathbb{T}}(U/K),
			\end{displaymath}
			This is a purely analytic condition on the arithmetic divisors $\overline{D}_n$, in the sense that it does not impose further restrictions on the geometric divisors $D_n$.
			\item The second option is to control the sets $\textup{dom}(\vartheta_n)$. For example, we may require
			\begin{displaymath}
				\textup{cl}(\textup{dom}(\vartheta_n ))= \textup{cl}(\textup{dom}(\vartheta )), \quad \textup{ for all } n \in \mathbb{N}.
			\end{displaymath}
			By Theorem~\ref{nefcorrespondence}, this is equivalent to the sequence of nef toric compactified divisors $\lbrace D_n \rbrace_{n \mathbb{N}}$ being constant. This is a geometric condition.
		\end{enumerate}
		Since the toric height is defined only for toric model arithmetic divisors $\overline{D}_n$, the strategy is clear. First, we take the approach of (1') to generalize the toric height for nef compactified arithmetic divisors $\overline{D}$ such that $g_D - g_{D,\textup{can}}$ is bounded. Then, we take the approach of (2') to remove the boundedness condition on $g_D - g_{D,\textup{can}}$. This motivates the following definition.
	\end{rem}
	\begin{dfn}\label{local-toric-h-dfn}
		Let $\overline{D}_0, \ldots, \overline{D}_d$ be nef toric compactified arithmetic divisors on $U/K$. The \textit{local toric height} of $U/K$ with respect to $\overline{D}_0 , \ldots \overline{D}_d$, denoted by $\textup{h}^{\textup{tor}}(U; \overline{D}_{0}, \ldots , \overline{D}_{d})$, is defined as follows:
		\begin{enumerate}[label=(\roman*)]
			\item Suppose that, for each $0 \leq i \leq d$, the function $g_{D_i} - g_{D_{i},\textup{can}}$ is bounded. The number $\textup{h}^{\textup{tor}}(U; \overline{D}_{0}, \ldots , \overline{D}_{d})$ is given by the limit
			\begin{displaymath}
				\textup{h}^{\textup{tor}}(U; \overline{D}_{0}, \ldots , \overline{D}_{d})\coloneqq \lim_{n \rightarrow \infty} \textup{h}^{\textup{tor}}(U; \overline{D}_{0,n}, \ldots , \overline{D}_{d,n} ),
			\end{displaymath}
			where $\lbrace \overline{D}_{i,n} \rbrace_{n \in \mathbb{N}}$ is any decreasing sequence of nef toric model arithmetic divisors of $U/K$, converging in the boundary topology to $\overline{D}_i$, and such that the sequence $\lbrace g_{D_{i,n}} - g_{D_{i,n}, \textup{can}} \rbrace_{n \in \mathbb{N}}$ of Green's functions is bounded uniformly in $n$.
			\item For each $0 \leq i \leq d$, write $\overline{D}_i = (D_i , g_i)$. The number $\textup{h}^{\textup{tor}}(U; \overline{D}_{0}, \ldots , \overline{D}_{d})$ is given by
			\begin{displaymath}
				\textup{h}^{\textup{tor}}(U; \overline{D}_{0}, \ldots , \overline{D}_{d})\coloneqq \textup{inf} \lbrace \textup{h}^{\textup{tor}}(U; \overline{E}_{0}, \ldots , \overline{E}_{d})  \, \vert \, (\overline{E}_0 , \ldots, \overline{E}_d) \textup{ satisfying } (\star) \rbrace,
			\end{displaymath}
			where the condition $(\star )$ means that, for each $0 \leq i \leq d$, the nef toric compactified arithmetic divisor $\overline{E}_i = (D_i, g_{i}^{\prime})$ is such that $g_{i}^{\prime} \geq g_i$ and $g_{i}^{\prime} - g_{D_{i} ,\textup{can}}$ is bounded.
		\end{enumerate}
		This definition is extended to integrable toric compactified divisors by linearity. Additionally, if $\overline{D}_0 =  \ldots = \overline{D}_d = \overline{D}$, we write $\textup{h}^{\textup{tor}}(U; \overline{D}) = \textup{h}^{\textup{tor}}(U; \overline{D}_{0}, \ldots , \overline{D}_{d})$.
	\end{dfn}
	\begin{rem}\label{mixed-energy-2}[Comparison with Burgos-Kramer's mixed relative energy]
		Suppose that $K$ is archimedean and for each $i \in \lbrace 0, \ldots, d \rbrace$, let $\overline{D}_i$ be a nef toric compactified arithmetic divisor. It follows directly from Proposition \ref{toric-mixed-energy}, and Theorem~4.3~of~\cite{BK24} applied to the divisors $\overline{D}_i$ and $\overline{D}_{i}^{\prime} = \overline{D}_{i, \textup{can}}$, that the local toric height $\textup{h}^{\textup{tor}}(U; \overline{D}_{0}, \ldots , \overline{D}_{d})$ defined above coincides with the mixed relative energy $I_{\pmb{\varphi}^{\prime}}(\pmb{\varphi})$ of the functions $\varphi_i$ and $\varphi_{i}^{\prime}$ in the statement of the cited result.
	\end{rem}
	Now, we show that the integral formula in Theorem~\ref{local-toric-height-proj} holds in this setting. First, we show this for the case in which the function $g_{D} - g_{D,\textup{can}}$ is bounded. Afterwards, we remove this restriction.
	\begin{thm}\label{local-toric-h-UK-1}
		Let $U/K$ be the $d$-dimensional split torus and $\overline{D} = (D,g_D)$ be a nef toric compactified arithmetic divisor of $U/K$. Suppose that $\lbrace \overline{D}_n \rbrace_{n \in \mathbb{N}}$ is a decreasing sequence of nef toric model arithmetic divisors of $U/K$, converging in the boundary topology to $\overline{D}$, and such that the sequence of Green's functions $\lbrace g_{D_n} - g_{D_n, \textup{can}} \rbrace_{n \in \mathbb{N}}$ is bounded uniformly in $n$. Then, the function $g_D - g_{D,\textup{can}}$ is bounded and the following identity holds
		\begin{displaymath}
			\textup{h}^{\textup{tor}}(U, \overline{D}) \coloneqq \lim_{n \rightarrow \infty} \textup{h}^{\textup{tor}}(U;\overline{D}_n) = (d+1)! \int_{\Delta_D} \vartheta_{\overline{D}} \, \textup{dVol}_{M} > - \infty,
		\end{displaymath} 
		where $\vartheta_{\overline{D}}$ is the local roof function of $\overline{D}$ and $\Delta_D$ is the compact convex set associated to $D$. Conversely, if the function $g_D - g_{D,\textup{can}}$ is bounded, then such a sequence $\lbrace \overline{D}_n \rbrace_{n \in \mathbb{N}}$ exists.
	\end{thm}
	\begin{proof}
		Let $\lbrace \overline{D}_n \rbrace_{n \in \mathbb{N}}$ be any such sequence and $C$ be a constant such that
		\begin{displaymath}
			\sup \lbrace |g_{D_n}(p) - g_{D_n,\textup{can}}(p)| \, \vert \, p \in U^{\textup{an}} \textup{ and } n \in \mathbb{N} \rbrace \leq C.
		\end{displaymath}
		It follows that $g_{D} - g_{D,\textup{can}}$ is bounded by $C$. By the definition of the map $\mathcal{G}$ and Theorem~\ref{arith-toric-div-UK}, the sequence $\lbrace \gamma_{\overline{D}_n} \rbrace_{n \in \mathbb{N}}$ of tropical Green's functions is increasing and converges in the $\mathcal{G}$-norm to the tropical Green's function $\gamma_{\overline{D}}$. Moreover, the sequence $\lbrace \gamma_{\overline{D}_n} - \Psi_{D_n} \rbrace_{n \in \mathbb{N}}$ is bounded by $C$. By Lemma~\ref{mono-conv-G}, we get that:
		\begin{enumerate}[label=(\alph*)]
			\item The sequence of compact convex sets $\lbrace \Delta_{D_n} \rbrace_{n \in \mathbb{N}}$ is decreasing and converges in the Hausdorff distance to the compact convex set $\Delta_D$.
			\item The sequence of roof functions $\lbrace \vartheta_{\overline{D}_n} \rbrace_{n \in \mathbb{N}}$ is decreasing, bounded by $C$, and converges pointwise to the roof function $\gamma_{\overline{D}}$.
		\end{enumerate}
		On the other hand, for each $n \in \mathbb{N}$, Theorem~\ref{local-toric-height-proj} gives the identity
		\begin{displaymath}
			\textup{h}^{\textup{tor}}(U;\overline{D}_n) = (d+1)! \int_{\Delta_{D_n}} \vartheta_{\overline{D}_n} \, \textup{dVol}_{M}.
		\end{displaymath} 
		For each $n \in \mathbb{N}$ we have an inclusion $\Delta_{D} \subset \Delta_{D_n}$. Then
		\begin{displaymath}
			\int_{\Delta_{D_n}} \vartheta_{\overline{D}_n} \, \textup{dVol}_{M} = \int_{\Delta_{D}} \vartheta_{\overline{D}_n} \, \textup{dVol}_{M} + \int_{\Delta_{D_n} \setminus \Delta_D} \vartheta_{\overline{D}_n} \, \textup{dVol}_{M}.
		\end{displaymath}
		By property (b), the monotonic convergence theorem implies
		\begin{displaymath}
			\lim_{n \rightarrow \infty} \int_{\Delta_{D}} \vartheta_{\overline{D}_n} \, \textup{dVol}_{M}  = \int_{\Delta_{D}} \vartheta_{\overline{D}} \, \textup{dVol}_{M} .
		\end{displaymath}
		For the second term, property (b) gives the estimate
		\begin{displaymath}
			\left\vert \int_{\Delta_{D_n} \setminus \Delta_D} \vartheta_{\overline{D}_n} \, \textup{dVol}_{M} \right\vert \leq \int_{\Delta_{D_n} \setminus \Delta_D} | \vartheta_{\overline{D}_n} | \, \textup{dVol}_{M} \leq C \cdot  \textup{Vol}_{M} (\Delta_{D_n} \setminus \Delta_D ).
		\end{displaymath}
		By property (a) and Proposition~\ref{comp-int-num}, we know that $\lim_{n} \textup{Vol}_M (\Delta_{D_n}) = \textup{Vol}_M (\Delta_D)$. Therefore
		\begin{displaymath}
			\lim_{n \rightarrow \infty} \int_{\Delta_{D_n} \setminus \Delta_D} \vartheta_{\overline{D}_n} \, \textup{dVol}_{M} = 0.
		\end{displaymath}
		Combining the two identities, we obtain
		\begin{displaymath}
			\lim_{n \rightarrow \infty} \textup{h}^{\textup{tor}}(U;\overline{D}_n) = (d+1)! \lim_{n \rightarrow \infty} \int_{\Delta_{D_n}} \vartheta_{\overline{D}_n} \, \textup{dVol}_{M} = (d+1)! \int_{\Delta_D} \vartheta_{\overline{D}} \, \textup{dVol}_{M}.
		\end{displaymath}
		Now assume that $\overline{D}$ is nef and the function $g_{D} - g_{D,\textup{can}}$ is bounded by $C$. By definition of the map $\mathcal{G}$ and Theorem~\ref{nef-cone-UK}, we know that the tropical Green's function $\gamma_{\overline{D}}$ and the support function $\Psi_{D}$ are concave and their difference $\gamma_{\overline{D}} - \Psi_D$ is bounded by $C$. By Proposition~\ref{m-d-approx-1}, there exists a decreasing sequence $\lbrace D_n \rbrace_{n \in \mathbb{N}}$ of nef toric model divisors converging to $D$. In particular, the sequence $\lbrace \Psi_{D_n} \rbrace_{n \in \mathbb{N}}$ converges increasingly to $\Psi_D$. Then, apply Lemma~\ref{mono-approx-3} to construct the tropical Green's function $\gamma_{\overline{D}_n}$ of the nef toric model arithmetic divisor $\overline{D}_n$. The result follows.
	\end{proof}
	\begin{cor}\label{local-toric-h-UK-1-mixed}
		Let $\overline{D}_0, \ldots, \overline{D}_d$ be nef toric compactified arithmetic divisors on $U/K$ such that the functions $g_{D_i} - g_{D_{i},\textup{can}}$ are bounded. For each $0 \leq i \leq d$, let $\lbrace \overline{D}_{i,n} \rbrace_{n \in \mathbb{N}}$ be a decreasing sequence of nef toric model arithmetic divisors of $U/K$, converging in the boundary topology to $\overline{D}_i$, and such that the sequence $\lbrace g_{D_{i,n}} - g_{D_{i,n}, \textup{can}} \rbrace_{n \in \mathbb{N}}$ of Green's functions is bounded uniformly in $n$. Then, the following identity holds
		\begin{displaymath}
			\textup{h}^{\textup{tor}}(U; \overline{D}_{0}, \ldots , \overline{D}_{d})\coloneqq \lim_{n \rightarrow \infty} \textup{h}^{\textup{tor}}(U; \overline{D}_{0,n}, \ldots , \overline{D}_{d,n} ) = \textup{MI}_{M}(\vartheta_{0}, \ldots , \vartheta_{d}),
		\end{displaymath}
		where $\vartheta_i$ is the roof function of $\overline{D}_i$ and $\textup{MI}_{M}$ is the mixed integral of Definition~\ref{mixed-integral}.
	\end{cor}
	\begin{proof}
		Observe that for all $0 \leq i,j \leq d$, the toric compactified arithmetic divisor $\overline{D}_i + \overline{D}_j$ is nef and the function
		\begin{displaymath}
			g_{D_i}+g_{D_j} - (g_{D_i, \textup{can}} + g_{D_j, \textup{can}})
		\end{displaymath}
		is bounded. Proceeding as in Corollary~\ref{comp-int-num}, the result follows by the multilinearity of the local toric height and Theorem~\ref{local-toric-h-UK-1}. 
	\end{proof}
	Now, we remove the boundedness condition on the difference $g_D - g_{D,\textup{can}}$.
	\begin{thm}\label{local-toric-h-UK-2}
		Let $U/K$ be the $d$-dimensional split torus and $\overline{D} = (D,g)$ be a nef toric compactified arithmetic divisor of $U/K$. Denote by $\vartheta_{\overline{D}}$ and $\Delta_D$ the corresponding roof function and compact convex set, respectively. Then, the following identity holds
		\begin{displaymath}
			\textup{h}^{\textup{tor}}(U, \overline{D}) \coloneqq \inf_{\overline{E}}  \textup{h}^{\textup{tor}}(U; \overline{E})  = (d+1)! \int_{\Delta_D} \vartheta_{\overline{D}} \, \textup{dVol}_{M},
		\end{displaymath}
		where $\overline{E}$ runs over all nef toric compactified arithmetic divisors $(D, g^{\prime})$ such that $g^{\prime} \geq g$ and $g^{\prime} - g_{D,\textup{can}}$ is bounded. In particular, $\textup{h}^{\textup{tor}}(U, \overline{D}) > - \infty $ if and only if $\vartheta_{\overline{D}} \in L^{1}(\Delta_D)$.
	\end{thm}
	\begin{proof}
		The last part of the theorem immediately follows the identity, which we now prove. Let $\overline{E} = (D,g')$ be such a divisor. Since $\overline{D}$ and $\overline{E}$ have the same divisorial part, their corresponding compact convex sets are the same. The inequality $\overline{E} \geq \overline{D}$ implies the inequality of roof functions $\vartheta_{\overline{E}} \geq \vartheta_{\overline{D}}$. By Theorem~\ref{local-toric-h-UK-1} and the monotonicity of integrals,
		\begin{displaymath}
			\textup{h}^{\textup{tor}}(U; \overline{E})  = (d+1)! \int_{\Delta_D} \vartheta_{\overline{E}} \, \textup{dVol}_{M} \geq (d+1)! \int_{\Delta_D} \vartheta_{\overline{D}} \, \textup{dVol}_{M}.
		\end{displaymath}
		Taking the infimum over $\overline{E}$, we obtain the inequality
		\begin{displaymath}
			\inf_{\overline{E}}  \textup{h}^{\textup{tor}}(U; \overline{E})  \geq (d+1)! \int_{\Delta_D} \vartheta_{\overline{D}} \, \textup{dVol}_{M}.
		\end{displaymath}
		It remains to show the reverse inequality. By Lemma~\ref{mono-approx-1}, we can find a decreasing sequence of continuous concave functions $\lbrace \vartheta_n \rbrace_{n \in \mathbb{N}}$ with effective domain $\textup{dom}(\vartheta_n)=\Delta_D$ and converging pointwise to the function $\vartheta_{\overline{D}}$. The sequence $\lbrace \vartheta_n \rbrace_{n \in \mathbb{N}}$ is bounded above by a constant $C$. Then, without loss of generality, we may assume $C=0$. By Fatou's lemma (stated for non-positive functions), we get
		\begin{displaymath}
			\limsup_{n \in \mathbb{N}} \int_{\Delta_D} \vartheta_n \, \textup{dVol}_{M} \leq \int_{\Delta_D} \limsup_{n \in \mathbb{N}} \vartheta_n \, \textup{dVol}_{M}.
		\end{displaymath}
		Recall that $\lbrace \vartheta_n \rbrace_{n \in \mathbb{N}}$ converges pointwise to $\vartheta_{\overline{D}}$. Then,
		\begin{displaymath}
			\limsup_{n \in \mathbb{N}} \vartheta_n (y) = \lim_{n \in \mathbb{N}} \vartheta (y) = \vartheta_{\overline{D}} (y), \quad \textup{for all } y \in \Delta_D.
		\end{displaymath}
		Since the limit superior is always larger than the infimum, we obtain
		\begin{displaymath}
			\inf_{n \in \mathbb{N}}  \int_{\Delta_D} \vartheta_n \, \textup{dVol}_{M}  \leq \limsup_{n \in \mathbb{N}} \int_{\Delta_D} \vartheta_n \, \textup{dVol}_{M} \leq \int_{\Delta_D} \vartheta_{\overline{D}} \, \textup{dVol}_{M}.
		\end{displaymath}
		On the other hand, by Corollary~\ref{roof-functions-UK}, the decreasing sequence  $\lbrace \vartheta_n \rbrace_{n \in \mathbb{N}}$ corresponds to a decreasing sequence of nef toric compactified arithmetic divisors
		\begin{displaymath}
			\lbrace \overline{E}_n \rbrace_{n \in \mathbb{N}}, \quad \overline{E}_n = (D, g_n),
		\end{displaymath}
		which converges to $\overline{D}$ and such that the differences $g_n - g_{D,\textup{can}}$ are bounded. Applying Theorem~\ref{local-toric-h-UK-1} to the previous inequality of integrals, we obtain
		\begin{displaymath}
			\inf_{\overline{E}}  \textup{h}^{\textup{tor}}(U; \overline{E}) \leq \inf_{n \in \mathbb{N}} \textup{h}^{\textup{tor}}(U; \overline{E}_n) \leq (d+1)! \int_{\Delta_D} \vartheta_{\overline{D}} \, \textup{dVol}_{M}.
		\end{displaymath}
		Therefore, the desired identity holds. 
	\end{proof}
	\begin{rem}
		The previous argument also implies $\textup{h}^{\textup{tor}}(U; \overline{D})  = \lim_{n \in \mathbb{N}}  \textup{h}^{\textup{tor}}(U; \overline{D}_n)$, for any decreasing sequence $\lbrace \overline{D}_n \rbrace_{n \in \mathbb{N}}$ of nef toric compactified divisors with divisorial part $D$, converging to $\overline{D}$, and such that the difference $g_{D_n} - g_{D,\textup{can}}$ is bounded.
	\end{rem}
	Naturally, we want to obtain the mixed version of Theorem~\ref{local-toric-h-UK-2}. Recall that the mixed integral, introduced in Definition~\ref{mixed-integral}, is a sum of integrals of sup-convolutions. In the context of Corollary~\ref{local-toric-h-UK-1-mixed}, each of these integrals corresponds to the local toric height of a sum of nef toric compactified arithmetic divisors. We cannot repeat the same argument of Corollary~\ref{local-toric-h-UK-1-mixed} in the unbounded situation: Let $\overline{D}_1$ and $\overline{D}_2$ be nef toric compactified arithmetic divisors. Even if both $\textup{h}^{\textup{tor}}(U; \overline{D}_1)$ and $\textup{h}^{\textup{tor}}(U; \overline{D}_2)$ are finite, is not necessarily true that $\textup{h}^{\textup{tor}}(U; \overline{D}_1 +\overline{D}_2)$ is finite. We give an example below, inspired by Example 3.5~of~\cite{DN15}.
	\begin{ex}\label{sum-not-finite}
		Let $U/K$ be the $2$-dimensional split torus over $K$ and identify $N_{\mathbb{R}}$ with $\mathbb{R}^2$. Define the compact convex sets
		\begin{displaymath}
			\Delta_1 \coloneqq \lbrace (x,y) \, \vert \, x \geq 0, \, y \geq 0 \textup{ and } 1 \geq x+y \rbrace \quad \textup{and} \quad \Delta_2 = [0,1] \times \lbrace 0 \rbrace,
		\end{displaymath}
		the standard $2$-dimensional simplex and an embedding of the unit interval, respectively. Then, consider the closed concave functions $\vartheta_1 \colon \Delta_1 \rightarrow \mathbb{R}_{-\infty}$ and $\vartheta_2 \colon \Delta_2 \rightarrow \mathbb{R}_{-\infty}$, given by
		\begin{displaymath}
			\vartheta_1 (x,y) \coloneqq -(1-y)^{-1} \quad \textup{ and } \quad \vartheta_2(x,y) \coloneqq 0.
		\end{displaymath}
		By Corollary~\ref{roof-functions-UK}, the functions $\vartheta_1$ and $\vartheta_2$ correspond to nef toric arithmetic divisors $\overline{D}_1$ and $\overline{D}_2$ respectively. By Theorems~\ref{local-toric-h-UK-1}~and~\ref{local-toric-h-UK-2}, we get
		\begin{displaymath}
			\textup{h}^{\textup{tor}}(U; \overline{D}_1) = 3! \int_{\Delta_1} \vartheta_{1} \, \textup{dVol} = -6 \int_{0}^{1} \int_{0}^{1-y} (1-y)^{-1} \, \textup{d}x \, \textup{d}y = -6.
		\end{displaymath}
		Similarly, we obtain $\textup{h}^{\textup{tor}}(U; \overline{D}_2) = 0$. On the other hand, the nef toric compactified arithmetic divisor $\overline{D}_3 = \overline{D}_1 + \overline{D}_2$ has a corresponding compact convex set
		\begin{displaymath}
			\Delta_3 = \Delta_1 + \Delta_2 = \lbrace (x,y) \, \vert \, x \geq 0, \, 1 \geq y \geq 0, \textup{ and } 2 \geq x+y \rbrace.
		\end{displaymath}
		By Proposition~\ref{prop-conv}, the roof function $\vartheta_3 \colon \Delta_3 \rightarrow \mathbb{R}_{-\infty}$ is given by the sup-convolution $\vartheta_3 = \vartheta_1 \boxplus \vartheta_2$. A quick calculation shows that
		\begin{displaymath}
			\vartheta_3 (x,y) =  -(1-y)^{-1}.
		\end{displaymath}
		By Theorem~\ref{local-toric-h-UK-2}, we know that
		\begin{displaymath}
			\textup{h}^{\textup{tor}}(U; \overline{D}_3) = 3! \int_{\Delta_3} \vartheta_{3} \, \textup{dVol}.
		\end{displaymath}
		Since the function $\vartheta_3$ is non-positive and $[0,1]^2$ is contained in $\Delta_3$, we get the bound
		\begin{displaymath}
			\int_{\Delta_3} \vartheta_{3} \, \textup{dVol} \leq  \int_{[0,1]^2} \vartheta_{3} \, \textup{dVol} = -\int_{0}^{1} \int_{0}^{1} (1-y)^{-1} \, \textup{d}x \, \textup{d}y = - \int_{0}^{1} (1-y)^{-1} \, \textup{d}y = -\infty.
		\end{displaymath}
		Summarizing, we have constructed nef toric compactified arithmetic divisors $\overline{D}_1$ and $\overline{D}_2$ with finite toric height, whose sum $\overline{D}_3$ does not have finite toric height.
	\end{ex}
	In the following way, we can generalize Corollary~\ref{local-toric-h-UK-1-mixed} to the unbounded setting.
	\begin{cor}\label{local-toric-h-UK-2-mixed}
		Let $\overline{D}_0, \ldots, \overline{D}_d$ be nef toric compactified arithmetic divisors on $U/K$, and write $\Delta_i$ and $\vartheta_i$ for the compact convex set and roof function of $\overline{D}_i$, respectively. For each $I \subset \lbrace 0,1, \ldots, d \rbrace$, denote
		\begin{displaymath}
			\overline{D}_I \coloneqq \sum_{i \in I} \overline{D}_i, \quad \Delta_I \coloneqq \sum_{i \in I} \Delta_i, \quad \textup{and} \quad  \vartheta_I \coloneqq \boxplus_{i \in I} \, \vartheta_i.
		\end{displaymath}
		Suppose that, for each $I$, we have $\vartheta_I \in L^{1}(\Delta_I)$. Then, the following identity holds
		\begin{displaymath}
			\textup{h}^{\textup{tor}}(U; \overline{D}_{0}, \ldots , \overline{D}_{d})= \sum_{I \subset \lbrace 0,1, \ldots, d \rbrace} (-1)^{d-|I|} \,  \textup{h}^{\textup{tor}}(U; \overline{D}_I) =  \textup{MI}_{M}(\vartheta_{0}, \ldots , \vartheta_{d}) > -\infty.
		\end{displaymath}
	\end{cor}
	\begin{proof}
		For each $0 \leq i \leq d$, write $\overline{D}_i = (D_i,g_i)$. By elementary properties of infimums, we can find a sequence $\lbrace (\overline{E}_{0,n}, \ldots, \overline{E}_{d,n})\rbrace_{n \in \mathbb{N}}$ satisfying the following properties:
		\begin{enumerate}[label=(\alph*)]
			\item For each $0 \leq i \leq d$ and each $n \in \mathbb{N}$, the nef toric compactified arithmetic divisor $\overline{E}_{i,n} = (D_i,g_{i,n})$ is such that $g_i \leq g_{i,n}$ and $g_{i,n} - g_{D_{i},\textup{can}}$ is bounded.
			\item For each $0 \leq i \leq d$, the sequence $\lbrace \overline{E}_{i,n} \rbrace_{n \in \mathbb{N}}$ is decreasing and
			\begin{displaymath}
				\textup{h}^{\textup{tor}}(U; \overline{D}_{0}, \ldots , \overline{D}_{d}) =\lim_{n \in \mathbb{N}} \textup{h}^{\textup{tor}}(U; \overline{E}_{0,n}, \ldots , \overline{E}_{d,n}).
			\end{displaymath}
		\end{enumerate}
		Moreover, we can assume that $\overline{D}_i = \lim_{n \in \mathbb{N}} \overline{E}_{i,n}$. Indeed, we may find a sequence $\lbrace(D_i, g_{i,n}^{\prime} ) \rbrace_{n \in \mathbb{N}}$ satisfying the condition (a) and converging to $\overline{D}_i$. Then, we may replace $g_{i,n}$ by $\min \lbrace g_{i,n}, \, g_{i,n}^{\prime} \rbrace$. By monotonicity of heights, we know that property (b) is still satisfied. On the other hand, Corollary~\ref{local-toric-h-UK-1-mixed} implies
		\begin{displaymath}
			\textup{h}^{\textup{tor}}(U; \overline{E}_{0,n}, \ldots , \overline{E}_{d,n}) = \sum_{I \subset \lbrace 0,1, \ldots, d \rbrace} (-1)^{d-|I|} \,  \textup{h}^{\textup{tor}}(U; \overline{E}_{I,n}),
		\end{displaymath}
		where $ \overline{E}_{I,n}$ is defined similarly to $\overline{D}_I$. The result follows from the combination of property (b), the above identity, and Theorem~\ref{local-toric-h-UK-2}.
	\end{proof}
	We state the analogous result for integrable toric compactified arithmetic divisors.
	\begin{cor}\label{local-toric-h-UK-2-integrable}
		Let $\overline{D}_0, \ldots, \overline{D}_d \in  \overline{\textup{Div}}^{\textup{int}}_{\mathbb{T}}(U/K)_{\mathbb{Q}}$. For each $i = 0, \ldots , d$, write $\overline{D}_{i}$ as a difference $\overline{D}_{i,+} - \overline{D}_{i,-}$ of nef toric compactified arithmetic divisors. Denote by $\vartheta_{i,+}$ and $\vartheta_{i,-}$ the corresponding roof function of $\overline{D}_{i,+}$ and $\overline{D}_{i,-}$, respectively. Copying the notation in Corollary~\ref{local-toric-h-UK-2-mixed}, assume that $\vartheta_{I,\pm} \in L^{1}(\Delta_{I,\pm})$. Then, the following identity holds
		\begin{displaymath}
			\textup{h}^{\textup{tor}}(U; \overline{D}_0 , \ldots , \overline{D}_d ) = \sum_{\epsilon_0 , \ldots , \epsilon_d \in \lbrace \pm 1 \rbrace} \epsilon_0 \ldots \epsilon_d \, \textup{MI}_M ( \vartheta_{0,\epsilon_0}, \ldots , \vartheta_{d,\epsilon_d }) > -\infty.
		\end{displaymath}
	\end{cor}
	\begin{proof}
		It is immediate from Corollary~\ref{local-toric-h-UK-2-mixed}.
	\end{proof}
	\begin{ex}\label{local-height-singular}
		The following example gives a nef toric compactified arithmetic divisor $\overline{D}$ on the one dimensional torus $U = \mathbb{G}_m / K$ such that difference $g_D - g_{D,\textup{can}}$ is not bounded and has finite local toric height $\textup{h}^{\textup{tor}}(U, \overline{D})$. Consider a real number $\alpha < 0$ and define the concave function
		\begin{displaymath}
			\vartheta (y) \coloneqq  \begin{cases} -\infty, & y=0 \\ 1-y^{\alpha} , & 0 < y \leq 1. \end{cases}
		\end{displaymath}
		By Corollary~\ref{roof-functions-UK}, the function $\vartheta$ of concave functions determines a s a nef toric compactified arithmetic divisor $\overline{D}$, whose geometric divisor $D$ is induced by the equation $x_0 = 0$ in $\mathbb{P}^1/K$. Now, use Theorem~\ref{local-toric-h-UK-2} to compute its local toric height:
		\begin{displaymath}
			\textup{h}^{\textup{tor}}(U, \overline{D}) = 2! \int_{0}^{1}\vartheta (y) \, \textup{d}y =  2 \int_{0}^{1} (1-y^{\alpha})  \, \textup{d}y =  \begin{cases} \frac{2 \alpha}{\alpha +1}, &  -1 < \alpha < 0 \\ - \infty,  & \alpha \leq -1. \end{cases}
		\end{displaymath}
		This cannot be computed using the methods of \cite{BPS}.
	\end{ex}
	\appendix
	\section{Convex-analytic tools} 
	This Appendix summarizes and develops the convex-analytic machinery needed to study the Arakelov geometry of toric quasi-projective varieties. The main results presented here are density properties of piecewise affine functions in certain Banach spaces and their cones of concave functions. This Appendix is subdivided into four subsections, which we now briefly describe. The first subsection consists of basic notation and properties of concave functions and their Legendre-Fenchel duals. The second subsection introduces the Banach spaces of continuous conical and sublinear functions, while the third studies the space of rational piecewise affine functions. In both of these subsections, we study in detail the convergence of sequences in these spaces. In the final subsection, we use the tools developed in the second and third subsections to prove the desired density theorems. Our main reference for this Appendix is Rockafeller's~book~\cite{Roc}, while we try to stick to the notation in Chapter~2~of~\cite{BPS}.
	
	\subsection{Convex functions and conical functions} Throughout this appendix, we fix a lattice $N$ of rank $d$ and follow the conventions from the introduction. We now recall the basic definitions from convex geometry.	Let $S$ be a subset of the Euclidean space $N_{\mathbb{R}}$. We denote by $\textup{Conv}(S)$ (resp. $c(S)$) the \textit{convex hull} of $S$  (resp. the \textit{cone} spanned by $S$). Note that the set $S$ is \textit{convex} (resp. a \textit{cone}) if and only if $\textup{Conv}(S) = S$ (resp. $c(S) = S$). For every pair of convex sets $A, B \subset N_{\mathbb{R}}$, their \textit{Minkowski sum} is the convex set given by $A+B \coloneqq \lbrace a + b \, \vert \, a \in A, \, b \in B \rbrace$. A convex set $C$ is \textit{strongly convex} if it contains no line. The \textit{span} of a convex set $C \subset N_{\mathbb{R}}$ is the smallest affine subspace $\textup{Span}(C)$ that contains it. The \textit{dimension} of $C$ is the dimension of  $\textup{Span}(C)$. The \textit{relative interior} $\textup{ri}(C)$ of $C$ is the interior of $C$ in its span.  The \textit{relative boundary} $\textup{rb}(C)$ of $C$ is boundary of $C$ in $\textup{Span}(C)$. The topological closure $\textup{cl}(C)$ of a convex set $C$ is also convex.
	
	A convex subset $F$ of a convex set $C$ is a \textit{face} if for every closed line segment $L \subset C$ such that $\textup{ri}(L) \cap F \neq \emptyset$ then $L \subset F$. We denote $F \leq C$ whenever $F$ is a face of $C$. A face is \textit{proper} if $F \neq \emptyset \, ,  C$. A \textit{facet} is a face of codimension one. A \textit{supporting hyperplane} $H$ of a convex set $C$ is a hyperplane $H \subset N_{\mathbb{R}}$ such that $C$ is contained in one of the closed half-spaces determined by $H$ and $\textup{rb}(C) \cap H \neq \emptyset$. An \textit{exposed face} $F$ is a face given by the intersection of $C$ with a closed half-space determined by a support hyperplane, and an \textit{exposed point} is an exposed face of dimension $0$. An \textit{extreme point} $x \in C$ is a point that is not contained in the interior of any closed line segment $L \subset C$. Every exposed face is a face, and every exposed point is an extreme point, but neither converse is true.
	
	We now recall the basic definitions from convex analysis. A function $f \colon N_{\mathbb{R}} \rightarrow \mathbb{R}_{- \infty}$ is \textit{concave} if it is not identically $-\infty$ and for all $x,y \in N_{\mathbb{R}}$ and $0 < t < 1$ we have
	\begin{displaymath}
		f(t \, x + (1-t) \, y) \geq t f(x) + (1-t) f(y).
	\end{displaymath}
	A function $g$ is \textit{convex} if $-g$ is concave. In this definition, we may replace $N_{\mathbb{R}}$ by a convex subset $C \subset N_{\mathbb{R}}$. Note, however, that any convex function $f$ on $C$ can be extended to the whole of $N_{\mathbb{R}}$ by defining it as $-\infty$ on the complement of $C$. We will use both approaches as we see fit. Now, let $C \subset N_{\mathbb{R}}$ be a cone. A function $f \colon  C \rightarrow \mathbb{R}$ is said to be \textit{conical} if for every $\lambda \in \mathbb{R}_{\geq 0}$ the identity $f( \lambda \cdot x ) = \lambda f(x)$ holds.
	
	Concavity of functions has a geometric characterization; namely, a function $f$ is concave if and only if its \textit{hypograph} $\textup{hyp}(f)$ is a convex set, where
	\begin{displaymath}
		\textup{hyp}(f) \coloneqq \lbrace (x , t) \in N_{\mathbb{R}} \times \mathbb{R} \, \vert \, - \infty < t \leq f(x) \rbrace.
	\end{displaymath}
	The \textit{effective domain} $\textup{dom}(f)$ and \textit{stability set} $\textup{stab}(f)$ of a concave function $f$ are given by
	\begin{displaymath}
		\textup{dom}(f) \coloneqq \lbrace x \in N_{\mathbb{R}} \, \vert \, f(x) \neq - \infty \rbrace, \quad \textup{stab}(f) \coloneqq \lbrace y \in M_{\mathbb{R}} \, \vert \, \langle y, x \rangle - f(x) \textup{ is bounded below} \rbrace.
	\end{displaymath}
	The effective domain $\textup{dom}(f)$ is the projection $p_1(\textup{hyp}(f))$ of the hypograph to the $x$-coordinate and, therefore, is convex. A concave function $f$ is \textit{closed} if it is upper semi-continuous. This is equivalent to the hypograph $\textup{hyp}(f)$ being a closed set. The \textit{closure} $\textup{cl}(f)$ of $f$ is the function whose hypograph is exactly $\textup{cl}(\textup{hyp}(f))$. It is a closed concave function and coincides with $f$ on $\textup{ri}(\textup{dom}(f))$. An important property of concave functions is that they are relatively continuous.
	\begin{prop}\label{cont-1}
		Let $f \colon  N_{\mathbb{R}} \rightarrow \mathbb{R}_{- \infty}$ be concave, then it is continuous as a function on $\textup{ri}(\textup{dom}(f))$ with the subspace topology. In particular, $\textup{dom}(f) = N_{\mathbb{R}}$ implies that $f$ is continuous, and therefore, closed.
	\end{prop}
	\begin{proof}
		This is just Theorem~10.1~of~\cite{Roc}.
	\end{proof}
	One of our goals is to construct sequences of concave functions that approximate a given concave function. The following concavity-preserving operations will be essential for these constructions. 
	\begin{prop}\label{prop-conv}
		Let $f,\, g  \colon  N_{\mathbb{R}} \rightarrow \mathbb{R}_{- \infty}$ be concave functions and let $c \in \mathbb{R}_{\geq 0}$, $\lambda \in \mathbb{R}_{>0}$, $x_0 \in N_{\mathbb{R}}$. We have:
		\begin{enumerate}[label=(\roman*)]
			\item If $\textup{dom}(f) \cap \textup{dom}(g) \neq \emptyset$, the sum $f+g$ is a concave function with effective domain $\textup{dom}(f) \cap \textup{dom}(g)$.
			\item The function $c \cdot f$ is a concave function with effective domain $\textup{dom}(f)$.
			\item The dilation $(\varpi_{\lambda} f) (x) = f(\lambda \cdot x)$ is concave with effective domain $(1/\lambda ) \cdot \textup{dom}(f)$.
			\item The translation $(\tau_{x_{0}}f) (x) = f(x - x_0)$ is concave with effective domain $\textup{dom}(f) + x_0$.
			\item If  $\textup{stab}(f) \cap \textup{stab}(g) \neq \emptyset$, the \textit{sup-convolution} of $f$ and $g$, defined as 
			\begin{displaymath}
				(f \boxplus g) (x) \coloneqq \sup_{y + z = x} f(y) + g(z)
			\end{displaymath}
			is concave with effective domain $\textup{dom}(f) + \textup{dom}(g)$. Moreover,
			\begin{displaymath}
				\textup{hyp}(f \boxplus g ) = \textup{hyp}(f) +  \textup{hyp}(g).
			\end{displaymath}
			\item The \textit{recession function} of $f$ is defined as
			\begin{displaymath}
				\textup{rec}(f)(x)\coloneqq \inf_{y \in \textup{dom}(f)} f(x+y) - f(y).
			\end{displaymath}
			It is a conical and concave function. If $f$ is closed, it can be defined as
			\begin{displaymath}
				\textup{rec}(f)(x) \coloneqq \lim_{\lambda \rightarrow + \infty} f (y_0 + \lambda \, x)/ \lambda
			\end{displaymath}
			where $y_0 \in \textup{dom}(f)$. The function $\textup{rec}(f)$ is closed in this case.
		\end{enumerate}
	\end{prop}
	\begin{proof}
		Parts \textit{(i)} to \textit{(v)} follow immediately from the definition. The last part is proven in Theorem~8.5~of~\cite{Roc}.
	\end{proof}
	Another essential property of concavity is that it is preserved under pointwise limits and pointwise infima.
	\begin{prop}\label{pointwise-conv}
		Let $\lbrace f_i \rbrace_{i \in \mathbb{N}}$ and  $\lbrace g_i \rbrace_{i \in I}$ be a sequence and a collection of concave functions on $N_{\mathbb{R}}$ respectively. Then, the pointwise limit and the pointwise infimum
		\begin{displaymath}
			f(x) \coloneqq \lim_{i \in \mathbb{N}} f_i (x), \quad g(x) \coloneqq \inf_{i \in I} g_i (x)
		\end{displaymath}
		are concave functions whenever they exist and are not identically $- \infty$.
	\end{prop}
	\begin{proof}
		Follows easily from definitions.
	\end{proof}
	The most important tool in our study is the \textit{Legendre-Fenchel transform}. 
	\begin{dfn}\label{LF-dfn}
		Let $f \colon N_{\mathbb{R}} \rightarrow \mathbb{R}_{- \infty}$ be a concave function. The \textit{Legendre-Fenchel transform} of $f$ is the function $f^{\vee} \colon M_{\mathbb{R}} \rightarrow \mathbb{R}_{- \infty}$ given by
		\begin{displaymath}
			f^{\vee}(y) \coloneqq \inf_{x \in N_{\mathbb{R}}} \lbrace \langle y , x \rangle - f(x) \rbrace.
		\end{displaymath}
	\end{dfn}
	\begin{rem}\label{LF-duality}
		The Legendre-Fenchel transform $f^{\vee}$ of a concave function $f$ is always a closed concave function, satisfying $(f^{\vee})^{\vee}=\textup{cl}(f)$. If the function $f$ is closed, then $(f^{\vee})^{\vee}=f$. In this case, we call $f^{\vee}$ the \textit{dual} of $f$. Note that the stability set of $f$ is exactly the effective domain of $f^{\vee}$, namely $\textup{stab}(f) = \textup{dom}(f^{\vee})$. If $f$ is closed, we have $\textup{dom}(f)=\textup{stab}(f^{\vee})$.
	\end{rem}
	\begin{ex}\label{LF-constant}
		Let $c \in N_{\mathbb{R}}$ and $f  \colon  N_{\mathbb{R}} \rightarrow \mathbb{R}$ the constant function $c$. Then
		\begin{displaymath}
			f^{\vee}(y) = \inf_{x \in N_{\mathbb{R}}} \lbrace \langle y , x \rangle - c \rbrace = \begin{cases} - c ,& y = 0 \\ - \infty, &\textup{otherwise}. \end{cases}
		\end{displaymath}
	\end{ex}
	\begin{rem}\label{conical-polyhedron}
		The Legendre-Fenchel transform establishes a correspondence between closed convex sets and closed conical concave functions. Let $\Delta \subset M_{\mathbb{R}}$, its \textit{indicator function} $\iota_{\Delta} \colon  M_{\mathbb{R}}  \rightarrow \mathbb{R}_{- \infty} $ is given by
		\begin{displaymath}
			\iota_{\Delta}(y) \coloneqq \begin{cases} 0 & y \in \Delta \\ -\infty & \textup{else} \end{cases}.
		\end{displaymath}
		This is just the logarithm of its characteristic function. If $\Delta$ is convex, its \textit{support function} $\Psi_{\Delta} \colon  N_{\mathbb{R}} \rightarrow \mathbb{R}_{- \infty}$ is given by
		\begin{displaymath}
			\Psi_{\Delta}(x) \coloneqq \inf_{y \in \Delta} \lbrace \langle y, x \rangle \rbrace.
		\end{displaymath}
		Note that the function $\Psi_{\Delta}$ is a conical function and satisfies the relations $\iota_{\Delta}^{\vee} = \Psi_{\Delta}$ and $\Psi_{\Delta}^{\vee} = \textup{cl}(\iota_{\Delta}) = \iota_{\textup{cl}(\Delta)}$. If $\Delta$ is closed, so are $\iota_{\Delta}$ and $\Psi_{\Delta}$. Therefore, the Legendre-Fenchel transform induces a bijection between closed convex sets $\Delta \subset M_{\mathbb{R}}$ and conical closed concave functions $f \colon  N_{\mathbb{R}}  \rightarrow \mathbb{R}_{- \infty}$.
	\end{rem}
	\begin{ex}\label{example-ball}
		Let $\| \cdot \|$ be any norm on $N_{\mathbb{R}}$. The triangle inequality and conical property imply that the function $f(x) = - \| x \|$ is conical and concave. Since it is continuous, it is also closed. By Remark~\ref{conical-polyhedron}, its Legendre-Fenchel dual is determined by its stability set. Let $y \in M_{\mathbb{R}}$, for each $x \in N_{\mathbb{R}}$ write $x = \| x \| \cdot \hat{x}$, where $\hat{x}$ is an element of norm $1$. By multilinearity,
		\begin{displaymath}
			\langle y, x \rangle + \| x \| = \| x \| \cdot ( \langle y, \hat{x} \rangle + 1).
		\end{displaymath}
		By definition of the dual norm, $\| y \| = \sup_{ \hat{x} } \lbrace |\langle y, \hat{x} \rangle | \rbrace$. Thus, we obtain the inequality
		\begin{displaymath}
			1 - \| y \| \leq ( \langle y, \hat{x} \rangle + 1) \leq 1 + \|y \|.
		\end{displaymath}
		Since we are in a finite-dimensional Banach space, the above inequalities are sharp. Then,
		\begin{displaymath}
			\| x \| \cdot ( 1 - \| y \| ) \leq \langle y, x \rangle + \| x \|.
		\end{displaymath}
		Since $\| x \| \geq 0$ can be chosen to be arbitrarily large, the left-hand side is bounded below exactly when $0 \leq 1 - \| y \|$. Therefore, $\textup{stab}(f) = \overline{\textup{B}}(0,1)$.
	\end{ex}
	The following proposition lists essential properties of the Legendre-Fenchel transform.
	\begin{prop}\label{LF-prop}
		Let $f,\, g  \colon  N_{\mathbb{R}} \rightarrow \mathbb{R}_{- \infty}$ be concave functions. The following statements hold:
		\begin{enumerate}[label=(\roman*)]
			\item Assume that $\textup{stab}(f) \cap \textup{stab}(g) \neq \emptyset$, then $(f \boxplus g)^{\vee} = f^{\vee} + g^{\vee}$. The stability sets satisfy the relation $\textup{stab}(f \boxplus g) = \textup{stab}(f) \cap \textup{stab}(g)$.
			\item Assume that  $\textup{dom}(f) \cap \textup{dom}(g) \neq \emptyset$, then $(\textup{cl}(f)+\textup{cl}(g))^{\vee} = \textup{cl}(f^{\vee} \boxplus g^{\vee})$. If additionally $\textup{ri}(\textup{dom}(f)) \cap \textup{ri}(\textup{dom}(g)) \neq \emptyset$, then we can omit the closures in the equation above.
			\item The function $\textup{rec}(f^{\vee})$ is the support function of $\textup{dom}(f)$. If $f$ is closed, then $\textup{rec}(f)$ is the support function of $\textup{stab}(f)$.
			\item Let $\lambda \in \mathbb{R}_{>0}$, then $(\lambda \cdot f)^{\vee} = \varpi_{\lambda} f^{\vee}$ and $(\varpi_{\lambda} f)^{\vee} = \lambda f^{\vee}$.
			The stability sets satisfy $\textup{stab}(\lambda \cdot f) = \lambda \cdot \textup{stab}(f)$ and $\textup{stab}(\varpi_{\lambda} f ) = \textup{stab}(f)$.
			\item Let $x_0 \in N_{\mathbb{R}}$, then $(\tau_{x_0} f)^{\vee} = f^{\vee} + x_0$ and $\textup{stab}((\tau_{x_0} f)^{\vee}) = \textup{stab}(f)$.
			\item Assume $f(x) \leq g(x)$ for all $x \in N_{\mathbb{R}}$, then $f^{\vee}(y) \geq g^{\vee}(y)$ for all $y \in M_{\mathbb{R}}$. In particular, $\textup{stab}(g) \subset \textup{stab}(f)$. 
		\end{enumerate}
	\end{prop}
	\begin{proof}
		Parts \textit{(i)} and \textit{(ii)} follow from Theorem~16.4~of~\cite{Roc}. Part \textit{(iii)} follows from Theorem 13.3~of~\cite{Roc}, and the remaining parts follow easily from the definition.
	\end{proof}
	Note that the Legendre-Fenchel transform maps increasing sequences into decreasing sequences and vice versa. We are combining parts \textit{(ii)} and \textit{(vi)} above with Example~\ref{LF-constant} to get the following corollary.
	\begin{cor}\label{LF-conv}
		Let $\lbrace f_i \rbrace_{i \in \mathbb{N}}$ be a sequence concave functions on $N_{\mathbb{R}}$. The following statements hold:
		\begin{enumerate}[label=(\roman*)]
			\item If $\lbrace f_i \rbrace_{i \in \mathbb{N}}$ converges uniformly to $f$, then $\lbrace f^{\vee}_{i} \rbrace_{i \in \mathbb{N}}$ converges uniformly to $f^{\vee}$.
			\item Assume that  $\lbrace f_i \rbrace_{i \in \mathbb{N}}$ is increasing and it converges pointwise to $f$. Then $\lbrace f^{\vee}_{i} \rbrace_{i \in \mathbb{N}}$ is decreasing and converges pointwise to $f^{\vee}$.
		\end{enumerate}
	\end{cor}
	\begin{proof}
		This is a standard result in convex analysis. We include a proof to illustrate the type of argument that will appear later in this section. Let us start with \textit{(i)}. Let $\varepsilon >0$ and $i_0$ such that for all $i \geq i_0$ we have
		\begin{displaymath}
			f_i - \varepsilon \leq f \leq f_i + \varepsilon.
		\end{displaymath}
		Using that $(f_i + c )^{\vee} = f^{\vee}_{i} - c$ combined with part \textit{(vi)} of Proposition~\ref{LF-prop} above, we get
		\begin{displaymath}
			f^{\vee}_{i} - \varepsilon \leq f^{\vee} \leq f^{\vee}_{i} + \varepsilon.
		\end{displaymath}
		The uniform convergence follows.
		
		Now, we prove \textit{(ii)}. If $\lbrace f_{i} \rbrace_{i \in \mathbb{N}}$ is increasing, by part \textit{(vi)} of Proposition~\ref{LF-conv} we have $f^{\vee}_{i} \geq f^{\vee}$. Let $g$ be the pointwise limit of the sequence $\lbrace f^{\vee}_{i} \rbrace_{i \in \mathbb{N}}$. Then, $f^{\vee}_{i} \geq g \geq f^{\vee}$. By duality, we obtain $f_i \leq g^{\vee} \leq f$. Taking the limit over $i$, we obtain $g^{\vee} =f$. Since $g$ is a closed concave function, we get $g=f^{\vee}$.
	\end{proof}
	We finish this section by recalling some results relating the boundedness of concave functions to their stability sets. 
	\begin{prop}\label{bounds-concave}
		Let $f  \colon  N_{\mathbb{R}} \rightarrow \mathbb{R}_{- \infty }$ be a closed concave function. Then:
		\begin{enumerate}[label=(\roman*)]
			\item The function $f$ is bounded above if and only if $0 \in \textup{stab} (f)$. Moreover,
			\begin{displaymath}
				\sup_{x \in N_{\mathbb{R}}} f(x) = - f^{\vee}(0).
			\end{displaymath}
			\item Let $C \subset \textup{dom}(f)$ be convex. If $f|_{C}$ attains its infimum in $\textup{ri}(C)$, it is constant on $C$.
			\item Let $C \subset \textup{ri}( \textup{dom}(f))$ be compact convex, then $f$ attains its minimum at an extreme point of $C$. 
		\end{enumerate}
	\end{prop}
	\begin{proof}
		This is shown in Theorems~27.1,~32.1,~and Corollary~32.3.2~of~\cite{Roc} respectively.
	\end{proof}
	\begin{rem}
		The condition $C \subset \textup{ri}( \textup{dom}(f))$ in part \textit{(iii)} of the above proposition cannot be relaxed. There exist closed concave functions with compact effective domains that are not bounded, or for which the infimum is not attained. For examples, see~\cite{Roc}~p. 345--345. If the effective domain of a concave function is a polytope, then the function is bounded and attains its minimum (see~Corollary~32.3.4~of~\cite{Roc}). 
	\end{rem}
	\subsection{The space of conical functions and sublinear functions}
	This subsection studies globally Lipschitz concave functions on $N_{\mathbb{R}}$. Its main goal is to realize the set of all such functions on $N_{\mathbb{R}}$ as a closed cone in a Banach space, which consists of all continuous functions that can be written as the sum of a conical continuous function and a sublinear continuous function. We start by recalling the definition of globally Lipschitz.
	\begin{dfn} 
		A function $f \colon  N_{\mathbb{R}}\rightarrow \mathbb{R}$ is \textit{Lipschitz} on $S \subset N_{\mathbb{R}}$ if there exist a constant $C$, called a \textit{Lipschitz constant}, such that for every $x,y \in S$ we have $	|f(x) - f(y)| \leq C \cdot \| x - y \|$. If $S=N_{\mathbb{R}}$, we say that $f$ is \textit{globally Lipschitz}.
	\end{dfn}
	The Lipschitz condition on a concave function $f \colon N_{\mathbb{R}} \rightarrow \mathbb{R}$ can be characterized in terms of the conical concave function $\textup{rec}(f)$, or in terms of its Legendre-Fenchel dual $f^{\vee}$. 
	\begin{prop}\label{Lipschitz}
		Let $f \colon N_{\mathbb{R}} \rightarrow \mathbb{R}_{-\infty}$ be concave. The following are equivalent:
		\begin{enumerate}[label=(\roman*)]
			\item The function $f$ is globally Lipschitz.
			\item The function $f$ satisfies $\textup{dom}(f)=\textup{dom}(\textup{rec}(f)) = N_{\mathbb{R}}$.
			\item The function $f$ satisfies $\textup{dom}(f) = N_{\mathbb{R}}$ and $\textup{stab}(f)$ is bounded.
		\end{enumerate}
		Moreover, the minimal Lipschitz constant is given by $C = \sup \lbrace \| y \| \, \vert \, y \in \textup{stab}(f) \rbrace.$
	\end{prop}
	\begin{proof}
		This combines Theorem~10.5 and Corollary~13.3.3~of~\cite{Roc}.
	\end{proof}
	We will study the connections between the Lipschitz condition, the Legendre-Fenchel transform, and the recession function. Let us rephrase Proposition~\ref{Lipschitz} as follows.
	\begin{prop}\label{conical-compact}
		The Legendre-Fenchel transform induces the following bijections:
		\begin{enumerate}[label=(\roman*)]
			\item The map given by $\Psi \mapsto \Delta = \textup{stab}(\Psi)$ is a bijection between the set of conical concave functions $\Psi \colon  N_{\mathbb{R}} \rightarrow \mathbb{R}$ and the set of compact convex sets $\Delta \subset M_{\mathbb{R}}$.
			\item The map given by $f \mapsto g= f^{\vee}$ is a bijection between the set of globally Lipschitz concave functions $f \colon N_{\mathbb{R}} \rightarrow \mathbb{R}$ and the set of closed concave functions $g \colon  M_{\mathbb{R}} \rightarrow \mathbb{R}_{- \infty}$ with $\textup{dom}(g)$ bounded.
		\end{enumerate}
		Moreover, the condition $\textup{dom}(\textup{rec}(f)) = N_{\mathbb{R}}$ is equivalent to the existence of a positive constant $C$ such that $|f(x)| \leq C (1 + \| x \|)$ for all $x \in N_{\mathbb{R}}$.
	\end{prop}
	\begin{proof}
		By Remark~\ref{LF-duality}, the assignment $f \mapsto f^{\vee}$ is a bijection between the cone of closed concave functions on $N_{\mathbb{R}}$ and the cone of closed concave functions on $M_{\mathbb{R}}$. It remains to show that this bijection restricts to the desired subsets.
		
		Proof of \textit{(i)}: By Remark~\ref{conical-polyhedron}, the Legendre-Fenchel dual of $\Psi$ is the indicator function $\iota_{\Delta}$ of the closed convex set $\Delta$. By definition of the indicator function, $\Delta$ is bounded if and only if there exist $r \in \mathbb{R}_{>0}$ such that $\iota_{ \overline{\textup{B}}(0,r)} \geq \iota_{\Delta}$. By Proposition~\ref{LF-prop} and Example~\ref{example-ball}, this inequality is equivalent to
		\begin{displaymath}
			\Psi(x) \geq \Psi_{\overline{\textup{B}}(0,r)} (x) =  - r \cdot \| x \|
		\end{displaymath}
		for all $x \in N_{\mathbb{R}}$. This holds if and only if $\textup{dom}(\Psi) = N_{\mathbb{R}}$; it is enough to choose
		\begin{displaymath}
			r = \sup \lbrace |\Psi (x) | \, | \, x \in \overline{\textup{B}}(0,1) \rbrace.
		\end{displaymath}
		
		Proof of \textit{(ii)}: By Proposition~\ref{Lipschitz}, the function  $f \colon N_{\mathbb{R}} \rightarrow \mathbb{R}$ is globally Lipschitz if and only if $\textup{dom}(g)$ is bounded. By Proposition~\ref{Lipschitz}, $\textup{dom}(\textup{rec}(f)) = N_{\mathbb{R}}$ is equivalent to $f$ being globally Lipschitz. Let $C$ be a Lipschitz constant, then for all $x \in N_{\mathbb{R}}$, $|f(x) - f(0)| \leq C \cdot \|x \|$. Apply the triangle inequality to get
		\begin{displaymath}
			|f(x)| \leq |f(x) - f(0)| + |f(0)| \leq  C \cdot \|x \| + |f(0)|.
		\end{displaymath}
		By enlarging $C$ if necessary, we get the result. The reverse implication follows from the bound
		\begin{displaymath}
			| \textup{rec}(f) (x) | =   \lim_{\lambda \rightarrow + \infty} |f ( \lambda \, x)|/ \lambda \leq  \lim_{\lambda \rightarrow + \infty} C(1+ \|\lambda \, x \| )/ \lambda = C \cdot \| x \|.
		\end{displaymath}
	\end{proof}
	To gain an intuition for the general case, we first investigate the set of conical concave functions. By Propositions~\ref{Lipschitz}~and~\ref{conical-compact}, these functions are globally Lipschitz. We will realize this set as a closed cone of the Banach space introduced below.
	\begin{dfn}\label{C-norm}
		Let $\mathcal{C}(N_{\mathbb{R}})$ be the vector space of real-valued continuous conical functions on $N_{\mathbb{R}}$. We equip it with the $\mathcal{C}$-norm, given by
		\begin{displaymath}
			\| f \|_{\mathcal{C}} \coloneqq \inf \lbrace \varepsilon \in \mathbb{Q}_{>0} \, \vert \, |f(x)| \leq \varepsilon \cdot \| x \| \textup{ for all } x \in N_{\mathbb{R}} \rbrace.
		\end{displaymath}
		Denote by $\mathcal{C}^{+}(N_{\mathbb{R}})$ the subset of $\mathcal{C}(N_{\mathbb{R}})$ consisting of concave functions.
	\end{dfn}
	\begin{rem}\label{comp-conical}
		The normed space $\mathcal{C}(N_{\mathbb{R}})$ is complete, as it is isomorphic to the Banach space $C^{0}(\mathcal{S}^{d-1})$ of real-valued continuous functions on the sphere $\mathcal{S}^{d-1} \subset N_{\mathbb{R}}$. Indeed, by mapping a continuous conical function $f$ to its restriction $ f|_{ \mathcal{S}^{d-1}}$ and observing that
		\begin{align*}
			\|f\|_{\mathcal{C}} = \sup_{ x \in \mathcal{S}^{d-1}} |f(x)| 
		\end{align*}
		we obtain an isometric isomorphism $\mathcal{C}(N_{\mathbb{R}}) \rightarrow C^{0}(\mathcal{S}^{d-1})$. Propositions~\ref{prop-conv}~and~\ref{pointwise-conv} show that $\mathcal{C}^{+}(N_{\mathbb{R}})$ is a closed cone in $\mathcal{C}(N_{\mathbb{R}})$.
	\end{rem}
	By Proposition~\ref{conical-compact}, every function $\Psi \in \mathcal{C}^{+}(N_{\mathbb{R}})$ is identified uniquely with the compact convex set $\textup{stab}(\Psi )$. We will show that this assignment is continuous. Let us introduce the necessary notation to formalize this statement.
	\begin{dfn}
		Denote by $\mathcal{K}(M_{\mathbb{R}})$ the set of compact subsets of $M_{\mathbb{R}}$. Given a compact set $C \in \mathcal{K}(M_{\mathbb{R}})$ and a non-negative number $\varepsilon$, the $\varepsilon$\textit{-thickening} of $C$ is defined as
		\begin{displaymath}
			C + \overline{\textup{B}}(0, \varepsilon ) \coloneqq \lbrace x \in M_{\mathbb{R}} \, \vert \, \textup{dist}(x,\, C) \leq \varepsilon \rbrace.
		\end{displaymath}
		The \textit{Hausdorff distance} on $\mathcal{K}(M_{\mathbb{R}})$ is given by
		\begin{displaymath}
			\textup{d}_{\mathcal{H}}( \Delta_1 , \Delta_2 ) \coloneqq \inf \lbrace \varepsilon > 0 \, \vert \, \Delta_1 \subset \Delta_2 + \overline{\textup{B}}(0,\varepsilon) \textup{ and } \Delta_2 \subset \Delta_1 + \overline{\textup{B}}(0,\varepsilon) \rbrace.
		\end{displaymath}
		We denote by $\mathcal{K}^{+}(M_{\mathbb{R}})$ the subspace of $\mathcal{K}(M_{\mathbb{R}})$ consisting of of convex sets.
	\end{dfn}
	By Chapter~7,~pp.~280--281~of~\cite{Mun00}, the pair $(\mathcal{K}(M_{\mathbb{R}}), \textup{d}_{\mathcal{H}})$ is a complete metric space. The following lemma is a geometric interpretation of convergence in the $\mathcal{C}$-norm. 
	\begin{lem}\label{geom-convergence}
		Let $\Psi \in \mathcal{C}^{+}(N_{\mathbb{R}})$ with stabilty set $\textup{stab}(\Psi) = \Delta$. Let $\lbrace \Psi_n \rbrace_{n \in \mathbb{N}}$ be a sequence in $\mathcal{C}^{+}(N_{\mathbb{R}})$ and write $\Delta_n = \textup{stab}(\Psi_n)$. Then, the sequence $\lbrace \Psi_n  \rbrace_{n \in \mathbb{N}}$ converges to $\Psi$ in the $\mathcal{C}$-norm if and only if the sequence $\lbrace \Delta_n  \rbrace_{n \in \mathbb{N}}$ converges to $\Delta$ in the Hausdorff distance, i.e. for each $\varepsilon > 0$ there is a $n_{\varepsilon}$ such that for all $n \geq n_{\varepsilon}$ we have 
		\begin{displaymath}
			\Delta_n \subset \Delta + \overline{\textup{B}}(0,\varepsilon) \quad \textup{and} \quad \Delta \subset \Delta_n + \overline{\textup{B}}(0,\varepsilon).
		\end{displaymath}
		Moreover, if the sequence $\lbrace \Psi_n  \rbrace_{n \in \mathbb{N}}$ converges to $\Psi$ in the $\mathcal{C}$-norm and $\Delta$ has non-empty interior, then for every $\varepsilon >0$ there is $n_{\varepsilon}$ such that for all $n \geq n_\varepsilon$ we have $\Delta_{\varepsilon} \subset \Delta_n$, where
		\begin{displaymath}
			\Delta_{\varepsilon} \coloneqq \textup{Conv} (\lbrace y \in \Delta \, \vert \, \textup{dist}(y,\textup{rb}(\Delta)) \geq \varepsilon \rbrace).
		\end{displaymath}
	\end{lem}
	\begin{proof}
		By definition, $\lbrace \Psi_n  \rbrace_{n \in \mathbb{N}}$ converges to $\Psi$ in the $\mathcal{C}$-norm if and only if for each $\varepsilon > 0$ there is a $n_{\varepsilon}$ such that for all $n \geq n_{\varepsilon}$ we have $\| \Psi - \Psi_n \|_{\mathcal{C}} \leq \varepsilon$. We can rewrite this inequality as
		\begin{displaymath}
			\Psi (x) - \varepsilon \cdot \| x \|  \leq \Psi_n (x) \quad\textup{and} \quad \Psi_n (x) - \varepsilon \cdot \| x \|  \leq \Psi (x),
		\end{displaymath}
		for all $x \in N_{\mathbb{R}}$. By Proposition~\ref{LF-prop}, this is equivalent to
		\begin{displaymath}
			\Delta_n \subset \Delta + \overline{\textup{B}}(0,\varepsilon) \quad \textup{and} \quad \Delta \subset \Delta_n + \overline{\textup{B}}(0,\varepsilon).
		\end{displaymath}
		Therefore, we get the first part of the lemma. Assume the convergence and that $\Delta$ has a non-empty interior. Then, $\Delta$ is of full dimension, and the relative boundary and boundary coincide. The case $\Delta_{\varepsilon} = \emptyset$ is trivial. Thus, we assume the opposite. Let $y \in \Delta$ such that $\textup{dist}(y,\textup{rb}(\Delta)) \geq \varepsilon$. By definition, $y + \overline{\textup{B}}(0,\varepsilon) \subset \Delta$. This implies
		\begin{displaymath}
			\lbrace y \in \Delta \, \vert \, \textup{d}(y,\textup{rb}(\Delta)) \geq \varepsilon \rbrace + \overline{\textup{B}}(0,\varepsilon) \subset \Delta.
		\end{displaymath}
		Taking the Minkowski sum and convex hull commute, therefore
		\begin{displaymath}
			\textup{Conv} (\lbrace y \in \Delta \, \vert \, \textup{d}(y,\textup{rb}(\Delta)) \geq \varepsilon \rbrace) + \overline{\textup{B}}(0,\varepsilon) = \textup{Conv} \left( \lbrace y \in \Delta \, \vert \, \textup{d}(y,\textup{rb}(\Delta)) \geq \varepsilon \rbrace + \overline{\textup{B}}(0,\varepsilon) \right).
		\end{displaymath}
		It follows that for all $n \geq n_{\varepsilon}$,
		\begin{displaymath}
			\Delta_{\varepsilon} + \overline{\textup{B}}(0,\varepsilon) \subset \Delta \subset \Delta_n + \overline{\textup{B}}(0,\varepsilon).
		\end{displaymath}
		Let $\Psi_{\Delta_{\varepsilon}}$ be the support function of  $\Delta_{\varepsilon}$. Using Proposition~\ref{LF-prop}, translate the above contention into the inequality
		\begin{displaymath}
			\Psi_{\Delta_{\varepsilon}} (x) - \varepsilon \cdot \|x \| \geq \Psi_n (x) - \varepsilon \cdot \|x \|
		\end{displaymath}
		for all $x \in N_{\mathbb{R}}$. This immediately implies $\Psi_{\Delta_{\varepsilon}}  \geq \Psi_n$. Using Proposition~\ref{LF-prop} again, we conclude $\Delta_{\varepsilon} \subset \Delta_n$.
	\end{proof}
	\begin{rem}
		If the interior of $\Delta$ in Lemma~\ref{geom-convergence} is empty, the conclusion on the set $\Delta_{\varepsilon}$ might fail. Consider the following convex subsets of $\mathbb{R}^2$:
		\begin{displaymath}
			\Delta_n = \left[ 0 , 1 \right] \times \lbrace 1/n \rbrace \quad \textup{and} \quad \Delta = \left[ 0 , 1 \right] \times \lbrace 0 \rbrace.
		\end{displaymath}
		The sequence $\lbrace \Delta_n  \rbrace_{n \in \mathbb{N}}$ satisfies the convergence condition, but $\Delta_n \cap \Delta = \emptyset$ for all $n$. 
	\end{rem}
	The previous discussion can be summarized as follows.
	\begin{cor}\label{Kompact}
		The map $\textup{stab} \colon \mathcal{C}^{+}(N_{\mathbb{R}}) \rightarrow \mathcal{K}^{+}(M_{\mathbb{R}})$ given by $\Psi \mapsto \textup{stab}(\Psi)$ is an isometry that commutes with sums. Its inverse is given by the assignment $\Delta \mapsto \Psi_{\Delta}$. In particular, the metric space $(\mathcal{K}^{+}(M_{\mathbb{R}}), \textup{d}_{\mathcal{H}})$ is complete.
	\end{cor}
	\begin{proof}
		This follows immediately from Lemma~\ref{geom-convergence} and Propositions~\ref{LF-prop}~and~\ref{conical-compact}.
	\end{proof}
	Now, we realize the set of globally Lipschitz concave functions as a closed cone in a Banach space. Intuitively, the recession function $\textup{rec}(f)$ is the conical convex function that best approximates $f$. Thus, the difference between $f$ and $\textup{rec}(f)$ should grow slower than any continuous conical function. This motivates the following definition.
	\begin{dfn}\label{G-def}
		Let $f \colon  N_{\mathbb{R}} \rightarrow \mathbb{R}$ be continuous. We have the following definitions:
		\begin{enumerate}[label=(\roman*)]
			\item The function $f$ is \textit{sublinear} if $f(x)$ is $o(1 + \| x \| )$ as $\| x \| \rightarrow \infty$, that is, for each $\varepsilon > 0$ there exist $R_{\varepsilon}>0$ such that, whenever $\|x \|>R_{\varepsilon}$, we have $|f(x)| \leq \varepsilon \cdot (1 + \|x\|).$ Denote by $\mathcal{SL}(N_{\mathbb{R}})$ the vector space of all continuous sublinear functions on $N_{\mathbb{R}}$. 
			\item  Define the normed vector space $\mathcal{G}(N_{\mathbb{R}}) \coloneqq \mathcal{SL}(N_{\mathbb{R}}) + \mathcal{C}(N_{\mathbb{R}}) \subset C^{0}(N_{\mathbb{R}})$. We equip it with the $\mathcal{G}$-norm, given by
			\begin{displaymath}
				\| f \|_{\mathcal{G}} \coloneqq \inf \lbrace \varepsilon \in \mathbb{Q}_{>0} \, \vert \, |f (x) | \leq \varepsilon \cdot (1 + \| x \| ) \textup{ for all } x \in N_{\mathbb{R}} \rbrace.
			\end{displaymath}
			Denote by $\mathcal{G}^{+}(N_{\mathbb{R}})$ its subset of concave functions.
			\item The \textit{extended recession map} $\textup{rec} \colon \mathcal{G}(N_{\mathbb{R}})  \rightarrow \mathcal{C}(N_{\mathbb{R}})$ is given by
			\begin{displaymath}
				\textup{rec}(f) (x)\coloneqq  \lim_{\lambda \rightarrow \infty} f (\lambda \, x)/\lambda.
			\end{displaymath}
		\end{enumerate}
	\end{dfn}
	\begin{rem}\label{G-VS}
		We make the following observations:
		\begin{enumerate}[label=(\arabic*)]
			\item The linear map $\mathcal{G}(N_{\mathbb{R}}) \rightarrow C^{0}_{\textup{bd}} (N_{\mathbb{R}})$ given by the assignment
			\begin{displaymath}
				f \longmapsto \left( x \mapsto f(x)/(1+ \| x \|) \right)
			\end{displaymath}
			is an isometry (in particular, injective). Therefore, it identifies the image of $\mathcal{G}(N_{\mathbb{R}})$ under this map as a normed subspace of the Banach space $C^{0}_{\textup{bd}}(N_{\mathbb{R}})$ of continuous bounded functions on $N_{\mathbb{R}}$.
			\item By Proposition~\ref{prop-conv}, $\mathcal{G}^{+}(N_{\mathbb{R}})$ is a cone in $\mathcal{G}(N_{\mathbb{R}})$. Convergence in the $\mathcal{G}$-norm implies pointwise convergence. Then, Proposition~\ref{pointwise-conv} shows that $\mathcal{G}^{+}(N_{\mathbb{R}})$ is closed in $\mathcal{G}(N_{\mathbb{R}})$.
			\item Let $f \in \mathcal{G}(N_{\mathbb{R}})$ and write $f =  f_0 + \Psi$ with $f_0$ sublinear and $\Psi$ conical. It is immediate that $\textup{rec}(f) = \Psi$.  If $f$ is concave, Proposition~\ref{prop-conv} shows that the extended recession function of $f$ and the usual recession function coincide. Using Proposition~\ref{conical-compact}, we conclude that every function in $\mathcal{G}^{+}(N_{\mathbb{R}})$ is globally Lipschitz.
			\item By Proposition~\ref{conical-compact}, the $\mathcal{G}$-norm of a globally Lipschitz concave function $f \colon N_{\mathbb{R}} \rightarrow \mathbb{R}$ is finite.
		\end{enumerate}
		It remains to show that $\mathcal{G}(N_{\mathbb{R}})$ is complete and the closed cone $\mathcal{G}^{+}(N_{\mathbb{R}})$ coincides with the set of globally Lipschitz concave functions on $N_{\mathbb{R}}$. First, we will show the completeness. Then, we will exhibit each globally Lipschitz concave function $f$ as the limit in the $\mathcal{G}$-norm of a sequence in $\mathcal{G}^{+}(N_{\mathbb{R}})$. By completeness, $f \in \mathcal{G}^{+}(N_{\mathbb{R}})$.
	\end{rem}
	We defined $\mathcal{G}(N_{\mathbb{R}})$ as the sum of the spaces $\mathcal{SL}(N_{\mathbb{R}}) $ and $ \mathcal{C}(N_{\mathbb{R}})$. To show that  $\mathcal{G}(N_{\mathbb{R}})$ is Banach, it is enough to prove that this is a direct sum of Banach spaces. We prove the following auxiliary lemmas, which then yield the desired completeness result.
	\begin{lem}
		Let $\Psi \in \mathcal{C}(N_{\mathbb{R}})$. Then, the identity $\| \Psi \|_{\mathcal{C}} = \| \Psi \|_{\mathcal{G}}$ holds. In particular, $\mathcal{C}(N_{\mathbb{R}})$ is a complete subspace of $\mathcal{G}(N_{\mathbb{R}})$.
	\end{lem}
	\begin{proof}
		The inequality $\| \Psi \|_{\mathcal{C}} \geq \| \Psi \|_{\mathcal{G}}$ follows trivially from the definitions. We only need to show the reverse inequality. Let $\varepsilon >0$ such that for all $x \in N_{\mathbb{R}}$ we have $|\Psi (x) | \leq \varepsilon \cdot (1 + \| x \| )$. Let $x_0 \in \mathcal{S}^{d-1}$ and $\lambda >0$. Replacing $x$ by $\lambda \cdot x_0$ we get $|\Psi ( \lambda \cdot x_0 ) | \leq \varepsilon (1 + \| \lambda \cdot x_0\|)$. Dividing by $1 + \| \lambda \cdot x_0 \|$ and taking the limit as $\lambda \rightarrow \infty$, we obtain
		\begin{displaymath}
			\lim_{\lambda \rightarrow  \infty} \frac{|\Psi ( \lambda \cdot x_0 ) |}{1 + \| \lambda \cdot x_0\|}  = \lim_{\lambda \rightarrow  \infty} \frac{ \lambda \cdot |\Psi (x_0 ) |}{1 + \lambda \cdot \|x_0\|}  = \lim_{\lambda \rightarrow  \infty} \frac{  |\Psi (x_0 ) |}{1/ \lambda + 1 }  = |\Psi(x_0)| \leq \varepsilon.
		\end{displaymath}
		Since $x_0 \in \mathcal{S}^{d-1}$ is arbitrary, taking the supremum we get that $\| \Psi \|_{\mathcal{C}} \leq \varepsilon$. The result follows.
	\end{proof}
	\begin{lem}\label{SL-banach}
		The pair $(\mathcal{SL}(N_{\mathbb{R}}), \| \cdot \|_{\mathcal{G}})$ is a Banach space.
	\end{lem}
	\begin{proof}
		Denote by $N_{\mathbb{R}}^{\ast}$ the Alexandroff (one-point) compactification of $N_{\mathbb{R}}$. It is clear from the definition of sublinearity that the map $\mathcal{SL}(N_{\mathbb{R}})  \rightarrow C^0 (N_{\mathbb{R}}^{\ast})$
		given by the assignment
		\begin{displaymath}
			f \longmapsto \left( x \mapsto f(x)/(1+ \| x \|) \right) 
		\end{displaymath}
		identifies $\mathcal{SL}(N_{\mathbb{R}})$ with the closed subspace $C^{0}_{0}(N_{\mathbb{R}}^{\ast}) \subset C^0 (N_{\mathbb{R}}^{\ast})$ of functions vanishing at $\infty \in N_{\mathbb{R}}^{\ast}$. Indeed, the inverse map is given by multiplication by $(1 + \|x\|)$ followed by restriction to $N_{\mathbb{R}}$. Hence, the space $\mathcal{SL}(N_{\mathbb{R}})$ is complete.
	\end{proof}
	\begin{thm}\label{G-banach}
		The extended recession map is continuous and linear. Its kernel is $\mathcal{SL}(N_{\mathbb{R}})$ and it fixes the subspace $\mathcal{C}(N_{\mathbb{R}})$. It induces a decomposition
		\begin{displaymath}
			\mathcal{G}(N_{\mathbb{R}}) = \mathcal{SL}(N_{\mathbb{R}}) \oplus  \mathcal{C}(N_{\mathbb{R}}).
		\end{displaymath}
		In particular, the space $\mathcal{G}(N_{\mathbb{R}})$ is complete. 
	\end{thm}
	\begin{proof}
		Completeness of $\mathcal{G}(N_{\mathbb{R}})$ and the decomposition as a direct sum are direct implications of the assertions on the extended recession map. Linearity is obvious. We show that the kernel is $\mathcal{SL}(N_{\mathbb{R}})$. Let $f = \Psi + f_0$ be the sum of a conical and a sublinear function, respectively. Then,
		\begin{displaymath}
			\textup{rec}(f) (x) = \lim_{\lambda \rightarrow + \infty} f( \lambda \, x)/ \lambda = \Psi( x) +  \lim_{\lambda \rightarrow + \infty} f_0 ( \lambda \, x) / \lambda = \Psi(x).
		\end{displaymath}
		It remains to show that the recession map is continuous. Let $\varepsilon > 0$. For all $x \neq 0$ apply the triangle inequality and divide by $\|x\|$ to obtain
		\begin{displaymath}
			\frac{|\Psi(x)|}{\| x \| }  \leq \frac{|f (x)|}{ \| x \|} + \frac{|f_{0} (x)|}{ \| x \|} = \frac{1+ \|x \| }{ \|x \|}  \cdot \left( \frac{|f (x)|}{1 + \| x \|} +  \frac{|f_{0} (x)|}{ 1+\| x \|} \right).
		\end{displaymath}
		Let $f$ be such that $\|f\|_{\mathcal{G}} \leq \varepsilon /4$. Since $\Psi$ is conical, we may assume that $\|x \| \geq R > 1$, where $R$ is large enough so that $|f_0 (x)| \leq \varepsilon / 4 \cdot (1 + \|x\|)$. Then
		\begin{displaymath}
			\frac{|\Psi(x)|}{\| x \| } \leq \frac{1+ \|x \| }{ \|x \|}  \cdot \left( \frac{|f (x)|}{1 + \| x \|} +  \frac{|f_{0} (x)|}{ 1+ \| x \|} \right) \leq 2 \cdot ( \| f \|_{\mathcal{G}} + \varepsilon/4 )= \varepsilon.
		\end{displaymath}
		Take the supremum over $\|x\| = R$ to show that $\| \Psi \|_{\mathcal{C}}=\| \Psi \|_{\mathcal{G}} \leq \varepsilon$. Therefore, the map $\textup{rec}$ is continuous.
	\end{proof}
	We want to show that every globally Lipschitz function $f$ is the limit in the $\mathcal{G}$-norm of a sequence $\lbrace f_n \rbrace_{n \in \mathbb{N}}$ in $\mathcal{G}^{+}(N_{\mathbb{R}})$. In fact, for each $n$, the function $f_n - \textup{rec}(f_n)$ will be bounded. To distinguish this class of functions, we introduce the following terminology.
	\begin{dfn}\label{min-sing-NR}
		Let $f \in \mathcal{G}(N_{\mathbb{R}})$. We say that $f$ is \textit{relatively bounded} if the difference $f - \textup{rec}(f)$ is bounded on $N_{\mathbb{R}}$. Then, denote by $\mathcal{G}_{\textup{rbd}}(N_{\mathbb{R}})$ the set of all relatively bounded functions and $\mathcal{G}^{+}_{\textup{rbd}}(N_{\mathbb{R}})$ its subset of concave functions.
	\end{dfn}
	Now, we prove the following pair of auxiliary lemmas. Afterwards, we show the main theorem of this subsection.
	\begin{lem}\label{cauchy-concave-G}
		Let $f_1, f_2 \in \mathcal{G}^{+}( N_{\mathbb{R}})$ and denote $g_i = f^{\vee}_{i}$. Let $\varepsilon > 0$ and $c \in \mathbb{R}$, then the following are equivalent:
		\begin{enumerate}[label = (\roman*)]
			\item For all $x \in N_{\mathbb{R}}$, we have $|f_1 (x) - f_2(x) | \leq \varepsilon \cdot (c + \| x \|) $.
			\item The inequalities $g_1 \boxplus \iota_{\overline{\textup{B}}(0,\varepsilon)} + \varepsilon \, c \geq g_2$ and $g_2 \boxplus \iota_{\overline{\textup{B}}(0,\varepsilon)} + \varepsilon  \, c \geq g_1$ hold.
		\end{enumerate}
	\end{lem}
	\begin{proof}
		We use the definition of the absolute value and rearrange the inequalities to show that $|f_1 (x) - f_2(x) | \leq \varepsilon \cdot (c + \| x \|) $ is equivalent to
		\begin{displaymath}
			f_1 (x) - \varepsilon \, c - \varepsilon \cdot \| x \| \leq f_2 (x) \quad \textup{and} \quad f_2 (x) - \varepsilon \, c - \varepsilon \cdot \| x \| \leq f_1 (x).
		\end{displaymath}
		Note that the function $h (x) = - \varepsilon \, c - \varepsilon \cdot \| x \|$ is concave. By symmetry, it is enough to prove one of the inequalities in \textit{(ii)}. By Proposition~\ref{LF-prop}, the inequality on the left-hand side holds for all $x \in N_{\mathbb{R}}$ if and only if
		\begin{displaymath}
			(f_1 + h)^{\vee} = g_1 \boxplus h^{\vee} \geq g_2.
		\end{displaymath}
		By Proposition~\ref{LF-prop} and Examples~\ref{LF-constant},~\ref{example-ball}, we know that $	h^{\vee} = \iota_{\overline{\textup{B}}(0,\varepsilon)} + \varepsilon \, c$. The result follows from the identity $(f_1 + h)^{\vee} = g_1 \boxplus \iota_{\overline{\textup{B}}(0,\varepsilon)} + \varepsilon \, c$.
	\end{proof}
	\begin{lem}\label{mono-approx-1}
		Let $\Delta \subset M_{\mathbb{R}}$ be a non-empty compact convex set. Let $g \colon  M_{\mathbb{R}} \rightarrow \mathbb{R}_{ -\infty }$ be a proper closed concave function such that $\textup{cl}(\textup{dom}(g)) = \Delta$. Let $\lbrace \Delta_k \rbrace_{k \in \mathbb{N}}$ be a decreasing sequence of compact convex subsets of $M_{\mathbb{R}}$ converging to $\Delta$ in the Hausdorff distance. Then, there exist a decreasing sequence $\lbrace g_n  \rbrace_{n \in \mathbb{N}}$ of closed concave functions on $M_{\mathbb{R}}$ satisfying:
		\begin{enumerate}[label=(\roman*)]
			\item The sequence $\lbrace \textup{dom}(g_n) \rbrace_{n \in \mathbb{N}}$ is a subsequence of $\lbrace \Delta_k  \rbrace_{k \in \mathbb{N}}$.
			\item For each $n$, the function $g_n$ is continuous on $\textup{dom}(g_n)$. In particular, $g_n$ is bounded.
			\item The sequence $\lbrace g^{\vee}_{n}  \rbrace_{n \in \mathbb{N}}$ is an increasing sequence in $\mathcal{G}^{+}_{\textup{rbd}}(N_{\mathbb{R}})$, it is Cauchy in the $\mathcal{G}$-norm, with limit $g^{\vee}$.
		\end{enumerate}
	\end{lem}
	\begin{proof}
		The sequence $\lbrace \Delta_k \rbrace_{k \in \mathbb{N}}$ is decreasing and converging to $\Delta$. Therefore, we may choose a decreasing subsequence $\lbrace \Delta_{k_n}  \rbrace_{n \in \mathbb{N}}$ such that $\Delta_{k_n} \subset \Delta + \textup{B}(0,\min \lbrace 1, 1/n \rbrace)$. Without loss of generality, assume that $n \geq 1$. As $\textup{cl}(\textup{dom}(g)) = \Delta$, we get that $\Delta_{k_n}$ is a compact convex set contained in an open subset of $\textup{dom}(g)+ \overline{\textup{B}}(0,1/n)$. By Propositions~\ref{prop-conv}~and~\ref{LF-prop}, the function
		\begin{displaymath}
			h_n = g \boxplus  \iota_{\overline{\textup{B}}(0,1/n)} + 1/n
		\end{displaymath}
		is a closed concave function with effective domain $\textup{dom}(g)+ \overline{\textup{B}}(0,1/n)$. Denote by $g_n$ the restriction of $h_n$ to $\Delta_{k_n}$ (extended by $- \infty$ on the complement of $\Delta_{k_n}$). Since the function $h_n$ is continuous on the interior of its domain, the function $g_n$ is a closed concave function, continuous on its effective domain $\Delta_{k_n}$. It follows that $g^{\vee}_n \in \mathcal{G}^{+}_{\textup{rbd}}(N_{\mathbb{R}})$. By construction, the sequence $\lbrace g_n  \rbrace_{n \in \mathbb{N}}$ is decreasing and the inequalities
		\begin{displaymath}
			g \boxplus  \iota_{\overline{\textup{B}}(0,1/n)} + 1/n \geq g_n, \quad g_n \boxplus  \iota_{\overline{\textup{B}}(0,1/n)} + 1/n \geq g
		\end{displaymath}
		are satisfied trivially. By Lemma~\ref{cauchy-concave-G}, the above inequalities imply that for all $x \in N_{\mathbb{R}}$,
		\begin{displaymath}
			|g^{\vee} (x) - g^{\vee}_n (x) | \leq 1/n \cdot (1 + \| x \|).
		\end{displaymath}
		This inequality, combined with the triangle inequality, shows that the sequence $\lbrace (g)^{\vee}_n  \rbrace_{n \in \mathbb{N}}$ is increasing and Cauchy in the $\mathcal{G}$-norm. By fixing $x$ in the above inequality, we obtain that $\lbrace (g)^{\vee}_n  \rbrace_{n \in \mathbb{N}}$ converges pointwise to $g^{\vee}$, thus, its limit in the $\mathcal{G}$-norm must be $g^{\vee}$. By Corollary~\ref{LF-conv}, the sequence $g_n$ converges pointwise to $g$.
	\end{proof}
	\begin{thm}\label{closedcone}
		The closed cone $\mathcal{G}^{+}(N_{\mathbb{R}})$ coincides with the set of globally Lipschitz concave functions on $N_{\mathbb{R}}$.
	\end{thm}
	\begin{proof}
		From Remark~\ref{G-VS}, we know that every function in $\mathcal{G}^{+}(N_{\mathbb{R}})$ is globally Lipschitz. Assume now that $f \colon N_{\mathbb{R}} \rightarrow \mathbb{R}$ is globally Lipschitz. By Proposition~\ref{conical-compact}, the Legendre-Fenchel dual $g = f^{\vee}$ satisfies the conditions of Lemma~\ref{mono-approx-1}. Let $\Delta= \textup{cl}(\textup{stab}(f))$. Trivially, the constant sequence $\lbrace \Delta \rbrace_{n \in \mathbb{N}}$ is decreasing and converges in the Hausdorff distance to $\Delta$. Then, use the lemma to obtain a sequence $\lbrace f_n \rbrace_{n \in \mathbb{N}}$ in $f \in \mathcal{G}^{+}_{\textup{rbd}}(N_{\mathbb{R}})$ converging to $f$ in the $\mathcal{G}$-norm. It follows that $f \in \mathcal{G}^{+}(N_{\mathbb{R}})$.
	\end{proof}
	The type of convergence in Lemma~\ref{mono-approx-1} is a particular case of the following result.
	\begin{lem}\label{mono-conv-G}
		Let $f$ and $\lbrace f_n  \rbrace_{n \in \mathbb{N}}$ be a function and an increasing sequence in $\mathcal{G}^{+}(N_{\mathbb{R}})$ respectively. The following are equivalent:
		\begin{enumerate}[label=(\roman*)]
			\item The sequence $\lbrace f_n \rbrace_{n \in \mathbb{N}}$ converges to $f$ in the $\mathcal{G}$-norm.
			\item The sequence $\lbrace \textup{rec}(f_n) \rbrace_{n \in \mathbb{N}}$ converges to $\textup{rec}(f)$ in the $\mathcal{C}$-norm and $\lbrace f_n \rbrace_{n \in \mathbb{N}}$ converges pointwise to $f$.
			\item The sequence $\lbrace \textup{cl}(\textup{stab}(f^{\vee}_{n}))  \rbrace_{n \in \mathbb{N}}$ converges to $ \textup{cl}(\textup{stab}(f^{\vee}))$ in the Hausdorff distance and $\lbrace f^{\vee}_{n} \rbrace_{n \in \mathbb{N}}$ converges pointwise to $f^{\vee}$.
		\end{enumerate}
	\end{lem}
	\begin{proof}
		The proof of \textit{(i)} $\Rightarrow$ \textit{(ii)} is clear: Convergence in the $\mathcal{G}$-norm implies pointwise convergence. The convergence of the recession functions follows from Theorem~\ref{G-banach}.
		
		Assume \textit{(ii)}, then \textit{(iii)} follows from Corollaries~\ref{Kompact} and~\ref{LF-conv}. The converse statement follows from the homomorphism in Corollary~\ref{Kompact}, together with the fact that $f$ is a closed concave function. Then, follow the proof of Corollary~\ref{LF-conv} to obtain the result.
		
		The main content of the lemma is that \textit{(ii)} implies \textit{(i)}. Fix $\varepsilon >0$. The sequence $\lbrace f_n \rbrace_{n \in \mathbb{N}}$ is increasing. Thus, it is enough to show that for all $x \in N_{\mathbb{R}}$ and all $n \in \mathbb{N}$ large enough, we have
		\begin{displaymath}
			f(x) \leq f_n (x) +  \varepsilon \cdot (1 + \| x \|).
		\end{displaymath}
		The sequence $\lbrace \textup{rec}(f_n) \rbrace_{n \in \mathbb{N}}$ of recession functions is increasing. Indeed, it is a consequence of the monotonicity of pointwise limits. By the convergence of recession functions, there exist $k = k(\varepsilon )$ such that for all $n \geq k$ and all $x \in N_{\mathbb{R}}$,
		\begin{displaymath}
			| \textup{rec}(f)(x) - \textup{rec}(f_n)(x) | \leq \varepsilon / 3 \cdot \| x \|.
		\end{displaymath}
		After removing the absolute value in the above inequality and rearranging, we obtain
		\begin{displaymath}
			\textup{rec}(f)(x) \leq  \textup{rec}(f_k )(x) +  \varepsilon / 3 \cdot \| x \|.
		\end{displaymath}
		By sublinearity of $f_{k} - \textup{rec}(f_{k})$ and $f - \textup{rec}(f)$, there exist a positive $R_{k} > 1$ such that for all $x$ satisfying $\| x \|>R_{k}$,
		\begin{displaymath}
			| \textup{rec}(f_{k})(x) - f_{k}(x) |  \leq  \varepsilon / 3 \cdot \| x \| \quad \textup{and} \quad | f(x) - \textup{rec}(f)(x) |  \leq  \varepsilon / 3 \cdot \| x \|.
		\end{displaymath}
		Similarly to the first inequality, we get
		\begin{displaymath}
			\textup{rec}(f_k )(x) \leq  f_{k}(x) +  \varepsilon / 3 \cdot \| x \|  \quad \textup{and} \quad  f(x) \leq  \textup{rec}(f)(x) +  \varepsilon / 3 \cdot \| x \|.
		\end{displaymath}
		Combining the three inequalities, we showed that for all $x$ such that $\| x \|>R_{k}$,
		\begin{displaymath}
			f(x) \leq  \textup{rec}(f)(x) +  \varepsilon / 3 \cdot \| x \| \leq \textup{rec}(f_k )(x) +  2 \,  \varepsilon / 3 \cdot \| x \| \leq f_{k}(x) + \varepsilon \cdot \| x \|.
		\end{displaymath}
		The sequence $\lbrace f_n \rbrace$ is increasing, therefore, for all $n \geq k$ and all $\|x \| >R_{k}$,
		\begin{displaymath}
			f(x) \leq f_{n}(x) + \varepsilon \cdot \| x \|.
		\end{displaymath}
		By Dini's theorem, the sequence $f_n$ converges uniformly to $f$ on the ball $\overline{\textup{B}}(0, R_k)$. Thus, we may find $n_{\varepsilon} \geq k$ such that for all $n \geq n_{\varepsilon}$ and all $\|x \| \leq R_k$, we have $f(x) \leq f_{n}(x) + \varepsilon$.	The convergence in the $\mathcal{G}$-norm follows.
	\end{proof}
	The result below will be useful to verify the pointwise convergence of sequences in $\mathcal{G}^{+}(N_{\mathbb{R}})$.
	\begin{prop}\label{equicontinuity}
		Let $\lbrace f_n \rbrace_{n \in \mathbb{N}}$ be a sequence in $\mathcal{G}^{+}(N_{\mathbb{R}})$ such that:
		\begin{enumerate}[label=(\roman*)]
			\item There exist $x_0 \in N_{\mathbb{R}}$ such that $\lbrace f_n (x_0) \rbrace_{n \in \mathbb{N}}$ is bounded.
			\item The set $S = \bigcup_{n} \textup{stab}(f_n)$ is bounded.
		\end{enumerate}
		Then, for each compact $K \subset N_{\mathbb{R}}$, there is a subsequence $\lbrace f_{n_k}  \rbrace_{k \in \mathbb{N}}$ converging uniformly on $K$.
	\end{prop}
	\begin{proof}
		Based on the two conditions above, the following constant is finite.
		\begin{displaymath}
			C = \max \left\lbrace \sup_{n \in \mathbb{N}} \lbrace |f_n (x_0)| \rbrace , \,  \sup_{y \in S} \lbrace \|y \| \rbrace \right\rbrace < \infty.
		\end{displaymath}
		By Proposition~\ref{Lipschitz}, every function $f_n$ is globally Lipschitz with the same constant $C$. It follows that the sequence is equicontinuous. Now, for every $x \in N_{\mathbb{R}}$ and every $n$,
		\begin{displaymath}
			|f_n (x)| \leq |f_n (x) - f_n (x_0)| + |f_n (0)| \leq C \cdot (1+ \| x - x_0 \|). 
		\end{displaymath}
		Let $K$ be a compact subset of $N_{\mathbb{R}}$. The estimate above implies that the sequence $\lbrace f_n |_{  K}  \rbrace_{n \in \mathbb{N}}$ is uniformly bounded. These are the requirements for the theorem of Arzel\`{a}-Ascoli, which states that the sequence $\lbrace f_n|_{ K}  \rbrace_{n \in \mathbb{N}}$ admits a subsequence converging uniformly on $K$.
	\end{proof}
	The following corollary is an improved version of Dini's theorem for $\mathcal{G}^{+}(N_{\mathbb{R}})$.
	\begin{cor}\label{dini-conc}
		Let $\lbrace f_n  \rbrace_{n \in \mathbb{N}}$ be an increasing sequence in $\mathcal{G}^{+}(N_{\mathbb{R}})$. Suppose there is $x_0 \in N_{\mathbb{R}}$ such that $\lbrace f_n (x_0)  \rbrace_{n \in \mathbb{N}}$ is bounded. Then, there exist a function $f \in \mathcal{G}^{+}(N_{\mathbb{R}})$ such that, for every compact subset $K$ of $N_{\mathbb{R}}$, the sequence $\lbrace f_n |_{ K}  \rbrace_{n \in \mathbb{N}}$ converges uniformly to $f|_{  K}$.
	\end{cor}
	\begin{proof}
		The sequence $\lbrace f_n  \rbrace_{n \in \mathbb{N}}$ is increasing, then the sequence $\lbrace \textup{stab}(f_n)  \rbrace_{n \in \mathbb{N}}$ of stability sets is decreasing. Let $R > 0 $, then apply Proposition~\ref{equicontinuity} to $K= \overline{\textup{B}}(0,R)$. Since the sequence $\lbrace f_n  \rbrace_{n \in \mathbb{N}}$ is increasing, the whole sequence converges uniformly on $K$. Since $R$ is arbitrary, the function $f$ defined as the pointwise limit of $\lbrace f_n  \rbrace_{n \in \mathbb{N}}$ is concave, has effective domain $N_{\mathbb{R}}$ and bounded stability set. This shows that $f \in \mathcal{G}^{+}(N_{\mathbb{R}})$.
	\end{proof}
	\begin{cor}\label{convergenceconicconcave}
		Every increasing sequence $\lbrace \Psi_n  \rbrace_{n \in \mathbb{N}}$ in $\mathcal{C}^{+}(N_{\mathbb{R}})$ converges in the $\mathcal{C}$-norm.
	\end{cor}
	\begin{proof}
		The sequence $\lbrace \Psi_n (0)  \rbrace_{n \in \mathbb{N}}$ is the constant $0$. By Corollary~\ref{dini-conc}, $\lbrace \Psi_n  \rbrace_{n \in \mathbb{N}}$ converges uniformly on the unit ball $\overline{\textup{B}}(0,1)$. This is the same as convergence in the $\mathcal{C}$-norm.
	\end{proof}
	\begin{rem}
		Let $\lbrace f_n  \rbrace_{n \in \mathbb{N}}$ be an increasing sequence in $\mathcal{G}^{+}(N_{\mathbb{R}})$. We have the following observations:
		\begin{enumerate}
			\item Suppose there is a dense subset $S \subset N_{\mathbb{R}}$ and a function $f \in \mathcal{G}^{+}(N_{\mathbb{R}})$ such that, for all $x \in S$, the sequence $\lbrace f_n (x) \rbrace_{n \in \mathbb{N}}$ converges to $f(x)$. Then, by Corollary~\ref{dini-conc}, the function $f$ is the pointwise limit of $\lbrace f_n  \rbrace_{n \in \mathbb{N}}$ on the whole of $N_{\mathbb{R}}$.
			\item Even if the sequence $\lbrace f_n  \rbrace_{n \in \mathbb{N}}$ converges pointwise and the sequence $\lbrace \textup{rec}(f_n)  \rbrace_{n \in \mathbb{N}}$ converges, it does not need to converge in the $\mathcal{G}$-norm. Consider the function $f_n  \colon  \mathbb{R} \rightarrow \mathbb{R}$ given by
			\begin{displaymath}
				f_n (x) \coloneqq \begin{cases} x, & x \leq n \\ n, & \textup{otherwise.} \end{cases}
			\end{displaymath}
			The sequence $\lbrace f_n  \rbrace_{n \in \mathbb{N}}$ converges pointwise to $f(x)=x$ but does not converge in the $\mathcal{G}$-norm. It is clear from
			\begin{displaymath}
				\textup{rec}(f_n)(x) = \begin{cases} x, & x \leq 0 \\ 0, & \textup{otherwise.} \end{cases}
			\end{displaymath}
			The important detail is that $\lbrace \textup{rec}(f_n)  \rbrace_{n \in \mathbb{N}}$ does not converges to $\textup{rec}(f)$.
		\end{enumerate}
	\end{rem}
	\subsection{Piecewise affine functions and directional derivatives} 
	This section studies the space of rational piecewise affine functions on $N_{\mathbb{R}}$. As seen in Subsection~\ref{DVR-1}, these functions describe torus invariant Cartier divisors on toric schemes over a discrete valuation ring. We also study the convergence of sequences of rational piecewise affine functions with respect to the $\mathcal{C}$-norm; this is a crucial step in the proof of Theorem~\ref{nef-cone-US}.
	
	We start by recalling some basic concepts from discrete geometry. A convex set $C$ is said to be a \textit{polyhedron} if it is determined by the intersection of a finite number of closed half-spaces. A polyhedron $C$ is a \textit{polytope} if $C$ is a bounded set. A polytope $C$ is \textit{lattice} (resp. \textit{rational}) if it satisfies $C = \textup{Conv}(S)$ where $S \subset N$ is finite (resp. $S \subset N_{\mathbb{Q}}$). A polyhedral cone $C$ is rational if it satisfies $C = c(S)$, where  $S \subset N_{\mathbb{Q}}$ is finite. Unlike polytopes, there is no distinction between rational and lattice polyhedral cones; For any finite set $S \subset N_{\mathbb{Q}}$, there is a positive integer $r$ such that $r \cdot S \subset N$. A \textit{ray} is a cone generated by a single non-zero element $v \in N_{\mathbb{R}}$. We write $ \mathbb{R}_{\geq 0} \cdot v = c( \lbrace v \rbrace ) = \tau$. If $\tau$ is a rational ray, its \textit{ray generator} is the vector $v=v_{\tau}$ that generates the semigroup $\tau \cap N$. The set $N^{\textup{pr}}$ of \textit{primitive elements} is defined as the subset of $N$ consisting of all ray generators.  A rational cone is \textit{smooth} if a set generating $\sigma$ extends to a $\mathbb{Z}$-basis of $N$.
	\begin{dfn}
		Let $\Pi$ be a non-empty collection of convex subsets of $N_{\mathbb{R}}$. The collection $\Pi$ is called a \textit{convex subdivision} if it satisfies the following:
		\begin{enumerate}[label=(\roman*)]
			\item Every face of an element of $\Pi$ is also an element of $\Pi$.
			\item Any two elements of $\Pi$ are either disjoint or intersect at a common face.
		\end{enumerate}
		The \textit{support} $|\Pi|$ of $\Pi$ is the subset of $N_{\mathbb{R}}$ obtained by the union of all elements of $\Pi$. We say that $\Pi$ is \textit{complete} if $|\Pi|= N_{\mathbb{R}}$. We say that $\Pi$ is a \textit{polyhedral complex}  if it is finite and every element of $\Pi$ is a polyhedron. A polyhedral complex $\Pi$ is \textit{conical} (resp. \textit{strongly convex, rational, simplicial, etc.}) if every element of $\Pi$ is a cone (resp. strongly convex, rational, simplex, etc.). We denote by $\Pi(n)$ the subset of $n$-dimensional elements of $\Pi$.
	\end{dfn}
	\begin{ex}[Fans]\label{fan}
		A \textit{fan} $\Sigma$ in $N_{\mathbb{R}}$ is a strongly convex, conical, rational, polyhedral complex. The fan is \textit{smooth} if every cone is smooth. An example of a smooth complete fan is the fan $\overline{\Sigma}$ constructed as follows: Let $N = \mathbb{Z}^{d}$ and $e_1,...,e_d$ be the standard basis. Write $e_0 = - (e_1 + ... + e_d)$. Then the cones of $\overline{\Sigma}$ are the cones spanned by all proper subsets of $\lbrace e_0, ..., e_d \rbrace$.
	\end{ex}
	\begin{ex}[Refinement by a ray]\label{star}
		Let $\Sigma$ be a fan in $N_{\mathbb{R}}$. Choose any cone $\sigma \in \Pi$ of dimension at least $2$ and a minimal generating set $S$ of $\sigma$. Let $v \in \textup{ri}(\pi)$. Let $\Sigma^{\prime}(\sigma, S,v)$ be the fan given by all the cones generated by proper subsets of $S \cup \lbrace v \rbrace$ distinct from $S$. Then, define the fan
		\begin{displaymath}
			\Sigma^{\ast}(\sigma, S, v) \coloneqq (\Sigma \setminus \lbrace \sigma \rbrace ) \cup \Sigma^{\prime}(\sigma, S, v).
		\end{displaymath}
		It is a fan such that every element of $\Sigma^{\ast}(\sigma, S, v)$  is contained in or is equal to an element of $\Sigma$. If $\sigma$ is a smooth cone, $S$ is the set of ray generators of $\sigma$, and $v = \sum_{s \in S} s$, the fan $\Sigma^{\ast} (\sigma) = \Sigma^{\ast}(\sigma, S, v)$ is known as the \textit{star subdivision} of $\Sigma$ along $\sigma$ (see~Definition~3.3.13~of~\cite{CLS}).
	\end{ex}
	\begin{ex}[Barycentric subdivision]\label{barycentric}
		Let $\Delta$ be a simplex of dimension $d$. The \textit{barycenter} of $\Delta$ the point given by
		\begin{displaymath}
			b_{\Delta} \coloneqq\frac{1}{d+1} \sum_{v \in \Delta(0)} v.
		\end{displaymath}
		Regarding $\Delta$ as a simplicial complex, we may inductively define the following complex:
		\begin{enumerate}[label=(\arabic*)]
			\item If $d=0$, then $\Delta$ is a single point and the process finishes.
			\item Assume that $d>0$. Let $1\leq n \leq d$ and suppose that the $(n-1)$-dimensional faces of $\Delta$ are already divided. Let $\pi \in \Delta(n)$ and note that each facet $\tau_i$ of $\pi$, $0 \leq i \leq n$, has been divided into simplices $\lbrace \tau_{i,j} \rbrace_{j \in J_i}$. Define the $n$-dimensional simplicial complex $\beta(\pi )$ whose $n$-dimensional simplices are given by $\textup{Conv}(\lbrace b_{\pi} \rbrace \cup \tau_{i,j} )$, where $0 \leq i \leq n$ and $j \in J_i$.
			\item The \textit{barycentric subdivision} of $\Delta$ is defined as the simplicial complex $\beta (\Delta)$.
		\end{enumerate}
		This construction generalizes to any simplicial complex. Let $\Pi$ be a $d$-dimensional simplicial complex, then its \textit{barycentric subdivision} $\beta (\Pi)$ is given by
		\begin{displaymath}
			\beta( \Pi ) \coloneqq \bigcup_{\pi  \in \Pi (d)} \beta(\pi).
		\end{displaymath}
	\end{ex}
	Note that the identity map on $N_{\mathbb{R}}$ is compatible with the polyhedral complex structures in the previous Examples~\ref{star},~\ref{barycentric}. This is generalized in the following definition.
	\begin{dfn}\label{compatible}
		Let $N_1, \, N_2$ be lattices, $\phi \colon (N_{1})_{\mathbb{R}} \rightarrow (N_{2})_{\mathbb{R}}$ a $\mathbb{R}$-linear map and $\Pi_i$ a convex subdivision in $(N_i)_{\mathbb{R}}$. The map $\phi$ is said to be \textit{compatible} with $\Pi_1$ and $\Pi_2$ if for every element $\pi_1 \in \Pi_1$ there is an element $\pi_2 \in \Pi_2$ such that $\phi (\pi_1) \subset \pi_2$. If the subdivisions are rational, we also require $\phi$ to be the extension of scalars of a $\mathbb{Z}$-linear map $\phi  \colon  N_{1} \rightarrow N_{2}$.
	\end{dfn}
	\begin{dfn}
		Let $\Pi_1$ and $\Pi_2$ be convex subdivisions in $N_{\mathbb{R}}$. We say that $\Pi_1$ is a \textit{refinement} of $\Pi_2$,  denoted $\Pi_1 \geq \Pi_2$, if the identity map is compatible with $\Pi_1$ and $\Pi_2$.
	\end{dfn}
	\begin{rem}\label{refin}
		Note that any two polyhedral complexes $\Pi_1$ and $\Pi_2$ with the same support admit a common refinement. Indeed, define $\Pi_1 \cdot \Pi_2 \coloneqq \lbrace \pi_1 \cap \pi_2 \, \vert \, \pi_i \in \Pi_i \rbrace$. Since each element of the $\Pi_i$'s is an intersection of closed half-spaces, so is every element of $\Pi_1 \cdot \Pi_2$. If $\pi_1 \in \Pi_1$ and $\pi_2 \in \Pi_2$ have non-empty intersection, then any face of $\pi_1 \cap \pi_2$ is given as the intersection
		\begin{displaymath}
			(\pi_1 \cap \pi_2) \cap H = ( \pi_1 \cap H ) \cap ( \pi_2 \cap H )
		\end{displaymath}
		where $H$ is a closed half-space. Note that $\pi_i \cap H \in \Pi_i$, as it is a face of $\pi_i$. To see the support condition, note that
		\begin{displaymath}
			| \Pi_1 \cdot \Pi_2 | = \bigcup_{\pi_i \in \Pi_i}  \pi_1 \cap \pi_2 = \left(  \bigcup_{\pi_1 \in \Pi_1} \pi_1  \right) \cap  \left(  \bigcup_{\pi_2 \in \Pi_2} \pi_2  \right) = |\Pi_1| \cap |\Pi_2|.
		\end{displaymath}
		Moreover, if $\Pi_1$ and $\Pi_2$ are both fans, so is $\Pi_1 \cdot \Pi_2$. In particular, the set of all polyhedral complexes (resp. fans) with common support, ordered by refinement, is a directed set. 
	\end{rem}
	Now, we consider functions compatible with the structure of a polyhedral complex.
	\begin{dfn}\label{piecewise-affine}
		Let $C \subset N_{\mathbb{R}}$ be a polyhedron. A function $f \colon  C \rightarrow \mathbb{R}$ is \textit{piecewise affine} (resp. \textit{piecewise linear}) if there is a finite cover $\lbrace C_i \rbrace_{i \in I}$ of $C$ by closed subsets such that the restrictions $f|_{C_i}$ are affine (resp. linear). Then:
		\begin{enumerate}[label=(\roman*)]
			\item If $\Pi$ is a polyhedral complex on $C$, we say that $f$ is piecewise affine (resp. \textit{piecewise linear}) on $\Pi$ if it is affine (resp. linear) on each $\pi \in \Pi$.
			\item If $C$ is rational, a piecewise affine function $f$ is \textit{rational} if $f(N_{\mathbb{Q}} \cap C) \subset \mathbb{Q}$.
			\item If $\Pi$ is a rational polyhedral complex, we denote by $\mathcal{PA}(\Pi, \mathbb{Q})$ the $\mathbb{Q}$-vector space of rational piecewise affine functions on $\Pi$.
			\item We denote by $\mathcal{PA}(N_{\mathbb{R}}, \mathbb{Q})$ the space of rational piecewise functions $f\colon N_{\mathbb{R}} \rightarrow \mathbb{R}$. That is, there exists a polyhedral complex $\Pi$ with support $N_{\mathbb{R}}$ such that $f$ is piecewise affine on $\Pi$.
			\item We write $\mathcal{PA}^{+}(\Pi, \mathbb{Q})$ and $\mathcal{PA}^{+}(N_{\mathbb{R}}, \mathbb{Q})$ for the corresponding subsets of concave functions.
		\end{enumerate}
	\end{dfn}
	\begin{ex}
		Let  $\overline{\Sigma}$ be the fan from Example~\ref{fan}. Each proper subset $S \subset \lbrace e_0, ..., e_d \rbrace$ is a basis for its span. Hence, any linear function $f$ on $c(S)$ is uniquely determined by its values on $S$. More generally, a piecewise affine function on $\overline{\Sigma}$ is uniquely determined by its values on the set $ \lbrace 0, e_0, ..., e_d \rbrace$. For instance, the function $\overline{\Psi}  \colon  N_{\mathbb{R}} \rightarrow \mathbb{R}$ determined by $\overline{\Psi} (e_i) = -1$ is a rational piecewise linear on $\overline{\Sigma}$, while the function $\overline{\gamma} \colon  N_{\mathbb{R}}\rightarrow \mathbb{R}$ given by $\overline{\gamma} \coloneqq \overline{\Psi} - 1$ is rational piecewise affine on $\overline{\Sigma}$. This is an example of a virtual support function, introduced below.
	\end{ex}
	\begin{dfn}\label{def-sup-fun}
		Let $\Sigma$ be a fan in $N_{\mathbb{R}}$. We introduce the following notation:
		\begin{enumerate}[label = (\roman*)]
			\item A \textit{virtual support function} on the fan $\Sigma$ is a piecewise linear function $f \colon  |\Sigma| \rightarrow \mathbb{R}$ on $\Sigma$. Then, denote by $\mathcal{SF}(\Sigma, \mathbb{R})$ the space of virtual support functions on $\Sigma$ and by $\mathcal{SF}(\Sigma, \mathbb{Q})$ its subspace of rational functions.
			\item We denote by $\mathcal{SF}(N_{\mathbb{R}}, \mathbb{R})$ the space of all virtual support functions $\Psi \colon N_{\mathbb{R}} \rightarrow \mathbb{R}$. That is, there exists a complete fan $\Sigma$ in $N_{\mathbb{R}}$ such that $\Psi$ is a virtual support function on $\Sigma$. We write $\mathcal{SF}(N_{\mathbb{R}}, \mathbb{Q})$ for its subspace of rational support functions.
			\item For $K = \mathbb{R} \textup{ or } \mathbb{Q}$, denote by $\mathcal{SF}^{+}(\Sigma, K)$ and $\mathcal{SF}^{+}(N_{\mathbb{R}}, K)$ the respective subsets consisting of concave functions.
		\end{enumerate}
	\end{dfn}
	\begin{rem}\label{direct-lim}
		Let $K = \mathbb{R} \textup{ or } \mathbb{Q}$. By Remark~\ref{refin}, the poset consisting of complete fans $\Sigma$ in $N_{\mathbb{R}}$ ordered by refinement is a directed set. Note that $\Sigma' \geq \Sigma$ implies $\mathcal{SF}(\Sigma, K) \subset \mathcal{SF}(\Sigma', K)$. These relations yield an isomorphism
		\begin{displaymath}
			\mathcal{SF}(N_{\mathbb{R}}, K) = \bigcup_{\Sigma \textup{ complete}}  \mathcal{SF}(\Sigma, K) \cong  \varinjlim_{\Sigma \textup{ complete fan}} \mathcal{SF}(\Sigma, K).
		\end{displaymath}
		Similarly, the poset consisting of complete rational polyhedral complexes $\Pi$ in $N_{\mathbb{R}}$ ordered by refinement is a directed set, and we have an isomorphism
		\begin{displaymath}
			\mathcal{PA}(N_{\mathbb{R}}, \mathbb{Q}) \cong   \varinjlim_{\Pi \textup{ complete}} \mathcal{PA}(\Pi, \mathbb{Q}).
		\end{displaymath}
		Clearly, $\mathcal{SF}(N_{\mathbb{R}}, \mathbb{Q})$ is a $\mathbb{Q}$-vector subspace of $\mathcal{PA}(N_{\mathbb{R}}, \mathbb{Q})$. Moreover, the recession map is a retraction of the space $\mathcal{PA}(N_{\mathbb{R}}, \mathbb{Q})$ onto $\mathcal{SF}(N_{\mathbb{R}}, \mathbb{Q})$. Indeed, the recession function $\textup{rec}(f)$ of a rational piecewise affine function $f$ on $N_{\mathbb{R}}$ is a rational support function on some fan $\Sigma$ in $N_{\mathbb{R}}$ (see Proposition~2.6.3~of~\cite{BPS}).
	\end{rem}
	Now, we study the convergence in the $\mathcal{C}$-norm of sequences $\lbrace f_n \rbrace_{n \in \mathbb{N}}$ in the space $\mathcal{PA}(N_{\mathbb{R}}, \mathbb{Q})$. We will copy the strategy in Remark~\ref{comp-conical} and Lemma~\ref{SL-banach}, that is, to relate the convergence of the sequence $\lbrace f_n \rbrace_{n \in \mathbb{N}}$ with uniform convergence of a sequence $\lbrace \overline{\eta}(f_n) \rbrace_{n \in \mathbb{N}}$ of continuous functions over a compact set. The notion of a one-sided directional derivative arises naturally in the definition of the functions $\overline{\eta}(f_n)$. To motivate this, we have the following remark.
	\begin{rem}\label{C-conv}
		Assume that $\lbrace f_n \rbrace_{n \in \mathbb{N}}$ is a sequence in $\mathcal{G}(N_{\mathbb{R}})$ which converges to a function $f$ with respect to the $\mathcal{C}$-norm (we will later specialize to the rational piecewise affine case). Comparing Definitions~\ref{C-norm}~and~\ref{G-def}, it is clear that the sequence $\lbrace f_n \rbrace_{n \in \mathbb{N}}$ converges to $f$ with respect to the $\mathcal{G}$-norm. By completeness, we get that $f \in \mathcal{G}(N_{\mathbb{R}})$. Convergence in the $\mathcal{C}$-norm means that for each $\varepsilon >0$ there is a natural number $n_{\varepsilon}$ such that for all $n\geq n_{\varepsilon}$ and $x \in N_{\mathbb{R}}$ we have
		\begin{displaymath}
			|f(x) - f_n (x)| \leq \varepsilon \cdot \| x \|.
		\end{displaymath}
		In particular, $f_n (0) = f(0)$. Now consider the homeomorphism $N_{\mathbb{R}} \setminus \lbrace 0 \rbrace \rightarrow (0, \infty) \times \mathcal{S}^{d-1}$ given by writing $x$ in polar coordinates, namely, $x \mapsto ( \|x \| , \hat{x})$ where $\hat{x} \coloneqq x/ \| x \|$. Then, define the sequence $\lbrace g_n \rbrace_{n \in \mathbb{N}}$ in the Banach space $C^{0}_{\textup{bd}} ((0, \infty) \times \mathcal{S}^{d-1})$ of bounded continuous real-valued functions on the open cylinder $(0, \infty) \times \mathcal{S}^{d-1}$, where
		\begin{displaymath}
			g_n ( \| x \| , \hat{x}) \coloneqq (f(x) - f_n (x))/ \| x \| .
		\end{displaymath}
		We conclude that the sequence $\lbrace f_n \rbrace_{n \in \mathbb{N}}$ converges to $f$ in the $\mathcal{C}$-norm if and only if the sequence $\lbrace f_n (0 ) \rbrace_{n \in \mathbb{N}}$ is eventually the constant $f(0)$ and the sequence $\lbrace g_n \rbrace_{n \in \mathbb{N}}$ converges uniformly to $0$.
	\end{rem}
	Motivated by the previous discussion, for each function $f \in \mathcal{G}(N_{\mathbb{R}})$ we define a continuous function $\eta(f) \colon (0, \infty) \times \mathcal{S}^{d-1} \rightarrow \mathbb{R}$. It is given by the assignment
	\begin{displaymath}
		\eta(f)( \| x \|, \hat{x}) \coloneqq (f(x) - f(0))/\| x \|,
	\end{displaymath}
	where $(\|x\| , \hat{x})$ is the representation in polar coordinates of the non-zero element $x \in N_{\mathbb{R}}$. Now, we restrict to the $\mathbb{Q}$-vector space $\mathcal{PA}(N_{\mathbb{R}}, \mathbb{Q})$ of rational piecewise affine functions on $N_{\mathbb{R}}$ and describe the properties of the assignment $f  \mapsto \eta(f)$.
	\begin{lem}\label{eta}
		With the previous notation, the assignment $f \mapsto \eta(f)$ induces a $\mathbb{Q}$-linear map $\eta \colon \mathcal{PA} (N_{\mathbb{R}}, \mathbb{Q}) \rightarrow C^{0}_{\textup{bd}} ((0, \infty) \times \mathcal{S}^{d-1})$.
		Moreover, a sequence $\lbrace f_n \rbrace_{n \in \mathbb{N}}$ in $\mathcal{PA} (N_{\mathbb{R}}, \mathbb{Q})$ converges in the $\mathcal{C}$-norm to a function $f \in \mathcal{G}(N_{\mathbb{R}})$ if and only if the sequence $\lbrace f_n (0) \rbrace_{n \in \mathbb{N}}$ is eventually the constant $f(0)$ and the sequence $\lbrace \eta(f)_n \rbrace_{n \in \mathbb{N}}$ converges uniformly to the continuous function $\eta(f)$.
	\end{lem}
	\begin{proof}
		It is clear from the definition that the assignment $f \mapsto \eta(f)$ is $\mathbb{Q}$-linear, and the function $\eta(f)$ is continuous. We must show that $\eta(f)$ is bounded. This is equivalent to the existence of a constant $C$ satisfying $|f(x) - f(0)| \leq C \cdot \|x \|$ for all $x \in N_{\mathbb{R}}$. That is, the function $h = f - f(0)$ has finite $\mathcal{C}$-norm. This follows from the following facts:
		\begin{enumerate}
			\item The function $h$ is piecewise linear when restricted to a sufficiently small ball $\overline{\textup{B}}(0,r)$ centered at $0$. Then, there exists a constant $C_1$ such that $|h(x)| \leq C_1 \cdot \| x \|$ for all $x \in \overline{\textup{B}}(0,r)$.
			\item The function $f$ is an element of the Banach space $\mathcal{G}(N_{\mathbb{R}})$. 
			By definition of the $\mathcal{G}$-norm, for every $x \in N_{\mathbb{R}}$ we have $|h(x)| \leq \| h \|_{\mathcal{G}} \cdot (1 + \| x \|)$. Given $r > 0$, for each $x $ of norm larger than $r$, we have $|h(x)| \leq (1 + 1/r ) \| h \|_{\mathcal{G}} \, \| x \|$.
		\end{enumerate}
		Clearly, the constant $C = \max \lbrace C_1, (1 + 1/r ) \| h \|_{\mathcal{G}} \rbrace$ satisfies the desired condition. The continuity property follows immediately from the discussion in Remark~\ref{C-conv}.
	\end{proof}
	\begin{rem}\label{eta-2}
		Unlike Remark~\ref{comp-conical} and Lemma~\ref{SL-banach}, the target of the map $\eta$ is not a space of continuous functions on a compact set. Then, our next objective is to show that, for each function $f \in \mathcal{PA}(N_{\mathbb{R}}, \mathbb{Q})$, the function $\eta (f)$ extends to a continuous function $\overline{\eta}(f)$ on the compactification $[0, \infty] \times \mathcal{S}^{d-1}$ of $(0, \infty) \times \mathcal{S}^{d-1}$. Suppose $\overline{\eta}(f)$ is such extension. Fixing a vector $\hat{x} \in \mathcal{S}^{d-1}$, continuity of $\overline{\eta}(f)$ implies the identities
		\begin{align*}
			\overline{\eta}(f)(\infty ,\hat{x}) =& \lim_{\lambda \rightarrow \infty} \eta(f)(\lambda, \hat{x}) = \lim_{\lambda \rightarrow \infty}  (f(\lambda \, \hat{x}) - f(0))/ \lambda = \textup{rec}(f)(\hat{x})\\
			\overline{\eta}(f)(0 ,\hat{x}) =& \lim_{\lambda \rightarrow + 0} \eta(f)(\lambda, \hat{x}) = \lim_{\lambda \rightarrow + 0}  (f(\lambda \, \hat{x}) - f(0))/ \lambda =\textup{D}_{+}(f, \hat{x})(0). 
		\end{align*}
		Here, $\lambda \rightarrow + 0$ means that $\lambda$ is positive and converges to $0$, and $\textup{D}_{+}(f, \hat{x})(0)$ is the one-sided directional derivative of $f$ at $0$ in the direction of $\hat{x}$. We recall its definition and show that the functions $\eta(f)$ extend continuously to the compactification $[0, \infty] \times \mathcal{S}^{d-1}$.
	\end{rem}
	\begin{dfn}
		Let $f \colon N_{\mathbb{R}} \rightarrow \mathbb{R}$ be a function. Given $x , y \in N_{\mathbb{R}}$ with $y \neq 0$, the \textit{one-sided directional derivative} of $f$ at $x$ in the direction of $y$ is defined as the limit
		\begin{displaymath}
			\textup{D}_{+}(f,y)(x) \coloneqq \lim_{\lambda \rightarrow + 0}  (f(\lambda \, y + x) - f(x))/ \lambda.
		\end{displaymath}
	\end{dfn}
	\begin{prop}\label{bar-eta}
		For each rational piecewise affine function $f \in \mathcal{PA}(N_{\mathbb{R}}, \mathbb{Q})$, the function $\eta(f) \colon (0, \infty) \times \mathcal{S}^{d-1} \rightarrow \mathbb{R}$ extends to the continuous function $\overline{\eta}(f) \colon [0, \infty] \times \mathcal{S}^{d-1} \rightarrow \mathbb{R}$ given by
		\begin{displaymath}
			\overline{\eta}(f) (\| x \|, \hat{x}) \coloneqq \begin{cases} \textup{D}_{+}(f,\hat{x})(0), & \| x \| = 0 \\ (f(x) - f(0))/\| x \|, & \|x \| \in (0,\infty) \\ \textup{rec}(f)(\hat{x}), & \|x \| = \infty. \end{cases}
		\end{displaymath}
		The assignment $f \mapsto \overline{\eta}(f)$ induces a $\mathbb{Q}$-linear map $\overline{\eta} \colon \mathcal{PA}(N_{\mathbb{R}}, \mathbb{Q}) \rightarrow C^0 ([0, \infty] \times \mathcal{S}^{d-1})$. Moreover, a sequence $\lbrace f_n \rbrace_{n \in \mathbb{N}}$ in $\mathcal{PA}(N_{\mathbb{R}}, \mathbb{Q})$ converges in the $\mathcal{C}$-norm if and only if the sequence $\lbrace f_n(0) \rbrace_{n \in \mathbb{N}}$ is eventually constant and $\lbrace \overline{\eta}(f_n) \rbrace_{n \in \mathbb{N}}$ converges uniformly.
	\end{prop}
	\begin{proof}
		Let $ f \in \mathcal{PA}(N_{\mathbb{R}}, \mathbb{Q})$. We must show that the function $\overline{\eta}(f)$ is continuous. By Lemma~\ref{eta}, the function $\overline{\eta}(f)$ is continuous on $(0,\infty) \times \mathcal{S}^{d-1}$. Now, we show continuity on $\lbrace \infty \rbrace \times \mathcal{S}^{d-1}$. Fix an element $\hat{x}_0 \in \mathcal{S}^{d-1}$. Since $\mathcal{PA}(N_{\mathbb{R}}, \mathbb{Q})$ is a subspace of $\mathcal{G}(N_{\mathbb{R}})$, the function $f$ can be written as $f = f_0 + \textup{rec}(f)$, where $f_0$ is sublinear and $\textup{rec}(f)$ is the recession function of $f$, which is a continuous conical function. Then, for each $\varepsilon>0$ there exists $\delta_{\varepsilon} >0$ such that 
		\begin{displaymath}
			|\textup{rec}(f)(\hat{x}) - \textup{rec}(f)(\hat{x}_0) | \leq \varepsilon /2
		\end{displaymath}
		for every $\hat{x} \in \mathcal{S}^{d-1}$ satisfying $\| \hat{x} - \hat{x}_0 \| \leq \delta_{\varepsilon}$. By sublinearity of $f_0 = f - \textup{rec}(f)$, there exists $R_{\varepsilon}>1$ such that for all $\|x \| \geq R_{\varepsilon}$ we have
		\begin{displaymath}
			|f(x) - \textup{rec}(f)(x)| \leq \varepsilon/4 \cdot \|x \| .
		\end{displaymath}
		Enlarging $R_{\varepsilon}$ if necessary, we can assume that $|f(0)|/R_{\varepsilon} \leq \varepsilon /4$. Then, for all $\|x \| \geq R_{\varepsilon}$, the triangle inequality gives
		\begin{displaymath}
			| \overline{\eta}(f)(\|x \|, \hat{x}) -  \textup{rec}(f)(\hat{x}) | = |f(x) - f(0) - \textup{rec}(f)(x)|/\|x \| \leq  (|f(x) - \textup{rec}(f)(x)| + |f(0)|)/\|x \| \leq \varepsilon/2.
		\end{displaymath}
		A second application of the triangle inequality shows
		\begin{displaymath}
			| \overline{\eta}(f)(\|x \|, \hat{x}) -  \textup{rec}(f)(\hat{x}_0) |  \leq | \overline{\eta}(f)(\|x \|, \hat{x}) -  \textup{rec}(f)(\hat{x}) | + |\textup{rec}(f)(\hat{x}) - \textup{rec}(f)(\hat{x}_0) | \leq \varepsilon
		\end{displaymath}
		for all $(\|x \|, \hat{x})$ such that $\|x \| \geq R_{\varepsilon}$ and $\| \hat{x} - \hat{x}_0 \| \leq \delta_{\varepsilon}$. Therefore, the function $\overline{\eta}(f)$ is continuous at $(\infty, \hat{x}_0)$.
		
		Now, we show continuity on $\lbrace 0 \rbrace\times \mathcal{S}^{d-1}$. Following the proof of Lemma~\ref{eta}, the function $h = f-f(0)$ is piecewise linear on a small ball $ \overline{\textup{B}}(0,\delta)$. In particular, the restriction of $h$ to $ \overline{\textup{B}}(0,\delta)$ is conical. This means that for every $\lambda >0$ and $x \in  \overline{\textup{B}}(0,\delta) \setminus \lbrace 0 \rbrace$ such that $\lambda \cdot x \in \overline{\textup{B}}(0,\delta)$ we have $h(\lambda \cdot x) = \lambda \, h(x)$. Therefore,
		\begin{displaymath}
			\eta (f) ( \lambda \cdot \|x \|, \hat{x} ) = h(\lambda \cdot x)/\| \lambda \cdot x \| = h(x)/ \| x \| = \eta (f) (\|x\|, \hat{x}).
		\end{displaymath}
		This means that the function $\eta (f)$ is constant on every set of the form $\left(0, \delta \right] \times \lbrace \hat{x} \rbrace$. It follows immediately from the definition of the one-sided directional derivative that
		\begin{displaymath}
			\overline{\eta}(f)(0,\hat{x}) = \textup{D}_{+}(f,\hat{x})(0) = \eta (f) ( \delta , \hat{x}).
		\end{displaymath}
		Since $\eta(f)$ is continuous on $\left(0, \delta \right] \times \mathcal{S}^{d-1}$, the function $\overline{\eta}(f)$ is continuous.
		
		It remains to show the continuity property of the $\mathbb{Q}$-linear map $\overline{\eta}$. By Lemma~\ref{eta}, if $\lbrace f_n \rbrace_{n \in \mathbb{N}}$ converges in the $\mathcal{C}$-norm to a function $f$, then the sequence $\lbrace f_n(0) \rbrace_{n \in \mathbb{N}}$ is eventually constant and the sequence $\lbrace \eta(f)_n \rbrace_{n \in \mathbb{N}}$ converges uniformly on $(0,\infty) \times \mathcal{S}^{d-1}$ to a function $\eta(f)$. Since $(0,\infty) \times \mathcal{S}^{d-1}$ is dense in $[0, \infty] \times \mathcal{S}^{d-1}$ and each $\eta(f_n)$ extends continuously to $\overline{\eta}(f_n)$, the sequence $\lbrace \overline{\eta}(f_n) \rbrace_{n \in \mathbb{N}}$ is Cauchy in the uniform norm. Completeness of $C^0 ([0, \infty] \times \mathcal{S}^{d-1})$ implies the uniform convergence. Conversely, assume that the sequence $\lbrace f_n (0) \rbrace_{n \in \mathbb{N}}$ is eventually the  constant $c$ and $\lbrace \overline{\eta}(f_n) \rbrace_{n \in \mathbb{N}}$ converges uniformly. Denote by $g \colon [0, \infty] \times \mathcal{S}^{d-1} \rightarrow \mathbb{R}$ the function given by the limit of $\lbrace \overline{\eta}(f_n) \rbrace_{n \in \mathbb{N}}$. Then, define the function $f \colon N_{\mathbb{R}} \rightarrow \mathbb{R}$ given by
		\begin{displaymath}
			f(x) \coloneqq \|x \| \cdot g(\| x \|, \hat{x})  + c.
		\end{displaymath}
		Following Remark~\ref{eta-2}, the fact that $g$ is uniformly continuous on $[0, \infty] \times \mathcal{S}^{d-1}$  shows that $f \in \mathcal{G}(N_{\mathbb{R}})$. By construction, $\eta(f)$ coincides with the restriction of $g$ to $(0, \infty) \times \mathcal{S}^{d-1}$. Finally, Lemma~\ref{eta} implies that the sequence $\lbrace f_n \rbrace_{n \in \mathbb{N}}$ converges in the $\mathcal{C}$-norm to $f$.
	\end{proof}
	\begin{rem}\label{dir-der-0}
		Let $f \in \mathcal{G}(N_{\mathbb{R}})$ and suppose that for each $\hat{x} \in \mathcal{S}^{d-1}$ the one-sided directional derivative $\textup{D}_{+}(f, \hat{x} )(0)$ exists and is continuous as a function of $\hat{x}$. Then, we may define the function $\overline{\eta}(f) \colon [0, \infty] \times \mathcal{S}^{d-1} \rightarrow \mathbb{R}$ as in the previous result. That is,
		\begin{displaymath}
			\overline{\eta}(f) (\| x \|, \hat{x}) \coloneqq \begin{cases} \textup{D}_{+}(f,\hat{x})(0), & \| x \| = 0 \\ (f(x) - f(0))/\| x \|, & \|x \| \in (0,\infty) \\ \textup{rec}(f)(\hat{x}), & \|x \| = \infty. \end{cases}
		\end{displaymath}
		Suppose further that $(f(x) - f(0))/\| x \|$ converges to $\textup{D}_{+}(f,\hat{x})(0)$ uniformly on $\hat{x} \in \mathcal{S}^{d-1}$. Concretely, for every $\varepsilon >0$ there is $\delta_{\varepsilon} >0$ such that
		\begin{displaymath}
			|\textup{D}_{+}(f,\hat{x})(0) - (f(x) - f(0))/\| x \| | \leq \varepsilon
		\end{displaymath}
		for all $x$ such that $\| x \| \leq \delta_{\varepsilon}$. Copying the previous proof, we can show that the function $\overline{\eta}(f)$ is continuous on the compact set $ [0, \infty] \times \mathcal{S}^{d-1}$. As we will see now, every globally Lipschitz concave function $f \in \mathcal{G}^{+} (N_{\mathbb{R}})$ satisfies these assumptions.
	\end{rem}
	Theorem~23.1~of~\cite{Roc} gives the following lemma, which states that the one-sided directional derivative of a concave function $f \colon N_{\mathbb{R}} \rightarrow \mathbb{R}$ always exists.
	\begin{lem}\label{dir-der-1}
		Let $f \colon N_{\mathbb{R}} \rightarrow \mathbb{R}$ be a concave function. Then, for each $x, y \in N_{\mathbb{R}}$, the difference quotient $(f(\lambda \, y + x) - f(x))/ \lambda$ is a non-increasing function of $\lambda >0$. In particular, the one-sided directional derivative $\textup{D}_{+}(f,y)(x)$ exists and is a conical concave function of $y$.
	\end{lem}
	By Remark~\ref{conical-polyhedron}, for each $x_0 \in N_{\mathbb{R}}$, the conical concave function $\textup{D}_{+}(f,y)(x_0)$ is the support function of a compact convex set, which can be described in terms of the \textit{sup-differential}.
	\begin{dfn}\label{sup-diff}
		Let $f \colon N_{\mathbb{R}} \rightarrow \mathbb{R}$ be a concave function and $x \in N_{\mathbb{R}}$. The \textit{sup-differential} of $f$ at $x$ is the set $\partial f (x) \coloneqq \lbrace z \in M_{\mathbb{R}} \, \vert \, \langle z, y-x \rangle \geq f(y) - f(x) \textup{ for all } y \in N_{\mathbb{R}} \rbrace$. We say that $f$ is \textit{sup-differentiable} at $x$ if $\partial f (x) \neq \emptyset$. The \textit{domain} $\textup{dom}(\partial f )$ of the sup-differential is the set of points where $f$ is sup-differentiable. If $E \subset N_{\mathbb{R}}$, then
		\begin{displaymath}
			\partial f (E) \coloneqq \bigcup_{x \in E}  \partial f (x) \subset M_{\mathbb{R}}.
		\end{displaymath}
	\end{dfn}
	\begin{lem}\label{dir-der-2}
		Let $f\colon N_{\mathbb{R}} \rightarrow \mathbb{R}$ be a concave function. Then, for each $x_0 \in N_{\mathbb{R}}$, the function $f$ is sup-differentiable at $x_0$ and the conical concave function $\textup{D}_{+}(f,y)(x_0)$ is the support function of the bounded convex set $\partial f (x_0)$. 
	\end{lem}
	\begin{proof}
		Follows from Theorem~23.4~of~\cite{Roc}.
	\end{proof}
	Finally, we give a version of Proposition~\ref{bar-eta} for globally Lipschitz concave functions.
	\begin{prop}\label{eta-3}
		For each globally Lipschitz concave function $f \in \mathcal{G}^{+}(N_{\mathbb{R}})$, the function $\overline{\eta}(f) \colon [0, \infty] \times \mathcal{S}^{d-1} \rightarrow \mathbb{R}$ given by
		\begin{displaymath}
			\overline{\eta}(f) (\| x \|, \hat{x}) \coloneqq \begin{cases} \textup{D}_{+}(f,\hat{x})(0), & \| x \| = 0 \\ (f(x) - f(0))/\| x \|, & \|x \| \in (0,\infty) \\ \textup{rec}(f)(\hat{x}), & \|x \| = \infty \end{cases}
		\end{displaymath}
		is continuous. Moreover, an increasing sequence $\lbrace f_n \rbrace_{n \mathbb{N}}$ in $\mathcal{G}^{+}(N_{\mathbb{R}})$ converges to $f$ in the $\mathcal{C}$-norm if and only if the sequence $\lbrace f_n (0) \rbrace_{n \in \mathbb{N}}$ is eventually the constant $f(0)$ and the sequence $\lbrace \overline{\eta}(f_n) \rbrace_{n \in \mathbb{N}}$ converges pointwise to $\overline{\eta}(f)$.
	\end{prop}
	\begin{proof}
		By Lemma~\ref{dir-der-1}, the function $\textup{D}_{+}(f,\hat{x})(0)$ on the sphere $\mathcal{S}^{d-1}$ is continuous. The same lemma implies that $\overline{\eta}(f) (\| x \|, \hat{x})$ regarded as a function on the sphere $\mathcal{S}^{d-1}$ is increasing as $\| x \|$ decreases. Since the sphere $\mathcal{S}^{d-1}$ is compact, Dini's theorem implies that $(f(x) - f(0))/\| x \|$ converges to $\textup{D}_{+}(f,\hat{x})(0)$ uniformly on $\hat{x} \in \mathcal{S}^{d-1}$ (in the sense of Remark~\ref{dir-der-0}). Thus, we can copy the proof of Proposition~\ref{eta-2} to show that $\overline{\eta}(f)$ is continuous.
		
		The argument in the proof of Proposition~\ref{eta-2} shows that the sequence $\lbrace f_n \rbrace_{n \mathbb{N}}$ converges in the $\mathcal{C}$-norm if and only if the sequence $\lbrace f_n (0) \rbrace_{n \in \mathbb{N}}$ is eventually the constant $f(0)$ and the sequence $\lbrace \overline{\eta}(f_n) \rbrace_{n \in \mathbb{N}}$ converges uniformly to $\overline{\eta}(f)$. Since the sequence $\lbrace f_n \rbrace_{n \mathbb{N}}$ is increasing, so is  $\lbrace \overline{\eta}(f_n) \rbrace_{n \in \mathbb{N}}$. Again by Dini's theorem, if the increasing sequence $\lbrace \overline{\eta}(f_n) \rbrace_{n \in \mathbb{N}}$ converges pointwise to the continuous function $\overline{\eta}(f)$ on the compact set $[0, \infty] \times \mathcal{S}^{d-1}$, it converges uniformly. The result follows.
	\end{proof}
	\subsection{Density theorems}
	In this final section, we will prove a density result for each of these spaces: The space $\mathcal{C}(N_{\mathbb{R}})$ of continuous conical functions, the cone $\mathcal{C}^{+}(N_{\mathbb{R}})$ of conical concave functions, and the cone $\mathcal{G}^{+}(N_{\mathbb{R}})$ of globally Lipschitz concave functions. In each case, the dense subset satisfies a rationality property. Finally, we compute the closure in the $\mathcal{C}$-norm of the cone $\mathcal{PA}^{+}(N_{\mathbb{R}}, \mathbb{Q})$ of concave rational piecewise affine functions.
	\begin{rem}\label{conventions}
		Recall that all norms on $N_{\mathbb{R}}$ are equivalent. Therefore, our results are independent of the choice of norm. The same is true for the lattice structure $N$. Then, without loss of generality, we may pick an isomorphism of lattices $N \cong \mathbb{Z}^d \subset \mathbb{R}^d$ and choose the norm $\| \cdot \|$ given by
		\begin{displaymath}
			\|(x_1,...,x_d) \| \coloneqq \max_{1 \leq i \leq d} | x_i | .
		\end{displaymath}
		This choice has the following advantages:
		\begin{enumerate}[label=(\arabic*)]
			\item For every rational number $R>0$, the closed ball $\overline{\textup{B}}(0,R)$ is equal to the hypercube $\left[ -R,R \right]^d$. Therefore, it is a rational polyhedral convex set.
			\item If $x \in N_{\mathbb{Q}}$, the norm $\|x \|$ is a rational number.
			\item Let $\Sigma_{0}^{\prime}$ be the fan whose $d$-dimensional cones are of the form $c(F)$, where $F$ runs over all facets of the unit hypercube. Let $\Sigma$ be any simplicial fan refining $\Sigma_{0}^{\prime}$. For each $\sigma \in \Sigma(d)$, the intersection $\hat{\sigma} = \sigma \cap \mathcal{S}^{d-1}$ is contained in a facet of the unit hypercube. Thus, the set $\hat{\sigma}$ is a rational $(d-1)$-simplex.
			\item By doing successive refinements of $\Sigma$ by rays (Example~\ref{star}), we can obtain fans such that the diameters of the sets $\hat{\sigma}$ are arbitrarily small.
		\end{enumerate}
		In the following, we will use these advantages to construct sequences of support functions.
	\end{rem}
	\paragraph*{\textbf{Density in} $\mathcal{C}(N_{\mathbb{R}})$.} The following theorem shows that we can approximate any continuous conical function by rational support functions.
	\begin{thm}\label{density-1}
		The space $\mathcal{SF}(N_{\mathbb{R}}, \mathbb{Q})$ of rational support functions on $N_{\mathbb{R}}$ is dense in the space $\mathcal{C}(N_{\mathbb{R}})$ of continuous conical functions, equipped with the $\mathcal{C}$-norm.
	\end{thm}
	\begin{proof}
		Let $f \in \mathcal{C}(N_{\mathbb{R}})$. We will construct a sequence $\lbrace (\Sigma_n , f_n ) \rbrace_{n \in \mathbb{N}}$ consisting of complete fans $\Sigma_n$ in $N_{\mathbb{R}}$ and functions $f_n \in \mathcal{SF}(\Sigma, \mathbb{Q})$ such that $f_n$ converges to $f$ with respect to the $\| \cdot \|_{\mathcal{C}}$-norm. We first construct the fans. Since $f$ is continuous on the compact set $S^{d-1}$, it is uniformly continuous. For each $n>0$ there is a $\delta_n > 0$ such that for every pair $x,y \in \mathcal{S}^{d-1}$ satisfying $\| x-y \| \leq \delta_n$, we have $|f(x) - f(y)| \leq 1/2n$. Following Remark~\ref{conventions}, we can find a sequence of simplicial fans $\lbrace \Sigma_n  \rbrace_{n \in \mathbb{N}}$ such that
		\begin{enumerate}[label=(\roman*)]
			\item For every $n$, $\Sigma_{n+1} \geq \Sigma_n$ and $\Sigma_1 \geq \Sigma_{0}^{\prime}$.
			\item For every cone $\sigma \in \Sigma_n (d)$, the intersection $\hat{\sigma} = \sigma \cap \mathcal{S}^{d-1}$ has diameter at most $\delta_n$.
		\end{enumerate}
		Now, we define the functions $f_n$. Since the fans $\Sigma_n$ are simplicial, a function $g \in \mathcal{SF}(\Sigma_n , \mathbb{Q})$ is uniquely determined by its values on the set 
		\begin{displaymath}
			L_n \coloneqq \lbrace v \in N_{\mathbb{Q}} \, \vert \, v \textup{ is a vertex of } \hat{\sigma}, \, \sigma \in \Sigma_n \rbrace.
		\end{displaymath}
		For each positive $n \in \mathbb{N}$ and vertex $v \in L_n$, choose $r_{n,v} \in \mathbb{Q}$ such that $|r_{n,v} - f(v)| < 1/2n$. Then, define $f_n \in \mathcal{SF}(\Sigma_n , \mathbb{Q})$ by declaring $f_n (v) \coloneqq r_{n,v}$ for each $v \in L_n$. We claim that the sequence $f_n$ converges to $f$ in the $\mathcal{C}$-norm. Let $x \in \mathcal{S}^{d-1}$ and choose $\sigma \in \Sigma(d)$ such that $x \in \hat{\sigma}$. Denote by $v_1,...,v_d$ the vertices of $\hat{\sigma}$ and consider a convex combination $x = \sum_{i=1}^{d} \lambda_i \cdot v_i$. Recall that $f$ is conical and $f_n$ is linear on $\sigma$, then
		\begin{displaymath}
			f(x) - f_n (x) = \sum_{i=1}^{d} \lambda_i (f(x) - f_n (v_i)) = \sum_{i=1}^{d} \lambda_i (f(x)  - f(v_i)) +  \sum_{i=1}^{d} \lambda_i (f(v_i)  - f_n (v_i)).
		\end{displaymath}
		Apply the triangle inequality and substitute to obtain the estimate
		\begin{align*}
			|f(x) - f_n (x)| & \leq \sum_{i=1}^{d}  \lambda_i |f(x) - f(v_i)| + \sum_{i=1}^{d}  \lambda_i |f(v_i) - f_n (v_i)|\\
			& \leq \sum_{i=1}^{d}  \lambda_i/2n + \sum_{i=1}^{d} \lambda_i \cdot |r_{n,v_i} - f(v_i)| \leq 1/n.
		\end{align*}
		This implies $\| f - f_n \|_{\mathcal{C}} < 1/n$.
	\end{proof}
	
	\paragraph*{\textbf{Density in} $\mathcal{C}^{+}(N_{\mathbb{R}})$.} Our goal is to show that the cone $\mathcal{SF}^{+}(N_{\mathbb{R}}, \mathbb{Q})$ of rational support functions is dense in the cone $\mathcal{C}^{+}(N_{\mathbb{R}})$ of conical concave functions. First, we remove the rationality condition on the support functions. Then, we prove the density result.
	\begin{lem}\label{density-2}
		Let $\Sigma$ be a complete fan in $N_{\mathbb{R}}$. Then, the set $\mathcal{SF}^{+}(\Sigma, \mathbb{Q})$ of concave rational support functions on $\Sigma$ is dense in the cone $\mathcal{SF}^{+}(\Sigma, \mathbb{R})$ of concave support functions on $\Sigma$, equipped with the $\mathcal{C}$-norm.
	\end{lem}
	\begin{proof}
		Let $\Psi \in \mathcal{SF}^{+}(\Sigma, \mathbb{R})$. Since $\Psi$ is linear on each cone of $\Sigma$, it is determined by its values at rays. Since the $\mathcal{C}$-norm is determined by the values of $\Psi$ on the sphere $\mathcal{S}^{d-1}$, it is enough to approximate the values of $\Psi$ at the normalized ray generators $ \lbrace v_{\rho} \rbrace_{\rho \in \Sigma (1)}$, which are rational by the conventions of Remark~\ref{conventions}. Let $\varepsilon > 0$, for each $\rho$ choose a rational number $r_{\rho, \varepsilon}$ such that
		\begin{displaymath}
			\Psi(v_{\rho}) - \varepsilon \leq r_{\rho, \varepsilon} \leq \Psi(v_{\rho}).
		\end{displaymath}
		Define the function $\Psi_{\varepsilon}  \colon  N_{\mathbb{R}} \rightarrow \mathbb{R}$ as
		\begin{displaymath}
			\Psi_{\varepsilon} (x) \coloneqq \sup \left \lbrace \left. \sum_{ \rho } \lambda_{\rho}  \, r_{\rho, \varepsilon} \right\vert \, \lambda_{\rho} \geq 0, \, x = \sum_{ \rho }  \lambda_{\rho} \, v_{\rho}  \right \rbrace.
		\end{displaymath}
		This is the unique closed concave function whose hypograph is the convex hull of the rays generated by $(v_{\rho}, r_{\rho, \varepsilon})$ and $(0, -1)$ in $N_{\mathbb{R}} \times \mathbb{R}$. Then, $\Psi_{\varepsilon} \in \mathcal{SF}^{+}(\Sigma, \mathbb{Q})$ and for each $\rho \in \Sigma (1)$ satisfies
		\begin{displaymath}
			\Psi(v_{\rho}) \geq \Psi_{\varepsilon}(v_{\rho}) \geq r_{\rho, \varepsilon} \geq \Psi(v_{\rho}) - \varepsilon,
		\end{displaymath}
		where the leftmost inequality is given by the concavity of $\Psi$ and the definition of $\Psi_{\varepsilon}$.
	\end{proof}
	\begin{thm}\label{mono-approx-2}
		The set $\mathcal{SF}^{+}(N_{\mathbb{R}}, \mathbb{Q})$ of concave rational support functions on $N_{\mathbb{R}}$ is dense in the cone $\mathcal{C}^{+}(N_{\mathbb{R}})$ of concave conical functions, equipped with the $\mathcal{C}$-norm. Moreover, for each concave conical function $\Psi \in \mathcal{C}^{+}(N_{\mathbb{R}})$, we can choose an approximating sequence $\lbrace \Psi_n  \rbrace_{n \in \mathbb{N}}$ in $\mathcal{SF}^{+}(N_{\mathbb{R}}, \mathbb{Q})$ which is increasing to $\Psi$.
	\end{thm}
	\begin{proof}
		Let $\Psi \in \mathcal{C}^{+}(N_{\mathbb{R}})$. By Lemma~\ref{density-2}, it is enough to construct a sequence $ \lbrace \Psi_n  \rbrace_{n \in \mathbb{N}}$ in $\mathcal{SF}^{+}(N_{\mathbb{R}}, \mathbb{R})$ converging to $\Psi$. By Corollary~\ref{convergenceconicconcave}, it is enough to show that the constructed sequence converges pointwise for each $x \in N_{\mathbb{Q}}$. Furthermore, by the conical property of $\Psi$, it is enough to prove convergence for the ray generators $x \in N^{\textup{pr}}$.
		
		Start with any complete fan $\Sigma$ in $N_{\mathbb{R}}$. Denote by $v_{\rho}$ the ray generator of $\rho \in \Sigma (1)$. Consider the countably infinite set $S \coloneqq N^{\textup{pr}} \setminus \lbrace v_{\rho} \, | \, \rho \in \Sigma(1) \rbrace$ and choose a bijection $s \colon  \mathbb{N} \rightarrow S$. Define the sequence of nested sets $\lbrace S_n \rbrace_{n \in \mathbb{N}}$, where
		\begin{displaymath}
			S_n \coloneqq \lbrace v_{\rho} \, | \, \rho \in \Sigma(1) \rbrace \cup s( \lbrace 0,...,n \rbrace ).
		\end{displaymath}
		For each $n \in \mathbb{N}$, let $\Psi_n  \colon  N_{\mathbb{R}} \rightarrow \mathbb{R}$ be the function given by
		\begin{displaymath}
			\Psi_n (x) = \sup \left\lbrace \left. \sum_{v \in S_n} \lambda_v \Psi (v) \, \right\vert \, \lambda_v \geq 0, \, x = \sum_{v \in S_n} \lambda_v \, v  \right \rbrace.
		\end{displaymath}
		By construction, the hypograph of $\Psi_n$ is the convex hull $C_n$ of the rays generated by the vectors $(v,\Psi (v))$, $v \in S_n$, and $(0,-1)$. Denote by $p_1 \colon  N_{\mathbb{R}} \times \mathbb{R} \rightarrow N_{\mathbb{R}}$ the first projection. Then, $\Psi_n$ is a concave support function on the complete fan $\Sigma_n \coloneqq \lbrace p_1 (F) \, | \, F \textup{ proper face of }C_n \rbrace$. The fan $\Sigma_n$ is rational; each cone $p(F)$ is generated by elements of $S_n$. We claim that $\lbrace \Psi_n \rbrace_{n \in \mathbb{N}}$ is the desired sequence. We first prove the monotonicity. The sets $S_n$ are nested, and $\Psi$ is closed and concave. Then, for each $n \leq m$ we have $C_n \subset C_m \subset \textup{hyp}(\Psi)$. It follows that $\Psi_n (x) \leq \Psi_m (x) \leq \Psi(x)$ for all $x \in N_{\mathbb{R}}$. Now, let $v \in S_n$.  By definition of $\Psi_n$, we have $\Psi_n (v) \geq \Psi (v)$. It follows that $\Psi_n (v) = \Psi (v)$ for all $v \in S_n$. The sets $S_n$ are nested, and their union is $N^{\textup{pr}}$. Therefore, for each $v \in N^{\textup{pr}}$ the sequence $\lbrace \Psi_n (v) \rbrace_{n \in \mathbb{N}}$ is eventually constant with limit $\Psi (v)$. The claim follows.
	\end{proof}
	
	\paragraph*{\textbf{Density in} $\mathcal{G}^{+}(N_{\mathbb{R}})$.} Now, we show that the set $\mathcal{PA}^{+}(N_{\mathbb{R}}, \mathbb{Q})$ of concave rational piecewise affine functions is dense in the cone of globally Lipschitz concave functions $\mathcal{G}^{+}(N_{\mathbb{R}})$, equipped with the $\mathcal{G}$-norm. Before proving this, we introduce the following notation and show some auxiliary results.
	\begin{dfn}\label{G-min-rec}
		Let $\Sigma$ be a complete fan in $N_{\mathbb{R}}$. Denote by $\mathcal{G}_{\textup{rbd}}(\Sigma, \mathbb{Q})$ the space consisting of functions $f \in \mathcal{G}(N_{\mathbb{R}})$ such that $f$ is relatively bounded and $\textup{rec}(f) \in \mathcal{SF}(\Sigma, \mathbb{Q})$. Then, we write $\mathcal{G}^{+}_{\textup{rbd}}(\Sigma, \mathbb{Q})$ for its subset of concave functions.
	\end{dfn}
	\begin{lem}\label{density-polytope}
		Consider a globally Lipschitz concave relatively bounded function $f \in \mathcal{G}^{+}_{\textup{rbd}}(N_{\mathbb{R}})$ (\ref{min-sing-NR}) and write $\Psi = \textup{rec}(f)$ for its recession function. Assume further that $\Psi \in \mathcal{SF}^{+}(N_{\mathbb{R}}, \mathbb{Q})$. Then, there exist an increasing sequence $\lbrace f_n \rbrace_{n}$ of concave piecewise affine rational functions on $N_{\mathbb{R}}$ with recession function $\textup{rec}(f_n) = \Psi$ converging uniformly to $f$.
	\end{lem}
	\begin{proof}
		This is implied by Propositions~2.5.23~and~2.5.24~in~\cite{BPS}.
	\end{proof}
	The next step is to remove the polyhedral condition $\Psi \in \mathcal{SF}^{+}(N_{\mathbb{R}}, \mathbb{Q})$ in Lemma~\ref{density-polytope}.
	\begin{lem}\label{mono-approx-3}
		Consider a function $f \in \mathcal{G}^{+}_{\textup{rbd}}(N_{\mathbb{R}})$ and write $\Psi$ for its recession function. Let $\lbrace \Psi_n \rbrace_{n \in \mathbb{N}}$ be an increasing sequence in $\mathcal{C}^{+}(N_{\mathbb{R}})$ converging to $\Psi$ in the $\mathcal{C}$-norm. Then, there exists an increasing sequence $\lbrace f_n \rbrace_{n \in \mathbb{N}}$ in $\mathcal{G}^{+}_{\textup{rbd}}(N_{\mathbb{R}})$ and a positive constant $C$ such that:
		\begin{enumerate}[label=(\roman*)]
			\item For each $n$, $\textup{rec}(f_n) = \Psi_n$ and the function $| f_n - \Psi_n |$ is bounded by $C$.
			\item The sequence $\lbrace f_n \rbrace_{n \in \mathbb{N}}$ converges to $f$ in the $\mathcal{G}$-norm.
		\end{enumerate}
	\end{lem}
	\begin{proof}
		Define the real number $C \coloneqq \sup | f- \Psi |$. For each $n \in \mathbb{N}$, let $f_n$ be the function given by
		\begin{displaymath}
			f_n (x) \coloneqq \min \lbrace f (x) , \, \Psi_n (x) + C \rbrace.
		\end{displaymath}
		We claim that the sequence $\lbrace f_n \rbrace_{n \in \mathbb{N}}$ satisfies the desired properties. Indeed, $f_n$ is the pointwise minimum of two closed concave functions; hence, it is closed concave. The sequence $\lbrace \Psi_n  \rbrace_{n \in \mathbb{N}}$ is increasing, therefore $\lbrace f_n  \rbrace_{n \in \mathbb{N}}$ is increasing. Fix $x \in N_{\mathbb{R}}$, then
		\begin{align*}
			\lim_{n \rightarrow \infty} f_n (x) &= \lim_{n \rightarrow \infty} \min \left \lbrace f (x) , \, \Psi_n (x) + C \right \rbrace \\
			&= \min \left \lbrace f (x) , \,  \lim_{n \rightarrow \infty} \Psi_n (x) + C \right \rbrace \\
			&= \min \lbrace f (x) , \, \Psi (x) + C \rbrace = f(x).
		\end{align*}
		Now, we show that $|f_n - \Psi_n|$ is bounded by $C$. By definition, for each $x \in N_{\mathbb{R}}$ we have
		\begin{displaymath}
			f_n(x) - \Psi_n (x) =  \min \lbrace f (x) - \Psi_n (x), \, C \rbrace.
		\end{displaymath}
		The bound is trivial if $f_n(x) = \Psi_n(x)+C$. Otherwise, use the definition of $C$ and that $\lbrace \Psi_n  \rbrace_{n \in \mathbb{N}}$ is increasing to get
		\begin{displaymath}
			-C \leq f(x) - \Psi(x) \leq  f(x) - \Psi_n (x) = f_n(x) - \Psi_n (x)  \leq C.
		\end{displaymath}
		Finally, apply Lemma~\ref{mono-conv-G} to see that $\lbrace f_n \rbrace_{n \in \mathbb{N}}$ converges in the $\mathcal{G}$-norm.
	\end{proof}
	Combining the previous two lemmas yields the following corollary, which we then apply to establish the desired density result.
	\begin{cor}
		The cone $\mathcal{PA}^{+}(N_{\mathbb{R}}, \mathbb{Q})$ of concave rational piecewise affine functions on $N_{\mathbb{R}}$ is a dense subset of the cone $\mathcal{G}^{+}_{\textup{rbd}}(N_{\mathbb{R}})$ of globally Lipschitz concave relatively bounded functions on $N_{\mathbb{R}}$, equipped with the $\mathcal{G}$-norm. 
	\end{cor}
	\begin{proof}
		Let $f \in  \mathcal{G}^{+}_{\textup{rbd}}(N_{\mathbb{R}})$. By Theorem~\ref{density-2}, there exist an an increasing sequence $\lbrace \Psi_n  \rbrace_{n \in \mathbb{N}}$ in $\mathcal{SF}^{+}(N_{\mathbb{R}}, \mathbb{Q})$ converging to $\textup{rec}(f)$ in the $\mathcal{C}$-norm. Then, by Lemma~\ref{mono-approx-3}, there exist an increasing sequence $\lbrace f_n  \rbrace_{n \in \mathbb{N}}$ of globally Lipschitz concave functions converging to $f$ in the $\mathcal{G}$-norm and such that $\textup{rec}(f_n) = \Psi_n$. Finally, by Lemma~\ref{density-polytope}, we may choose $f_n$ to be a rational piecewise affine function.
	\end{proof}
	\begin{thm}\label{closure-P}
		The cone $\mathcal{PA}^{+}(N_{\mathbb{R}},\mathbb{Q})$ of concave rational piecewise affine functions on $N_{\mathbb{R}}$ is dense in the cone $\mathcal{G}^{+}(N_{\mathbb{R}})$ of globally Lipschitz concave functions on $N_{\mathbb{R}}$, equipped with the $\mathcal{G}$-norm. 
	\end{thm}
	\begin{proof}
		By the previous corollary $\mathcal{PA}^{+}(N_{\mathbb{R}},\mathbb{Q})$ is dense in $\mathcal{G}^{+}_{\textup{rbd}}(N_{\mathbb{R}})$. By Theorem~\ref{closedcone}, $\mathcal{G}^{+}_{\textup{rbd}}(N_{\mathbb{R}})$ is dense in $\mathcal{G}^{+}(N_{\mathbb{R}})$. We conclude that $\mathcal{PA}^{+}(N_{\mathbb{R}},\mathbb{Q})$ is dense in $\mathcal{G}^{+}(N_{\mathbb{R}})$.
	\end{proof}
	\paragraph*{\textbf{The closure in the }$\mathcal{C}$\textbf{-norm of} $\mathcal{PA}^{+}(N_{\mathbb{R}}, \mathbb{Q})$.} By Proposition~\ref{eta-2}, if a globally Lipschitz concave function $f \in \mathcal{G}^{+}(N_{\mathbb{R}})$ is the limit $\mathcal{C}$-norm of a sequence $\lbrace f_n  \rbrace_{n \in \mathbb{N}}$ of concave rational piecewise affine functions, then the number $f(0)$ is rational. The final result of this appendix is that this condition is sufficient.
	\begin{thm}\label{closure-P-C}
		Consider a globally Lipschitz concave function $f \in \mathcal{G}^{+}(N_{\mathbb{R}})$ such that $f(0) \in \mathbb{Q}$. Given an increasing sequence $\lbrace \Psi_n  \rbrace_{n \in \mathbb{N}}$ in $\mathcal{SF}^{+}(N_{\mathbb{R}}, \mathbb{Q})$ converging in the $\mathcal{C}$-norm to $\Psi \coloneqq \textup{rec}(f)$, there exist an increasing sequence $\lbrace f_n  \rbrace_{n \in \mathbb{N}}$ in $\mathcal{PA}^{+}(N_{\mathbb{R}}, \mathbb{Q})$ satisfying:
		\begin{enumerate}[label = (\roman*)]
			\item For each $n \in \mathbb{N}$, we have $f_n (0) = f(0)$.
			\item For each $n \in \mathbb{N}$, the recession function $\textup{rec}(f_n)$ is the function $\Psi_n$.
			\item For each $n \in \mathbb{N}$, the inequality $\| f - f_n \|_{\mathcal{C}} \leq 2 \, \varepsilon_n$ holds, where $\varepsilon_n \coloneqq  \| \textup{rec}(f) - \Psi_n \|_{\mathcal{C}}$.
		\end{enumerate}
		In particular, the sequence $\lbrace f_n \rbrace_{n \in \mathbb{N}}$ converges to $f$ in the $\mathcal{C}$-norm.
	\end{thm}
	\begin{proof}
		By Theorem~\ref{mono-approx-2}, there exists an increasing sequence $\lbrace \Psi_n  \rbrace_{n \in \mathbb{N}}$ in $\mathcal{SF}^{+}(N_{\mathbb{R}}, \mathbb{Q})$ converging in the $\mathcal{C}$-norm to $\Psi$. Then, we only need to show that such a sequence $\lbrace f_n \rbrace_{n \in \mathbb{N}}$ exists. We construct it inductively. Without loss of generality, we may assume that $f(0) = 0$. Proceeding as in the proof of Lemma~\ref{mono-conv-G}, there exist $R_n >1$ such that $f(x) \leq  \Psi_n (x) + 2 \, \varepsilon_n \cdot \|x \|$ for all $x$ with norm at least $R_n$. We may assume that $\lbrace R_n  \rbrace_{n \in \mathbb{N}}$ is an increasing, unbounded sequence of rational numbers. Any $f \in \mathcal{G}^{+}(N_{\mathbb{R}})$ such that $f(0)=0$ satisfies $f \geq \textup{rec}(f)$. This is immediate from concavity; if $\lambda >0$ and $x \neq 0$ then
		\begin{displaymath}
			f(x) \geq  f(0) + f( \lambda x )/\lambda = f( \lambda x )/ \lambda.
		\end{displaymath}
		By construction, $\textup{hyp}(\Psi_n)$ is a rational polyhedral cone satisfying $\textup{hyp}(\Psi_n) \subset \textup{hyp}(\Psi) \subset \textup{hyp}(f)$. By compactness, the continuous function $\overline{\eta}(f)$ is uniformly continuous on $[0, R_n] \times \mathcal{S}^{d-1}$. Then, for each $n$ there exist  $\delta_n >0$ such that for all $z,w \in [0, R_n] \times \mathcal{S}^{d-1}$ at distance at most $2 \delta_n$ we have $|\overline{\eta}(f) (z) - \overline{\eta}(f) (w)| \leq \varepsilon_n$. Without loss of generality, we may assume that $\lbrace \delta_n  \rbrace_{n \in \mathbb{N}}$ is a decreasing sequence of rational numbers strictly smaller than $1$, with limit $0$.
		
		Following the conventions in Remark~\ref{conventions}, the set $[ \delta_n, R_n] \times \mathcal{S}^{d-1}$ identifies with the $d$-annulus
		\begin{displaymath}
			\overline{\textup{B}}(0, R_n) \setminus \overline{\textup{B}}(0, \delta_n) = [- R_n, R_n]^d \setminus [- \delta_n , \delta_n ]^d.
		\end{displaymath}
		For each $n$, let $\Pi_n$ be a rational simplicial complex with support $|\Pi_n |= [ \delta_n  , R_n] \times \mathcal{S}^{d-1}$ satisfying the following properties:
		\begin{enumerate}[label=(\arabic*)]
			\item The complex obtained by intersecting $\Pi_n$ with $[\delta_{n-1} , R_{n-1} ] \times \mathcal{S}^{d-1}$ is a subcomplex of $\Pi_n$ which refines $\Pi_{n-1}$.
			\item The diameter of each simplex in $\Pi_n$ is smaller than $\delta_n$.
		\end{enumerate}
		The sequence $\lbrace \Pi_n  \rbrace_{n \in \mathbb{N}}$ always exists; start with a rational simplicial complex with support $ [ \delta_1, R_1] \times \mathcal{S}^{d-1}$ and do successive barycentric subdivisions until the diameter condition is satisfied. Then, the sequence $\lbrace \Pi_n  \rbrace_{n \in \mathbb{N}}$ is constructed inductively. If $\Pi_n$ is already constructed, extend $\Pi_n$ to a rational simplicial complex $\Pi_{n+1}^{\prime}$ with support $ [ \delta_{n+1} , R_{n+1}] \times \mathcal{S}^{d-1}$. The simplicial complex $\Pi_{n+1}$ is obtained by successive barycentric subdivisions until the diameter condition is satisfied.
		
		Now, for each vertex $x \in \Pi_n (0)$, we inductively choose a rational number $r_{x,n} \in \mathbb{Q}$ such that:
		\begin{enumerate}[label=(\arabic*)]
			\item If $x \in  \Pi_{n-1} (0)$, then $r_{x,n} \geq r_{x,n-1}$.
			\item The inequality $\max \lbrace \Psi_n(x), \, f(x) - \varepsilon_n \cdot \| x \| \rbrace  \leq r_{x,n} \leq f(x)$ holds.
		\end{enumerate}
		For each $n$, define a rational polytope $C_n$ and a rational polyhedral set $H_n$ by
		\begin{displaymath}
			C_n \coloneqq \textup{Conv}( \lbrace (x , r_{x,n}) \, \vert \, x \in \Pi_n (0) \rbrace \cup \lbrace 0 \rbrace), \quad H_n \coloneqq C_n + \textup{hyp}(\Psi_n).
		\end{displaymath}
		Since $f$ is closed and concave, $C_n \subset \textup{hyp}(f)$. Then, theorem 19.6~of~\cite{Roc} shows that
		\begin{displaymath}
			H_n = \textup{cl}(\textup{Conv}(C_n \cup \textup{hyp}(\Psi_n))) \subset \textup{hyp}(f).
		\end{displaymath}
		Let $f_n$ be the unique closed concave function such that $\textup{hyp}(f_n) = H_n$. By construction, $\lbrace f_n  \rbrace_{n \in \mathbb{N}}$ is an increasing sequence in $\mathcal{PA}^{+}_{0} (N_{\mathbb{R}}, \mathbb{Q})$, bounded from above by $f$, and such that $\textup{rec}(f_n) = \Psi_n$.
		
		We claim that $\| f_n - f \|_{\mathcal{C}} \leq  2 \, \varepsilon_n$. Note that $f_n (0) = 0$, hence $\Psi_n \leq f_n$. Let $x \in N_{\mathbb{R}}$ such that $\| x \| \geq R_n$, then
		\begin{displaymath}
			f_n (x) \leq f(x) \leq \Psi_n (x) + 2 \, \varepsilon_n \cdot \|x \| \leq f_n (x) +2 \, \varepsilon_n \cdot \|x \| .
		\end{displaymath}
		Let $x \in  [ \delta_n  , R_n] \times \mathcal{S}^{d-1}$ and choose a simplex $S_x \in \Pi_n (0)$ containing $x$. Let $S_x (0)$ be the set of vertices of $S_x$ and write $x$ as a convex combination $x = \sum_{z \in S_x (0)} \lambda_z \cdot z$. By concavity of $f_n$, we know that
		\begin{displaymath}
			f_n (x) \geq \sum_{z \in S_x (0)} \lambda_z \cdot f_n(z) \geq  \sum_{z \in S_x (0)} \lambda_z \left( f(z) - \varepsilon_n \cdot \| z \| \right) = \sum_{z \in S_x (0)} \lambda_z \cdot \| z \| (  f(z)/\| z\| - \varepsilon_n ).
		\end{displaymath}
		The diameter of $S_x$ is less than $\delta_n$, hence $\varepsilon_n \geq \overline{\eta}(f) (x) - \overline{\eta}(f) (z) $ for all $z \in S_x$. After substitution and rearranging, we get $f(z)/ \|z \| \geq f(x)/\|x \| - \varepsilon_n$. We substitute this estimate in the above equation to obtain
		\begin{displaymath}
			f_n (x) \geq  \sum_{z \in S_x (0)} \lambda_z \cdot \| z \| ( f(z)/ \| z\| -  \varepsilon_n ) \geq  \sum_{z \in S_x (0)} \| \lambda_z \cdot z \| ( f(x)/\|x \|  - 2 \,\varepsilon_n ).
		\end{displaymath}
		By the triangle inequality, $ \sum_{z \in S_x (0)} \| \lambda_z \cdot z \| \geq \| x \|$. After substitution in the above inequality,
		\begin{displaymath}
			f_n (x) \geq   ( f(x)/ \|x \| - 2 \,\varepsilon_n )  \sum_{z \in S_x (0)} \| \lambda_z \cdot z \| \geq ( f(x)/\|x \|  - 2 \,\varepsilon_n ) \|x \| = f(x) - 2 \,\varepsilon_n \cdot \|x\|.
		\end{displaymath}
		If $x \in  \overline{\textup{B}}(0, \delta_n)$, there exist a simplex $S_x \in \Pi_n$ such that $S_x \subset \lbrace \delta_n \rbrace \times \mathcal{S}^{d-1}$ and the convex hull $\textup{Conv}(S_x \cup \lbrace 0 \rbrace )$ is contained in $\overline{\textup{B}}(0, \delta_n)$. Then, we may assume $x \neq 0$ and write $x$ as a convex combination of the form $x = \lambda_0 \cdot 0 +  \sum_{z \in S_x (0)} \lambda_z \cdot z$. Note that $f_n (0) = f(0) = 0$. By the concavity of $f_n$ and repeating the argument above, we obtain
		\begin{displaymath}
			f_n (x)  \geq \sum_{z \in S_x (0)} \lambda_z \cdot \| z \| (  f(z) / \| z\| - \varepsilon_n ).
		\end{displaymath}
		Since $x, z \neq 0$ and the diameter of the ball $\overline{\textup{B}}(0, \delta_n)$ is $2 \, \delta_n$, we can repeat the above computation to obtain the bound $	f_n (x)  \geq f(x) - 2 \,\varepsilon_n \cdot \|x\|$.
	\end{proof}

\end{document}